\documentclass[reqno,10pt]{amsart}
\usepackage{a4wide}
\usepackage{mathbbol}
\usepackage{mathrsfs}
\usepackage{amssymb,mathrsfs,graphicx,extpfeil}
\usepackage{epsfig}
\usepackage{indentfirst, latexsym, amssymb, enumerate,amsmath,graphicx}
\usepackage{float}
\usepackage{colortbl}
\usepackage{epsfig,subfigure}
\allowdisplaybreaks[4]
\usepackage{amsmath}
\title[Global Hilbert expansion for some non-relativistic kinetic equations]{Global Hilbert expansion for some non-relativistic kinetic equations }

\author[Y. J. Lei]{Yuanjie Lei}
\address[Y. J. Lei]{{\newline School of Mathematics and Statistics, Huazhong University of Science and Technology, Wuhan 430074, China}}
\email{leiyuanjie@hust.edu.cn}

\author[S. Q. Liu]{Shuangqian Liu}
\address[S. Q. Liu]{{\newline School of Mathematics and Statistics, and Key Laboratory of Nonlinear Analysis $\&$ Applications (Ministry of Education), Central China Normal University, Wuhan 430079, P. R. China}}
\email{sqliu@ccnu.edu.cn}

\author[Q. H. Xiao]{Qinghua Xiao}
\address[Q. H. Xiao]{\newline Innovation Academy for Precision Measurement Science and Technology, Chinese Academy of Sciences, Wuhan 430071, China}
\email{xiaoqh@apm.ac.cn}

\author[H. J. Zhao]{Huijiang Zhao}
\address[H. J. Zhao]{{\newline School of Mathematics and Statistics, and Computational
Science Hubei Key Laboratory, Wuhan University, Wuhan 430072, China}}
\email{hhjjzhao@whu.edu.cn}

\newtheorem{theorem}{Theorem}[section]
\newtheorem{lemma}{Lemma}[section]
\newtheorem{corollary}{Corollary}[section]
\newtheorem{proposition}{Proposition}[section]

\newtheorem{remark}{Remark}[section]

\def\charf {\mbox{{\text 1}\kern-.30em {\text l}}}

\def\nablax{\nabla_x}

\newcommand{\dis}{\displaystyle}

\newcommand{\R}{\mathbb{R}}

\newcommand{\FM}{\mathbf{M}}
\newcommand{\FP}{\mathbf{P}}

\newcommand{\FI}{\mathbf{I}}

\newcommand{\FZ}{\mathbf{Z}}

\newcommand{\CA}{\mathcal{A}}

\newcommand{\CC}{\mathcal{C}}

\newcommand{\CK}{\mathcal{K}}
\newcommand{\CL}{\mathcal{L}}

\newcommand{\CQ}{\mathcal{Q}}

\newcommand{\CU}{\mathcal{U}}

\newcommand{\na}{\nabla}

\newcommand{\al}{\alpha}
\newcommand{\bet}{\beta}
\newcommand{\ga}{\gamma}

\newcommand{\de}{\delta}
\newcommand{\si}{\sigma}
\newcommand{\pa}{\partial}

\newcommand{\eps}{\epsilon}

\newcommand{\ta}{\theta}

\newcommand{\vps}{\varepsilon}

\newcommand{\Ga}{\Gamma}

\newcommand{\lag}{\langle}
\newcommand{\rag}{\rangle}

\numberwithin{equation}{section}

\begin{document}

\date{\today}


\thanks{\textbf{Acknowledgment.} The research of Yuanjie Lei was supported by the National Natural Science Foundation of China grants 11971187 and 12171176, the research of Shuangqian Liu was supported by the National Natural Science Foundation of China grant 12325107, the research of Qinghua Xiao was supported by the National Natural Science Foundation of China grants 11871469 and 12271506, the research of Huijiang Zhao was supported by the National Natural Science Foundation of China grants 11731008 and 12221001. This work was also partially supported by the Fundamental Research Funds for the Central Universities.}


\begin{abstract}

The Vlasov-Maxwell-Landau (VML) system and the Vlasov-Maxwell-Boltzmann (VMB) system are fundamental models in dilute collisional plasmas. In this paper, we are concerned with the hydrodynamic limits of both the VML and the non-cutoff VMB systems in the entire space. Our primary objective is to rigorously prove that, within the framework of Hilbert expansion, the unique classical solution of the VML or non-cutoff VMB system converges globally over time to the smooth global solution of the Euler-Maxwell system as the Knudsen number approaches zero.

The core of our analysis hinges on deriving novel interplay energy estimates for the solutions of these two systems, concerning both a local Maxwellian and a global Maxwellian, respectively. Our findings address a problem in the hydrodynamic limit for Landau-type equations and non-cutoff Boltzmann-type equations with a magnetic field. Furthermore, the approach developed in this paper can be seamlessly extended to assess the validity of the Hilbert expansion for other types of kinetic equations.

\end{abstract}

\maketitle
\thispagestyle{empty}

\tableofcontents

\section{Introduction}

\setcounter{equation}{0}
The smooth transition from kinetic equations describing perfect gases or particles to hydrodynamics is closely linked to Hilbert's sixth problem. Hilbert himself introduced the renowned Hilbert expansion as the inaugural example of this program. In this expansion, the Knudsen number $\varepsilon$ plays a central role, as detailed in references such as \cite{Hilbert, Saint-Raymond-2009}. Rigorously substantiating the legitimacy of this expansion for kinetic equations has remained a  challenging problem since its inception.

In this article, at the kinetic level, we investigate the rescaled Vlasov-Maxwell-Landau (VML) and Vlasov-Maxwell-Boltzmann (VMB) systems. These systems represent the fundamental models governing the dynamics of dilute charged particles,
\begin{align}\label{main1}
\begin{aligned}
& \partial_t F^{\varepsilon} + v\cdot \nabla_x F^{\varepsilon}- e_-\Big(E^{\epsilon}+v\times B^{\varepsilon} \Big)\cdot\nabla_v F^{\varepsilon} = \frac{1}{\varepsilon}\mathcal{C}(F^{\varepsilon},F^{\varepsilon}),\\
 & \partial_tE^{\varepsilon}-  c\nabla_x \times B^{\varepsilon} =4\pi \int_{\mathbb R^3} v F^{\varepsilon} dv, \\
 &\partial_tB^{\varepsilon}+ c\nabla_x \times E^{\varepsilon}=0,\\
& \nabla_x\cdot E^{\varepsilon}=4\pi e_-\Big(1 -\int_{\mathbb R^3}  F^{\varepsilon} dv\Big), \qquad \nabla_x\cdot B^{\varepsilon}=0.
\end{aligned}
\end{align}
Here, $\varepsilon>0$ represents the Knudsen number, which characterizes the ratio of the mean free path to the characteristic length. The operators $\nabla_x=(\partial_{x_1}, \partial_{x_2}, \partial_{x_3})$ and $\nabla_v=(\partial_{v_1}, \partial_{v_2}, \partial_{v_3})$ denote spatial and velocity gradients, respectively. The function $F^{\varepsilon}= F^{\varepsilon}(t, x, v)\geq0$ represents the number density function for electrons at time $t\geq0$, position $x=(x_1, x_2, x_3) \in \mathbb R^3$, and velocity $v=(v_1, v_2, v_3)\in \mathbb R^3$. In this context, $-e_-$ denotes the charge of the electrons, and $c$ stands for the speed of light. Furthermore, we normalize both constants $e_-$ and $c$ to one for simplicity in the rest of this paper.

Regarding the bilinear collision operator $\mathcal{C}(\cdot,\cdot)$ in equation \eqref{main1}, it can be associated with either the Landau collision operator or the Boltzmann collision operator:

\begin{itemize}
\item  For the Landau collision operator, denoted as $\CC(\cdot,\cdot)$, it is defined as follows:
\begin{eqnarray*}
\CC(G,H)&=&\nabla_v\cdot\left\{\int_{{\R}^{3}}\phi(v-v')\left[G(v')\nabla_vH(v)-H(v)\nabla_{v'}G(v')\right]dv'\right\}\nonumber\\
&=&\sum\limits_{i,j=1}^{3}\partial_{v_i}\int_{{\R}^{3}}\phi^{ij}(v-v')\left[G(v')\partial_{v_j}H(v)-H(v)\partial_{v'_j}G(v')\right]dv',
\end{eqnarray*}
where $\phi(v)=\left(\phi^{ij}(v)\right)_{3\times 3}$ is a non-negative matrix with components $\phi^{ij}(v)$ given by
\begin{equation}\label{cker}
\phi^{ij}(v)=\left\{\delta_{ij}-\frac{v_iv_j}{|v|^2}\right\}|v|^{\ga+2}, \qquad \gamma\geq-3.
\end{equation}
Here, $\delta_{ij}$ represents the Kronecker delta. We are particularly interested in the case where $\gamma=-3$, which corresponds to the original Landau collision operator with Coulombic interactions.

\item
For the Boltzmann collision operator, also denoted as $\CC(\cdot,\cdot)$, it is defined as follows:
\begin{equation*}
  \mathcal{C}(G_1,G_2)(v)
  =\int_{\mathbb{R}^3\times \mathbb{S}^2}B(v-v_\ast, \sigma)\left\{G_1(v_\ast')G_2(v')-G_1(v_\ast)G_2(v)\right\}
  dv_\ast d\sigma.
\end{equation*}
Here, $(v,v_\ast)$ and $(v',v_\ast')$ represent the pre-post collision velocities of the particles, and they satisfy
 \begin{align}
v'=\frac{v+v_\ast}{2}+\frac{|v-v_\ast|}{2}\si,\ v_\ast'=\frac{v+v_\ast}{2}-\frac{|v-v_\ast|}{2}\si.\notag
\end{align}
The collision kernel $B(v-v_\ast, \sigma)$ is a non-negative function of the relative velocity $|v-v_\ast|$ and the deviation angle $\theta$:
 \begin{equation*}
B(v-v_\ast, \sigma)=C_v|v-v_\ast|^{\gamma}b(\cos\theta),\qquad \gamma>-3,
\end{equation*}
where $C_v>0$, and $\cos \theta=[(v-v_\ast)\cdot \sigma]/|v-v_\ast|$. It's worth noting that through standard symmetrization techniques, we can assume that $B(v-v_\ast, \sigma)$ is supported on the set $0\leq \theta\leq \frac{\pi}{2}$ with $\cos\theta\geq0$.
For the cutoff case, the function $b(\cos\theta)$ satisfies the following inequality:
$$
0\leq b(\cos\theta)\leq C |\cos \theta|, \qquad C>0.
$$
In contrast, for the non-cutoff case, $b(\cos\theta)$ obeys the following bounds:
\begin{equation*}
\frac{1}{C_b \theta^{1+2s}}\leq \sin\theta b(\cos\theta)\leq \frac{C_b}{ \theta^{1+2s}}, \qquad \forall s\in (0,1), \quad\theta\in\left(0,\frac{\pi}{2}\right],\quad C_b>1.
\end{equation*}

In this paper, our focus is specifically on the non-cutoff case, where $\gamma$ lies within the range:
\begin{equation*}
\max\Big\{-3, -\frac{3}{2}-2s\Big\}<\gamma<0.
\end{equation*}
\end{itemize}

The objective of this paper is to provide a rigorous justification for the global validity of the Hilbert expansion in both the VML and the non-cutoff VMB systems, as given in equation \eqref{main1}, within the entire space. This justification is carried out under the initial data conditions:
\begin{align}\label{VL-id}
F^\vps(0,x,v)=F_0^\vps(x,v),\  E^\vps(0,x)=E^\vps_0(x),\ B^\vps(0,x)=B^\vps_0(x).
\end{align}
\subsection{Hilbert expansions}

In this subsection, motivated by Caflisch's groundbreaking work \cite{Caflisch-CPAM-1980}, we outline the standard Hilbert expansion procedures for the system \eqref{main1}.


\subsubsection{Hilbert expansion of the system \eqref{main1}}
Let $k\geq 2$, by plugging the Hilbert expansion
\begin{align}\label{expan}
\begin{aligned}
&F^{\varepsilon}=\sum_{n=0}^{2k-1}\varepsilon^nF_n+\varepsilon^kF^{\varepsilon}_R,\quad
E^{\varepsilon}=E+\sum_{n=1}^{2k-1}\varepsilon^n E_n+\varepsilon^kE^{\varepsilon}_R,\quad
B^{\varepsilon}=B+ \sum_{n=1}^{2k-1}\varepsilon^n B_n+\varepsilon^kB^{\varepsilon}_R
\end{aligned}
\end{align}
into \eqref{main1},
we then deduce that the coefficients $[F_n(t, x, v),$ $E_n(t, x), B_n(t, x)]$ $(0\leq n\leq 2k-1)$ and  the remainder $[F_R^{\varepsilon}(t,x,v), E_R^{\varepsilon}(t,x),$ $B_R^{\varepsilon}(t,x)]$
satisfy
\begin{align}
\frac{1}{\varepsilon}:&\quad \mathcal{C}(F_0,F_0)=0,\nonumber
\end{align}
\begin{align}\label{expan2}
\varepsilon^0:\quad
\left\{\begin{array}{rll}&\partial_tF_0+v\cdot\nabla_x F_0-\Big(E+v \times B \Big)\cdot\nabla_vF_0=\mathcal{C}(F_1,F_0)+\mathcal{C}(F_0,F_1),\\
 & \partial_tE-  \nabla_x \times B =4\pi \dis{\int_{\mathbb R^3}} vF_0 dv,\\
 &\partial_tB+ \nabla_x \times E=0,\\
&\nabla_x\cdot E=4\pi \Big(1 -\dis{\int_{\mathbb R^3}}  F_0 dv\Big), \qquad \nabla_x\cdot B=0,
\end{array}\right.
\end{align}
$$\cdots\cdots\cdots\cdots\cdots$$
\begin{align}
\varepsilon^n:&\quad \partial_tF_n+v\cdot \nabla_xF_n-\Big(E_n+v \times B_n \Big)\cdot\nabla_pF_0-\Big(E+v \times B \Big)\cdot\nabla_vF_n\nonumber\\
&\qquad=\sum_{\substack{i+j=n+1\\i,j\geq0}}\mathcal{C}(F_i,F_j)+\sum_{\substack{i+j=n\\i,j\geq1}}\Big(E_i+v \times B_i \Big)\cdot\nabla_vF_j, \label{Fn}\\
&\quad\partial_tE_n-\nabla_x \times B_n=4\pi \int_{\mathbb R^3} v F_n dv, \nonumber\\
 &\quad \partial_t B_n+ \nabla_x \times E_n=0,\nonumber\\
&\quad \nabla_x\cdot E_n=-4\pi \int_{\mathbb R^3}  F_n dv, \qquad \nabla_x\cdot  B_n=0,\nonumber\\
&\hspace{4.6cm}\cdots\cdots\cdots\cdots\cdots\nonumber\\
\varepsilon^{2k-1}:& \quad\partial_tF_{2k-1}+v\cdot\nabla_x F_{2k-1}-\Big(E_{2k-1}+v \times B_{2k-1} \Big)\cdot\nabla_vF_0-\Big(E+v \times B \Big)\cdot\nabla_vF_{2k-1}\nonumber\\
& \quad=\sum_{\substack{i+j=2k\\i,j\geq1}}\mathcal{C}(F_i,F_j)+\sum_{\substack{i+j=2k-1\\i,j\geq1}}\Big(E_i+v \times B_i \Big)\cdot\nabla_vF_j, \nonumber\\
&\quad\partial_tE_{2k-1}-\nabla_x \times B_{2k-1}=4\pi \int_{\mathbb R^3} v F_{2k-1} dv, \nonumber\\
 &\quad\partial_t B_{2k-1}+ \nabla_x \times E_{2k-1}=0,\nonumber\\
&\quad \nabla_x\cdot E_{2k-1}=-4\pi \int_{\mathbb R^3}  F_{2k-1} dv, \qquad \nabla_x\cdot  B_{2k-1}=0,\nonumber
\end{align}
and
\begin{align}
&\partial_tF_R^{\varepsilon}+v\cdot\nabla_x F_R^{\varepsilon}-\Big(E_R^{\varepsilon}+v \times B_R^{\varepsilon} \Big)\cdot\nabla_vF_0\nonumber\\
 &\qquad-\Big(E+v \times B \Big)\cdot\nabla_vF_R^{\varepsilon}-\frac{1}{\varepsilon}[\mathcal{C}(F_R^{\varepsilon},F_0)+\mathcal{C}(F_0,F_R^{\varepsilon})]\nonumber\\
&\quad=\varepsilon^{k-1}\mathcal{C}(F_R^{\varepsilon},F_R^{\varepsilon})+\sum_{n=1}^{2k-1}\varepsilon^{n-1}[\mathcal{C}(F_i, F_R^{\varepsilon})+\mathcal{C}(F_R^{\varepsilon}, F_i)]+\varepsilon^k\Big(E_R^{\varepsilon}+v \times B_R^{\varepsilon}\Big)\cdot\nabla_vF_R^{\varepsilon}\nonumber\\
 &\qquad+\sum_{n=1}^{2k-1}\varepsilon^n\Big[\Big(E_n+v \times B_n \Big)\cdot\nabla_vF_R^{\varepsilon}+\Big(E_R^{\varepsilon}+v \times B_R^{\varepsilon} \Big)\cdot\nabla_vF_i\Big]+\varepsilon^{k}\CQ,
\nonumber\\
&\partial_tE_R^{\varepsilon}-\nabla_x \times B_R^{\varepsilon}=4\pi \int_{\mathbb R^3} v F_R^{\varepsilon} dv, \label{remain VML}\\
 &\partial_t B_R^{\varepsilon}+ \nabla_x \times E_R^{\varepsilon}=0,\nonumber\\
&\nabla_x\cdot E_R^{\varepsilon}=-4\pi \int_{\mathbb R^3}  F_R^{\varepsilon} dv, \qquad \nabla_x\cdot  B_R^{\varepsilon}=0.\nonumber
\end{align}
Here
\begin{align*}
\begin{aligned}
\CQ=\sum_{\substack{i+j\geq 2k+1\\2\leq i,j\leq2k-1}}\varepsilon^{i+j-2k-1}Q(F_i,F_j)
+\sum_{\substack{i+j\geq 2k\\1\leq i,j\leq2k-1}}\varepsilon^{i+j-2k}\Big(E_i+v \times B_i \Big)\cdot\nabla_vF_j.
\end{aligned}
\end{align*}

\subsubsection{Determination of the coefficients}
Let's briefly outline how to determine the coefficients $[F_0(t, x, v), E(t, x), B(t, x)]$ and $[F_n(t, x, v), E_n(t, x), B_n(t, x)]$ for $1 \leq n \leq 2k-1$.

Starting with \eqref{expan2}$_1$, we can deduce that the zeroth-order coefficient is given by

\begin{equation*}
	F_0(t,x,v)=\FM=\FM_{[\rho,u,T]}= \frac{\rho(t,x) }{(2\pi RT(t,x))^{3/2}} \exp\left(-\frac{|v-u(t,x)|^2}{2RT(t,x)}\right),
\end{equation*}
where $\rho(t,x), u(t,x)$, and $T(t,x)$ represent the fluid density, bulk velocity, and temperature, respectively. Furthermore, $[\rho, u, T, E, B]$ must satisfy the following compressible Euler-Maxwell system:
\begin{eqnarray}\label{EM}
\partial_t\rho + \nabla_x\cdot(\rho u) &=&0,\nonumber\\
\partial_t[\rho u] + \nabla_x\cdot\left(\rho u\otimes u\right)+\nablax (R\rho T)+\rho \big(E+u\times B\big)&=&0,\nonumber\\
 \partial_t\Big[\rho \Big(e+\frac12|u|^2\Big)\Big] +  \nabla_x\cdot\Big[\rho u\Big(e+\frac12|u|^2\Big)\Big]+\nabla_x\big(R\rho Tu\big)+\rho u\cdot E&=&0,\nonumber\\
 \partial_tE-  \nabla_x \times B &=&4\pi \rho u, \\
 \partial_tB+ \nabla_x \times E&=&0,\nonumber\\
\nabla_x\cdot E&=&4\pi \big(1 -\rho\big), \nonumber\\
\nabla_x\cdot B&=&0.\nonumber
\end{eqnarray}
Here the pressure $p$ and the internal energy $e$ satisfy the following constitutive relations take the form:
 $$p=R\rho T, \qquad e=\frac{3}{2}RT$$
 with $R>0$ being the gas constant. Without loss of generality, we can assume $R=\frac23$.

Given the initial datum
\begin{align}\label{em-id}
[\rho,u,T,E,B](0,x)=[\rho_{0},u_{0},T_{0},E_{0},B_{0}](x), \ T_0=\rho_0^{2/3},
\end{align}
it is shown in \cite[Theorem 1.3, p. 2359]{Ionescu-Pausader-JEMS-2014} that the Cauchy problem \eqref{EM} and \eqref{em-id} does admit a global smooth solution $[\rho(t,x), u(t,x), T(t,x), E(t,x), B(t,x)]$ in the whole space.
\begin{proposition}[Global solvability of the Euler-Maxwell system]\label{em-ex-lem}
Let $N_0>10^4$ be a sufficiently large integer and denote
$$
\left\|\left[\rho_0-1, u_0, T_0-1, E_0, B_0\right]\right\|_{H^{N_0+1}}+\left\|(\FI-\Delta)^{1/2}u_0\right\|_Z+\left\|(\FI-\Delta)^{1/2}E_0\right\|_Z=\epsilon_1,
$$
where the norm $\|\cdot\|_Z$ is defined by \cite[(2.18), p. 2371]{Ionescu-Pausader-JEMS-2014}.

Assume that
$$
\rho_0=-\na_x\cdot E_0,\ B_0=\na_x\times u_0,
$$
and $\epsilon_1>0$ is chosen suitably small, there exists a unique global smooth solution $[\rho(t,x), u(t,x), $ $T(t,x), E(t,x), B(t,x)]$ with $T(t,x)=\rho(t,x)^{2/3}$ to the Cauchy problem \eqref{EM} and \eqref{em-id} which satisfies
\begin{align}\label{em-decay}
&\sup_{t\in[0,\infty]}\left\|\left[\rho(t,x)-1, T(t,x)-1, u(t,x), E(t,x), B(t,x)\right]\right\|_{H^{N_0+1}}\\
&+\sup_{t\in[0,\infty]}\left[(1+t)^{p_0}\sup_{i\leq N_1}\left(\left\|\nabla_x^i[\rho(t,x)-1]\right\|_{\infty}
+\left\|\nabla_x^i[T(t,x)-1]\right\|_{\infty}\right)\right]\nonumber\\
&+\sup_{t\in[0,\infty]}\Big[(1+t)^{p_0}\sup_{i\leq N_1+1}
\left\|\left[\nabla_x^i u(t,x), \nabla_x^i E(t,x), \nabla_x^i B(t,x)\right]\right\|_{\infty}\Big]\leq C \epsilon_1,\nonumber
\end{align}
where $3\leq N_1\ll N_0$, $p_0>1$, $C$ is some positive constant independent of $\epsilon_1$.
\end{proposition}
Once  $[\rho(t,x), u(t,x), $ $T(t,x), E(t,x), B(t,x)]$ is constructed as in Proposition \ref{em-ex-lem}, we can then determine $[F_1(t,x,v),$ $E_1(t, x), B_1(t, x)]$ as follows. Firstly, by setting $F_1=\FM^{\frac{1}{2}}f_1$ and using the definition \eqref{fn-mac-def}, we write
	\begin{align*}
	{\bf P}_{\mathbf{M}}[f_1]=&\frac{\rho_1}{\sqrt{\rho}} \chi_0+\sum\limits_{i=1}^3\frac{1}{\sqrt{R\rho T}}u_1^i\chi_i+\frac{T_1}{\sqrt{6\rho}}\chi_4
\end{align*}
and
\begin{align}
F_1=&\Big[\frac{\rho_1}{\rho} \FM+\sum\limits_{i=1}^3\frac{v^i-u^i}{R\rho T}u_1^i\FM+\frac{T_1}{6\rho}\Big(\frac{|v-u|^2}{RT}-3\Big)\FM\Big]+\FM^{\frac{1}{2}}({\bf I-P}_{\mathbf{M}})[f_1].\label{F1}
	\end{align}
Then it is straightforward to see that the microscopic part of $f_{1}=\FM^{-\frac{1}{2}}F_{1}$ can be determined by
\begin{align*}
({\bf I-P}_{\mathbf{M}})[f_{1}]=&\mathcal{L}_{\FM}^{-1}\Big[-\FM^{-\frac{1}{2}}\Big(\partial_tF_{0}+v\cdot \nabla_xF_{0}-\Big(E+v \times B \Big)\cdot\nabla_vF_0\Big], \nonumber
\end{align*}
according to $\eqref{expan2}_1$, where the linearized collision operator $\mathcal{L}_{\FM}$ is defined in \eqref{lm-def}.

To further derive the macroscopic part of $F_{1}$, which is given in \eqref{F1}, by taking velocity moments $1, v, |v|^2$ of \eqref{Fn} with $n=1$, we get
\begin{align}\label{examp1}
		&\partial_t\rho_1+\na_x\cdot(\rho_1 u+u_1)=0,\notag\\
&\partial_tu_1+(u\cdot\nabla_x)u_1+(u_1\cdot\nabla_x)u+(\nabla_x\cdot u)u_1+\frac{1}{3}\nabla_x(3RT\rho_1+RTT_1)-\frac{R\rho_1}{\rho}\nabla_x(\rho T)\notag\\
&+\rho\big[E_1+\big(u\times B_1\big)\big]+u_1\times B=-\nabla_x\cdot\int_{\mathbb R^3}\Big[(v-u)\otimes(v-u)-\frac{|v-u|^2}{3}\FI\Big] F_1 dv,\notag\\
&RT\big[\partial_tT_1+u\cdot \nabla_x T_1+(\nabla_x\cdot u) T_1\big]+2RT\nabla_x\cdot u_1+3R\nabla_xT\cdot u_1\notag\\
&-\frac{2RT}{\rho}\nabla_x\rho\cdot u_n-2\big(u\times B\big)\cdot u_n\\
=&2u^i\partial_{x_j}\Big[\int_{\mathbb R^3}\Big[(v^i-u^i)(v^j-u^j)-\delta_{ij}\frac{|v-u|^2}{3}\Big] F_1 dv\Big]\notag\\
&-\nabla_x\cdot\int_{\mathbb R^3}(v-u)\big[|v-u|^2-5RT\big] F_1 dv\notag\\
&-2\partial_{x_j}\Big[u^i\int_{\mathbb R^3}\Big[(v^i-u^i)(v^j-u^j)-\delta_{ij}\frac{|v-u|^2}{3}\Big] F_1 dv\Big],\notag\\
&\pa_t E_1-\na_x\times B_1=4\pi\big(u_1+\rho_1 u\big),\notag\\
			&\pa_t B_1+\na_x\times E_1=0,\notag\\
			&\na_x\cdot E_1=-\rho_1,\ \na_x\cdot B_1=0,\notag
			\end{align}
where $\FI$ stands for the unit $3\times 3$ matrix. Fortunately, \eqref{examp1} is a linear symmetrizable hyperbolic system with respect to the variables
$[\rho_1,u_1,T_1,E_1,B_1]$ and its Cauchy problem  can be solved appropriately, cf. \cite{majda-book}.

We can then proceed with the above procedure step by step and ultimately determine all the coefficients $[F_n(t, x, v), E_n(t, x), B_n(t, x)]$ under suitable initial data $[\rho_n, u_n, T_n, E_n, B_n](0, x)$ for $1 \leq n \leq 2k-1$. Therefore, our main objective, in order to verify \eqref{expan}, is to solve the remainders $[F_R^{\varepsilon}, E_R^{\varepsilon}, B_R^{\varepsilon}]$.

\subsection{The remainders}
In this subsection, our aim is to explore the procedure for solving the remainder equations \eqref{remain VML}.

In order to solve $F_R^{\varepsilon}$, it is standard to introduce $f^{\varepsilon}$ as
\begin{equation*}
  F_R^{\varepsilon}=\FM^{\frac{1}{2}} f^{\varepsilon}.
\end{equation*}
Then we can rewrite \eqref{remain VML} as
\begin{align}
\partial_tf^{\varepsilon}&+v\cdot\nabla_xf^{\varepsilon}+\frac{\big(E_R^{\varepsilon}+v \times B_R^{\varepsilon} \big) }{ RT}\cdot \big(v-u\big)\mathbf{M}^{\frac{1}{2}}\nonumber\\
&+\Big(E+v \times B \Big)\cdot\frac{v-u }{ 2RT}f^{\varepsilon}-\Big(E+v \times B \Big)\cdot\nabla_vf^{\varepsilon}+\frac{\mathcal{L}_{\mathbf{M}}[f^{\varepsilon}]}{\varepsilon}\nonumber\\
=&-\mathbf{M}^{-\frac{1}{2}}f^{\varepsilon}\Big[\partial_t+v\cdot\nabla_x-\Big(E+v \times B \Big)\cdot\nabla_v\Big]\mathbf{M}^{\frac{1}{2}}+\varepsilon^{k-1}\Gamma_{\mathbf{M}}(f^{\varepsilon},f^{\varepsilon})\nonumber\\
 &+\sum_{n=1}^{2k-1}\varepsilon^{n-1} \Big[\Gamma_{\mathbf{M}}(\mathbf{M}^{-\frac{1}{2}}F_n, f^{\varepsilon})+\Gamma_{\mathbf{M}}(f^{\varepsilon}, \mathbf{M}^{-\frac{1}{2}}F_n)\Big]+\varepsilon^k \Big(E_R^{\varepsilon}+v \times B_R^{\varepsilon}\Big)\cdot\nabla_vf^{\varepsilon}\nonumber\\
 &-\varepsilon^k \Big(E_R^{\varepsilon}+v \times B_R^{\varepsilon}\Big) \cdot\frac{v-u }{ 2RT}f^{\varepsilon}\label{VMLf}\\
 &+\sum_{n=1}^{2k-1}\varepsilon^n\Big[\Big(E_n+v \times B_n \Big)\cdot\nabla_vf^{\varepsilon}+\Big(E_R^{\varepsilon}+v \times B_R^{\varepsilon} \Big)\cdot\mathbf{M}^{-\frac{1}{2}}\nabla_v F_n\Big]\nonumber\\
 &-\sum_{n=1}^{2k-1}\varepsilon^n\Big[\Big(E_n+v \times B_n \Big)\cdot\frac{\big(v-u\big)}{ 2RT}f^{\varepsilon}\Big]+\varepsilon^{k}\CQ_0,
\nonumber
\end{align}
and
\begin{align}\label{fM-2}
\begin{aligned}
&\partial_tE_R^{\varepsilon}-\nabla_x \times B_R^{\varepsilon}=4\pi\int_{\mathbb R^3} v\mathbf{M}^{\frac{1}{2}}f^{\varepsilon} dv, \\
 &\partial_t B_R^{\varepsilon}+ \nabla_x \times E_R^{\varepsilon}=0,\\
& \nabla_x\cdot E_R^{\varepsilon}=-\int_{\mathbb R^3}  \mathbf{M}^{\frac{1}{2}} f^{\varepsilon} dv, \qquad \nabla_x\cdot  B_R^{\varepsilon}=0,
\end{aligned}
\end{align}
where $\CQ_0=\mathbf{M}^{-\frac{1}{2}}\CQ$, $\mathcal{L}_{\mathbf{M}}[f]$ is the linearized Landau/Boltzmann collision operator defined by
\begin{align}\label{lm-def}
-\mathcal{L}_{\mathbf{M}}[f]=&\mathbf{M}^{-\frac{1}{2}} \big[\mathcal{C}( \mathbf{M} , \mathbf{M}^{\frac{1}{2}} f)+\mathcal{C}( \mathbf{M}^{\frac{1}{2}} f, \mathbf{M}\big],
\end{align}
and
$\Gamma_{\mathbf{M}}(f,g)$ is the nonlinear Landau/Boltzmann collision operator given by
\begin{align}\label{gm-def}
\Gamma_{\mathbf{M}}(f,g)=&\mathbf{M}^{-\frac{1}{2}} \mathcal{C}( \mathbf{M}^{\frac{1}{2}} f, \mathbf{M}^{\frac{1}{2}} g).
\end{align}

On the other hand, if we introduce the global Maxwellian  $\mu$
 \begin{equation}\label{Global-Maxwellian}
\mu=  \frac{1 }{(2\pi RT_c)^{3/2}
} \exp\left(-\frac{|v|^2}{ 2RT_c}\right)
\end{equation}
with the constant temperature $T_c$ satisfying $\frac{1}{2}\sup_{t,x}T(t,x)<T_c<\inf_{t,x}T(t,x)$ and set
\begin{equation*}
F_R^{\varepsilon}=\mu^{\frac{1}{2}} h^{\varepsilon},
\end{equation*}
then it is straightforward to see that $h^{\varepsilon}$ satisfies the following equation
\begin{align}\label{h-equation}
\partial_th^{\varepsilon}&+v\cdot\nabla_xh^{\varepsilon}+\frac{\big(E_R^{\varepsilon}+v \times B_R^{\varepsilon} \big) }{ RT}\cdot \big(v-u\big)\mu^{-\frac{1}{2}}\mathbf{M}\nonumber\\
&+\frac{E\cdot v }{ 2RT_c}h^{\varepsilon}-\Big(E+v \times B \Big)\cdot\nabla_vh^{\varepsilon}+\frac{\mathcal{L}[h^{\varepsilon}]}{\varepsilon}\nonumber\\
 =&-\frac{\mathcal{L}_d[h^{\varepsilon}]}{\varepsilon}+\varepsilon^{k-1}\Gamma(h^{\varepsilon},h^{\varepsilon})+\sum_{n=1}^{2k-1}\varepsilon^{n-1}[\Gamma(\mu^{-\frac{1}{2}}F_n, h^{\varepsilon})+\Gamma(h^{\varepsilon}, \mu^{-\frac{1}{2}}F_n)]\nonumber\\
 &+\varepsilon^k \Big(E_R^{\varepsilon}+v \times B_R^{\varepsilon}\Big)\cdot\nabla_vh^{\varepsilon}
 -\varepsilon^k \frac{E_R^{\varepsilon}\cdot v}{ 2RT_c}h^{\varepsilon}\notag\nonumber\\
 &+\sum_{n=1}^{2k-1}\varepsilon^n\Big[\Big(E_n+v \times B_n \Big)\cdot\nabla_vh^{\varepsilon}+\Big(E_R^{\varepsilon}+v \times B_R^{\varepsilon} \Big)\cdot\mu^{-\frac{1}{2}}\nabla_v F_n\Big]\nonumber\\
& -\sum_{n=1}^{2k-1}\varepsilon^n\frac{E_n\cdot v}{ 2RT_c}h^{\varepsilon}+\varepsilon^{k}\CQ_1.
\end{align}
Here $\CQ_1=\mu^{-\frac{1}{2}}\CQ$, $\mathcal{L}[f]$,
and $\Gamma(f,g)$ are defined as \eqref{lm-def} and \eqref{gm-def} respectively, with $\mathbf{M}$ replaced by $\mu$.
In addition, $\mathcal{L}_d [f]$ is a linear collision operator defined as
\begin{equation}\label{dL}
-\mathcal{L}_d [f]=\mu^{-\frac{1}{2}}\left[\mathcal{C}( \mathbf{M}- \mu,\mu^{\frac{1}{2}}f  )+\mathcal{C}( \mu^{\frac{1}{2}} f, \mathbf{M}- \mu )\right].
\end{equation}


The primary objective of this paper is to establish specific estimates for both $f^\vps$ and $h^\vps$, ensuring the global applicability of the Hilbert expansion for both the VML and the non-cutoff VMB systems \eqref{main1} in the whole space.

\subsection{Notations}

Throughout this paper, $C$ denotes a generic positive constant which is independent of the Knudsen number $\varepsilon$ but may change line by line. The notation $a \lesssim b$ means that there exists a generic positive constant $C$ independent of the Knudsen number $\varepsilon$ such that $a \leq Cb$ and $a \approx b$ means that both $a \lesssim b$ and $b \lesssim a$ hold.

The multi-indices $ \alpha= [\alpha_1,\alpha_2, \alpha_3]$ and $\beta = [\beta_1, \beta_2, \beta_3]$ will be used to record spatial and velocity derivatives, respectively. And $\partial^{\alpha}_{\beta}=\partial^{\alpha_1}_{x_1}\partial^{\alpha_2}_{x_2}\partial^{\alpha_3}_{x_3} \partial^{\beta_1}_{ v_1}\partial^{\beta_2}_{ v_2}\partial^{\beta_3}_{ v_3}$. Similarly, the notation $\partial^{\alpha}$ will be used when $\beta=0$ and likewise for $\partial_{\beta}$. The length of $\alpha$ is denoted by $|\alpha|=\alpha_1 +\alpha_2 +\alpha_3$. $\alpha'\leq  \alpha$ means that no component of $\alpha'$ is greater than the corresponding component of $\alpha$, and $\alpha'<\alpha$ means that $\alpha'\leq  \alpha$ and $|\alpha'|<|\alpha|$. And it is convenient to write $\nabla_x^k=\partial^{\alpha}$ with $|\alpha|=k$.

$(\cdot, \cdot)$  is used to denote the $L^2$ inner product in ${\mathbb{ R}}^3_{ v}$, with the norm $|\cdot|_{L^2}$. For notational simplicity, $\langle\cdot,\cdot\rangle$  denotes the ${L^2}$ inner product either in ${\mathbb{ R}}^3_{x}\times{\mathbb{ R}}^3_{ v }$ or in ${\mathbb{ R}}^3_{x}$ with the norm $\|\cdot\|$. For each non-negative integer $k$ and $1\leq p\leq +\infty$, we also use $W^{k,p}$ to denote the standard Sobolev spaces for $(x, v)\in {\mathbb R}^3 \times {\mathbb R}^3$ or $x \in {\mathbb R}^3$, and denote $H^s=W^{s,2}$ with
$\|f\|^2_{H^s}:=\sum_{|\alpha|+|\beta|=0}^s\|\partial^{\alpha}_{\beta}f\|^2.$ Recall that $\nabla_x=(\partial_{x_1}, \partial_{x_2}, \partial_{x_3})$, $\nabla_v=(\partial_{v_1}, \partial_{v_2}, \partial_{v_3})$.

We now introduce some norms for the VML system and the non-cutoff VMB system case by case.
\subsubsection{Norms for the VML system}

Let
$$
\sigma^{ij}(v)=\int_{{\R}^{3}}\phi^{ij}(v-v')\mu(v')dv'
$$
with $\phi^{ij}(v)$ defined by \eqref{cker}, we now introduce the dissipation norm in $v\in {\mathbb R}^3$ by
\begin{eqnarray}\label{D-norm}
    |f(t,x)|_{\bf D}^2&:=&\int_{{\mathbb R}^3}\sigma^{ij}(v)\partial_{v_i} f(t,x,v)\partial_{v_j}f(t,x,v)\, dv\nonumber\\
    &&+\frac{1}{4R^2T^{2}_c}\int_{{\mathbb R}^3} \sigma^{ij}(v)v_iv_j|f(t,x,v)|^2\, dv,
\end{eqnarray}
the corresponding dissipation norm in $(x, v)\in {\mathbb R}^3\times {\mathbb R}^3$ by
\begin{eqnarray*}
    \|f(t)\|_{\bf D}^2&:=&\iint_{{\mathbb R}^3\times{\mathbb R}^3}\sigma^{ij}(v)\partial_{v_i} f(t,x,v)\partial_{v_j}f(t,x,v)\, dv dx\\
    &&+\frac{1}{4R^2T^{2}_c}\iint_{{\mathbb R}^3\times{\mathbb R}^3}\sigma^{ij}(v)v_iv_j|f(t,x,v)|^2\, dv dx,
\end{eqnarray*}
and the Sobolev dissipation norm
\begin{align*}
     \|f(t)\|^2_{H^m_{\bf D}}:=\sum_{|\alpha|=0}^m\left\|\partial^{\alpha}f(t)\right\|^2_{\bf D},\qquad m\leq2.
\end{align*}
Here and in the sequel, Einstein's summation convention is used.

For constant $\ell\geq m$, we set $\lag v\rag=\sqrt{1+|v|^2}$ and introduce the following time-velocity dependent weight functions
\begin{align}\label{tt 01}
    w_{i}(t,v):=\lag v\rag^{\ell-i}\exp\left(\frac{\lag v\rag^2}{8RT_c\ln(\mathrm{e}+t)}\right),\quad  0\leq i\leq m
\end{align}
and then define the weighted norms
\begin{align*}
    \|wf\|^2_{H^m}:=\sum_{i=0}^m\left\|w_{i}\nabla_x^i f\right\|^2 ,\qquad
    \|wf\|^2_{H^m_{\bf D}}:=\sum_{i=0}^m\left\|w_{i}\nabla_x^i f\right\|^2_{\bf D}.
\end{align*}

\subsubsection{Norms for the non-cutoff VMB system}

For improved clarity, we will introduce separate notations for the non-cutoff VMB system.

As described in \cite{AMUXY-JFA-2012, Duan-Liu-Yang-Zhao-2013-KRM}, we introduce the following norms for the non-cutoff VMB system:
 \begin{equation}
\begin{split}
\left |f \right |^2_{\bf D} :=&\iint_{\R^6\times \mathbb{S}^2} B(v  -v _{\ast},\sigma){\bf M}_*
(f '-f)^2 d v _{\ast}dv d\sigma \\
&+ \iint_{\R^6 \times \mathbb{S}^2} B(v  -v _{\ast},\sigma)f_*^2  \Big(\sqrt{{\bf M}_*^{\prime}} - \sqrt{{\bf
M}_*} \Big)^2 dv dv _{\ast}d\sigma, \\
\| f\|_{\bf D}^2 :=&\left\||f|_{\bf D}\right\|^2_{L^2_{x}},
\end{split}\notag
\end{equation}
 For $\ell\in{\mathbb R}$, $L^2_{\ell}$ represents the weighted space with the norm
$$|f|^2_{L^2_{\ell}}=\displaystyle\int_{{\mathbb R}^3} \langle v \rangle^{\ell}
|f(v )|^2dv, \quad \|f\|_{L^2_{\ell}}=\left\||f|_{L^2_{\ell}}\right\|_{L^2_x}.$$
The weighted fractional Sobolev norm is given by
$$
\left|\langle v \rangle^{\ell}f\right|^2_{H^s_{\gamma}}=\left|\langle v \rangle^{\ell}u\right|^2_{L^2_{\gamma}}
+\displaystyle\int_{{\mathbb R}^3} \int_{{\mathbb R}^3}
\frac{\left|\langle v \rangle^{\ell}f(v )-
\langle v '\rangle^{\ell}f(v ')\right|^2}{|v -v '|^{3+2s}}
\chi_{| v- v '|\leq1}d v 'd v,$$
which is equivalent to
$$
\left|\langle v \rangle^{\ell}f\right|^2_{H^s_{\gamma}}=\displaystyle\int_{{\mathbb R}^3} \langle v \rangle^{\gamma}\left|\left(1-\triangle_{v }\right)^{\frac{s}{2}}\left(\langle v \rangle^{\ell}(v )f(v )\right)\right|^2dv.
$$
We introduce the following time-velocity dependent weight function:
\begin{align*}
    \overline{w}(t,v):=\exp\left(\frac{\lag v\rag}{8RT_c\ln(\mathrm{e}+t)}\right),
\end{align*}
and define the corresponding weighted norms:
\begin{align*}
    \|\overline{w}h\|^2_{H^m}:=\sum_{|\alpha|+|\beta|=0}^m\left\|\overline{w}\pa^\alpha _\beta  h\right\|^2 ,\qquad
    \|\overline{w}h\|^2_{H^m_{\bf D}}:=\sum_{|\alpha|+|\beta|=0}^m\left\|\overline{w}\pa^\alpha _\beta h\right\|^2_{\bf D}.
\end{align*}

For simplicity, let's also denote
\begin{align*}
    X(t):=\exp\left(\frac{1}{8RT_c\ln(\mathrm{e}+t)}\right),
\end{align*}
it is straightforward to see that
\begin{align*}
    Y(t):=-\frac{X'(t)}{X(t)}=\frac{1}{8RT_c(\mathrm{e}+t)\big(\ln(\mathrm{e}+t)\big)^2}>0.
\end{align*}

Next, it is straightforward to determine the null space of the linearized Landau/Boltzmann operator $\mathcal{L}_{\mathbf{M}}$,
\[\mathcal {N}_\FM=\mbox{span}\left\{\mathbf{M}^{\frac{1}{2}}, (v_i-u_i)\mathbf{M}^{\frac{1}{2}}(1\leq i\leq3),\left(\frac{|v-u|^2}{RT}-3\right)\mathbf{M}^{\frac{1}{2}}\right\}.
\]
Similarly, for the linearized operator $\mathcal{L}$,
\[\mathcal {N}=\mbox{span}\left\{\mu^{\frac{1}{2}}, v_i\mu^{\frac{1}{2}}(1\leq i\leq3),(|v|^2-3)\mu^{\frac{1}{2}}\right\}.
\]
Now, let's define ${\bf P}_{\mathbf{M}}$ as the orthogonal projection from $L^2_v$ onto $\mathcal {N}_{\bf M}$ and ${\bf P}$ as the orthogonal projection from $L^2_v$ onto $\mathcal {N}$. Additionally, we will use $\chi_j$ $(0\leq j\leq4)$ to denote the orthonormal basis of the null space of $\CL_\FM$, given by
\begin{align*}
\chi_0= \frac{1}{\sqrt{\rho}}\sqrt{\FM},\quad
\chi_{i}=\frac{v_i-u_i}{{\sqrt{R\rho T}}}
\sqrt{\FM},\ (i=1,2,3),\quad
\chi_{4}=\frac{1}{\sqrt{6\rho}}\left(\frac{|v-u|^2}
{RT}-3\right)\sqrt{\FM}.
\end{align*}
Using these orthonormal basis functions, we can represent ${\bf P}_{\mathbf{M}}[f_n] (1\leq n\leq 2k-1)$ as
\begin{align}\label{fn-mac-def}
{\bf P}_{\mathbf{M}}[f_n]=\frac{\rho_n}{\sqrt{\rho}} \chi_0+\sum\limits_{i=1}^3\frac{1}{\sqrt{R\rho T}}u_n^i\chi_i+\frac{T_n}{\sqrt{6\rho}}\chi_4.
\end{align}
Here, $[\rho_n,u_n,T_n]$ are the hydrodynamic fields representing the density, velocity, and temperature fluctuations of $f_n$, respectively.

From Proposition \ref{em-ex-lem}, we can easily deduce that
\begin{corollary}\label{corollary-1}
For any given generic positive constant $C_0\geq 1$, if we assume further that $\eps_1$ is sufficiently small, then one can always find a positive constant $T_c$ such that
 \begin{align}
\frac{ \epsilon_1}{C_0}\leq  T(t,x)-T_c\leq C_0\epsilon_1 \label{tt1}
 \end{align}
holds for all $(t,x)\in[0,\infty)\times\R^3$.
\end{corollary}

\subsection{Main results}
In this subsection, we will present our main results concerning the global Hilbert expansion of the VML system and the non-cutoff VMB system. To start, we will first state the result for the VML system.

\subsubsection{Main result for the VML system}
Setting $\kappa=\frac13$, we define the energy functional
\begin{align}\label{eg-vml}
&\mathcal{E}(t)=\sum_{i=0}^2 \varepsilon^i\left[\left(\left\|\sqrt{4\pi RT}\nabla_x^if^{\varepsilon}(t)\right\|^2+\left\|\nabla_x^iE_R^{\varepsilon}(t)\right\|^2+\left\|\nabla_x^i B_R^{\varepsilon}(t)\right\|^2\right)
+\varepsilon^{1+\kappa}\left\|w_i\nabla_x^ih^{\varepsilon}(t)\right\|^2\right],
\end{align}
and the dissipation rate
\begin{align}\label{dn-vml}
\mathcal{D}(t)\backsimeq&\sum_{i=0}^2 \varepsilon^i\left[\frac{1}{\varepsilon}\left\|\nabla_x^i({\bf I}-{\bf P}_{\mathbf{M}})\left[f^{\varepsilon}\right](t)\right\|^2_{\bf D}+\varepsilon^{\kappa}\left\|w_i\nabla_x^i h^{\varepsilon}(t)\right\|^2_{\bf D}\right.\\
&\left.+\varepsilon^{1+\kappa}Y(t)\left\|(1+|v|)w_i\nabla_x^i h^{\varepsilon}(t)\right\|^2\right].\nonumber
\end{align}
Our main result can now be stated as follows:

\begin{theorem}[Global Hilbert expansion of the VML system]\label{resultVML} Assume that
\begin{itemize}
\item The initial datum $[F^\varepsilon_0(x,v), E^{\varepsilon}_0(x), B^{\varepsilon}_0(x)]$ satisfies
\begin{eqnarray*}
F^\vps_0&=&\sum_{n=0}^{2k-1}\varepsilon^nF_{n,0}+\varepsilon^kF^{\varepsilon}_{R,0}\geq0,
\end{eqnarray*}
and
\begin{align*}
E^{\varepsilon}_0=E_0+\sum_{n=1}^{2k-1}\varepsilon^n E_{n,0}+\varepsilon^kE^{\varepsilon}_{R,0},\
B^{\varepsilon}_0=B_0+\sum_{nB=1}^{2k-1}\varepsilon^n B_{n,0}+\varepsilon^kB^{\varepsilon}_{R,0},
\end{align*}
where $[\rho(t,x), u(t,x), T(t,x), E(t,x), B(t,x)]$ is the smooth solution constructed in Proposition \ref{em-ex-lem} for the Euler-Maxwell system \eqref{EM};
\item $T_c$ is suitably chosen such that there exists a positive constant $C_0 > 1$ and \eqref{tt1} is valid for all $(t,x) \in [0,\infty) \times \mathbb{R}^3$;
\item There exists a constant $C > 0$ such that for $N \geq 2$:
\begin{align}
\sum\limits_{\al_0+|\al|\leq N+4k-2n+2}&\left\|\pa_t^{\al_0}\pa^\al[\rho_{n,0},u_{n,0},T_{n,0},E_{n,0},B_{n,0}](x)\right\|\leq C,\ \ 1\leq n\leq2k-1;\label{Fn-id-vml}
\end{align}
\item Let $k\geq3$ and
\begin{align*}
    \mathcal{E}(0)\lesssim 1.
\end{align*}
\end{itemize}
Then, there exists a constant $\varepsilon_0 > 0$ such that for $0 < \varepsilon \leq \varepsilon_0$, the Cauchy problem of the VML system \eqref{main1} and \eqref{VL-id} admits a unique smooth solution $[F^{\varepsilon}(t,x,v), E^{\varepsilon}(t,x), B^{\varepsilon}(t,x)]$ satisfying
 \begin{align}
\begin{aligned}
&F^{\varepsilon}=\sum_{n=0}^{2k-1}\varepsilon^nF_n+\varepsilon^kF^{\varepsilon}_R\geq0,\quad
E^{\varepsilon}=E+\sum_{n=1}^{2k-1}\varepsilon^n E_n+\varepsilon^kE^{\varepsilon}_R,\quad
B^{\varepsilon}=B+\sum_{n=1}^{2k-1}\varepsilon^n B_n+\varepsilon^kB^{\varepsilon}_R.
\end{aligned}\notag
\end{align}
Furthermore, for $\kappa=\frac13$, it holds that
\begin{align}\label{TVML1}
&\mathcal{E}(t)+\int_0^{t}\mathcal{D}(s)\, d s\lesssim \mathcal{E}(0)+1,\ 0\leq t\leq\varepsilon^{-\kappa},
\end{align}
and
\begin{align*}
   \lim_{\varepsilon\rightarrow 0^+} \sup_{0\leq t\leq \varepsilon^{-\kappa}}\Big(\left\|\FM^{-\frac{1}{2}}(t)\left(F^{\varepsilon}(t)-\mathbf{M}(t)\right)\right\|_{H^2}+\left\|E^{\varepsilon}(t)-E(t)\right\|_{H^2}
   +\left\|B^{\varepsilon}(t)-B(t)\right\|_{H^2}\Big)=0.
\end{align*}

\end{theorem}
The following remark is concerned with the initial condition of the coefficients $F_n$$(1\leq n\leq2k-1)$ in the Hilbert expansion.
\begin{remark}
The $t-$derivatives in \eqref{Fn-id-vml} are implied by the moment equations from the evolution equations in \eqref{expan2}. For example, $\partial_t[\rho_{n,0},u_{n,0},T_{n,0},E_{n,0},B_{n,0}]$ can be defined through the evolution equation of $\rho_{n},u_{n},T_{n},E_{n},B_{n}$, as described in \eqref{em-abc-eq}. It's important to note that only the initial data of the macroscopic part of the coefficients are specified, as the microscopic component is determined iteratively according to the relations in \eqref{Fn-id-vml}.

\end{remark}

\subsubsection{Main result for the non-cutoff VMB system}
Our second result is concerned with the global Hilbert expansion of the non-cutoff VMB system.

In addition to the previous definitions, we also define the following energy functional:
\begin{align}\label{eg-vmb}
\mathcal{E}(t)=&\sum\limits_{|\alpha|\leq 4}\vps^{|\alpha|}\big(\|\sqrt{4\pi RT}\partial^{\alpha} f^{\varepsilon}\|^2+\|\partial^{\alpha} E_R^{\varepsilon}\|^2+\|\partial^{\alpha} B^{\varepsilon}_R\|^2\big)\nonumber\\
&+\sum\limits_{|\alpha|+|\beta|\leq 4,|\beta|\geq1}\vps^{|\alpha|+|\beta|}\|\pa^\alpha _\beta ({\bf I}-{\bf P}_{\mathbf{M}})[f^{\varepsilon}]\|^2+\sum_{|\alpha|+|\beta|\leq 4}\vps^{|\alpha|+|\beta|+1+\kappa}\|\overline{w} \partial^\alpha_\beta  h^{\varepsilon}\|^2,
\end{align}
where $\kappa=\frac{1}{3}$, and the dissipation rate:
\begin{align}\label{dn-vmb}
\mathcal{D}(t)\backsimeq	&\sum\limits_{|\alpha|+|\beta|\leq 4}\vps^{|\alpha|+|\beta|-1}\|\pa^\alpha _\beta({\bf I}-{\bf P}_{\mathbf{M}})[f^{\varepsilon}]\|^2_{\bf D}+Y(t)\sum_{|\alpha|+|\beta|\leq 4}\vps^{|\alpha|+|\beta|+\kappa}\big\|{\lag v\rag}^{\frac{1}{2}}\overline{w} \partial^\alpha_\beta h^{\varepsilon}\big\|^2   \nonumber\\
&+\sum_{|\alpha|+|\beta|\leq 4}\vps^{|\alpha|+|\beta|+\kappa}\|\overline{w} \partial^\alpha_\beta
  h^{\varepsilon}]\|^2_{\bf D}.
\end{align}

Our second main result can now be stated as follows

\begin{theorem}[Global Hilbert expansion of the non-cutoff VMB system]\label{resultVMB} Assume that
$$\max\left\{-3, -\frac{3}{2}-2s\right\} < \gamma < 0$$
 for the non-cutoff collision kernel, and consider the following conditions:
\begin{itemize}
\item The initial datum $[F^\varepsilon_0(x,v), E^{\varepsilon}_0(x), B^{\varepsilon}_0(x)]$ satisfies
\begin{eqnarray*}
F^\vps_0&=&\sum_{n=0}^{2k-1}\varepsilon^nF_{n,0}+\varepsilon^kF^{\varepsilon}_{R,0}\geq0,
\end{eqnarray*}
and
\begin{align*}
E^{\varepsilon}_0=E_0+\sum_{n=1}^{2k-1}\varepsilon^n E_{n,0}+\varepsilon^kE^{\varepsilon}_{R,0},\quad
B^{\varepsilon}_0=B_0+\sum_{n=0}^{2k-1}\varepsilon^n B_{n,0}+\varepsilon^kB^{\varepsilon}_{R,0},
\end{align*}
where $[\rho(t,x), u(t,x), T(t,x), E(t,x), B(t,x)]$ is the smooth solution constructed in Proposition \ref{em-ex-lem} for the Euler-Maxwell system \eqref{EM};
\item Choose $T_c$ suitably so that \eqref{tt1} is valid for all $(t,x) \in [0,\infty) \times \mathbb{R}^3$;
\item There exists a constant $C > 0$ such that for $N \geq 4$,
\begin{align}
\sum\limits_{\al_0+|\al|\leq N+4k-2n+2}&\left\|\pa_t^{\al_0}\pa^\al[\rho_{n,0},u_{n,0},T_{n,0},E_{n,0},B_{n,0}](x)\right\|\leq C,\ \ 1\leq n\leq2k-1;\label{Fn-id}
\end{align}
\item  Let  $k\geq 4$ and
\begin{align*}
    \mathcal{E}(0)\lesssim 1.
\end{align*}
\end{itemize}
Then, there exists a constant $\varepsilon_1 > 0$ such that for $0 < \varepsilon \leq \varepsilon_1$, the Cauchy problem of the VMB system \eqref{main1} and \eqref{VL-id} admits a unique smooth solution $[F^{\varepsilon}(t,x,v), E^{\varepsilon}(t,x), B^{\varepsilon}(t,x)]$ satisfying
 \begin{align}
\begin{aligned}
&F^{\varepsilon}=\sum_{n=0}^{2k-1}\varepsilon^nF_n+\varepsilon^kF^{\varepsilon}_R\geq0,\quad
E^{\varepsilon}=E+\sum_{n=1}^{2k-1}\varepsilon^n E_n+\varepsilon^kE^{\varepsilon}_R,\quad
B^{\varepsilon}=B+\sum_{n=1}^{2k-1}\varepsilon^n B_n+\varepsilon^kB^{\varepsilon}_R.
\end{aligned}\notag
\end{align}
Furthermore, for $\kappa = \frac{1}{3}$, it holds that
\begin{align}\label{TVML1}
&\mathcal{E}(t)+\int_0^{t}\mathcal{D}(s)\, d s\lesssim \mathcal{E}(0)+1,\ 0\leq t\leq\varepsilon^{-\kappa},
\end{align}
and
\begin{align*}
   \lim_{\varepsilon\rightarrow 0^+} \sup_{0\leq t\leq \varepsilon^{-\kappa}}\Big(\left\|\FM^{-\frac{1}{2}}(t)\left(F^{\varepsilon}(t)-\mathbf{M}(t)\right)\right\|_{H^4}+\left\|E^{\varepsilon}(t)-E(t)\right\|_{H^4}
   +\left\|B^{\varepsilon}(t)-B(t)\right\|_{H^4}\Big)=0.
\end{align*}

\end{theorem}
\begin{remark}
 A similar result also holds for the cutoff VMB system with $-3<\gamma\leq 1$. The only difference is that we replace the weight function $\overline{w}$ with $w=\exp\big(\frac{\langle v\rangle^2}{8RT_c\ln(\mathrm{e}+t)}\big)$.
\end{remark}

\subsection{Strategies and ideas}

In this subsection, we will outline the main ideas and strategies applied in our analysis.

A prevalent challenge encountered when validating the Hilbert expansion for non-relativistic kinetic equations within the realm of perturbation lies in a phenomenon known as ``velocity growth". It can be expressed as follows:

\begin{align}\label{nordiff}
	f^{\varepsilon}(t,x,v)\mathbf{M}^{-\frac12}(t,x,v)\Big\{\partial_t +v\cdot\nabla_x\Big\}\mathbf{M}^{\frac12}(t,x,v) \sim |v|^3f^{\varepsilon}(t,x,v).
\end{align}

To overcome this difficulty, Caflisch \cite{Caflisch-CPAM-1980} introduced a significant decomposition:
 $$F^{\varepsilon}_R=\underbrace{\mathbf{M}^{\frac12}\bar{f}^{\varepsilon}}_\text{low velocity part}+\underbrace{\mu^{\frac12}\bar{h}^{\varepsilon}}_\text{high velocity part},$$
 by partitioning $F^{\varepsilon}_R$ in this manner, the equation corresponding to the angular cutoff Boltzmann equation was artificially separated into a coupled system involving $\bar{f}^{\varepsilon}$ and $\bar{h}^{\varepsilon}$. Employing an energy-based approach, $\bar{f}^{\varepsilon}$ and $\bar{h}^{\varepsilon}$ were subsequently solved, with the assumption that $\bar{f}^{\varepsilon}(0,x,v)=\bar{h}^{\varepsilon}(0,x,v)=0$. However, it's worth pointing out that this assumption may not guarantee the non-negativity of the density distribution function $F^{\varepsilon}(t,x,v)$.
 To address this concern, Guo-Jang-Jiang \cite{Guo-JJ-KRM-2009, Guo-JJ-CPAM-2010} introduced the $L^2-L^{\infty}$ method, initially proposed in \cite{Guo-ARMA-2010} for solving the initial boundary value problem of the Boltzmann equation. They utilized this method to justify the Hilbert expansion of the angular cutoff Boltzmann equation and established the non-negativity of the density distribution function $F^{\varepsilon}(t,x,v)$. In their works \cite{Guo-JJ-KRM-2009, Guo-JJ-CPAM-2010}, the intricate term \eqref{nordiff} was estimated as follows:
 \begin{align}
 	\Big\langle\Big| \mathbf{M}^{-\frac12}(\partial_t +v\cdot\nabla_x)\mathbf{M}^{\frac12}\Big|\bar{f}^{\varepsilon},\bar{f}^{\varepsilon}\Big\rangle\lesssim \|\nabla_{t,x}(\rho, u, T)\| \|\bar{h}^{\varepsilon}\|_{\infty}\|\bar{f}^{\varepsilon}\|.\label{gjj-in}
 \end{align}
 Here, $\bar{h}^{\varepsilon}$ satisfies
 \begin{align*}
 	(1+|v|)^{-\beta}\mu^{\frac12}\bar{h}^{\varepsilon}=\mathbf{M}^{\frac12}\bar{f}^{\varepsilon}=F^{\varepsilon}_R , \quad \beta>\frac{9-2\gamma}{2}.
 \end{align*}
 We also emphasize that \eqref{gjj-in} holds under the constraint that $\frac{1}{2}\sup_{t,x}T(t,x)<T_c<\inf_{t,x}T(t,x)$ for a positive constant $T_c$. Furthermore, it's noteworthy that the estimates for $\|\bar{h}^{\varepsilon}\|_{\infty}$ were established using a characteristic method, and the total energy was balanced through a clever interplay between norms $\|\bar{f}^{\varepsilon}\|$ and $\|\bar{h}^{\varepsilon}\|_{\infty}$.

 This method found success in demonstrating the validity of the Hilbert expansion for various systems, including the relativistic Boltzmann equation \cite{Speck-Strain-CMP-2011}, the Vlasov-Poisson-Boltzmann system \cite{Guo-Jang-CMP-2010}, and the relativistic VMB system \cite{Guo-Xiao-CMP-2021}.

 It's worth noting that the adapted $L^2-L^{\infty}$ method developed in \cite{Guo-Xiao-CMP-2021} faces limitations when applied to the non-relativistic VMB system. This limitation arises due to the absence of a corresponding Glassey-Strauss expression \cite{GS-1986} for the electromagnetic field within the relativistic framework. Furthermore, since this method relies on the characteristic method and the structured nature of the linearized Boltzmann collision operator with cut-off potentials, it becomes challenging to adapt this method to tackle the same problem for Landau-type equations and non-cutoff Boltzmann-type equations.

 In the relativistic framework, the troublesome term akin to \eqref{nordiff} exhibits linear growth only in momentum $p$. Exploiting this gradual momentum growth, Ouyang-Wu-Xiao \cite{Ouyang-Wu-Xiao-arxiv-2022-rLM} recently introduced a weighted energy method that involves the local Maxwellian to investigate the Hilbert expansion of relativistic Landau-type equations. Their key strategy was to control the linear momentum growth term by incorporating an additional dissipation term induced by a special time-velocity dependent weight function. However, this method exclusively applies to relativistic kinetic equations, and addressing the case for non-relativistic Landau-type equations and non-cutoff Boltzmann-type equations remains an open problem.

 This article centers on precisely this problem. Our primary objective is to establish the global validity of the Hilbert expansion for both the VML system and the non-cutoff VMB system, which serve as fundamental models for studying the dynamics of dilute charged particles.

 Inspired by the works of \cite{Caflisch-CPAM-1980} and \cite{Ouyang-Wu-Xiao-arxiv-2022-rLM}, we aspire to develop an enhanced energy method that accommodates both the local Maxwellian $\FM$ and the global Maxwellian $\mu$. Our approach hinges on establishing the following pivotal relationship:

 \begin{align}\label{relation-local-globsl}
 	F^\varepsilon_R = \sqrt{\FM}f^\varepsilon = \sqrt{\mu}h^\varepsilon, \quad \text{with} \quad T_c < \inf_{t,x}T(t,x).
 \end{align}

 The crux of our analysis revolves around the interplay of energy estimates for both $f^{\varepsilon}$ and $h^{\varepsilon}$. This harmonious interplay enables us to effectively close the energy estimates and, as a result, establish the well-posedness of the remainder $F_R^\varepsilon$ within the framework of Sobolev space. It's noteworthy that the global Maxwellian $\mu$ adheres to the same functional form as presented in \eqref{Global-Maxwellian}. The specific range of the positive constant $T_c > 0$ is meticulously specified in Corollary \ref{corollary-1}, underscoring the precision of our approach.
 To make our presentation more clear, we will present precise explanations for each case individually.

For the VML system, our main ideas are to introduce an interplay of energy estimates of $f^{\varepsilon}$ and $h^{\varepsilon}$, which can be summarized as follows.
\begin{itemize}
\item  To effectively estimate $f^{\varepsilon}$ and tackle the difficult term represented by \eqref{nordiff}, we adopt a strategy that involves dividing the velocity space ${\mathbb R}^3_v$ into two distinct regions: the low-velocity portion and the high-velocity portion.
In the low-velocity region, we are able to provide bounds for both the microscopic and macroscopic components of $f^{\varepsilon}$. Specifically, we employ terms like $\varepsilon^{-1}\|\partial^{\alpha}({\bf I}-{\bf P}_{\mathbf{M}})[f^{\varepsilon}]\|_{\bf D}^2$ and $\sum_{|\alpha'|\leq|\alpha|}\|\partial^{\alpha'}f^{\varepsilon}\|^2$ to effectively control these components.
However, for the high-velocity portion, the following significant estimate was used
\begin{align*}
	\left|\partial^{\alpha}f^{\varepsilon}\right|
	\lesssim\sum_{j=0}^{|\alpha|}{\langle v\rangle}^{2(|\alpha|-j)}\exp\left(-\frac{(T-T_c)|v|^2}{8C_0RTT_c}\right) \left|\nabla_x^jh^{\varepsilon}\right|,\quad 0\leq |\alpha|\leq 2.
\end{align*}
This estimate directly follows from the assumption that $\mathbf{M}$ and $\mu$ are sufficiently close. To effectively control the high-velocity portion, we utilize norms like $\sum_{|\alpha'|\leq|\alpha|}
\|\partial^{\alpha'}h^{\varepsilon}\|^2$. Similar techniques can also be applied to handle other nonlinear terms, such as $-(E+v \times B)\cdot\nabla_vf^{\varepsilon}$, which involve velocity growth and/or velocity derivatives.
The dissipation norms $\|\partial^{\alpha}h^{\varepsilon}\|^2_{{\bf D}}$ for $0\leq |\alpha|\leq 2$ are specifically applied to control estimates related to $\nabla_vf^{\varepsilon}$, since the energy functional $\mathcal{E}(t)$ for the VML system does not include any velocity derivatives.
\item
For the estimates of $h^{\varepsilon}$, we encounter a problematic term  $\varepsilon^{-1}\big|\big\langle\partial^{\alpha}\mathcal{L}_d[h^{\varepsilon}], \partial^{\alpha}h^{\varepsilon}\big\rangle\big|$, which can be bounded as follows:
\begin{align}\label{Ld ex1}
\frac{1}{\varepsilon}\big|\big\langle\partial^{\alpha}\mathcal{L}_d[h^{\varepsilon}], \partial^{\alpha}h^{\varepsilon}\big\rangle\big|\lesssim \frac{\epsilon_1}{\varepsilon}\big\|h^{\varepsilon}\big\|_{H^{|\alpha|}_{\bf D}}^2\lesssim \frac{\epsilon_1}{\varepsilon}\Big(\big\|({\bf I}-{\bf P})[h^{\varepsilon}]\big\|_{H^{|\alpha|}_{\bf D}}^2+\big\|{\bf P}[h^{\varepsilon}]\big\|_{H^{|\alpha|}_{\bf D}}^2\Big).
\end{align}

However, controlling the macroscopic component $\epsilon_1\varepsilon^{-1}\big\|{\bf P}[h^{\varepsilon}]\big\|_{H^{|\alpha|}_{\bf D}}^2$ presents difficulties. To address this issue, our paper introduces another approach to bound this troublesome term by using $f^\varepsilon$. It can be shown that
\begin{align}\label{Ld ex2}
\left\|{\bf P}\left[h^{\varepsilon}\right]\right\|^2_{H^{|\alpha|}_{\bf D}}\lesssim& \left\|{\bf P}\left[f^{\varepsilon}\right]\right\|^2_{H^{|\alpha|}_{\bf D}} \lesssim \left\|f^{\varepsilon}\right\|^2_{H^{|\alpha|}},
\end{align}
which follows from \eqref{relation-local-globsl} once again.

In addition to this main difficulty in estimating $h^\vps$, we also face additional difficulties arising from the interaction between particles and the electromagnetic field. These include terms like $\frac{E\cdot v }{ 2RT_c}h^{\varepsilon}$, $(E+v \times B)\cdot\nabla_vh^{\varepsilon}$, etc., which involve velocity growth and derivatives and thus cannot be directly controlled in the energy estimates.

In fact, direct estimates of the delicate term $\frac{E\cdot v }{ 2RT_c}h^{\varepsilon}$ leads to temporal decay but velocity growth in the form of
\begin{align}
\epsilon_1(1+t)^{-p_0}\|\sqrt{{\lag v\rag}}w_0h^{\varepsilon}\|^2\label{e-gth}
\end{align}
for some $p_0>1$. This growth cannot be controlled by either the energy functional $\mathcal{E}(t)$ via Gronwall's inequality or the dissipation terms induced by the linearized Landau collision operator. To tackle this issue, we introduce a time-velocity dependent exponential weight function $\exp\left(\frac{\lag v\rag^2}{8RT_c\ln(\mathrm{e}+t)}\right)$
which leads to the following additional dissipative term:
$$
\frac{1}{8RT_c(\mathrm{e}+t)\big(\ln(\mathrm{e}+t)\big)^2}\|{\lag v\rag}w_0h^{\varepsilon}\|^2.
$$
With this additional term and the favorable temporal decay of solutions to the Euler-Maxwell system \eqref{EM}, we can then control \eqref{e-gth} in a suitable manner. Note that this kind of time-velocity exponential weight was originally introduced in \cite{DYZ-2012,DYZ-2013,Duan-2014-vml}to construct global smooth solutions near Maxwellians to some complex kinetic equations.

Regarding the term $(E+v \times B)\cdot\nabla_vh^{\varepsilon}$, we employ a similar approach to that in \cite{Guo-JAMS-2011}. We design a polynomial weight depending on the order of the $x-$derivatives to alleviate the velocity growth. However, this special design of the polynomial part also introduces new difficulties for the weighted trilinear estimates of the nonlinear collision operator $\varepsilon^{k-1}\Gamma(h^{\varepsilon},h^{\varepsilon})$, where Sobolev's inequalities like $\|f\|_{L^{\infty}_x}\lesssim \|f\|_{H^2_x}$ have to be used, and additional polynomial weights may arise.
Our  approach to solve this problem is to replace the additional polynomial weights with the exponential weight $\exp\left(\frac{\langle v\rangle^2}{8RT_c\ln(\mathrm{e}+t)}\right)$, as shown in estimates \eqref{gah-1}-\eqref{zh2} for more details.
\item
To capture the interplay between the estimates of $f^\varepsilon$ and $h^\varepsilon$, we introduce the energy functional $\mathcal{E}(t)$ defined in \eqref{eg-vml}. In fact, during our energy estimates for the function $f^\varepsilon$, as observed in \cite{Ouyang-Wu-Xiao-arxiv-2022-rLM}, we encounter a difficulty arising from the non-commutative structure of the linearized collision operator term $\varepsilon^{-1}\mathcal{L}_{\mathbf{M}} [f^{\varepsilon}]$ and ${\bf P}_{\mathbf{M}}[f^{\varepsilon}]$ with the spatial derivative operator $\partial^{\alpha}$. This leads to the appearance of singularities characterized by $\varepsilon^{-|\alpha|}$ in the estimate of $\|\partial^{\alpha}f^{\varepsilon}\|^2$. Similarly, when we derive an estimate for $\|h^{\varepsilon}\|^2_{H^{|\alpha|}}$ as a consequence of \eqref{Ld ex1} and \eqref{Ld ex2}, we encounter a term $\varepsilon^{-1}\|f^{\varepsilon}\|^2_{H^{|\alpha|}}$ with a singular factor of $\varepsilon^{-1}$ in its upper bound estimates.

To remove these singularities and ensure that $\|f^{\varepsilon}\|^2_{H^{|\alpha|}}$ remains integrable over the time interval $[0,\vps^{-\frac{1}{3}})$, which is the lifespan of our Hilbert expansion, we introduce coefficients $\varepsilon^{|\alpha|}$ for the norm $\|\partial^{\alpha}f^{\varepsilon} \|^2$ with $0\leq |\alpha|\leq 2$. Additionally, by multiplying the norms $\|\partial^{{\alpha}}h^{\varepsilon}\|^2$ by a weight factor of $\varepsilon^{|\alpha|+1+\kappa}$ and
 combining these energy estimates for $\|f^{\varepsilon}\|^2_{H^2}$ and $\|h^{\varepsilon}\|^2_{H^2}$, we are able to close the {\it a priori} estimates.

\end{itemize}

In the case of the non-cutoff VMB system, compared to the VML case, the corresponding problem becomes more intricate, and the main difficulties can be summarized in two aspects:
\begin{itemize}
\item [1.] The dissipation property of the non-cutoff linearized Boltzmann collision operator  is weaker than that of the Landau case;
  \item [2.]
  In contrast to the Landau case, where we have weighted estimates of the form $e^{q\langle v\rangle^2}$ (with $q > 0$) for the Landau collision operators, the existence of similar estimates for the non-cutoff Boltzmann collision operators remains uncertain at present.

%
\end{itemize}

 Our main strategies for addressing the aforementioned problems can be summarized as follows:

 Firstly, we recognize the importance of introducing additional terms involving velocity derivatives into our energy functional, as expressed in \eqref{eg-vmb}. This step is crucial for effectively handling nonlinear terms such as $\frac{E\cdot v }{2RT}h^{\varepsilon}$ and $-(E+v \times B)\cdot\nabla_vh^{\varepsilon}$ within the system.

%

Secondly, we leverage the estimates on the non-cutoff Boltzmann collision operators obtained in \cite{Duan-Liu-Yang-Zhao-2013-KRM, FLLZ-SCM-2018} to derive refined versions of these estimates with respect to the weight function $\overline{w}(t,v)=\exp\big(\frac{\lag v\rag}{8RT_c\ln(\mathrm{e}+t)}\big)$, as detailed in Lemma \ref{non-nonlinear}. With these refined estimates, we are able to handle the terms involving the interaction between particles and the electromagnetic field, the nonlinear collision terms, and the associated coefficients, as demonstrated in the following:

\begin{itemize}
\item   The interaction between particles and the electromagnetic field, exemplified by terms like $\frac{E\cdot v }{ 2RT_c}h^{\varepsilon}$ and $(E+v \times B)\cdot\nabla_vh^{\varepsilon}$, can be effectively controlled through the additional dissipation terms induced by the weight function $\overline{w}(t,v)$. It's worth noting that the weight function $\overline{w}(t,v)$ lacks polynomial components and remains unchanged regardless of the order of derivatives. This distinctive feature simplifies our estimations and renders them relatively more concise.

\item To estimate the nonlinear terms related to $F_1$, such as
$$\big|\big\langle[\Gamma(\mu^{-\frac{1}{2}}F_1,h^{\varepsilon})+\Gamma(
 h^{\varepsilon}, \mu^{-\frac{1}{2}} F_1)], \overline{w}^2h^{\varepsilon}\big\rangle\big|,$$
if we apply the estimates from
 \cite[Lemma 2.3]{Duan-Liu-Yang-Zhao-2013-KRM},
 such a term would be bounded by
\begin{equation}\label{unp-term}
	[\ln(\mathrm{e}+t)]^{-2}  \big\|\mu^{\frac{\lambda_0}{128}}F_1\big\|_{L^{\infty}_xL^2_v}\left\|\overline{w}h^{\varepsilon}\right\|^2_{1+\gamma/2},
\end{equation}
 which becomes unmanageable due to the linear time growth of $F_1$ as indicated in \eqref{vmb-Fn-es}.
 To overcome this difficulty, we improved the estimate of $\left|\left\langle {\Gamma}\left(g_1,g_2\right), \overline{w}^2 g_3\right\rangle\right|$ by employing a more precise angular integration estimate, resulting in the removal of undesirable terms like \eqref{unp-term}. For details, please refer to the proof of \eqref{Gamma-noncut-2} in Lemma \ref{non-nonlinear}. Indeed, the estimates in Lemma \ref{non-nonlinear} are inherently interesting.
 \item  For the coefficients $F_n$ ($1\leq n\leq 2k-1$), it is essential to establish velocity decay estimates of the form
 \begin{align}\label{vt-example}
 	\left|F_n\right|\lesssim \mu^{\frac{1}{2}} e^{-q\langle v\rangle},\qquad q>0.
 \end{align}
  In previous works such as \cite{Guo-JJ-CPAM-2010, Guo-Jang-CMP-2010}, similar estimates to \eqref{vt-example} were derived directly from $e^{q\langle v\rangle^2}$ ($q>0$) type weighted estimates on $\mathbf{M}^{-\frac{1}{2}}F_n$.
 However, at the current stage, it is evident that $\mathbf{M}\mu^{-1}\sim e^{c_1\langle v\rangle^2}$ for some $c_1>0$, even though $\mathbf{M}$ and $\mu$ are sufficiently close as our assumption \eqref{tt1}. This implies that \eqref{vt-example} cannot be derived through $e^{q\langle v\rangle}$ ($q>0$) type weighted estimates on $\mathbf{M}^{-\frac{1}{2}}F_n$. To overcome this difficulty, our strategy is to use $e^{q\langle v\rangle}$ ($q>0$) type weighted estimates on $\mathbf{\mu}^{-\frac{1}{2}}F_n$ instead of $\mathbf{M}^{-\frac{1}{2}}F_n$.

 In other words, when constructing the coefficients $F_n$ for $1\leq n\leq 2k-1$, we set $F_n=\mathbf{M}^{\frac{1}{2}}f_n=\mu^{\frac{1}{2}}h_n$. We then establish estimates like \eqref{vt-example} using $h_n$. For more details, please refer to Lemma \ref{vmb-Fn-lem}.
\end{itemize}


\subsection{Literatures}

Numerous studies have investigated the connections between kinetic equations and fluid dynamics systems, such as the Euler or Navier-Stokes equations. These relationships between the fluid-dynamic perspective and kinetic descriptions can be explored through various mathematical approaches. One such approach involves examining the equivalence of kinetic equations and fluid systems in a time-asymptotic sense. A lot of contributions to this area including works on Boltzmann-type equations like \cite{Kawashima-Matsumura-Nishida-1979, DL-2015, Huang-Xin-Yang, Huang-Yang, Li-Wang-Wang-2022, Liu-Yang-Yu-Zhao-2006}, and for Landau-type equations, references like \cite{DYY-Landau-RW, DYY-Landau-cw} have been made.

Another well-known strategy is to explore the small Knudsen number limits of kinetic equations. Typically, there are three distinct approaches to investigate such approximations:
\begin{itemize}
\item One approach is based on the Chapman-Enskog procedure \cite{Chapman-1990, Glassey-1996}, where the first-order approximation leads to the compressible Euler equations, and the second-order approximation results in the compressible Navier-Stokes system. In the context of the Boltzmann equation, the compressible Euler approximation was rigorously established in \cite{Nishida-1978, Ukai-Asano-1983} within the framework of analytical function spaces, and in \cite{Yu-2005, Xin-Zeng-2010, Huang-Wang-Yang-2010-1, Huang-Wang-Yang-2010-2, Huang-Wang-Wang-Yang} near the Riemann solutions of the one-dimensional Euler equations. Regarding the global validity of the compressible Navier-Stokes approximation for Boltzmann-type equations, references \cite{Lachowicz-1992, LYZ-2014, DL-2021} address this issue for the Boltzmann equation, and \cite{Duan-Yang-Yu-arXiv-VMB} considers the VMB system.

For the Landau equation, the validity of the Chapman-Enskog expansion was established in \cite{Duan-Yang-Yu-arXiv-Landau}.

\item  The second program focuses on the incompressible limit. For Boltzmann-type equations, the incompressible Navier-Stokes-Fourier limit has been extensively studied in various works such as \cite{Bardos-Golse-Levermore-1991, Bardos-Golse-Levermore-1993, Golse-S-04, Golse-S-09, Jiang-Masmoudi-CPAM-2017, Lions-Masmoudi-2001, Masmoudi-Saint-Raymond-2003, Saint-Raymond-2009, AS-2019} in the framework of renormalized solutions, and \cite{ELM, ELM2, Guo-CPAM-2006, Arsenio-ARMA-2012, Jang-ARMA-2009, Jiang-Xu-Zhao-Indiana-2018} for classical solutions. Additionally, investigations into the incompressible Euler limit for Boltzmann-type equations can be found in works like \cite{Lions-Masmoudi-2001, Saint-Raymond-ARMA-2003, Saint-Raymond-2009, Jang-Kim-2021}.

For the Landau equation, the incompressible Navier-Stokes-Fourier limit was explored in \cite{Guo-CPAM-2006, Rachid-2021}.

\item
The third program involves rigorously justifying the small Knudsen number limits of kinetic equations using the Hilbert expansion. For the Boltzmann equation, the validity of this expansion was established in \cite{Caflisch-CPAM-1980, Lachowicz-1987, Lachowicz-1991} through a decomposition of the remainder equation. However, it was noted that these solutions constructed using this method did not guarantee non-negativity. This issue was resolved in \cite{Guo-JJ-KRM-2009, Guo-JJ-CPAM-2010} with the help of the $L^2-L^\infty$ approach developed in \cite{Guo-ARMA-2010}. This approach has been successfully extended to address systems like the Vlasov-Poisson-Boltzmann system \cite{Guo-Jang-CMP-2010}, relativistic Boltzmann-type equations \cite{Speck-Strain-CMP-2011, Guo-Xiao-CMP-2021}, and more. For the initial boundary value problems associated with these equations, references like \cite{GHW-ARAM-2021, JLT-21-1, JLT-21-2} provide further insights.

It's important to note that the $L^2-L^\infty$ argument relies on the favorable properties of the linearized Boltzmann collision operator with cut-off potentials. Adapting this method to deal with the same problem for the VMB system with cut-off potentials, let alone the Landau-type equations and non-cutoff Boltzmann-type equations, appears challenging.

For the relativistic VML system, the problematic term analogous to \eqref{nordiff} shows linear growth only in momentum $p$. The global validity of its Hilbert expansion was established in \cite{Ouyang-Wu-Xiao-arxiv-2022-rLM} using a weighted energy method based on a time-momentum related exponential weight function $\exp\Big\{\frac{c_0\sqrt{1+|p|^2}}{\ln(\mathrm{e}+t)}\Big\}$, where $c_0>0$. However, it's important to note that this method is specific to relativistic kinetic equations. Addressing the same problem for non-relativistic Landau-type equations and non-cutoff Boltzmann-type equations remains an open problem. Considering that the VML  and the non-cutoff VMB systems serve as fundamental models for the dynamics of dilute charged particles, this paper concentrates on investigating these two systems.

\end{itemize}

It should be pointed out that partial of our results about the VML system was announced in \cite{Lei-Liu-Xiao-Zhao}.

\subsection{The structure}
The remainder of our paper is organized as follows. In Section \ref{h-VML}, we present the proof of the validity of the global Hilbert expansion for the VML system. Section \ref{sec-vmb} is dedicated to demonstrating the validity of the global Hilbert expansion for the non-cutoff VMB system. In Section \ref{sec-app}, we provide the methodology for determining the coefficients $[F_n, E_n, B_n] (1\leq n\leq 2k-1)$, particularly within the context of the non-cutoff VMB system.

\section{Global Hilbert expansion for the VML system} \label{h-VML}
In this section, our focus is on justifying the validity of the Hilbert expansion \eqref{expan} for the VML system \eqref{main1}. As discussed in the previous section, the proof relies on two main ingredients:
\begin{itemize}
	\item Determining the coefficients $[F_n, E_n, B_n]$ for $1\leq n\leq 2k-1$. These coefficients are obtained through the iteration equations \eqref{expan2}.
	\item Proving the well-posedness of the remainder terms $[f^\vps, E^\vps_R, B^\vps_R]$ and $h^\vps$ using the equations \eqref{VMLf} and \eqref{h-equation} with prescribed initial data given by \eqref{fEB-initial} and \eqref{hvps-initial}:
	\begin{align}\label{fEB-initial}
		f^\vps(0,x,v)=&f^\vps_0(x,v)=\vps^{-k}\FM^{-\frac{1}{2}}\left\{F^\vps_0(x,v)-\sum\limits_{n=0}^{2k-1}\vps^kF_{n,0}(x,v)\right\},\notag\\
		E_R^\vps(0,x)=&E^\vps_{R,0}(x)=\vps^{-k}\left\{E^\vps_0(x)-\sum\limits_{n=0}^{2k-1}\vps^kE_{n,0}(x)\right\},\\
		B_R^\vps(0,x)=&B^\vps_{R,0}(x)=\vps^{-k}\left\{B^\vps_0(x)-\sum\limits_{n=0}^{2k-1}\vps^kB_{n,0}(x)\right\},\notag
	\end{align}
	and
	\begin{align}\label{hvps-initial}
		h^\vps(0,x,v)=h^\vps_0(x,v)=\mu^{-\frac{1}{2}}\FM_{[\rho_0,u_0,T_0]}^{\frac{1}{2}}f^\vps_0(x,v).
	\end{align}
	These ingredients combined with the iteration equations \eqref{expan2} will provide the necessary framework to establish the Hilbert expansion for the VML system.
\end{itemize}

\setcounter{equation}{0}
\subsection{Preliminaries }
In this subsection, we will compile some estimates for the Landau collision operators $\mathcal{L}_{\mathbf{M}}$, $\Gamma_{\mathbf{M}}$, and $\mathcal{L}_d$ as defined in \eqref{lm-def}, \eqref{gm-def}, and \eqref{dL}, respectively. Although estimates for the operators $\mathcal{L}$ and $\Gamma$ have been well-established in previous works such as \cite{Guo-CMP-2002,Strain-Guo-ARMA-2008, Guo-JAMS-2011}, it is not straightforward to derive corresponding estimates in the scenario of the local Maxwellian $\mathbf{M}$. In this specific case, the formulation of Landau collision operators becomes more intricate, with additional terms involving the fluid velocity $u$. Apart from dealing with these additional terms, the main difficulty arises from the non-commutativity of the spatial derivatives $\partial^{\alpha}$ and the Landau collision operators, as well as the projection operator $\FP_{\mathbf{M}}$. However, due to the assumption of smallness in $u$ and the favorable time decay of $[\rho, u, T]$ in \eqref{em-decay}, the coercive property of the linearized collision operator $\mathcal{L}_{\mathbf{M}}$ is maintained. Trilinear estimates for $\Gamma_{\mathbf{M}}$, similar to those provided in \cite{Guo-CMP-2002, Strain-Guo-ARMA-2008, Guo-JAMS-2011}, can be derived through careful analysis and estimation.
			
			Before delving into the estimates for the Landau collision operators, we will first provide explicit expressions for $\mathcal{A}_{\mathbf{M}}$, $\mathcal{K}_{\mathbf{M}}$, $\mathcal{A}_d$, $\mathcal{K}_d$, and $\Gamma_{\mathbf{M}}$.

\begin{lemma}\label{AK0d} The expressions for the operators $\mathcal{A}_{\mathbf{M}}$, $\mathcal{K}_{\mathbf{M}}$, $\mathcal{A}_d$, $\mathcal{K}_d$, and $\Gamma_{\mathbf{M}}$ are provided as follows:
	For $\mathcal{A}_{\mathbf{M}}$ and $\mathcal{K}_{\mathbf{M}}$,
\begin{align}\label{AK0}
\left\{\begin{array}{rll}
\mathcal{A}_{\mathbf{M}}[f]
=&\partial_i\Big[\sigma^{ij}_{\mathbf{M}}\partial_jf
\Big]
-\sigma^{ij}_{\mathbf{M}}\frac{(v_i-u_i)(v_j-u_j)}{4R^2T^{2}}f+\partial_i\Big[\sigma^{ij}_{\mathbf{M}}\frac{v_j-u_j}{2RT}\Big]f,\\[2mm]
\mathcal{K}_{\mathbf{M}}[f]=&-\mathbf{M}^{-\frac{1}{2}}\partial_i\Big[\mathbf{M} \Big(\phi^{ij}\ast \mathbf{M}^{\frac{1}{2}}\Big(\partial_jf+\frac{v_j-u_j}{2RT}\Big)\Big)\Big].
\end{array}\right.
\end{align}
For $\mathcal{A}_d$ and $\mathcal{K}_d$,
\begin{align}\label{AKd}
\left\{\begin{array}{rll}
\mathcal{A}_d[h]
=&\partial_i\Big[\sigma^{ij}_{\mathbf{M}}\Big(\frac{-u_j}{ RT}+\frac{T_c-T}{RT_cT}v_j\Big)h
\Big]
-\sigma^{ij}_{\mathbf{M}}\frac{v_i}{ 2RT_c}\Big(\frac{-u_j}{ RT}+\frac{T_c-T}{RT_cT}v_j\Big)h\\[2mm]
&+\partial_i\Big[\sigma^{ij}_{(\mathbf{M}-\mu)}\partial_jh
\Big]
-\sigma^{ij}_{(\mathbf{M}-\mu)}\frac{v_iv_j}{4R^2T_cT}f+\partial_i\Big[\sigma^{ij}_{(\mathbf{M}-\mu)}\frac{v_j}{ 2RT_c}\Big]h,\\[2mm]
\mathcal{K}_d[h]=&\mu^{-\frac{1}{2}}\partial_i\Big[\Big(\frac{u_j}{ RT}\mathbf{M}+\frac{(T-T_c)\mu-T_c(\mathbf{M}-\mu)}{RT_cT}v_j\Big)\Big(\phi^{ij}\ast \mu^{\frac{1}{2}}h\Big)\Big]\\[2mm]
&- \mu^{-\frac{1}{2}}\partial_i\Big[\big(\mathbf{M}-\mu\big) \Big(\phi^{ij}\ast \mathbf{M}^{\frac{1}{2}}\Big(\partial_jh-\frac{v_j}{ 2RT_c}h\Big)\Big)\Big].
\end{array}\right.
\end{align}
For $\Gamma_{\mathbf{M}}$,
\begin{align*}
\Gamma_{\mathbf{M}}(f_1,f_2)
=&\partial_i\Big[\Big(\phi^{ij}\ast\mathbf{M}^{\frac{1}{2}}f_1\Big)
\partial_jf_2-\Big(\phi^{ij}\ast\mathbf{M}^{\frac{1}{2}}\partial_jf_1\Big)
f_2
\Big]
\\
&-\Big(\phi^{ij}\ast \frac{v_j-u_j}{2RT}\mathbf{M}^{\frac{1}{2}}f_1\Big)
\partial_jf_2+\Big(\phi^{ij}\ast \frac{v_j-u_j}{2RT}\mathbf{M}^{\frac{1}{2}}\partial_jf_1\Big)
f_2.\nonumber
\end{align*}
Here, $\partial_i = \partial_{v_i}$ for $i=1,2,3$, and
$$
\sigma^{ij}_{g}=\int_{{\R}^{3}}\phi^{ij}(v-v')g(v')dv'.
$$
\end{lemma}
\begin{proof}
The proof will be carried out in a similar manner as in \cite[Lemma 1, p. 395]{Guo-CMP-2002}. For brevity, we will only present \eqref{AKd}.

To this end, a direct calculation yields,
\begin{align*}
\mathcal{A}_d[h]=&\mu^{-\frac{1}{2}}
\mathcal{C}( \mathbf{M}-\mu,\mu^{\frac{1}{2}}h  )\\
=&\mu^{-\frac{1}{2}}\partial_i\Big[\sigma^{ij}_{(\mathbf{M}-\mu)}\mu^{\frac{1}{2}}\Big(\partial_jh-\frac{v_j}{ 2RT_c}h\Big)
-\mu^{\frac{1}{2}}h \Big(\phi^{ij}\ast \Big(\frac{-v_j+u_j}{ RT}\mathbf{M}+\frac{v_j}{ RT_c}\mu\Big)\Big)\Big]\\
=&\mu^{-\frac{1}{2}}\partial_i\Big[\mu^{\frac{1}{2}}h\Big(\phi^{ij}\ast \Big(\frac{-u_j}{ RT}+\frac{T_c-T}{RT_cT}v_j\Big)\mathbf{M}\Big)\Big]+\mu^{-\frac{1}{2}}\partial_i\Big[\mu^{\frac{1}{2}}\sigma^{ij}_{(\mathbf{M}-\mu)}
\Big(\partial_jh
+\frac{v_j}{ 2RT_c}h\Big)\Big]\\
=&\partial_i\Big[\sigma^{ij}_{\mathbf{M}}\Big(\frac{-u_j}{ RT}+\frac{T_c-T}{RT_cT}v_j\Big)h
\Big]
-\sigma^{ij}_{\mathbf{M}}\frac{v_i}{ 2RT_c}\Big(\frac{-u_j}{ RT}+\frac{T_c-T}{RT_cT}v_j\Big)h\\
&+\partial_i\Big[\sigma^{ij}_{(\mathbf{M}-\mu)}\partial_jh
\Big]
-\sigma^{ij}_{(\mathbf{M}-\mu)}\frac{v_iv_j}{4R^2T_cT}h+\partial_i\Big[\sigma^{ij}_{(\mathbf{M}-\mu)}\frac{v_j}{ 2RT_c}\Big]h.
\end{align*}
Similarly,
\begin{align*}
\mathcal{K}_d[h]=&\mu^{-\frac{1}{2}}
\mathcal{C}( \mu^{\frac{1}{2}}h, \mathbf{M}-\mu )\\
=&\mu^{-\frac{1}{2}}\partial_i\Big[\Big(\phi^{ij}\ast \mu^{\frac{1}{2}}h\Big)\Big(\frac{-v_j+u_j}{ RT}\mathbf{M}+\frac{v_j}{ RT_c}\mu\Big)\\
&-\big(\mathbf{M}-\mu\big) \Big(\phi^{ij}\ast \mu^{\frac{1}{2}}\Big(\partial_jh-\frac{v_j}{ 2RT_c}h\Big)\Big)\Big]\\
=&\mu^{-\frac{1}{2}}\partial_i\Big[\Big(\frac{u_j}{ RT}\mathbf{M}+\frac{(T-T_c)\mu-T_c(\mathbf{M}-\mu)}{RT_cT}v_j\Big)\Big(\phi^{ij}\ast \mu^{\frac{1}{2}}h\Big)\Big]\nonumber\\
&- \mu^{-\frac{1}{2}}\partial_i\Big[\big(\mathbf{M}-\mu\big) \Big(\phi^{ij}\ast \mathbf{M}^{\frac{1}{2}}\Big(\partial_jh-\frac{v_j}{ 2RT_c}h\Big)\Big)\Big].
\end{align*}
This is exactly \eqref{AKd} and the proof of Lemma \ref{AK0d} is complete.
\end{proof}

\subsubsection{ The linearized collision operators}\label{1cl}
The following lemma is quoted from \cite[Corollary 1, p. 399]{Guo-CMP-2002}, which states the lower bound of ${\bf D}-$norm defined by \eqref{D-norm}.
\begin{lemma}\label{lower norm}
Let the ${\bf D}$-norm be defined as in \eqref{D-norm}. Then there exists a constant $C > 0$ such that
\begin{equation}
|g|_{\bf D}^2\geq C\left\{\left|(1+|v|)^{-\frac{3}{2}}{\bf P}_v[\nabla_vg]\right|_{L^2}^2
+\left|(1+|v|)^{-\frac{1}{2}}({\bf I}-{\bf P}_v)[\nabla_vg]\right|_{L^2}^2
+\left|(1+|v|)^{-\frac{1}{2}}g\right|_{L^2}^2\right\},\notag
\end{equation}
where ${\bf P}_v$ is the projection defined by
$$
{\bf P}_vh_j=\sum{h_kv_k}\frac{v_j}{|v|^2}, \ \ 1\leq j\leq3,
$$
for any vector-valued function $h(v)=[h_1(v),h_2(v),h_3(v)]$.
\end{lemma}

We will now provide coercivity estimates for the linearized operators $\CL$ and $\CL_\FM$, along with an upper bound for the difference operator $\CL_d$.
\begin{lemma}
The linearized operators $\CL$ and $\CL_\FM$ satisfy the following coercivity estimates:
\begin{itemize}
\item [(i)] For $\CL$ as given by \eqref{lm-def}, we have
\begin{equation}
\big\lag \mathcal{L}h, h\big\rag\gtrsim |({\bf I}-{\bf P})h|_{\bf D}^2.\label{coLL}
\end{equation}
\item [(ii)] Let $\FM=\FM_{[\rho,u,T]},$ where $[\rho,u,T]$ is determined by Proposition \ref{em-ex-lem}, then there exists a constant $\de>0$ such that
\begin{align}
\big(\mathcal{L}_{\mathbf{M}}f, f\big)\geq\de \big| ({\bf I}-{\bf P}_{\mathbf{M}})f\big|_{\bf D}^2. \label{coLL0}
\end{align}
Moreover, for $\ell\geq 0$, $q\in [0,1)$, and $|\beta|>0$, we have
\begin{align}\label{wLLLM}
&\big(\partial^{\alpha}_{\beta}\mathcal{L}_{\mathbf{M}}f, \lag v\rag^{2(\ell-|\beta|)}\mathbf{M}^{-q} \partial_\beta^{\alpha}f\big)\nonumber\\
\geq& \de\big|\lag v\rag^{(\ell-|\beta|)}\mathbf{M}^{-q/2}\partial^{\alpha}_{\beta}f\big|_{\bf D}^2-\big(\eta+C\epsilon_1\big)\sum_{\al'\leq \alpha}\sum_{|\beta'|=|\beta|}\big|\lag v\rag^{(\ell-|\beta'|)}\mathbf{M}^{-q/2}\partial^{\al'}_{\beta'}f\big|_{\bf D}^2\\
&-C(\eta)\sum_{\al'\leq \alpha}\sum_{|\beta'|<|\beta|}\big|\lag v\rag^{(\ell-|\beta'|)}\mathbf{M}^{-q/2}\partial^{\al'}_{\beta'}f\big|^2_{\bf D}.\nonumber
\end{align}
For $|\beta|=0$, we have
\begin{align}\label{wLLLM0}
&\big(\partial^{\alpha}\mathcal{L}_{\mathbf{M}}f, \lag v\rag^{2\ell}\mathbf{M}^{-q} \partial^{\alpha}f\big)\nonumber\\
\geq& \de\big|\lag v\rag^{\ell}\mathbf{M}^{-q/2}\partial^{\alpha}f\big|_{\bf D}^2
-C\epsilon_1\sum_{\al'< \alpha}{\bf 1}_{|\alpha|\geq1}\big|\lag v\rag^{\ell}\mathbf{M}^{-q/2}\partial^{\al'}f\big|_{\bf D}^2\\
&-C(\eta)\sum_{\al'\leq \alpha}\big|\lag v\rag^{\ell}\mathbf{M}^{-q/2}\partial^{\al'}f\big|^2_{L^2(B_{C(\eta)})}.\nonumber
\end{align}
Here $\eta>0$ is some small constant,  $B_{C(\eta)}$ is the ball in ${\mathbb R}^3$ with center $0$ and radius $C(\eta)$.\\
In addition, for $|\alpha|>0$, there exists a polynomial $\CU(\cdot)$ with $\CU(0)=0$ such that
\begin{align}\label{coLLh}
\big\langle\partial^{\alpha}\mathcal{L}_{\mathbf{M}}f, \partial^{\alpha}f\big\rangle\geq&\frac{3\de}{4}\|\partial^{\alpha}({\bf I}-{\bf P}_{\mathbf{M}})f\|_{\bf D}^2-C\CU(\|\nabla_x[\rho,u,T]\|_{W^{|\alpha|-1,\infty}})\\
&\times\Big( \frac{1}{\varepsilon}\|({\bf I}-{\bf P}_{\mathbf{M}})f\|_{H^{|\alpha|-1}_{\bf D}}^2+\varepsilon\|f\|^2_{H^{|\alpha|}}\Big).\nonumber
\end{align}
\end{itemize}
\end{lemma}
\begin{proof} \eqref{coLL} is directly taken from \cite[Lemma 5, p. 400]{Guo-CMP-2002}. Since $[\rho, u, T]$ satisfies \eqref{em-decay}, \eqref{coLL0} follows once more from a similar argument as presented in the proof of \cite[Lemma 5, p. 400]{Guo-CMP-2002}.

We will now prove \eqref{wLLLM}. Recall that $-\CL_\FM=\mathcal{A}_{\mathbf{M}}+\mathcal{K_{\mathbf{M}}}$, where $\mathcal{A_{\mathbf{M}}}$ and $\mathcal{K_{\mathbf{M}}}$ are defined by \eqref{AK0}. Direct calculations yield the following expression:
\begin{align}
&\big(\partial^{\alpha}_{\beta}\mathcal{A}_{\mathbf{M}}f,\lag v\rag^{2(\ell-|\beta|)}\mathbf{M}^{-q} \partial^{\alpha}_\beta f\big)\notag\\
=&-\big|\lag v\rag^{(\ell-|\beta|)}\mathbf{M}^{-q/2}\partial^{\alpha}_{\beta}f\big|_{\tilde{{\bf D}}}^2\notag\\
&-{\bf 1}_{\al+\beta>0}\sum\limits_{0<\al'+\beta'\leq\al+\beta} C_\al^{\al'}C_\beta^{\beta'}
\big(\pa^{\al'}_{\beta'}\si^{ij}_\FM\pa^{\al-\al'}_{\beta-\beta'}\pa_jf,\lag v\rag^{2(\ell-|\beta|)}\mathbf{M}^{-q} \partial^{\alpha}_\beta\pa_if\big)\notag\\
&-\sum\limits_{\al'\leq\al,\beta'\leq\beta}C_\al^{\al'}C_\beta^{\beta'}
\big(\pa_i\left(\lag v\rag^{(\ell-|\beta|)}\mathbf{M}^{-q/2}\right)\pa^{\al'}_{\beta'}\si^{ij}_\FM\pa^{\al-\al'}_{\beta-\beta'}\pa_jf, \partial^{\alpha}_\beta\pa_if\big)\notag\\
&-{\bf 1}_{\al+\beta>0}\sum\limits_{0<\al'+\beta'\leq\al+\beta}C_\al^{\al'}C_\beta^{\beta'}
\big(\pa^{\al'}_{\beta'}\left(\sigma^{ij}_{\mathbf{M}}\frac{(v_i-u_i)(v_j-u_j)}{4R^2T^{2}}\right)\pa^{\al-\al'}_{\beta-\beta'}\pa_jf,\lag v\rag^{2(\ell-|\beta|)}\mathbf{M}^{-q} \partial^{\alpha}_\beta\pa_if\big)\notag\\
&+\sum\limits_{\al'+\beta'\leq\al+\beta}C_\al^{\al'}C_\beta^{\beta'}
\big(\pa^{\al'}_{\beta'}\left(\sigma^{ij}_{\mathbf{M}}\frac{v_j-u_j}{2RT}\right)\pa^{\al-\al'}_{\beta-\beta'}\pa_jf,\lag v\rag^{2(\ell-|\beta|)}\mathbf{M}^{-q} \partial^{\alpha}_\beta\pa_if\big).\label{FA-ep}
\end{align}
Here, we define:
\begin{align}
\big|\lag v\rag^{(\ell-|\beta|)}\mathbf{M}^{-q/2}\partial^{\alpha}_{\beta}f\big|_{\tilde{{\bf D}}}^2=&
\int_{{\mathbb R}^3}\lag v\rag^{2(\ell-|\beta|)}\mathbf{M}^{-q}\sigma_\FM^{ij}\partial_{v_i}\partial^{\alpha}_{\beta} f\partial_{v_j}\partial^{\alpha}_{\beta}f\, dv
\notag\\&+\frac{1}{4R^2T^{2}}\int_{{\mathbb R}^3}\lag v\rag^{2(\ell-|\beta|)}\mathbf{M}^{-q}\sigma_\FM^{ij}(v_i-u_i)(v_j-u_j)|\partial^{\alpha}_{\beta}f|^2\, dv.\notag
\end{align}

Considering that $[\rho,u,T]$ satisfies \eqref{em-decay}, we can establish the following inequality:
\begin{align}
\big|\lag v\rag^{(\ell-|\beta|)}\mathbf{M}^{-q/2}\partial^{\alpha}_{\beta}f\big|_{\tilde{{\bf D}}}^2\gtrsim&
\big|\lag v\rag^{(\ell-|\beta|)}\mathbf{M}^{-q/2}\partial^{\alpha}_{\beta}f\big|_{\bf D}^2-C\epsilon_1\big|\lag v\rag^{(\ell-|\beta|)}\mathbf{M}^{-q/2}\partial^{\alpha}_{\beta}f\big|_{\bf D}^2.\notag
\end{align}

Subsequently, by employing \eqref{em-decay} and conducting calculations analogous to those found in the proof of \cite[Lemma 8, p. 319]{Strain-Guo-ARMA-2008}, we observe that the remaining terms on the R.H.S. of \eqref{FA-ep} can be bounded by
\begin{align}
\big(\eta+C\epsilon_1\big)\sum_{\al'\leq \alpha}\sum_{|\beta'|=|\beta|}\big|\lag v\rag^{(\ell-|\beta'|)}\mathbf{M}^{-q/2}\partial^{\al'}_{\beta'}f\big|_{\bf D}^2+C(\eta)\sum_{\al'\leq \alpha}\sum_{|\beta'|<|\beta|}\big|\lag v\rag^{(\ell-|\beta'|)}\mathbf{M}^{-q/2}\partial^{\al'}_{\beta'}f\big|^2_{\bf D},\notag
\end{align}
where $0<\eta\ll1.$

Moving on to estimate $\CK_\FM$, we start with \eqref{AK0}, from which we derive the following expression:
\begin{align}
&\big(\partial^{\alpha}_{\beta}\CK_{\mathbf{M}}f,\lag v\rag^{2(\ell-|\beta|)}\mathbf{M}^{-q} \partial^{\alpha}_\beta f\big)\notag\\
=&-\left(\partial_i\left[\lag v\rag^{2(\ell-|\beta|)}\mathbf{M}^{-q} \partial^{\alpha}_{\beta}\left\{\mathbf{M}^{\frac{1}{2}} \Big(\phi^{ij}\ast \mathbf{M}^{\frac{1}{2}}\Big(\partial_jf+\frac{v_j-u_j}{2RT}f\Big)\Big)\right\}\right]
,\partial^{\alpha}_\beta f\right)\notag\\
&+\left(\partial_i\left[\lag v\rag^{2(\ell-|\beta|)}\mathbf{M}^{-q} \right]\partial^{\alpha}_{\beta}\left\{\mathbf{M}^{\frac{1}{2}} \Big(\phi^{ij}\ast \mathbf{M}^{\frac{1}{2}}\Big(\partial_jf+\frac{v_j-u_j}{2RT}f\Big)\Big)\right\}
,\partial^{\alpha}_\beta f\right)\notag\\
&+\left(\lag v\rag^{2(\ell-|\beta|)}\mathbf{M}^{-q} \partial^{\alpha}_{\beta}\left\{\frac{v_i-u_i}{2RT}\mathbf{M}^{\frac{1}{2}} \Big(\phi^{ij}\ast \mathbf{M}^{\frac{1}{2}}\Big(\partial_jf+\frac{v_j-u_j}{2RT}f\Big)\Big)\right\}
,\partial^{\alpha}_\beta f\right).\notag
\end{align}

As for any finite $l\in\mathbb{N}$ and any multi-indices $\alpha, \alpha', \beta',$ it can be deduced from \eqref{em-decay} that
\[\Big|\na_v^l\left\{\lag v\rag^{(\ell-|\beta|)}\mathbf{M}^{-q/2}\right\}\FM^{1/2}\Big|\leq C\FM^{\frac{1-q}{4}},\]
and
\[\Big|\pa_{\beta'}^{\al'}\left\{\frac{v_i-u_i}{2RT}\mathbf{M}^{\frac{1}{2}}\right\}\Big|\leq C\FM^{\frac{1}{4}}.\]
Using the same reasoning that was employed to derive estimates for ${\bf K}$ in \cite[Lemma 8, p. 319]{Strain-Guo-ARMA-2008}, we can express it as
\begin{align}
\big(\partial^{\alpha}_{\beta}&\CK_{\mathbf{M}}f,\lag v\rag^{2(\ell-|\beta|)}\mathbf{M}^{-q} \partial^{\alpha}_\beta f\big)\notag\\
\leq&\left\{\eta\big|\lag v\rag^{(\ell-|\beta|)}\partial^{\alpha}_{\beta}f\big|_{\bf D}+C(\eta)\big|\lag v\rag^{(\ell-|\beta|)}f\big|_{L^2(B_{C(\eta)})}\right\}\big|\lag v\rag^{(\ell-|\beta|)}\mathbf{M}^{-q/2}\partial^{\alpha}_{\beta}f\big|_{\bf D}.\notag
\end{align}

When combining the above estimates for $\CA_\FM$ and $\CK_\FM$, we can conclude that \eqref{wLLLM} holds true. \eqref{wLLLM0} can be proven similarly as described above.

Regarding \eqref{coLLh}, after applying \eqref{coLL0}, we can derive the following:
\begin{align*}
&\big\langle \partial^{\alpha}\mathcal{L}_{\mathbf{M}}f,\partial^{\alpha}_\beta f\big\rangle\\
=&\big\langle \mathcal{L}_{\mathbf{M}}\partial^{\alpha}({\bf I}-{\bf P}_{\mathbf{M}})f,({\bf I}-{\bf P}_{\mathbf{M}})\partial^{\alpha}f\big\rangle
+\big\langle \Lbrack\partial^{\alpha},\mathcal{L}_{\mathbf{M}}\Rbrack ({\bf I}-{\bf P}_{\mathbf{M}})f,\partial^{\alpha}f\big\rangle\\
=&\big\langle \mathcal{L}_{\mathbf{M}}\partial^{\alpha}({\bf I}-{\bf P}_{\mathbf{M}})f,\partial^{\alpha}({\bf I}-{\bf P}_{\mathbf{M}})f+\Lbrack\partial^{\alpha},{{\bf P}_{\mathbf{M}}}\Rbrack f\big\rangle
+\big\langle \Lbrack\partial^{\alpha},\mathcal{L}_{\mathbf{M}}\Rbrack ({\bf I}-{\bf P}_{\mathbf{M}})f,\partial^{\alpha}f\big\rangle\\
\geq&\delta\|\partial^{\alpha}({\bf I}-{\bf P}_{\mathbf{M}})f\|_{\bf D}^2+\big\langle \mathcal{L}_{\mathbf{M}}\partial^{\alpha}({\bf I}-{\bf P}_{\mathbf{M}})f,\Lbrack\partial^{\alpha},{{\bf P}_{\mathbf{M}}}\Rbrack f\big\rangle\\&+\big\langle \Lbrack \partial^{\alpha},\mathcal{L}_{\mathbf{M}}\Rbrack ({\bf I}-{\bf P}_{\mathbf{M}})f,\partial^{\alpha}f\big\rangle\\
\geq&\delta\|\partial^{\alpha}({\bf I}-{\bf P}_{\mathbf{M}})f\|_{\bf D}^2-C\|\Lbrack\partial^{\alpha},{{\bf P}_{\mathbf{M}}}\Rbrack f\|_{\bf D}^2+\big\langle \Lbrack \partial^{\alpha},\mathcal{L}_{\mathbf{M}}\Rbrack ({\bf I}-{\bf P}_{\mathbf{M}})f,\partial^{\alpha}f\big\rangle.
\end{align*}
Here, $\Lbrack A,B\Rbrack=AB-BA$ denotes the commutator of operators $A$ and $B$.

Noting that if $|\alpha| > 0$, $\partial^{\alpha}$ acts on the local Maxwellian $\mathbf{M}$ at least once in the commutators $\Lbrack\partial^{\alpha},{{\bf P}{\mathbf{M}}}\Rbrack$ and $\Lbrack \partial^{\alpha},\mathcal{L}{\mathbf{M}}\Rbrack$. By the definition of the operator ${{\bf P}_{\mathbf{M}}}$, we can conclude that there exists a polynomial $\CU(\cdot)$ with $\CU(0) = 0$ such that
\begin{align*}
\|\Lbrack\partial^{\alpha},{{\bf P}_{\mathbf{M}}}\Rbrack f\|_{\bf D}^2\lesssim&
\CU(\|\nabla_x[\rho,u,T]\|_{W^{|\alpha|-1,\infty}})\|{{\bf P}_{\mathbf{M}}}f\|^2_{H^{|\alpha|-1}}\\
\lesssim&
\CU(\|\nabla_x[\rho,u,T]\|_{W^{|\alpha|-1,\infty}})
\|f\|^2_{H^{|\alpha|-1}}.
\end{align*}

Furthermore, by using \eqref{AK0} and Cauchy-Schwarz's inequality with $\varepsilon > 0$, we obtain
\begin{align*}
&\big|\big\langle \Lbrack \partial^{\alpha},\mathcal{L}_{\mathbf{M}}\Rbrack ({\bf I}-{\bf P}_{\mathbf{M}})f,\partial^{\alpha}f\big\rangle\big|\\
\lesssim &
\CU(\|\nabla_x[\rho,u,T]\|_{W^{|\alpha|-1,\infty}})\|({\bf I}-{\bf P}_{\mathbf{M}})f\|_{H^{|\alpha|-1}_{\bf D}}\|\partial^{\alpha}f\|_{\bf D}\\
\lesssim&
\CU(\|\nabla_x[\rho,u,T]\|_{W^{|\alpha|-1,\infty}})\|({\bf I}-{\bf P}_{\mathbf{M}})f\|_{H^{|\alpha|-1}_{\bf D}}
\big(\|\partial^{\alpha}({\bf I}-{\bf P}_{\mathbf{M}})f\|_{\bf D}+\|\partial^{\alpha}{{\bf P}_{\mathbf{M}}}f\|_{\bf D}\big)\\
\lesssim& \epsilon_1\|\partial^{\alpha}({\bf I}-{\bf P}_{\mathbf{M}})f\|_{\bf D}^2+
\CU(\|\nabla_x[\rho,u,T]\|_{W^{|\alpha|-1,\infty}})
\Big( \frac{1}{\varepsilon}\|({\bf I}-{\bf P}_{\mathbf{M}})f\|_{H^{|\alpha|-1}_{\bf D}}^2+\varepsilon\|f\|^2_{H^{|\alpha|}}\Big).
\end{align*}

Combining the aforementioned estimates yields the expression for \eqref{coLLh}.
\end{proof}

\subsubsection{ The nonlinear collision operators}
Recalling that the nonlinear collision operators $\Gamma$ and $\Gamma_{\mathbf{M}}$ are defined as
\begin{align*}
\Gamma(f,g)=\mu^{-\frac{1}{2}} \mathcal{C}( \mu^{\frac{1}{2}} f, \mu^{\frac{1}{2}} g)
\end{align*}
and
\begin{align*}
\Gamma_{\mathbf{M}}(f,g)=&\mathbf{M}^{-\frac{1}{2}} \mathcal{C}( \mathbf{M}^{\frac{1}{2}} f, \mathbf{M}^{\frac{1}{2}} g)
\end{align*}
respectively, we now focus on the trilinear estimates of these operators. The first result is concerned with estimates of $\Gamma$ and $\Gamma_{\mathbf{M}}$  without velocity weight.
\begin{lemma}\label{GLL0} It holds that
\begin{align}\label{GLL}
\big|\big(\partial^{\alpha}\Gamma(f,g),\partial^{\alpha} h\big)\big|
\lesssim& \sum_{\al'+\tilde{\alpha}= \alpha}\Big(\big|\partial^{\al'}f\big|_{L^2}\big|\partial^{\tilde{\alpha}}g\big|_{\bf D}+\big|\partial^{\al'}f\big|_{\bf D}
\big|\partial^{\tilde{\alpha}}g\big|_{L^2}\Big)\big|\partial^{\alpha}h\big|_{\bf D}
\end{align}
and
\begin{align}\label{GLLM}
&\big|\big(\partial^{\alpha}\Gamma_{\mathbf{M}}(f,g), \partial^{\alpha}h\big)\big|\nonumber\\
\lesssim& \sum_{\al'+\tilde{\alpha}= \alpha}\Big(\big|\partial^{\al'}f\big|_{L^2}\big|\partial^{\tilde{\alpha}}g\big|_{\bf D}+\big|\partial^{\al'}f\big|_{\bf D}
\big|\partial^{\tilde{\alpha}}g\big|_{L^2}\Big)\big|\partial^{\alpha}h\big|_{\bf D}\\
&+\sum_{\al'+\tilde{\alpha}+\overline{\alpha}= \alpha} {\bf 1}_{|\overline{\alpha}|\geq1}\CU(\|\nabla_x[\rho,u,T]\|_{W^{|\overline{\alpha}|-1,\infty}})\Big(\big|\partial^{\al'}f\big|_{L^2}\big|\partial^{\tilde{\alpha}}g\big|_{\bf D}+\big|\partial^{\al'}f\big|_{\bf D}
\big|\partial^{\tilde{\alpha}}g\big|_{L^2}\Big)\big|\partial^{\alpha}h\big|_{\bf D}.\nonumber
\end{align}
Here and in the sequel, ${\bf 1}_{\Omega}$ is the usual indicator function of the set $\Omega$.
\end{lemma}
\begin{proof} Using \cite[Lemma 3, p. 406]{Guo-CMP-2002} with $\ta=0$ and $\beta=0$ gives us \eqref{GLL}. Since \eqref{GLLM} is a special case of \eqref{GGGLm} below, we'll omit the details for brevity. This concludes the proof of Lemma \ref{GLL0}.
\end{proof}

The following lemma is dedicated to trilinear estimates involving the nonlinear operators $\Gamma$ and $\Gamma_{\mathbf{M}}$,  incorporating velocity weights and velocity derivatives.
\begin{lemma}\label{wggmL}
\begin{itemize}
\item [(i)]  For any $\ell\geq|\alpha|$, the following inequality holds
\begin{align}\label{wGGL}
&\big|\big(\partial^{\alpha}\Gamma(f,g), w^2_{|\alpha|}\partial^{\alpha}h\big)\big|\\
\lesssim&  \sum_{\al'\leq\alpha} \Big(\big|\lag v\rag^{\ell-|\al|}\partial^{\al'}f\big|_{L^2}|w_{|\alpha|}\partial^{\al-\al'}g|_{\bf D}
+\big|\lag v\rag^{\ell-|\al|}\partial^{\al'}f\big|_{\bf D}\big|w_{|\alpha|}\partial^{\al-\al'}g|_{L^2}\Big)
\big|w_{|\alpha|}\partial^{\alpha}h|_{\bf D}.\nonumber
\end{align}
\item [(ii)]Let $q\in[0,1)$ and $\ell\geq0$, the following inequality holds
\begin{align}\label{GGGLm}
&\big|\big(\partial^{\alpha}_{\beta}\Gamma_{\mathbf{M}}(f,g), {\lag v\rag}^{2(\ell-|\beta|)}\mathbf{M}^{-q}\partial^{\alpha}_{\beta}h\big)\big|\\
\lesssim& \sum_{\al'+\al''= \alpha\atop{\beta'+\beta''\leq \beta}} \Big(\big|{\lag v\rag}^{(\ell-|\beta'|)}\partial^{\al'}_{\beta'}f\big|_{L^2}\big|{\lag v\rag}^{(\ell-|\beta''|)}\mathbf{M}^{-q/2}\partial^{\al''}_{\beta''}g\big|_{\bf D}\nonumber\\
&\qquad\qquad+\big|{\lag v\rag}^{(\ell-|\beta'|)}\partial^{\al'}_{\beta'}f\big|_{\bf D}
\big|{\lag v\rag}^{(\ell-|\tilde{\beta}|)}\mathbf{M}^{-q/2}\partial^{\tilde{\alpha}}_{\tilde{\beta}}g\big|_{L^2}\Big)\big|{\lag v\rag}^{(\ell-|\beta|)}\mathbf{M}^{-q/2}\partial^{\alpha}_{\beta}h\big|_{\bf D}\nonumber\\
&+\sum_{|\al'+\al''|\leq|\alpha|-m-1\atop{\beta'+\beta''\leq \beta}}{\bf 1}_{m\geq0}\CU(\|\nabla_x[\rho,u,T]\|_{W^{m,\infty}})\Big(\big|{\lag v\rag}^{(\ell-|\beta'|)}\partial^{\al'}_{\beta'}f\big|_{L^2}\big|{\lag v\rag}^{(\ell-|\beta''|)}\mathbf{M}^{-q/2}\partial^{\al''}_{\beta''}g\big|_{\bf D}\notag\\&\qquad\qquad+\big|{\lag v\rag}^{(\ell-|\beta'|)}\partial^{\al'}_{\beta'}f\big|_{\bf D}
\big|{\lag v\rag}^{(\ell-|\beta''|)}\mathbf{M}^{-q/2}\partial^{\al''}_{\beta''}g\big|_{L^2}\Big)\big|{\lag v\rag}^{(\ell-|\beta|)}\mathbf{M}^{-q/2}\partial^{\alpha}_{\beta}h\big|_{\bf D}.\nonumber
\end{align}
\end{itemize}
\end{lemma}
\begin{proof} \eqref{wGGL} directly follows from \cite[Lemma 10, p. 327]{Strain-Guo-ARMA-2008}.
	We will only prove \eqref{GGGLm} here. To do this, we first denote $W={\lag v\rag}^{(\ell-|\beta|)}\mathbf{M}^{-q/2}$ for simplicity, and then write
	\begin{align}
		\big|\big(\partial^{\alpha}_{\beta}&\Gamma_{\mathbf{M}}(f,g),W^2\partial^{\alpha}_{\beta}h\big)\big|=\sum C_\al^{\al',\al''}C_\beta^{\beta'}G_{\al',\al'',\beta'},\notag
	\end{align}
	where $G_{\al',\al'',\beta'}=\sum\limits_{i=1}^6G_{\al',\al'',\beta'}^i$ with
	\begin{align}
		G_{\al',\al'',\beta'}^1=-\left(W^2\left(\phi^{ij}\ast\pa_{\beta'}\left\{\pa^{\al''}
		\left\{\mathbf{M}^{\frac{1}{2}}\right\}\pa^{\al'-\al''}f\right\}\right)
		\partial_j\pa_{\beta-\beta'}^{\al-\al'}g,\partial_i\partial^{\alpha}_{\beta}h\right),\notag
	\end{align}
\begin{align}
G_{\al',\al'',\beta'}^2=\left(W^2\left(\phi^{ij}\ast\pa_{\beta'}\left\{\pa^{\al''}
\left\{\mathbf{M}^{\frac{1}{2}}\right\}\pa^{\al'-\al''}\partial_jf\right\}\right)
\pa_{\beta-\beta'}^{\al-\al'}g,\partial_i\partial^{\alpha}_{\beta}h\right),\notag
\end{align}
\begin{align}
G_{\al',\al'',\beta'}^3=-\left(W^2\left(\phi^{ij}\ast\pa_{\beta'}
\left\{\pa^{\al''}\left\{\mathbf{M}^{\frac{1}{2}}\frac{v_i-u_i}{2RT}\right\}\pa^{\al'-\al''}f\right\}\right)
\pa_{\beta-\beta'}^{\al-\al'}\partial_jg,\partial^{\alpha}_{\beta}h\right),\notag
\end{align}
\begin{align}
G_{\al',\al'',\beta'}^4=\left(W^2\left(\phi^{ij}\ast\pa_{\beta'}
\left\{\pa^{\al''}\left\{\mathbf{M}^{\frac{1}{2}}\frac{v_i-u_i}{2RT}\right\}\pa^{\al'-\al''}\pa_jf\right\}\right)
\pa_{\beta-\beta'}^{\al-\al'}g,\partial^{\alpha}_{\beta}h\right),\notag
\end{align}
\begin{align}
G_{\al',\al'',\beta'}^5=-\left(\pa_iW^2\left(\phi^{ij}\ast\pa_{\beta'}\left\{\pa^{\al''}
\left\{\mathbf{M}^{\frac{1}{2}}\right\}\pa^{\al'-\al''}f\right\}\right)
\partial_j\pa_{\beta-\beta'}^{\al-\al'}g,\partial^{\alpha}_{\beta}h\right),\notag
\end{align}
\begin{align}
G_{\al',\al'',\beta'}^6=\left(\partial_iW^2\left(\phi^{ij}\ast\pa_{\beta'}\left\{\pa^{\al''}
\left\{\mathbf{M}^{\frac{1}{2}}\right\}\pa^{\al'-\al''}\partial_jf\right\}\right)
\pa_{\beta-\beta'}^{\al-\al'}g,\partial^{\alpha}_{\beta}h\right).\notag
\end{align}

If $\al''=0$, since it follows from \eqref{tt1} that
$$
\rho\sim1,\ u\sim0,\ T\sim T_c,
$$
then performing the same calculation as in the proof of \cite[Lemma 10, p. 327]{Strain-Guo-ARMA-2008}, one can see that all of $G_{\alpha',\alpha'',\beta'}^i$ $(1\leq i\leq 6)$ can be controlled by the first part on the R.H.S. of \eqref{GGGLm}.

If $|\alpha''|>0$, using \eqref{em-decay} again, we obtain
$$
\left|\pa^{\al''}
\mathbf{M}^{\frac{1}{2}}\right|\leq C\CU(\|\nabla_x[\rho,u,T]\|_{W^{|\al''|,\infty}})\FM^{\frac{3}{8}}
$$
and
$$
 \left|\pa^{\al''}\left\{\mathbf{M}^{\frac{1}{2}}\frac{v_i-u_i}{2RT}\right\}\right|\leq C\CU(\|\nabla_x[\rho,u,T]\|_{W^{|\al''|,\infty}})\FM^{\frac{3}{8}}.
$$

From the above estimates and by repeating the similar calculations as those used to derive \cite[Lemma 10, p. 327]{Strain-Guo-ARMA-2008}, one can see that all of $G_{\alpha',\alpha'',\beta'}^i$ $(1\leq i\leq 6)$ can be bounded by the second part of the right-hand side of \eqref{GGGLm}.
This completes the proof of Lemma \ref{wggmL}.
\end{proof}

\subsubsection{ The difference operator $\CL_d$}
The subsequent lemma focuses on estimates associated with the difference operator $\CL_d.$
\begin{lemma}\label{WGLLL0}
	Let $\rho(t,x), u(t,x)$, and $T(t,x)$ satisfy \eqref{em-decay}, and assume that \eqref{tt1} is valid. Define the weight function $w_i$ as in \eqref{tt 01}. Then the following estimates hold:
	\begin{itemize}
		\item [i)] For any multi-index $\alpha$ and the corresponding weight $w_{|\alpha|}$, we have
		\begin{align}
			\big|\big\langle\partial^{\alpha}\mathcal{L}_d[h], w_{|\alpha|}^2 \partial^{\alpha}h\big\rangle\big|\lesssim& \epsilon_1\big\|wh\big\|_{H^{|\alpha|}_{\bf D}}^2.\label{wGLLd}
		\end{align}
	\item [ii)]If we set $\sqrt{\FM}f=\sqrt{\mu}h$, then it holds that
	\begin{align}
		\big|\big\langle\partial^{\alpha}\mathcal{L}_dh, \partial^{\alpha}h\big\rangle\big|\lesssim& \epsilon_1\Big(\big\|({\bf I}-{\bf P})h\|_{H^{|\alpha|}_{\bf D}}^2+\big\|f\|_{H^{|\alpha|}_{\bf D}}^2\Big).\label{LLdh}
	\end{align}
	\item [iii)] There exists a positive constant $\delta$ such that for $\alpha$ and the corresponding weight $w_{|\alpha|}$:
	\begin{align}
		\big(\partial^{\alpha}\mathcal{L}[h], w_{|\alpha|}^2 \partial^{\alpha}h\big)\geq& \de\big|w_{|\alpha|}\partial^{\alpha}h\big|_{\bf D}^2-C\sum_{\al'\leq \alpha}|\partial^{\al'}f|^2_{L^2}.\label{wLLL}
	\end{align}
	\end{itemize}
	
\end{lemma}
\begin{proof}
By \eqref{dL}, we can rewrite it as
\begin{align}\label{dL-gaL}
-\CL_df=\Ga\left(\frac{\FM-\mu}{\sqrt{\mu}},f\right)+\Ga\left(f,\frac{\FM-\mu}{\sqrt{\mu}}\right).
\end{align}
From \eqref{tt1}, we have
\begin{align}\label{alphaxL}
 |\partial^{\alpha}\mathbf{M}|\lesssim \epsilon_1\mu^{5/6}
 \end{align}
for any multi-index $\al$ with $|\al|>0.$ Moreover, using \eqref{em-decay} and the mean value theorem, we get
\begin{align}\label{mvtL}
|\mathbf{M}-\mu|=\left|[\na_{[\rho,u,T]}\FM](\bar{\rho},\bar{u},\bar{T})\cdot(\rho-1,u,T-T_c)\right|
\lesssim\epsilon_1\mu^{5/6},
\end{align}
where $\bar{\rho}\in(\min{\rho,1},\max{\rho,1})$, $\bar{u}_j\in(\min{u_j,0},\max{0,u_j})$ with $1\leq j\leq3$, and $\bar{T}\in(T_c,T).$

Combining \eqref{alphaxL} and \eqref{mvtL}, we obtain
\begin{equation}\label{mmbL}
\Big|\partial^{\alpha}\Big[\mu^{-\frac{1}{2}}(\mathbf{M}-\mu)\Big]\Big|\lesssim \epsilon_1\mu^{\frac{1}{3}}.
\end{equation}
With \eqref{dL-gaL}, \eqref{mmbL}, and \eqref{wGGL} from Lemma \ref{wggmL}, we can derive \eqref{wGLLd}.

Once \eqref{wGLLd} is established, we can derive \eqref{LLdh} from
\begin{align}
 h=\mu^{-\frac{1}{2}}\mathbf{M}^{\frac{1}{2}}f\label{fh-rlL}
\end{align}
and the fact that
\begin{align*}
\left\|\partial^{\alpha}{\bf P}\left[h^{\varepsilon}\right]\right\|^2_{\bf D}\lesssim \left\|\partial^{\alpha}{\bf P}\left[h^{\varepsilon}\right]\right\|^2 \lesssim \left\|f^{\varepsilon}\right\|^2_{H^{|\alpha|}}.
\end{align*}

To complete the proof of \eqref{wLLL}, we use \cite[Lemma 9, p. 323]{Strain-Guo-ARMA-2008} with $q=1/(8\ln(\mathrm{e}+t))$ to obtain
\begin{align*}
\big(\partial^{\alpha}\mathcal{L}h, w_{|\alpha|}^2 \partial^{\alpha}h\big)\geq& \de\big|w_{|\alpha|}\partial^{\alpha}h\big|_{\bf D}^2-C(\eta)\big|\partial^{\alpha}h\big|^2_{L^2(B_{C(\eta)})}.
\end{align*}
This, combined with \eqref{fh-rlL}, yields \eqref{wLLL}. The proof of Lemma \ref{WGLLL0} is now complete.
\end{proof}

\subsection{Estimates of $f^{\varepsilon}, E_R^{\varepsilon}, B_R^{\varepsilon}$}
In this subsection, our goal is to derive estimates for the remainders $f^{\varepsilon}, E_R^{\varepsilon}$, and $B_R^{\varepsilon}$, which satisfy the equations \eqref{VMLf} and \eqref{fM-2} with the initial data specified in \eqref{fEB-initial}.

The global existence of \eqref{VMLf}, \eqref{fM-2}, and \eqref{h-equation} with initial condition given by
 \begin{align}
\left[f^\vps(0,x,v),E_R^\vps(0,x),B_R^\vps(0,x), h^\vps(0,x,v)\right]=\left[f^\vps_0(x,v),E^\vps_{R,0}(x),B^\vps_{R,0}(x), h^\vps_0(x,v)\right]\label{VML-id-pt}
\end{align}
can be constructed through the local-in-time existence, the {\it a priori} energy estimate and the continuation argument.
Here, we focus on proving the a priori energy estimate \eqref{TVML1} under the {\it a priori} assumption:
\begin{align}\label{aps-vml}
    \sup_{0\leq t\leq \vps^{-\kappa}}\mathcal{E}(t)\lesssim \varepsilon^{-\frac{1}{2}},\qquad \kappa=\frac{1}{3}.
\end{align}
Before deriving the estimates for the remainders $f^{\varepsilon}, E_R^{\varepsilon}$, and $B_R^{\varepsilon}$, we first establish crucial estimates in the following lemmas. The first lemma deals with the velocity growth terms involving the electromagnetic field without velocity derivatives. Our strategy involves partitioning ${\mathbb R}^3_v$ into $\vps$-dependent low-velocity and high-velocity regions. Subsequently, we estimate each region separately, utilizing the dissipation norms $\varepsilon^{-1}\|({\bf I}-{\bf P}_{\mathbf{M}})[f^{\varepsilon}]\|_{H^i_{\bf D}}^2 (i\leq 2)$ and exploiting the relationship between $f^{\varepsilon}$ and $h^{\varepsilon}$.

\begin{lemma}\label{dxlm VML}
Assume that $f^{\varepsilon}, E_R^{\varepsilon}, B_R^{\varepsilon}$ and $h^{\varepsilon}$ is the smooth solution to the Cauchy problem \eqref{VMLf},\eqref{fM-2}, \eqref{h-equation} and \eqref{VML-id-pt} for the VML system \eqref{main1} and satisfies \eqref{aps-vml}, then it holds that
\begin{align}\label{growdxi VML}
&\Big|\Big\langle \nabla_x^i\Big(f^{\varepsilon}\mathbf{M}^{-\frac{1}{2}}\big(\partial_t+v\cdot\nabla_x\big)\mathbf{M}^{\frac{1}{2}}\Big), 4\pi RT\nabla_x^if^{\varepsilon}\Big\rangle\Big|
    \lesssim\,\epsilon_1(1+t)^{-p_0}{\bf Z}_{1,i}(t),
   \end{align}
 \begin{align}\label{EBdxi}
   &\Big|\Big\langle \nabla_x^i\Big[\Big(E+v \times B \Big)\cdot\frac{v-u }{ T}f^{\varepsilon}\Big], 2\pi T \nabla_x^if^{\varepsilon}\Big\rangle\Big|
    \lesssim\,\epsilon_1(1+t)^{-p_0}{\bf Z}_{1,i}(t),
   \end{align}
  \begin{align}\label{EBRdxi}
  &\varepsilon^k\Big|\Big\langle \nabla_x^i\Big[\big(E_R^{\varepsilon}+v \times B_R^{\varepsilon}\big) \cdot\frac{(v-u)}{ T}f^{\varepsilon}\Big], 2\pi T \nabla_x^if^{\varepsilon}\Big\rangle\Big|
    \lesssim\,\varepsilon {\bf Z}_{1,i}(t),
   \end{align}
 and   \begin{align}\label{EBndxi}
  &\sum_{n=1}^{2k-1}\varepsilon^n\Big|\Big\langle \nabla^i_x\Big[\Big(E_n+v \times B_n \Big)\cdot\frac{(v-u)}{ T}f^{\varepsilon}\Big], 2\pi T \nabla_x^if^{\varepsilon}\Big\rangle\Big|
    \lesssim\,\varepsilon^{\frac{2}{3}}{\bf Z}_{1,i}(t),
   \end{align}
   for $i= 1, 2$, and $0\leq t\leq \vps^{-\kappa}$ with $\kappa=\frac{1}{3}$, where
   $${\bf Z}_{1,i}(t):=\|f^{\varepsilon}\|^2_{H^i}+ \frac{1}{\varepsilon}\left\|({\bf I}-{\bf P}_{\mathbf{M}})[f^{\varepsilon}]\right\|_{H^i_{\bf D}}^2 +C_{\epsilon_1}\exp\left(-\frac{\epsilon_1}{8C_0RT^2_c\sqrt{\varepsilon}}\right)\|h^{\varepsilon}\|^2_{H^i}.$$
\end{lemma}
\begin{proof}
	First of all, we will focus on proving \eqref{growdxi VML}, while noting that \eqref{EBdxi}, \eqref{EBRdxi}, and \eqref{EBndxi} can be established in a similar manner. We begin by considering the following:
 From \eqref{em-decay}, we can deduce that
\begin{align}\label{dxfhi}
   &\Big|\Big\langle \nabla_x^i\Big(f^{\varepsilon}\mathbf{M}^{-\frac{1}{2}}\big(\partial_t+v\cdot\nabla_x\big)\mathbf{M}^{\frac{1}{2}}\Big), \nabla_x^if^{\varepsilon}\Big\rangle\Big|\nonumber\\
    \lesssim& \sum_{j=0}^i\CU(\|\nabla_x[\rho,u,T]\|_{W^{j,\infty}})
    \left\|{\lag v\rag}^{3/2+2(i-j)}\nabla_x^jf^{\varepsilon}\right\|\left\|{\lag v\rag}^{3/2}\nabla_x^if^{\varepsilon}\right\|\\
     \lesssim& \sum_{j=0}^i\epsilon_1
    \left\|{\lag v\rag}^{3/2+2(i-j)}\nabla_x^jf^{\varepsilon}\right\|\left\|{\lag v\rag}^{3/2}\nabla_x^if^{\varepsilon}\right\|.\nonumber
\end{align}

In order to eliminate the velocity growth in the right hand side of \eqref{dxfhi}, we divide the domain of the $v-$integration ${\mathbb R}^3_v$ into $\vps-$dependent low-velocity and high-velocity regions. This separation allows us to bound the microscopic and macroscopic components of the low-velocity part by  $\varepsilon^{-1}\|({\bf I}-{\bf P}_{\mathbf{M}})[f^{\varepsilon}]\|_{H^i_{\bf D}}^2$ and $\|f^{\varepsilon}\|^2_{H^i}$ respectively. Simultaneously, the high velocity part can be controlled by $\|h^{\varepsilon}\|^2_{H^i}$, facilitated by the choice of $T_c$ in the global Maxwellian $\mu$.
More specifically, by \eqref{tt1}, one has for suitably small $\epsilon_1>0$,
\begin{align*}
\mathbf{M}^{-\frac{1}{2}}\mu^{\frac{1}{2}}=&\Big(\frac{1 }{ \rho}\Big)^{\frac{1}{2}}\Big(\frac{T }{ T_c}\Big)^{3/2}
 \exp\left(-\frac{(T-T_c)|v|^2}{4RTT_c}+\frac{v\cdot u}{2RT}+\frac{|u|^2}{4RT}\right)\\
\lesssim& \exp\left(-\frac{(T-T_c)|v|^2}{8C_0RTT_c}\right) \lesssim C_{\epsilon_1}\exp\left(-\frac{\epsilon_1|v|^2}{16C_0RT^2_c}\right),
\end{align*}
which further implies that
 \begin{align}\label{dxihfL}
\left|\nabla_x^if^{\varepsilon}\right|=&\left|\nabla_x^i\Big(\mathbf{M}^{-\frac{1}{2}}\mu^{\frac{1}{2}}h^{\varepsilon}\Big)\right|
\lesssim C_{\epsilon_1}\sum_{j=0}^i\left|\nabla_x^jh^{\varepsilon}\right| \exp\left(-\frac{\epsilon_1|v|^2}{16C_0RT^2_c}\right).
\end{align}
Thus we get from \eqref{dxihfL} that
\begin{align}\label{dxfhi0}
&\sum_{j=0}^i\left\|{\lag v\rag}^{3/2+2(i-j)}\nabla_x^jf^{\varepsilon}\right\|\nonumber\\
\lesssim&\sum_{j=0}^i\int_{{\mathbb R}^3}\int_{{\lag v\rag}^{4(i+1-j)}\leq\varepsilon^{-(i+1-j)}}{\lag v\rag}^{3+4(i-j)}\left|\nabla_x^jf^{\varepsilon}\right|^2\, dv dx\notag\\
&+\sum_{j=0}^i\int_{{\mathbb R}^3}\int_{{\lag v\rag}^{4(i+1-j)}\geq\varepsilon^{-(i+1-j)}}{\lag v\rag}^{3+4(i-j)}\left|\nabla_x^jf^{\varepsilon}\right|^2\, dv dx\nonumber\\
\lesssim& \sum_{j=0}^i\left\|{\lag v\rag}^{3/2+2(i-j)}\nabla_x^j{{\bf P}_{\mathbf{M}}}[f^{\varepsilon}]\right\|^2+\sum_{j=0}^i\frac{1}{\varepsilon^{(i+1-j)}} \left\|\nabla_x^j({\bf I}-{\bf P}_{\mathbf{M}})[f^{\varepsilon}]\right\|_{\bf D}^2\\
&+\sum_{j=0}^iC_{\epsilon_1}\int_{{\mathbb R}^3}\int_{{\lag v\rag}^{4(i+1-j)}\geq\varepsilon^{-(i+1-j)}} \exp\left(-\frac{\epsilon_1|v|^2}{8C_0RT^2_c}\right)|h^{\varepsilon}|^2\, dv dx\nonumber\\
\lesssim&\,\|f^{\varepsilon}\|^2_{H^i}+\sum_{j=0}^i\frac{1}{\varepsilon^{(i+1-j)}}\left\|\nabla_x^j({\bf I}-{\bf P}_{\mathbf{M}})[f^{\varepsilon}]\right\|_{\bf D}^2\notag\\ &+C_{\epsilon_1}\exp\left(-\frac{\epsilon_1}{8C_0RT^2_c\sqrt{\varepsilon}}\right)\|h^{\varepsilon}\|^2_{H^i}\lesssim\,{\bf Z}_{1,i}(t).\nonumber
\end{align}
\end{proof}

The next lemma is to control the velocity derivative  terms associated with the Lorentz force. This aspect leads to more intricacies, particularly since our energy functional $\mathcal{E}(t)$ lacks velocity derivative term. In addition to the techniques employed in Lemma \ref{dxlm VML}, our approach is to control $\nabla_v f^{\varepsilon}$ by using the dissipation norm $|\cdot|_{\bf D}$ derived from the linearized Landau collision operator, which includes first-order velocity derivative terms.
\begin{lemma}\label{dxvlm VML}
Under the assumptions in Lemma \ref{dxlm VML}, for $i=0, 1, 2$,  it holds that
  \begin{align}\label{EBLvdxi-0}
   &\Big|\Big\langle \nabla_x^i\Big[\Big(E+v \times B \Big)\cdot\nabla_vf^{\varepsilon}\Big], 4\pi R T \nabla_x^if^{\varepsilon}\Big\rangle\Big|
    \lesssim\,\epsilon_1(1+t)^{-p_0}{\bf 1}_{i\geq1}{\bf Z}_{2,i}(t),
   \end{align}
   \begin{align}\label{EBRLvdxi}
   &\varepsilon^k\Big|\Big\langle \nabla_x^i\Big[\big(E_R^{\varepsilon}+v \times B_R^{\varepsilon}\big) \cdot\nabla_vf^{\varepsilon}\Big], 4\pi RT \nabla_x^if^{\varepsilon}\Big\rangle\Big|
    \lesssim\,\varepsilon {\bf 1}_{i\geq1}{\bf Z}_{2,i}(t),
   \end{align}
  and for $\kappa=\frac13$,
   \begin{align}\label{EBLnvdxi}
   &\sum_{n=1}^{2k-1}\varepsilon^n\Big|\Big\langle \nabla_x^i\Big[\Big(E_n+v \times B_n \Big)\cdot\nabla_vf^{\varepsilon}+\Big(E_R^{\varepsilon}+v \times B_R^{\varepsilon} \Big)\cdot\nabla_v F_n\Big], 4\pi RT \nabla_x^if^{\varepsilon}\Big\rangle\Big|\nonumber\\
    \lesssim&\,\varepsilon^{2\kappa}\Big[\|E_R^{\varepsilon}\|^2_{H^i} +\|B_R^{\varepsilon}\|^2_{H^i} +\|f^{\varepsilon}\|^2_{H^i}+{\bf 1}_{i\geq1}{\bf Z}_{2,i}(t)\Big],
   \end{align}
   where $${\bf Z}_{2,i}(t):={\bf Z}_{1,i}(t)+C_{\epsilon_1}\exp\Big(-\frac{\epsilon_1}{8C_0RT^2_c\sqrt{\varepsilon}}\Big)
    \|h^{\varepsilon}\|^2_{H^{i}_{\bf D}}.$$
\end{lemma}
\begin{proof}
	We only prove \eqref{EBRLvdxi} for brevity, as analogous methods can be employed to derive \eqref{EBLvdxi-0} and \eqref{EBLnvdxi}.
	
	Note that by the \textit{a priori} assumption \eqref{aps-vml},
	
\begin{align}\label{fbd-vml}
\varepsilon^{k-1}\|f^{\varepsilon}\|_{H^2}\lesssim \varepsilon^{\frac{1}{2}}.
\end{align}

Applying \eqref{aps-vml} and Sobolev's inequalities, one has
\begin{align*}
&\varepsilon^k\Big|\Big\langle \nabla_x^i\Big[\big(E_R^{\varepsilon}+v \times B_R^{\varepsilon}\big) \cdot\nabla_vf^{\varepsilon}\Big], 4\pi RT \nabla_x^if^{\varepsilon}\Big\rangle\Big|\notag\\
\lesssim& \varepsilon^k\Big(\|E_R^{\varepsilon}\|_{H^2}+\|B_R^{\varepsilon}\|_{H^2}\Big)\Big(\sum_{j=0}^{i}\left\|{\lag v\rag}\nabla_v \nabla_x^{j}f^{\varepsilon}\right\|^2+\left\|\nabla_x^if^{\varepsilon}\right\|^2\Big)\\
\lesssim&\varepsilon\Big(\left\|{\lag v\rag}\nabla_v \nabla_x^{i}f^{\varepsilon}\right\|^2+\left\|\nabla_x^if^{\varepsilon}\right\|^2\Big).\nonumber
\end{align*}
Then it suffices to control $|{\lag v\rag}\nabla_v \nabla_x^if^{\varepsilon}|^2$ with $i=0, 1, 2.$

To see this, as in \eqref{dxihfL}, one has
\begin{align*}
\left|\nabla_v\nabla_x^if^{\varepsilon}\right|\lesssim& \sum_{j=0}^i\exp\left(-\frac{\epsilon_1|v|^2}{8C_0RTT_c}\right) \left[{\lag v\rag}^{2(i-j)+1}|h^{\varepsilon}|
+{\lag v\rag}^{2(i-j)}|\nabla_v\nabla_x^jh^{\varepsilon}|\right]\\
\lesssim& \sum_{j=0}^iC_{\epsilon_1}\exp\left(-\frac{\epsilon_1|v|^2}{12C_0RT^2_c}\right) \big[|h^{\varepsilon}|+|\nabla_v\nabla_x^jh^{\varepsilon}|\big].
\end{align*}

Subsequently, $\left\|{\lag v\rag}\nabla_v \nabla_x^{i}f^{\varepsilon}\right\|^2$ can be further bounded by using similar arguments to those in \eqref{dxfhi0}. The difference lies in the inclusion of additional $\|h^{\varepsilon}\|^2_{H^i_{\bf D}}$ norms to control terms $\sum_{j=0}^i|\nabla_v\nabla_x^jh^{\varepsilon}|$ for $i= 1, 2$. This leads to the desired estimate \eqref{EBRLvdxi}.

%
\end{proof}

With the above estimates, we are now prepared to derive the following energy estimates for the remainders $[f^{\varepsilon}, E_R^{\varepsilon}, B_R^{\varepsilon}]$.
\begin{proposition}\label{f-eng-vml}
Under the assumptions in Lemma \ref{dxlm VML}, it holds that for $\kappa = \frac{1}{3}$,
\begin{align}
	\label{f-eng-vml-sum}
	&\frac{\mathrm{d}}{\mathrm{d} t}\sum\limits_{i=0}^2 \varepsilon^i\Big(\|\sqrt{4\pi RT}\nabla_x^if^{\varepsilon}\|^2 + \|\nabla_x^iE_R^{\varepsilon}\|^2 + \|\nabla_x^iB_R^{\varepsilon}\|^2\Big) \nonumber \\
	&+ \delta\sum\limits_{i=0}^2 \varepsilon^{i-1}\|\nabla_x^i({\bf I}-{\bf P}_{\mathbf{M}})[f^{\varepsilon}]\|^2_{\bf D}   \\
	\lesssim& \sum\limits_{i=0}^2 \varepsilon^i\big[(1+t)^{-p_0} + \varepsilon^{\kappa}\big] \nonumber \\
	&\times \Big(\|\nabla_x^if^{\varepsilon}\|^2 + \left\|[E_R^{\varepsilon},B_R^{\varepsilon}]\right\|_{H^i}^2 + C_{\epsilon_1}\exp\left(-\frac{\epsilon_1}{8C_0RT^2_c\sqrt{\varepsilon}}\right)(\|h^{\varepsilon}\|_{H^i}^2 + \|h^{\varepsilon}\|_{H^i_{\bf D}}^2)\Big) \nonumber \\
	&+ \sum\limits_{i=0}^2 \left[\varepsilon^{2k+1+i}(1+t)^{4k+2} + \varepsilon^{k+i}(1+t)^{2k}\|\nabla_x^if^{\varepsilon}\|\right]. \nonumber
\end{align}
	
\end{proposition}
\begin{proof}
Our proof is divided into three steps.
\vskip 0.2cm

\noindent\underline{{\it Step 1. Basic energy estimate of the remainders.}} In this step, we derive the $L^2$ estimates on $[f^{\varepsilon}, E_R^{\varepsilon}, B_R^{\varepsilon}]$. Taking the $L^2$ inner product of \eqref{VMLf} with $4\pi RT f^{\varepsilon}$, we use \eqref{coLL0} to obtain
\begin{align}
	\label{L2f1 VML}
	&\frac{1}{2}\frac{\mathrm{d}}{\mathrm{d} t}\Big(\|\sqrt{4\pi RT}f^{\varepsilon}\|^2+\|E_R^{\varepsilon}\|^2+\|B_R^{\varepsilon}\|^2\Big)
	+\frac{\delta}{\varepsilon}\|({\bf I}-{\bf P}_{\mathbf{M}})[f^{\varepsilon}]\|^2_{\bf D} \\
	\leq&\;\frac{1}{2}\big|\big\langle \partial_tT f^{\varepsilon}, 4\pi R f^{\varepsilon}\big\rangle\big|+\Big|\Big\langle \big(E_R^{\varepsilon}+v \times B_R^{\varepsilon} \big) \cdot u\mathbf{M}^{\frac{1}{2}},4\pi  f^{\varepsilon}\Big\rangle\Big|\nonumber\\
	&+\Big|\Big\langle \Big(E+v \times B \Big)\cdot(v-u )f^{\varepsilon}, 2\pi  f^{\varepsilon}\Big\rangle\Big|\nonumber\\
	&+    \Big|\Big\langle f^{\varepsilon}\mathbf{M}^{-\frac{1}{2}}\big(\partial_t+v\cdot\nabla_x\big)\mathbf{M}^{\frac{1}{2}}, 4\pi RT f^{\varepsilon}\Big\rangle\Big|
	+\varepsilon^{k-1}\big|\big\langle\Gamma_{\mathbf{M}} ( f^{\varepsilon},
	f^{\varepsilon} ), 4\pi RT f^{\varepsilon}\big\rangle\big|\nonumber\\
	&+\sum_{n=1}^{2k-1}\varepsilon^{n-1}\big|\big\langle[\Gamma_{\mathbf{M}}(\mathbf{M}^{-\frac{1}{2}}F_n,f^{\varepsilon})+\Gamma_{\mathbf{M}}(
	f^{\varepsilon}, \mathbf{M}^{-\frac{1}{2}} F_n)\big], 4\pi RT f^{\varepsilon}\big\rangle\big|\nonumber\\
	&+\varepsilon^k \big|\big\langle \big(E_R^{\varepsilon}+v \times B_R^{\varepsilon}\big) \cdot(v-u)f^{\varepsilon},2\pi  f^{\varepsilon}\big\rangle\big|\nonumber\\
	&+\sum_{n=1}^{2k-1}\varepsilon^n\Big|\Big\langle \Big(E_R^{\varepsilon}+v \times B_R^{\varepsilon} \Big)\cdot\nabla_v F_n,4\pi RT f^{\varepsilon}\Big\rangle\Big|\nonumber\\
	&+\sum_{n=1}^{2k-1}\varepsilon^n\Big|\Big\langle \Big(E_n+v \times B_n \Big)\cdot(v-u)f^{\varepsilon},2\pi f^{\varepsilon}\Big\rangle\Big|
	+\big|\big\langle \CQ_0,4\pi RT f^{\varepsilon}\big\rangle\big|.\nonumber
\end{align}

Here we have used the following identity
\begin{align*}
\frac{1}{2}\frac{d}{dt}\left(\|E_R^{\varepsilon}(t)\|^2+\|B_R^{\varepsilon}(t)\|^2\right)=4\pi \Big\langle v\cdot E_R^{\varepsilon}\mathbf{M}^{\frac{1}{2}}, f^{\varepsilon}\Big\rangle
\end{align*}
from \eqref{fM-2}.

Now, let's estimate the terms on the right-hand side (R.H.S.) of \eqref{L2f1 VML} individually.

By \eqref{em-decay}, we can observe that the first two terms on the R.H.S. of \eqref{L2f1 VML} are bounded by
\begin{align*} C(1+t)^{-p_0}\epsilon_1 \Big(\|f^{\varepsilon}\|^2+\|E_R^{\varepsilon}\|^2
+\|B_R^{\varepsilon}\|^2\Big).
\end{align*}

Considering Lemma \ref{dxlm VML} and Lemma \ref{dxvlm VML}, we can control the 3rd, 5th, 7th-9th terms on the R.H.S. of \eqref{L2f1 VML} as follows:

\begin{align*}
 C\Big[\epsilon_1(1+t)^{-p_0}&+\varepsilon^{2\kappa}\Big]\Big(\|E_R^{\varepsilon}\|^2
+\|B_R^{\varepsilon}\|^2+{\bf Z}_{1,0}(t)\Big).
\end{align*}

For the 5th term on the R.H.S. of \eqref{L2f1 VML}, we use \eqref{GLLM}, Sobolev's inequality, and \eqref{aps-vml} to obtain
\begin{align}\label{f0L}
&\varepsilon^{k-1}\big|\big\langle\Gamma_{\mathbf{M}} ( f^{\varepsilon},
    f^{\varepsilon} ), 4\pi RT f^{\varepsilon}\big\rangle\big|\nonumber\\
    \lesssim&\varepsilon^{k-1}\|f^{\varepsilon}\|_{H^2}\Big(\|({\bf I}-{\bf P}_{\mathbf{M}})[f^{\varepsilon}]\|_{\bf D}+\|{{\bf P}_{\mathbf{M}}}[f^{\varepsilon}]\|_{\bf D}\Big)\|({\bf I}-{\bf P}_{\mathbf{M}})[f^{\varepsilon}]\|_{\bf D}\\
    \lesssim &\,\|({\bf I}-{\bf P}_{\mathbf{M}})[f^{\varepsilon}]\|_{\bf D}^2+\varepsilon\|{{\bf P}_{\mathbf{M}}}[f^{\varepsilon}]\|^2\lesssim \|({\bf I}-{\bf P}_{\mathbf{M}})[f^{\varepsilon}]\|_{\bf D}^2+\varepsilon\|f^{\varepsilon}\|^2 .\nonumber
\end{align}

For the 6th term on the R.H.S. of \eqref{L2f1 VML}, we can use \eqref{GLLM} and \eqref{em-Fn-es} to derive the following:
\begin{align}\label{fn0L}
&\sum_{n=1}^{2k-1}\varepsilon^{n-1}\big|\big\langle[\Gamma_{\mathbf{M}}(\mathbf{M}^{-\frac{1}{2}} F_n,f^{\varepsilon})+\Gamma_{\mathbf{M}}(
 f^{\varepsilon}, \mathbf{M}^{-\frac{1}{2}} F_n)\big], 4\pi RT f^{\varepsilon}\big\rangle\big|\nonumber\\
 \lesssim&\sum_{n=1}^{2k-1}\varepsilon^{n-1}(1+t)^n\|f^{\varepsilon}\|_{H^2}\Big(\|({\bf I}-{\bf P}_{\mathbf{M}})[f^{\varepsilon}]\|_{\bf D}+\|{{\bf P}_{\mathbf{M}}}[f^{\varepsilon}]\|_{\bf D}\Big)\|({\bf I}-{\bf P}_{\mathbf{M}})[f^{\varepsilon}]\|_{\bf D}\\
    \lesssim&\, \frac{o(1)}{\varepsilon}\|({\bf I}-{\bf P}_{\mathbf{M}})[f^{\varepsilon}]\|_{\bf D}^2+\varepsilon(1+t)^2\|{{\bf P}_{\mathbf{M}}}[f^{\varepsilon}]\|_{\bf D}^2\nonumber\\
    \lesssim& \frac{  o(1)}{\varepsilon}\|({\bf I}-{\bf P}_{\mathbf{M}})[f^{\varepsilon}]\|_{\bf D}^2+\varepsilon^{\kappa}\|f^{\varepsilon}\|^2\nonumber,
\end{align}
where we also used the fact that $0\leq t\leq \varepsilon^{-\kappa}$ with $\kappa=\frac13$.

For the last term on the R.H.S. of \eqref{L2f1 VML}, we can apply Cauchy's inequality, \eqref{em-Fn-es}, and \eqref{em-EB-es} to obtain
\begin{align*}
    \big|\big\langle \CQ_0, 4\pi RT f^{\varepsilon}\big\rangle\big|\lesssim& \frac{o(1)}{\varepsilon}\|({\bf I}-{\bf P}_{\mathbf{M}})[f^{\varepsilon}]\|_{\bf D}^2+\varepsilon
    \sum_{\substack{i+j\geq 2k+1\\2\leq i,j\leq2k-1}}\varepsilon^{2(i+j-k-1)}\|F_i\|^2_{H^2}\|F_j\|_{\bf D}^2\\
    &+\sum_{\substack{i+j\geq 2k\\1\leq i,j\leq2k-1}}\varepsilon^{i+j-k}\Big(\|E_n\|_{L^\infty}+\|B_n\|_{L^\infty}\Big)\big\|(1+v) \mathbf{M}^{-\frac{1}{2}}\nabla_vF_j\big\|\|f^{\varepsilon}\|\\
    \lesssim& \frac{o(1)}{\varepsilon}\|({\bf I}-{\bf P}_{\mathbf{M}})[f^{\varepsilon}]\|_{\bf D}^2+\varepsilon^{2k+1}(1+t)^{4k+2}+\varepsilon^{k}(1+t)^{2k}\|f^{\varepsilon}\|.\nonumber
\end{align*}

Substituting the above estimates into \eqref{L2f1 VML} leads to:
\begin{align}\label{L2f VML}
&\frac{\mathrm{d}}{\mathrm{d} t}\Big(\|\sqrt{4\pi RT}f^{\varepsilon}\|^2+\|E_R^{\varepsilon}\|^2+\|B_R^{\varepsilon}\|^2\Big)
    +\frac{\delta}{\varepsilon}\|({\bf I}-{\bf P}_{\mathbf{M}})[f^{\varepsilon}]\|^2_{\bf D}\nonumber\\
     \lesssim&\,\Big[(1+t)^{-p_0}+\varepsilon^{\kappa}\Big]\Big(\|f^{\varepsilon}\|^2+\|E_R^{\varepsilon}\|^2+\|B_R^{\varepsilon}\|^2
     +C_{\epsilon_1}\exp\left(-\frac{\epsilon_1}{8C_0RT^2_c\sqrt{\varepsilon}}\right)\|h^{\varepsilon}\|^2\Big)\\
     &+\varepsilon^{2k+1}(1+t)^{4k+2}+\varepsilon^{k}(1+t)^{2k}\|f^{\varepsilon}\|.\nonumber
\end{align}

\vskip 0.2cm
\noindent\underline{{\it Step 2. Estimates on the first order derivative of the remainders.}}

In this step, we proceed with the estimation of $\nabla_x f^{\varepsilon}$, $\nabla_x E^{\varepsilon}$, and $\nabla_x B_R^{\varepsilon}$. To do this, applying $\partial^{\alpha}$ ($1\leq |\alpha|\leq 2$) to \eqref{VMLf} gives
\begin{align}\label{m0xVMLL}
\big(\partial_t&+v\cdot\nabla_x\big)\partial^{\alpha}f^{\varepsilon}
        +\partial^{\alpha}\Big[\frac{\big(E_R^{\varepsilon}+v \times B_R^{\varepsilon} \big) }{ RT}\cdot \big(v-u\big)\mathbf{M}^{\frac{1}{2}}\Big]\nonumber\\
&+\partial^{\alpha}\Big[\Big(E+v \times B \Big)\cdot\frac{v-u }{ 2RT}f^{\varepsilon}\Big]-\partial^{\alpha}\Big[\Big(E+v \times B \Big)\cdot\nabla_vf^{\varepsilon}\Big]+\frac{\partial^{\alpha}\mathcal{L}_{\mathbf{M}}[f^{\varepsilon}]}{\varepsilon}\nonumber\\
 =&-\partial^{\alpha}\Big[\mathbf{M}^{-\frac{1}{2}}f^{\varepsilon}\Big[\partial_t+\hat{p}\cdot\nabla_x-\Big(E+v \times B \Big)\cdot\nabla_v\Big]\mathbf{M}^{\frac{1}{2}}\Big]+\varepsilon^{k-1}\partial^{\alpha}\Gamma_{\mathbf{M}}(f^{\varepsilon},f^{\varepsilon})\nonumber\\
 &+\sum_{n=1}^{2k-1}\varepsilon^{n-1}[\partial^{\alpha}\Gamma_{\mathbf{M}}(\mathbf{M}^{-\frac{1}{2}}F_n, f^{\varepsilon})+\partial^{\alpha}\Gamma_{\mathbf{M}}(f^{\varepsilon}, \mathbf{M}^{-\frac{1}{2}} F_n)\big]+\varepsilon^k \partial^{\alpha}\Big[\Big(E_R^{\varepsilon}+v \times B_R^{\varepsilon}\Big)\cdot\nabla_vf^{\varepsilon}\Big]\nonumber\\
 &-\varepsilon^k \partial^{\alpha}\Big[\Big(E_R^{\varepsilon}+v \times B_R^{\varepsilon}\Big) \cdot\frac{v-u }{ 2RT}f^{\varepsilon}\Big]\\
 &+\sum_{n=1}^{2k-1}\varepsilon^n\partial^{\alpha}\Big[\Big(E_n+v \times B_n \Big)\cdot\nabla_vf^{\varepsilon}+\Big(E_R^{\varepsilon}+v \times B_R^{\varepsilon} \Big)\cdot\nabla_v F_n\Big]\nonumber\\
 &-\sum_{n=1}^{2k-1}\varepsilon^n\partial^{\alpha}\Big[\Big(E_n+v \times B_n \Big)\cdot\frac{\big(v-u\big)}{ 2RT}f^{\varepsilon}\Big]+\varepsilon^{k}\partial^{\alpha}\CQ_0.\nonumber
\end{align}
Furthermore, by taking the inner product of $4\pi RT\partial^\alpha f^{\varepsilon}$ and \eqref{m0xVMLL} with $|\alpha|=1$, one obtains
\begin{align}\label{H1f1 VML}
&\frac{1}{2}\frac{\mathrm{d}}{\mathrm{d} t}\Big(\|\sqrt{4\pi RT}\nabla_xf^{\varepsilon}\|^2+\|[\nabla_xE_R^{\varepsilon},\nabla_xB_R^{\varepsilon}]\|^2\Big)
    +\frac{1}{\varepsilon}\big\langle \nabla_x\mathcal{L}_{\mathbf{M}}[f^{\varepsilon}],4\pi RT\nabla_xf^{\varepsilon}\big\rangle \\
    \lesssim&\;\big|\big\langle \partial_tT \nabla_xf^{\varepsilon},   2\pi R\nabla_xf^{\varepsilon}\big\rangle\big|+\sum_{|\alpha+\alpha'|=1}\Big|\Big\langle \big(E_R^{\varepsilon}+v \times \partial^\alpha B_R^{\varepsilon} \big) \cdot \partial^{\alpha'}\Big[\frac{v-u}{ T}\mathbf{M}^{\frac{1}{2}}\Big],4\pi T \nabla_xf^{\varepsilon}\Big\rangle\Big|\nonumber\\
    &+\sum\limits_{|\alpha|=1}\big|\big\langle v\big(\nabla_x \mathbf{M}^{\frac{1}{2}}\big)f^{\varepsilon},   \nabla_x E_R^{\varepsilon}\big\rangle\big|+\Big|\Big\langle \nabla_x\left[\big(E+v \times B \big) \cdot \nabla_v f^{\varepsilon}\right], 4\pi R T \nabla_xf^{\varepsilon}\Big\rangle\Big|\nonumber\\
    &+\Big|\Big\langle \nabla_x\Big[\Big(E+v \times B \Big)\cdot\frac{v-u }{ 2T}f^{\varepsilon}\Big],4\pi  T \nabla_xf^{\varepsilon}\Big\rangle\Big|\nonumber\\
    &+    \Big|\Big\langle \nabla_x\Big[f^{\varepsilon}\mathbf{M}^{-\frac{1}{2}}\big(\partial_t+v\cdot\nabla_x\big)\mathbf{M}^{\frac{1}{2}}\Big],4\pi R T \nabla_xf^{\varepsilon}\Big\rangle\Big|
    +\varepsilon^{k-1}\big|\big\langle\nabla_x\Gamma_{\mathbf{M}} ( f^{\varepsilon},
    f^{\varepsilon} ),4\pi R T \nabla_xf^{\varepsilon}\big\rangle\big|\nonumber\\
    &+\sum_{n=1}^{2k-1}\varepsilon^{n-1}\big|\big\langle[\nabla_x\Gamma_{\mathbf{M}}(\mathbf{M}^{-\frac{1}{2}}F_n, f^{\varepsilon})+\nabla_x\Gamma_{\mathbf{M}}(
 f^{\varepsilon}, \mathbf{M}^{-\frac{1}{2}} F_n)\big],4\pi R T \nabla_xf^{\varepsilon}\big\rangle\big|\nonumber\\
 &+\varepsilon^k \big|\big\langle \nabla_x\Big[\big(E_R^{\varepsilon}+v \times B_R^{\varepsilon}\big) \cdot\nabla_vf^{\varepsilon}\Big],4\pi R T  \nabla_xf^{\varepsilon}\big\rangle\big|\nonumber\\
 &+\varepsilon^k \big|\big\langle \nabla_x\Big[\big(E_R^{\varepsilon}+v \times B_R^{\varepsilon}\big) \cdot\frac{(v-u)}{ 2T}f^{\varepsilon}\Big],4\pi  T \nabla_xf^{\varepsilon}\big\rangle\big|\nonumber\\
 &+\sum_{n=1}^{2k-1}\varepsilon^n\Big|\Big\langle \nabla_x\Big[\big(E_n+v \times B_n\big) \cdot\nabla_vf^{\varepsilon}+\Big(E_R^{\varepsilon}+v \times B_R^{\varepsilon} \Big)\cdot\nabla_v F_n\Big],4\pi R T \nabla_xf^{\varepsilon}\Big\rangle\Big|\nonumber\\
 &+\sum_{n=1}^{2k-1}\varepsilon^n\Big|\Big\langle \nabla_x\Big[\Big(E_n+v \times B_n \Big)\cdot\frac{(v-u)}{2 T}f^{\varepsilon}\Big],4\pi  T\nabla_xf^{\varepsilon}\Big\rangle\Big|
  +\big|\big\langle \nabla_x\CQ_0,4\pi R T \nabla_xf^{\varepsilon}\big\rangle\big|.\nonumber
\end{align}

For the 2nd term on the L.H.S. of \eqref{H1f1 VML}, applying \eqref{coLLh} and \eqref{em-decay}, we have
\begin{align*}
\frac{1}{\varepsilon}\big\langle \nabla_x\mathcal{L}_{\mathbf{M}}[f^{\varepsilon}],4\pi RT \nabla_xf^{\varepsilon}\big\rangle
\geq&\frac{3\delta}{4\varepsilon}\|\nabla_x({\bf I}-{\bf P}_{\mathbf{M}})[f^{\varepsilon}]\|_{\bf D}^2-\frac{C\epsilon_1}{\varepsilon^2}\|({\bf I}-{\bf P}_{\mathbf{M}})[f^{\varepsilon}]\|_{\bf D}^2\\
&-\frac{C\epsilon_1}{\varepsilon}(1+t)^{-p_0}\big(\varepsilon\|f^{\varepsilon}\|^2_{H^1}+\|f^{\varepsilon}\|^2\big).
\end{align*}

Now, let's estimate the terms on the R.H.S. of \eqref{H1f1 VML} separately.

Referring again to \eqref{em-decay}, the upper bound of the 1st to 3rd terms on the R.H.S. of \eqref{H1f1 VML} is
$C\epsilon_1(1+t)^{-p_0}\Big(\|f^{\varepsilon}\|^2_{H^1}+\|E_R^{\varepsilon}\|^2_{H^1}
+\|B_R^{\varepsilon}\|^2_{H^1}\Big).
$
Lemmas \ref{dxlm VML} and \ref{dxvlm VML} show that the 4th-6th, 9th-12th terms on the R.H.S. of \eqref{H1f1 VML} are dominated by

\begin{align*}
C\Big[\epsilon_1(1+t)^{-p_0}+\varepsilon^{2\kappa}\Big]\Big[\|E_R^{\varepsilon}\|^2_{H^1} +\|B_R^{\varepsilon}\|^2_{H^1}+{\bf Z}_{2,1}(t)\Big].
\end{align*}
By applying \eqref{GLLM}, \eqref{fbd-vml}, and Sobolev's inequalities, we can bound the 7th term on the R.H.S. of \eqref{H1f1 VML} by
\begin{align*}
   C\Big[\|\nabla_x({\bf I}-{\bf P}_{\mathbf{M}})[f^{\varepsilon}]\|_{\bf D}^2+\|({\bf I}-{\bf P}_{\mathbf{M}})[f^{\varepsilon}]\|_{\bf D}^2+\epsilon_1(1+t)^{-p_0}\Big(\|f^{\varepsilon}\|^2+ \varepsilon\|f^{\varepsilon}\|_{H^1}^2\Big)\Big].
\end{align*}
Similarly, by virtue of \eqref{em-Fn-es}, the upper bound of the 8th term on the R.H.S. of \eqref{H1f1 VML} is
\begin{align*}
\frac{o(1)}{\varepsilon}\|\nabla_x({\bf I}-{\bf P}_{\mathbf{M}})[f^{\varepsilon}]\|_{\bf D}^2+C\Big[\frac{\epsilon_1}{\varepsilon}\|({\bf I}-{\bf P}_{\mathbf{M}})[f^{\varepsilon}]\|_{\bf D}^2 +\frac{\epsilon_1}{\varepsilon}(1+t)^{-p_0}\|f^{\varepsilon}\|^2+\varepsilon^{\kappa}\|f^{\varepsilon}\|^2_{H^1}\Big]\nonumber.
\end{align*}
For the last term on the R.H.S. of \eqref{H1f1 VML}, we can further use \eqref{em-EB-es} to have
\begin{align*}
    \big|\big\langle \nabla_x\CQ_0,4\pi R T \nabla_xf^{\varepsilon}\big\rangle\big|
    \lesssim& \frac{o(1)}{\varepsilon}\|\nabla_x({\bf I}-{\bf P}_{\mathbf{M}})[f^{\varepsilon}]\|_{\bf D}^2+\frac{\epsilon_1}{\varepsilon}\|({\bf I}-{\bf P}_{\mathbf{M}})[f^{\varepsilon}]\|_{\bf D}^2\\
    &+\frac{\epsilon_1}{\varepsilon}(1+t)^{-p_0}\|f^{\varepsilon}\|^2 +\varepsilon^{2k+1}(1+t)^{4k+2}+\varepsilon^{k}(1+t)^{2k}\|\nabla_xf^{\varepsilon}\|.\nonumber
\end{align*}
Finally, by plugging the above estimates into \eqref{H1f1 VML} and multiplying the resulting inequality by $\varepsilon$, we obtain
\begin{align}\label{H1f VML}
\frac{\mathrm{d}}{\mathrm{d} t}\Big[\varepsilon\Big(&\|\sqrt{4\pi RT}\nabla_xf^{\varepsilon}\|^2+\|\nabla_xE_R^{\varepsilon}\|^2+\|\nabla_xB_R^{\varepsilon}\|^2\Big)\Big]
    +\delta\|\nabla_x({\bf I}-{\bf P}_{\mathbf{M}})[f^{\varepsilon}]\|^2_{\bf D}  \\
     \lesssim&\, \varepsilon\Big((1+t)^{-p_0}+\varepsilon^{\kappa}\Big)\Big[\|\nabla_xf^{\varepsilon}\|^2+\frac{1}{\varepsilon}\|f^{\varepsilon}\|^2
     +\|E_R^{\varepsilon}\|^2_{H^1}+\|B_R^{\varepsilon}\|^2_{H^1}
     \nonumber\\
     &+ C_{\epsilon_1}\exp\left(-\frac{\epsilon_1}{8C_0RT^2_c\sqrt{\varepsilon}}\right)
    \Big(\|h^{\varepsilon}\|^2_{H^1}+\|h^{\varepsilon}\|^2_{H^1_{\bf D}}\Big)\Big]+\frac{\epsilon_1}{\varepsilon}\|({\bf I}-{\bf P}_{\mathbf{M}})[f^{\varepsilon}]\|^2_{\bf D}\nonumber\\
    &+\varepsilon^{2k+2}(1+t)^{4k+2}
     +\varepsilon^{k+1}(1+t)^{2k}\|\nabla_xf^{\varepsilon}\|.\nonumber
\end{align}
\vskip 0.2cm
\noindent\underline{{\it Step 3. Estimates on the second order derivative of the remainders.}}
In this step, we proceed to derive the estimate of $\|[\nabla^2_x f^{\varepsilon},\nabla_x^2 E_R^{\varepsilon}, \nabla_x^2 B_R^{\varepsilon}]\|$. For results in this direction, we have
\begin{align}\label{H2f VML}
\frac{\mathrm{d}}{\mathrm{d} t}\Big[\varepsilon^2&\Big(\|\sqrt{4\pi RT}\nabla_x^2f^{\varepsilon}\|^2+\|\nabla_x^2E_R^{\varepsilon}\|^2+\|\nabla_x^2B_R^{\varepsilon}\|^2\Big)\Big]
    +\delta\varepsilon\|\nabla_x^2({\bf I}-{\bf P}_{\mathbf{M}})[f^{\varepsilon}]\|^2_{\bf D}   \\
     \lesssim&\,\varepsilon^2\Big((1+t)^{-p_0}+\varepsilon^{\kappa}\Big)\Big[\|\nabla_x^2f^{\varepsilon}\|^2 +\frac{1}{\varepsilon}\|f^{\varepsilon}\|^2_{H^1}
     +\|E_R^{\varepsilon}\|^2_{H^2}+\|B_R^{\varepsilon}\|^2_{H^2}\Big)
     \nonumber\\
     &+ C_{\epsilon_1}\exp\left(-\frac{\epsilon_1}{8C_0RT^2_c\sqrt{\varepsilon}}\right)
    \Big(\|h^{\varepsilon}\|^2_{H^2}+\|h^{\varepsilon}\|^2_{H^2_{\bf D}}\Big)\Big]+\epsilon_1\|\nabla_x({\bf I}-{\bf P}_{\mathbf{M}})[f^{\varepsilon}]\|^2_{\bf D}\nonumber\\
    &+\epsilon_1\|({\bf I}-{\bf P}_{\mathbf{M}})[f^{\varepsilon}]\|^2_{\bf D}+\varepsilon^{2k+3}(1+t)^{4k+2}
     +\varepsilon^{k+2}(1+t)^{2k}\|\nabla_x^2f^{\varepsilon}\|.\nonumber
\end{align}
To prove \eqref{H2f VML}, we first take the inner product of $4\pi RT\partial^{\alpha} f^{\varepsilon}$ and \eqref{m0xVMLL} for $|\alpha|=2$ to obtain
\begin{align}\label{H2f1 VML}
&\frac{1}{2}\frac{\mathrm{d}}{\mathrm{d} t}\Big(\|\sqrt{4\pi RT}\nabla_x^2f^{\varepsilon}\|^2+\|\nabla_x^2E_R^{\varepsilon}\|^2+\|\nabla_x^2B_R^{\varepsilon}\|^2\Big)
    +\frac{1}{\varepsilon}\big\langle \nabla_x^2\mathcal{L}_{\mathbf{M}}[f^{\varepsilon}],4\pi RT\nabla_x^2f^{\varepsilon}\big\rangle \nonumber\\
    \lesssim&\;\big|\big\langle \partial_tT \nabla_x^2f^{\varepsilon}, 2\pi R \nabla_x^2f^{\varepsilon}\big\rangle\big|+\sum_{|\alpha+\alpha'|=2}\Big|\Big\langle \partial^\alpha\big(E_R^{\varepsilon}+v \times B_R^{\varepsilon} \big) \cdot \partial^{\alpha'}\Big[\frac{v-u}{ T}\mathbf{M}^{\frac{1}{2}}\Big],4\pi  T \nabla_x^2f^{\varepsilon}\Big\rangle\Big|\nonumber\\
    &+\sum_{|\alpha+\alpha'|=2}\big|\big\langle v\partial^\alpha\big(\mathbf{M}^{-\frac{1}{2}}\big)\partial^{\alpha'}f^{\varepsilon},   \nabla_x^2E_R^{\varepsilon}\big\rangle\big|+\Big|\Big\langle \nabla_x^2\Big[\big(E+v \times B \big) \cdot \nabla_v f^{\varepsilon}\Big],4\pi R T \nabla_x^2f^{\varepsilon}\Big\rangle\Big|\nonumber\\
    &+\Big|\Big\langle \nabla_x^2\Big[\Big(E+v \times B \Big)\cdot\frac{v-u }{ 2T}f^{\varepsilon}\Big],4\pi  T \nabla_x^2f^{\varepsilon}\Big\rangle\Big|\nonumber\\
    &+    \Big|\Big\langle \nabla_x^2\Big[f^{\varepsilon}\mathbf{M}^{-\frac{1}{2}}\big(\partial_t+v\cdot\nabla_x\big)\mathbf{M}^{\frac{1}{2}}\Big],4\pi R T \nabla_x^2f^{\varepsilon}\Big\rangle\Big|
    +\varepsilon^{k-1}\big|\big\langle\nabla_x^2\Gamma_{\mathbf{M}} ( f^{\varepsilon},
    f^{\varepsilon} ),4\pi R T \nabla_x^2f^{\varepsilon}\big\rangle\big|\nonumber\\
    &+\sum_{n=1}^{2k-1}\varepsilon^{n-1}\big|\big\langle[\nabla_x^2\Gamma_{\mathbf{M}}(\mathbf{M}^{-\frac{1}{2}}F_n, f^{\varepsilon})+\nabla_x^2\Gamma_{\mathbf{M}}(
 f^{\varepsilon}, \mathbf{M}^{-\frac{1}{2}} F_n)\big],4\pi R T \nabla_x^2f^{\varepsilon}\big\rangle\big|\\
 &+\varepsilon^k \big|\big\langle \nabla_x^2\Big[\big(E_R^{\varepsilon}+v \times B_R^{\varepsilon}\big) \cdot\nabla_vf^{\varepsilon}\Big],4\pi R  T  \nabla_x^2f^{\varepsilon}\big\rangle\big|\nonumber\\
 &+\varepsilon^k \big|\big\langle \nabla_x^2\Big[\big(E_R^{\varepsilon}+v \times B_R^{\varepsilon}\big) \cdot\frac{(v-u)}{ 2T}f^{\varepsilon}\Big],4\pi T \nabla_x^2f^{\varepsilon}\big\rangle\big|\nonumber\\
 &+\sum_{n=1}^{2k-1}\varepsilon^n\Big|\Big\langle \nabla_x^2\Big[\big(E_n+v \times B_n\big) \cdot\nabla_vf^{\varepsilon}+\Big(E_R^{\varepsilon}+v \times B_R^{\varepsilon} \Big)\cdot\nabla_v F_n\Big],4\pi R T \nabla_x^2f^{\varepsilon}\Big\rangle\Big|\nonumber\\
 &+\sum_{n=1}^{2k-1}\varepsilon^n\Big|\Big\langle \nabla_x^2\Big[\Big(E_n+v \times B_n \Big)\cdot\frac{(v-u)}{ 2T}f^{\varepsilon}\Big], 4\pi  T\nabla_x^2f^{\varepsilon}\Big\rangle\Big|
  +\big|\big\langle \nabla_x^2\CQ_0,4\pi R T \nabla_x^2f^{\varepsilon}\big\rangle\big|.\nonumber
\end{align}
By using \eqref{em-decay} and \eqref{coLLh}, the 2nd term on the L.H.S. of \eqref{H2f1 VML} can be estimated as
\begin{align*}
\frac{1}{\varepsilon}\big\langle \nabla_x^2\mathcal{L}_{\mathbf{M}}[f^{\varepsilon}],4\pi RT\nabla_x^2f^{\varepsilon}\big\rangle
\geq&\frac{3\delta}{4\varepsilon}\|\nabla_x^2({\bf I}-{\bf P}_{\mathbf{M}})[f^{\varepsilon}]\|_{\bf D}^2-\frac{C\epsilon_1}{\varepsilon^2}\|({\bf I}-{\bf P}_{\mathbf{M}})[f^{\varepsilon}]\|_{H^1_{\bf D}}^2\\
&-\frac{C\epsilon_1}{\varepsilon}(1+t)^{-p_0}\Big(\varepsilon\|f^{\varepsilon}\|^2_{H^2}+\|f^{\varepsilon}\|^2_{H^1}\Big).
\end{align*}
It follows from \eqref{em-decay} that the first three terms on the R.H.S. of \eqref{H2f1 VML} can be bounded by
$C\Big[\epsilon_1(1+t)^{-p_0}+\varepsilon^{2\kappa}\Big]\Big[\|E_R^{\varepsilon}\|^2_{H^2} +\|B_R^{\varepsilon}\|^2_{H^2}+{\bf Z}_{2,2}(t)\Big].$ In view of Lemma \ref{dxlm VML} and Lemma \ref{dxvlm VML}, the 4th-6th, 9th-12th terms on the R.H.S. of \eqref{H2f1 VML} are no more than
\begin{align*}
  C\Big[\epsilon_1(1+t)^{-p_0}&+\varepsilon^{2\kappa}\Big]\Big[\|E_R^{\varepsilon}\|^2_{H^2} +\|B_R^{\varepsilon}\|^2_{H^2}+{\bf Z}_{2,2}(t)\Big].
\end{align*}

Similar to the corresponding terms in \eqref{H1f1 VML}, the 7th and 8th terms on the R.H.S. of \eqref{H1f1 VML} can be bounded by
\begin{align*}
\frac{o(1)}{\varepsilon}\|({\bf I}-{\bf P}_{\mathbf{M}})[f^{\varepsilon}]\|_{H^2_{\bf D}}^2+\frac{C\epsilon_1}{\varepsilon}(1+t)^{-p_0}\|f^{\varepsilon}\|_{H^1}^2+C\varepsilon^{\kappa}\|f^{\varepsilon}\|^2_{H^2}\nonumber.
\end{align*}
By using \eqref{em-Fn-es} and \eqref{em-EB-es}, we can further estimate the last term on the R.H.S. of \eqref{H2f1 VML} as
\begin{align*}
    \big|\big\langle \nabla_x^2\CQ_0,4\pi R T \nabla_x^2f^{\varepsilon}\big\rangle\big|
    \lesssim& \frac{o(1)}{\varepsilon}\|\nabla_x^2({\bf I}-{\bf P}_{\mathbf{M}})[f^{\varepsilon}]\|_{\bf D}^2+\frac{\epsilon_1}{\varepsilon}\|\nabla_x({\bf I}-{\bf P}_{\mathbf{M}})[f^{\varepsilon}]\|_{\bf D}^2\\
    &+\frac{\epsilon_1}{\varepsilon}(1+t)^{-p_0}\|f^{\varepsilon}\|^2_{H^1} +\varepsilon^{2k+1}(1+t)^{4k+2}+\varepsilon^{k}(1+t)^{2k}\|\nabla_x^2f^{\varepsilon}\|.\nonumber
\end{align*}
Substituting the aforementioned estimates into \eqref{H2f1 VML} and multiplying the resulting inequality by $\varepsilon^2$ gives us \eqref{H2f VML}.

In conclusion, \eqref{f-eng-vml-sum} follows from \eqref{L2f VML}, \eqref{H1f VML}, and \eqref{H2f VML}.
\end{proof}


\subsection{Estimates of $h^{\varepsilon}$}

Before proceeding to derive the energy estimates for $h^{\varepsilon}$, let's establish some lemmas for later use. The first lemma deals with the velocity growth terms involving the electric field. In contrast to Lemma \ref{dxlm VML} for the corresponding estimates of $f^{\varepsilon}$, we control these velocity growth terms using additional dissipation terms $Y(t)\|{\lag v\rag}w_i \nabla_x^ih^{\varepsilon}\|^2$. These terms arise due to the special weight function $\exp\left(\frac{\langle v\rangle^2}{8RT_c\ln(\mathrm{e}+t)}\right)$. Additionally, the estimates make use of the terms $\|w h^{\varepsilon}\|^2_{H^i}$ for $i=0, 1, 2$ and rely on the favorable time-decay characteristics described in \eqref{em-decay} for $[\rho, u, T]$ as well as the smallness of $\varepsilon$.

\begin{lemma}\label{dxlmh VML}
	Under the assumptions in Lemma \ref{dxlm VML}, for $i=0, 1, 2$, the following inequalities hold:
  \begin{align}
   \Big|\Big\langle \nabla_x^i\Big[\frac{E\cdot v }{ 2RT_c}h^{\varepsilon}\Big], w_i^2 \nabla_x^ih^{\varepsilon}\Big\rangle\Big|\lesssim \epsilon_1Y(t)\|{\lag v\rag}w_i \nabla_x^ih^{\varepsilon}\|^2+\epsilon_1(1+t)^{-p_0}{\bf 1}_{i\geq1}\|w h^{\varepsilon}\|^2_{H^{i-1}},\notag
\end{align}
   \begin{align}
   \varepsilon^k\Big|\Big\langle \nabla_x^i\Big[\frac{E_R^{\varepsilon}\cdot v }{ 2RT_c}h^{\varepsilon}\Big], w_i^2 \nabla_x^ih^{\varepsilon}\Big\rangle\Big|
    \lesssim\,\varepsilon^{\frac{1}{2}}Y(t)\big\|{\lag v\rag}w_i \nabla_x^i h^{\varepsilon}\big\|^2+{\bf 1}_{i\geq1}\varepsilon\|w h^{\varepsilon}\|^2_{H^i}.\notag
   \end{align}
For $\kappa=\frac{1}{3}$, the following inequality holds
   \begin{align}\label{EBLndxih}
   &\sum_{n=1}^{2k-1}\varepsilon^n\Big|\Big\langle \nabla_x^i\Big[\frac{E_n\cdot v }{ 2RT_c}h^{\varepsilon}\Big], w_i^2 \nabla_x^ih^{\varepsilon}\Big\rangle\Big|
    \lesssim\,\varepsilon^{\frac{\kappa}{2}}Y(t)\|{\lag v\rag}w_i \nabla_x^ih^{\varepsilon}\|^2+\varepsilon^{2\kappa}{\bf 1}_{i\geq1}\|w h^{\varepsilon}\|^2_{H^{i-1}}.
   \end{align}
\end{lemma}
\begin{proof} To verify \eqref{EBLndxih}, we will focus on this particular inequality, as the other two can be handled in a similar manner.
	
	Using \eqref{em-decay}, for $t\leq \varepsilon^{-\kappa}$, we have the following:
\begin{align*}
\sum_{n=1}^{2k-1}\varepsilon^n\Big|\Big\langle \nabla_x^i\Big[\frac{E_n\cdot v }{ 2RT_c}h^{\varepsilon}\Big], w_i^2 \nabla_x^ih^{\varepsilon}\Big\rangle\Big|
\lesssim& \sum_{n=1}^{2k-1}\varepsilon^n(1+t)^{n}\Big(\|{\lag v\rag}w_i \nabla_x^ih^{\varepsilon}\|^2+{\bf I}_{i\geq1}\|w h^{\varepsilon}\|^2_{H^{i-1}}\Big)\\
\lesssim&\,\varepsilon(1+t)\Big(\|{\lag v\rag}w_i \nabla_x^ih^{\varepsilon}\|^2+{\bf I}_{i\geq1}\|w h^{\varepsilon}\|^2_{H^{i-1}}\Big)\\
\lesssim&\varepsilon^{\frac{\kappa}{2}}Y(t)\|{\lag v\rag}w_i \nabla_x^ih^{\varepsilon}\|^2+\varepsilon^{2\kappa}{\bf I}_{i\geq1}\|w h^{\varepsilon}\|^2_{H^{i-1}}.
\end{align*}
\end{proof}

The following lemma is dedicated to controlling the velocity derivative terms associated with the Lorentz force. Since the energy functional $\mathcal{E}(t)$ defined in \eqref{eg-vml} does not include any terms involving velocity derivatives, we need to bound the velocity derivative terms using the dissipation norm $\|\cdot\|^2_{\bf D}$. However, this norm becomes degenerate for large velocities, and we must combine it with additional dissipative terms induced by our special weight functions to compensate for the velocity degeneracy.
	
	Additionally, we need to deal with terms like $\varepsilon^k \big(v \times B_R^{\varepsilon}\big) \cdot\nabla_vh^{\varepsilon}$, where both the velocity $v$ and a velocity derivative are involved. Fortunately, as demonstrated in Lemma \ref{lower norm}, the velocity degeneracy for $({\bf I}-{\bf P}v)[\nabla_vh^{\varepsilon}]$ is of order $1$ in $|\cdot|^2_{\bf D}$, in contrast to the order $3$ degeneracy of ${\bf P}_v[\nabla_vh^{\varepsilon}]$. Although the $|v|^2$ growth in the dissipation norms $Y(t)\|{\lag v\rag}w_i \nabla_x^ih^{\varepsilon}\|^2 (i=0, 1, 2)$ may not fully compensate for the velocity degeneracy due to the possible order $3$ degeneracy of ${\bf P}_v[\nabla_vh^{\varepsilon}]$, we overcome this difficulty by specifically designing the polynomial velocity weight part $\lag v\rag^{\ell-i}$ (where $i$ is the order of $x$-derivatives) to compensate for the weak dissipation of the linearized Landau operator.
			
			However, this polynomial velocity weight also leads to velocity growth when estimating the nonlinear collision operator $\varepsilon^{k-1}\Gamma ( h^{\varepsilon}, h^{\varepsilon} )$. Furthermore, these growths can be absorbed by the exponential weight, as indicated in \eqref{pevi}, where we ensure that

\begin{align}\label{pevi}
	\sup_{v\in {\mathbb R}^3}\Big[{\lag v\rag}^i\exp\left(\frac{-\epsilon_1\langle v\rangle^2}{8C_0RT_c\ln(\mathrm{e}+t)}\right)\Big]\lesssim \Big(\frac{\ln(\mathrm{e}+t)}{\epsilon_1}\Big)^{\frac{i}{2}}.
\end{align}
Further details on this can be found in the proof of Proposition \ref{h-vml-eng-prop}, particularly in inequalities like \eqref{zh1} and \eqref{zh2}.

\begin{lemma}\label{key-VML-2} Under the assumptions of Lemma \ref{dxlm VML}, for $i=0, 1, 2$, we have
  \begin{align}\label{EBvdxih}
   &\Big|\Big\langle \nabla_x^i\Big[\Big(E+v \times B \Big)\cdot\nabla_vh^{\varepsilon}\Big], w_i^2 \nabla_x^ih^{\varepsilon}\Big\rangle\Big|\lesssim\epsilon_1\Big(Y(t)\|{\lag v\rag}w_i \nabla_x^ih^{\varepsilon}\|^2+{\bf I}_{i\geq1}\|w h^{\varepsilon}\|^2_{H^{i-1}_{\bf D}}\Big),
   \end{align}
   \begin{align}\label{EBRvdxih}
   &\varepsilon^k\Big|\Big\langle \nabla_x^i\Big[\big(E_R^{\varepsilon}+v \times B_R^{\varepsilon}\big) \cdot\nabla_vh^{\varepsilon}\Big], w_i^2 \nabla_x^ih^{\varepsilon}\Big\rangle\Big|
    \lesssim\,\varepsilon^{\frac{1}{2}}Y(t)\big\|{\lag v\rag}w_i \nabla_x^i h^{\varepsilon}\big\|^2+\varepsilon\|w h^{\varepsilon}\|^2_{H^i_{\bf D}}.
   \end{align}
 For $\kappa=\frac{1}{3}$,
   \begin{align}\label{EBnvdxih}
   \sum_{n=1}^{2k-1}\varepsilon^n\Big|\Big\langle \nabla_x^i\big[\big(E_n&+v \times B_n \big)\cdot\nabla_vh^{\varepsilon}+\big(E_R^{\varepsilon}+v \times B_R^{\varepsilon} \big)\cdot\nabla_v F_n\big], w_i^2 \nabla_x^ih^{\varepsilon}\Big\rangle\Big|\\
       \lesssim&\,\varepsilon^{\frac{\kappa}{2}}Y(t)\|{\lag v\rag}w_i \nabla_x^ih^{\varepsilon}\|^2+\varepsilon^{2\kappa}\Big(\|E_R^{\varepsilon}\|_{H^i}^2+\|B_R^{\varepsilon}\|_{H^i}^2+{\bf I}_{i\geq1}\|w h^{\varepsilon}\|^2_{H^{i-1}_{\bf D}}\Big).\nonumber
   \end{align}
\end{lemma}
\begin{proof}
	We will prove \eqref{EBRvdxih} and omit the proof of \eqref{EBvdxih} and \eqref{EBnvdxih} for the sake of brevity. For \eqref{EBRvdxih}, we have
\begin{align*}
&\varepsilon^k \Big|\Big\langle \nabla_x^i\Big[\Big(E_R^{\varepsilon}+v \times B_R^{\varepsilon} \Big)\cdot\nabla_vh^{\varepsilon}\Big], w^2_i \nabla_x^i h^{\varepsilon}\Big\rangle\Big|\\
    \lesssim&\,\varepsilon^k\sum_{\substack{|\alpha+\alpha'|=i\\ |\alpha'|<i}}{\bf I}_{i\geq1}\big|\big\langle \partial^{\alpha}E_R^{\varepsilon}\cdot\nabla_v\partial^{\alpha'}h^{\varepsilon},w^2_i \nabla_x ^i h^{\varepsilon}\big\rangle\big|
+\varepsilon^k\big|\big\langle E_R^{\varepsilon} \cdot\nabla_v\big(w^2_i\big)  \nabla_x^i h^{\varepsilon},  \nabla_x^i h^{\varepsilon}\big\rangle\big|\\
&+\varepsilon^k\sum_{\substack{|\alpha+\alpha'|=i\\ |\alpha'|<i}}{\bf I}_{i\geq1}\big|\big\langle\big( v \times \partial^{\alpha}B_R^{\varepsilon}\big)\cdot\big({\bf I}-{\bf P}_v)[\nabla_v\partial^{\alpha'}h^{\varepsilon}], w^2_i \nabla_x^i h^{\varepsilon}\big\rangle\big|\\
\lesssim &\varepsilon^k\sum_{\substack{|\alpha+\alpha'|=i\\ |\alpha'|<i}}{\bf I}_{i\geq1}\|\partial^{\alpha}E_R^{\varepsilon}\|_{L^6} \big\|{\lag v\rag}^{-1}w_i\nabla_v\partial^{\alpha'}h^{\varepsilon}\big\|_{L^2_vL^3_x}\|{\lag v\rag}w_i\nabla_x^i h^{\varepsilon}\|\\
&+\varepsilon^k\|E_R^{\varepsilon}\|_{H^2} \big\|{\lag v\rag}w_i \nabla_x^i h^{\varepsilon}\big\|^2\\
+&\varepsilon^k\sum_{\substack{|\alpha+\alpha'|=i\\ |\alpha'|<i}}{\bf I}_{i\geq1}\|\partial^{\alpha}B_R^{\varepsilon}\|_{L^6} \Big\|\frac{|v|}{{\lag v\rag}}w_i\big({\bf I}-{\bf P}_v)[\nabla_v \partial^{\alpha'}h^{\varepsilon}]\Big\|_{L^2_vL^3_x}\|{\lag v\rag}w_i\nabla_x^i h^{\varepsilon}\|.
\end{align*}
Furthermore, by Lemma \ref{lower norm}, for $|\alpha'|<i$ with $1\leq i\leq 2$, we can get that
\begin{align*}
&\big\|{\lag v\rag}^{-1}w_i\nabla_v\partial^{\alpha'}h^{\varepsilon}\big\|_{L^2_vL^3_x}^2=\int_{{\mathbb R}^3_v} {\lag v\rag}^{-2}w_i^2|\nabla_v\partial^{\alpha'}h^{\varepsilon}|^2_{L^3_x}dv\\
\leq& \int_{{\mathbb R}^3_v} \Big({\lag v\rag}^{-\frac{3}{2}}w_i|\nabla_v\partial^{\alpha'}h^{\varepsilon}|_{L^6_x}\Big)\Big({\lag v\rag}^{-\frac{1}{2}}w_i|\nabla_v\partial^{\alpha'}h^{\varepsilon}|_{L^2_x}\Big)\\
\lesssim& \left\|w_{i-1}\nabla_x^{|\alpha'|} h^{\varepsilon}\right\|_{\bf D} \left\|w_i\nabla_x^{|\alpha'|+1}h^{\varepsilon}\right\|_{\bf D}.
\end{align*}
And similarly,
\begin{align*}
\Big\|\frac{|v|}{{\lag v\rag}}w_i\big({\bf I}-{\bf P}_v)[\nabla_v \partial^{\alpha'}h^{\varepsilon}]\Big\|_{L^2_vL^3_x}^2
\lesssim& \left\|w_{i-1} \nabla_x^{|\alpha'|}h^{\varepsilon}\right\|_{\bf D} \left\|w_i\nabla_x^{|\alpha'|+1}h^{\varepsilon}\right\|_{\bf D}.
\end{align*}
The two aforementioned estimates are the very reason why we design our weight functions $w_i$ to depend on the order $i$ of spatial derivatives.

Noting that $(1+t)^{-\beta}\lesssim Y(t)$ and $\varepsilon\lesssim \varepsilon^{\frac{1}{2}}Y(t)$ for $t\leq \varepsilon^{-\kappa}$, we can further use \eqref{fbd-vml} to get
\begin{align*}
 &\varepsilon^k \Big|\Big\langle \nabla_x^i\Big[\Big(E_R^{\varepsilon}+v \times B_R^{\varepsilon} \Big)\cdot\nabla_vh^{\varepsilon}\Big], w^2_i \nabla_x^i h^{\varepsilon}\Big\rangle\Big|\\
\lesssim & \varepsilon^k\Big(\|E_R^{\varepsilon}\|_{H^2}+\|B_R^{\varepsilon}\|_{H^2}\Big)\Big( \big\|{\lag v\rag}w_i \nabla_x^i h^{\varepsilon}\big\|^2+ \|w_i\nabla_x^ih^{\varepsilon}\|_{\bf D}^2+{\bf I}_{i\geq1}\|w h^{\varepsilon}\|_{H^{i-1}_{\bf D}}^2\Big)\\
\lesssim&
\varepsilon^{\frac{1}{2}}Y(t)\big\|{\lag v\rag}w_i \nabla_x^i h^{\varepsilon}\big\|^2+\varepsilon\|w h^{\varepsilon}\|^2_{H^i_{\bf D}}.
\end{align*}
\end{proof}

We have now reached the stage of applying weighted energy estimates to the function $h^\varepsilon$.
\begin{proposition}\label{h-vml-eng-prop}
Under the assumptions of Lemma \ref{dxlm VML}, we can establish the following weighted energy estimate for $h^\varepsilon$ with $\kappa=\frac{1}{3}$:
\begin{align}\label{h-vml-eng}
\frac{\mathrm{d}}{\mathrm{d} t}\sum\limits_{i=0}^2&\big(\varepsilon^{i+1+\kappa}\|w_i\nabla_x^ih^{\varepsilon}\|^2\big)
    +\sum\limits_{i=0}^2\varepsilon^{i+1+\kappa}Y(t)\big\|{\lag v\rag}w_i\nabla_x^ih^{\varepsilon}\big\|^2
    +\de\sum\limits_{i=0}^2\varepsilon^{i+\kappa}\|w_i\nabla_x^ih^{\varepsilon}\|_{\bf D}^2 \\
    \lesssim&\sum\limits_{i=0}^2\varepsilon^{i+\kappa}\|f^{\varepsilon}\|_{H^i}^2
    +\sum\limits_{i=0}^2\varepsilon^{i+1+2\kappa}\Big(\|E_R^{\varepsilon}\|_{H^i}^2+\|B_R^{\varepsilon}\|_{H^i}^2+\|wh^{\varepsilon}\|_{H^i}^2\Big)
       \nonumber\\
    &+\varepsilon^{3+\kappa}\|wh^{\varepsilon}\|_{H^2_{\bf D}}^2+\epsilon_1\sum\limits_{i=0}^1\varepsilon^{i+1+\kappa}\Big((1+t)^{-p_0}\|w h^{\varepsilon}\|_{H^i}^2+\|w h^{\varepsilon}\|_{H^i_{\bf D}}^2\Big)
    \nonumber\\
    &+\sum\limits_{i=0}^2\varepsilon^{2k+i+2+\kappa}(1+t)^{4k+2}
    +\sum\limits_{i=0}^2\varepsilon^{k+i+1+\kappa}(1+t)^{2k}\|w_i\na_x^ih^{\varepsilon}\|.\nonumber
\end{align}
\end{proposition}

\begin{proof} The proof is divided into three steps.
\vskip 0.2cm
\noindent{\underline{\it Step 1. Basic Weighted Energy Estimate of $h^{\varepsilon}$.}} In this step, we deduce the weighted $L^2$ estimate of $h^{\varepsilon}$. For results in this direction, we have
\begin{align}\label{L2h VML}
    \frac{\mathrm{d}}{\mathrm{d} t}\big(\varepsilon^{1+\kappa}&\|w_0h^{\varepsilon}\|^2\big)+\varepsilon^{1+\kappa}Y\big\|{\lag v\rag}w_0h^{\varepsilon}\big\|^2
    +\delta\varepsilon^{\kappa}\|w_0h^{\varepsilon}\|_{\bf D}^2 \\
    \lesssim&\,\varepsilon^{\kappa}\|f^{\varepsilon}\|^2+\varepsilon^{1+2\kappa}\Big(\|E_R^{\varepsilon}\|^2+\|B_R^{\varepsilon}\|^2
    +\|w_0h^{\varepsilon}\|^2\Big)+\varepsilon^{3+\kappa}\|wh^{\varepsilon}\|_{H^2_{\bf D}}^2\nonumber\\
    &+\varepsilon^{2k+2+\kappa}(1+t)^{4k+2}+\varepsilon^{k+1+\kappa}(1+t)^{2k}\|w_0h^{\varepsilon}\|.\nonumber
\end{align}
To prove this, we first take the $L^2$ inner product of $w^2_0 h^{\varepsilon}$ with \eqref{h-equation} and use \eqref{wLLL} to obtain
\begin{align}\label{L2h1 VML}
    \frac{1}{2}\frac{\mathrm{d}}{\mathrm{d} t}\|w_0h^{\varepsilon}\|^2&+Y\|{\lag v\rag}w_0h^{\varepsilon}\|^2
    +\frac{\delta}{\varepsilon}\|w_0h^{\varepsilon}]\|^2_{\bf D} \\
    \leq&\;\frac{C}{\varepsilon}\|f^{\varepsilon}\|^2+\frac{1}{\varepsilon}\big|\big\langle \mathcal{L}_d[ h^{\varepsilon}], w^2_0h^{\varepsilon}\big\rangle\big|+\Big|\Big\langle \big(E_R^{\varepsilon}+v \times B_R^{\varepsilon} \big) \cdot \frac{v-u}{ RT}\mu^{-\frac{1}{2}}\mathbf{M},w^2_0 h^{\varepsilon}\Big\rangle\Big| \nonumber\\
    & +\Big|\Big\langle \big(E+v \times B \big) \cdot \nabla_v h^{\varepsilon}, w^2_0 h^{\varepsilon}\Big\rangle\Big|  +\Big|\Big\langle \frac{E\cdot v }{ 2RT_c}h^{\varepsilon}, w^2_0 h^{\varepsilon}\Big\rangle\Big|\nonumber\\
    &+\varepsilon^{k-1}\big|\big\langle\Gamma ( h^{\varepsilon},
    h^{\varepsilon} ), w^2_0 h^{\varepsilon}\big\rangle\big|+\sum_{n=1}^{2k-1}\varepsilon^{n-1}\big|\big\langle[\Gamma(\mu^{-\frac{1}{2}}F_n, h^{\varepsilon})
    +\Gamma(
 h^{\varepsilon}, \mu^{-\frac{1}{2}} F_n)], w^2_0 h^{\varepsilon}\big\rangle\big|\nonumber\\
 &+\varepsilon^k \big|\big\langle \Big(E_R^{\varepsilon}+v \times B_R^{\varepsilon} \Big)\cdot\nabla_vh^{\varepsilon},w^2_0 h^{\varepsilon}\big\rangle\big|+\varepsilon^k \big|\big\langle \frac{E_R^{\varepsilon} \cdot v}{ 2RT_c}h^{\varepsilon},w^2_0 h^{\varepsilon}\big\rangle\big|\nonumber\\
 &+\sum_{n=1}^{2k-1}\varepsilon^n\Big|\Big\langle \big(E_n+v \times B_n \Big)\cdot\nabla_vh^{\varepsilon}+\big(E_R^{\varepsilon}+v \times B_R^{\varepsilon} \big)\cdot\nabla_v F_n,w^2_0 h^{\varepsilon}\Big\rangle\Big|\nonumber\\
 &+\sum_{n=1}^{2k-1}\varepsilon^n\Big|\Big\langle \frac{E_n\cdot v}{ 2RT_c}h^{\varepsilon},w^2_0 h^{\varepsilon}\Big\rangle\Big|
  +\big|\big\langle \CQ_1,w^2_0 h^{\varepsilon}\big\rangle\big|.\nonumber
\end{align}
We now proceed to estimate the terms on the R.H.S. of \eqref{L2h1 VML} term by term.

From \eqref{wGLLd}, the upper bound of the 2nd term on the R.H.S. of \eqref{L2h1 VML} is $C\epsilon_1\varepsilon^{-1} |w_0h^{\varepsilon}]|^2_{\bf D}$. Noting that ${\lag v\rag}w_0\mu^{-\frac{1}{2}}\mathbf{M}$ is bounded, the 3rd term can be bounded by $ \frac{o(1)}{\varepsilon}\|w_0h^{\varepsilon}]\|^2_{\bf D} +C\varepsilon\big(\|E_R^{\varepsilon}|^2+\|B_R^{\varepsilon}|^2\big)$. Applying Lemma \ref{dxlmh VML} and Lemma \ref{key-VML-2}, we can control the 4th, 5th, and 8th-11th terms by
\begin{eqnarray}
	&&\frac{o(1)}{\varepsilon}\|w_0h^{\varepsilon}]\|^2_{\bf D} +C\epsilon_1 Y\|{\lag v\rag}w_0h^{\varepsilon}\|^2+C\varepsilon^{2\kappa}\Big(\|E_R^{\varepsilon}\|^2+\|B_R^{\varepsilon}\|^2\Big).\nonumber
\end{eqnarray}

For the 6th term on the R.H.S. of \eqref{L2h1 VML},
we get from \eqref{wGGL} and Sobolev's inequality that
\begin{align}\label{gah-1}
    &\varepsilon^{k-1}\big|\big\langle\Gamma ( h^{\varepsilon},
    h^{\varepsilon} ), w^2_0  h^{\varepsilon}\big\rangle\big|\nonumber\\
    \lesssim&\varepsilon^{k-1}\int_{{\mathbb R}^3}\Big(\big|{\lag v\rag}^{\ell}h^{\varepsilon}\big|_{L^2}|w_0h^{\varepsilon}|_{\bf D}
    +|w_0h^{\varepsilon}|_{L^2}\big|{\lag v\rag}^{\ell}h^{\varepsilon}\big|_{\bf D}\Big)|w_0h^{\varepsilon}|_{\bf D}dx\, \\
    \lesssim&\varepsilon^{k-1}\Big(\big\|{\lag v\rag}^{\ell}h^{\varepsilon}\big\|_{H^2}\|w_0h^{\varepsilon}\|_{\bf D}^2
    +\|w_0h^{\varepsilon}\|\big\|{\lag v\rag}^{\ell}h^{\varepsilon}\big\|_{H^2_{\bf D}}\|w_0h^{\varepsilon}\|_{\bf D}\Big).\notag
\end{align}
Here, we've used the Sobolev's inequality $\|f\|_{L^{\infty}_x}\lesssim \|f\|_{H^2_x}$ to bound both $\big|{\lag v\rag}^{\ell}h^{\varepsilon}\big|_{L^2}$ and $\big|{\lag v\rag}^{\ell}h^{\varepsilon}\big|_{\bf D}$, as either $|w_0h^{\varepsilon}|_{L^2}$ or $|w_0h^{\varepsilon}|_{\bf D}$ cannot accommodate more $x-$derivatives. Note that in norms $\big\|{\lag v\rag}^{\ell}h^{\varepsilon}\big\|_{H^2}$ and $\big\|{\lag v\rag}^{\ell}h^{\varepsilon}\big\|_{H^2_{\bf D}}$, $\nabla_x^2h^{\varepsilon}$ can only bear the polynomial weight ${\lag v\rag}^{\ell-2}$ and an additional polynomial weight ${\lag v\rag}^{2}$ arises. We replace ${\lag v\rag}^{2}$ with the exponential weight $\exp\left(\frac{\langle v\rangle^2}{8RT_c\ln(\mathrm{e}+t)}\right)$ to obtain
\begin{align}\label{zh1}
\big\|{\lag v\rag}^{\ell}h^{\varepsilon}\big\|_{H^2}
\lesssim& \sup_{v\in {\mathbb R}^3}\Big[{\lag v\rag}^2\exp\left(\frac{-\epsilon_1\langle v\rangle^2}{8C_0RT_c\ln(\mathrm{e}+t)}\right)\Big] \big\|wh^{\varepsilon}\big\|_{H^2}\lesssim \frac{1}{\epsilon_1}\ln(\mathrm{e}+t) \|wh^{\varepsilon}\|_{H^2}.
\end{align}
Similarly, for the 7th term, we have
\begin{align}\label{zh2}
\big\|{\lag v\rag}^{\ell}h^{\varepsilon}\big\|_{H^2_{\bf D}}\lesssim& \sup_{v\in {\mathbb R}^3}\Big[{\lag v\rag}^2\exp
\left(\frac{-\epsilon_1\langle v\rangle^2}{8C_0RT_c\ln(\mathrm{e}+t)}\right)\Big] \big\|wh^{\varepsilon}\big\|_{H^2_{\bf D}}\lesssim \frac{1}{\epsilon_1}\ln(\mathrm{e}+t) \|wh^{\varepsilon}\|_{H^2_{\bf D}}.
\end{align}
Moreover, for $k\geq3$, it follows from the {\it a priori} assumption \eqref{aps-vml} that
\begin{align}\label{hf2b VML}
\varepsilon^{k-1}\|wh^{\varepsilon}\|_{H^i}\lesssim \varepsilon^{k-2-\frac{i}{2}}\lesssim \varepsilon^{1-\frac{i}{2}}.
\end{align}

Thus, by using $t\leq \varepsilon^{-\kappa}$ and the assumption $\varepsilon\ll\epsilon_1$, we can plug \eqref{zh1} and \eqref{zh2} into \eqref{gah-1} to get
\begin{align*}
    \varepsilon^{k-1}\big|\big\langle\Gamma ( h^{\varepsilon},
    h^{\varepsilon} ), w^2_0  h^{\varepsilon}\big\rangle\big|
     \lesssim& \frac{1}{\epsilon_1}\ln(\mathrm{e}+t)\|w_0h^{\varepsilon}\|_{\bf D}^2+\frac{o(1)}{\varepsilon}\|w_0h^{\varepsilon}\|_{\bf D}^2+\epsilon_1^{-2}\ln^2(\mathrm{e}+t)\varepsilon^3 \big\|wh^{\varepsilon}\big\|_{H^2_{\bf D}}\\
           \lesssim& \frac{o(1)}{\varepsilon}\|w_0h^{\varepsilon}\|_{\bf D}^2
    +\varepsilon^{2}\|wh^{\varepsilon}\|_{H^2_{\bf D}}^2.
\end{align*}

For the 7th term on the R.H.S. of \eqref{L2h1 VML}, we use
 \eqref{em-Fn-es} and Lemma \ref{wggmL} to bound it by
\begin{align*}
C\sum_{i=1}^{2k-1}[\varepsilon(1+t)]^{i-1}(1+t)\Big(\|w_0h^{\varepsilon}\|+\|w_0h^{\varepsilon}\|_{\bf D}\Big)\|w_0h^{\varepsilon}\|_{\bf D}
\lesssim&\frac{o(1)}{\varepsilon}\|w_0h^{\varepsilon}\|_{\bf D}^2+\varepsilon^{\kappa}\|w_0h^{\varepsilon}\|^2\nonumber.
\end{align*}

Similarly, for the last term on the R.H.S. of \eqref{L2h1 VML}, we have
\begin{align*}
\big|\big\langle \CQ_1, w^2_0 h^{\varepsilon}\big\rangle\big|\lesssim \frac{\epsilon_1}{\varepsilon} \|w_0h^{\varepsilon}]\|^2_{\bf D} +\varepsilon^{2k+1}(1+t)^{4k+2}+\varepsilon^{k}(1+t)^{2k}\|w_0h^{\varepsilon}\|.
\end{align*}

Putting these estimates into \eqref{L2h1 VML} and multiplying the resulting inequality by $\varepsilon^{1+\kappa}$ yields \eqref{L2h VML}.
\vskip 0.2cm
\noindent\underline{{\it Step 2. First order derivative estimates of $h^{\varepsilon}$ with weight.} }
In this step, we continue deducing the estimate of $\|w_1\nabla_xh^{\varepsilon}\|$. We apply $\partial^{\alpha}$ for $1\leq |\alpha|\leq 2$ to \eqref{h-equation}, resulting in the following equation:
\begin{align}\label{mh1x VML}
 \partial_t\partial^{\alpha}h^{\varepsilon}&+v\cdot\nabla_x\partial^{\alpha}h^{\varepsilon}+\partial^{\alpha}\Big[\frac{\big(E_R^{\varepsilon}+v \times B_R^{\varepsilon} \big) }{ RT}\cdot \big(v-u\big)\mu^{-\frac{1}{2}}\mathbf{M}\Big]\nonumber\\
&+\partial^{\alpha}\Big[\frac{E\cdot v }{ 2RT_c}h^{\varepsilon}\Big]-\partial^{\alpha}\big[\big(E+v \times B \big)\cdot\nabla_vh^{\varepsilon}\big]+\frac{\partial^{\alpha}\mathcal{L}[h^{\varepsilon}]}{\varepsilon}\nonumber\\
 =&-\frac{\partial^{\alpha}\mathcal{L}_d[h^{\varepsilon}]}{\varepsilon}+\varepsilon^{k-1}\partial^{\alpha}\Gamma(h^{\varepsilon},h^{\varepsilon})
 +\sum_{n=1}^{2k-1}\varepsilon^{n-1}[\partial^{\alpha}\Gamma(\mu^{-\frac{1}{2}}F_n, h^{\varepsilon})+\partial^{\alpha}\Gamma(h^{\varepsilon}, \mu^{-\frac{1}{2}}F_n)]\nonumber\\
 &+\varepsilon^k \partial^{\alpha}\Big[\Big(E_R^{\varepsilon}+v \times B_R^{\varepsilon}\Big)\cdot\nabla_vh^{\varepsilon}\Big]-\varepsilon^k  \partial^{\alpha}\Big[\frac{E_R^{\varepsilon}\cdot v}{ 2RT_c}h^{\varepsilon}\Big]\\
 &+\sum_{n=1}^{2k-1}\varepsilon^n\partial^{\alpha}\Big[\Big(E_n+v \times B_n \Big)\cdot\nabla_vh^{\varepsilon}+\Big(E_R^{\varepsilon}+v \times B_R^{\varepsilon} \Big)\cdot\nabla_v F_n\Big]\nonumber\\
 &-\sum_{n=1}^{2k-1}\varepsilon^n\partial^{\alpha}\Big(\frac{E_n \cdot v}{ 2RT_c}h^{\varepsilon}\Big)+\varepsilon^{k}\partial^{\alpha}\CQ_1.
\nonumber
\end{align}
Next, we take the $L^2$ inner product of $w^2_1 \partial^{\alpha} h^{\varepsilon}$ with \eqref{mh1x VML}, considering $|\alpha|=1$, and apply \eqref{wLLL}:
\begin{align}\label{H1h1 VML}
&\frac{1}{2}\frac{\mathrm{d}}{\mathrm{d} t}\|w_1 \nabla_xh^{\varepsilon}\|^2+Y(t)\big\|{\lag v\rag}w_1\nabla_xh^{\varepsilon}\big\|^2
    +\frac{\delta}{\varepsilon}\|w_1 \nabla_xh^{\varepsilon}]\|^2_{\bf D} \\
    \leq&\;\frac{C}{\varepsilon}\|f^{\varepsilon}\|^2_{H^1}+\frac{1}{\varepsilon}\big|\big\langle \nabla_x\mathcal{L}_d[ h^{\varepsilon}], w^2_1 \nabla_xh^{\varepsilon}\big\rangle\big|\nonumber\\
    &+\Big|\Big\langle \nabla_x\Big[\big(E_R^{\varepsilon}+v \times B_R^{\varepsilon} \big) \cdot \frac{v-u}{ RT}\mu^{-\frac{1}{2}}\mathbf{M}\Big], w^2_1 \nabla_xh^{\varepsilon}\Big\rangle\Big|
    +\Big|\Big\langle \nabla_x\Big[\frac{E\cdot v }{ 2RT_c}h^{\varepsilon}\Big], w^2_1 \nabla_xh^{\varepsilon}\Big\rangle\Big|\nonumber\\
    &  +\big|\big\langle \nabla_x\big[\big(E+v \times B \big)\cdot\nabla_vh^{\varepsilon}\big],w^2_1 \nabla_x h^{\varepsilon}\big\rangle\big|  +\varepsilon^{k-1}\big|\big\langle \nabla_x\Gamma ( h^{\varepsilon},
    h^{\varepsilon} ), w^2_1 \nabla_x h^{\varepsilon}\big\rangle\big|\nonumber\\
    &+\sum_{n=1}^{2k-1}\varepsilon^{n-1}\big|\big\langle[\nabla_x\Gamma(\mu^{-\frac{1}{2}}F_n,h^{\varepsilon})
    +\nabla_x\Gamma(
 h^{\varepsilon}, \mu^{-\frac{1}{2}} F_n)], w^2_1 \nabla_x h^{\varepsilon}\big\rangle\big|\nonumber\\
 &+\varepsilon^k \big|\big\langle \nabla_x\Big[\Big(E_R^{\varepsilon}+v \times B_R^{\varepsilon} \Big)\cdot\nabla_vh^{\varepsilon}\Big], w^2_1 \nabla_x h^{\varepsilon}\big\rangle\big|+\varepsilon^k \Big|\Big\langle \nabla_x\Big[\frac{E_R^{\varepsilon} \cdot v}{ 2RT_c}h^{\varepsilon}\Big], w^2_1 \nabla_x h^{\varepsilon}\Big\rangle\Big|\nonumber\\
 &+\sum_{n=1}^{2k-1}\varepsilon^n\Big|\Big\langle \nabla_x\Big[\big(E_n+v \times B_n \big)\cdot\nabla_vh^{\varepsilon}+\big(E_R^{\varepsilon}+v \times B_R^{\varepsilon} \big)\cdot\nabla_v F_n\Big], w^2_1 \nabla_x h^{\varepsilon}\Big\rangle\Big|\nonumber\\
 &+\sum_{n=1}^{2k-1}\varepsilon^n\Big|\Big\langle \nabla_x\Big(\frac{E_n\cdot v}{ 2RT_c}h^{\varepsilon}\Big), w^2_1 \nabla_x h^{\varepsilon}\Big\rangle\Big|
  +\big|\big\langle \nabla_x\CQ_1, w^2_1 \nabla_x h^{\varepsilon}\big\rangle\big|.\nonumber
\end{align}
Now let's estimate the R.H.S. of equation \eqref{H1h1 VML} individually.

The upper bound for the 2nd term on the R.H.S. of \eqref{H1h1 VML} is
$C\epsilon_1\varepsilon^{-1}\|wh^{\varepsilon}\|^2_{H^1_{\bf D}}.$
The 3rd term on the R.H.S. of \eqref{H1h1 VML} can be bounded by
$ o(1)\varepsilon^{-1}\|w_1 \nabla_xh^{\varepsilon}\|^2_{\bf D}+C\varepsilon\big(\|E_R^{\varepsilon}\|_{H^1}^2+\|B_R^{\varepsilon}\|_{H^1}^2\big).$
Using Lemma \ref{dxlmh VML} and Lemma \ref{key-VML-2}, we can bound the 4th, 5th, and 8th to 11th terms on the R.H.S. of \eqref{H1h1 VML} by
\begin{align*}
C\Big[\epsilon_1Y(t)\|{\lag v\rag}w_1 \nabla_xh^{\varepsilon}\|^2+\big[(1+t)^{-p_0}+\varepsilon^{2\kappa}\big]\big(\|w_0 h^{\varepsilon}\|^2
+\|E_R^{\varepsilon}\|_{H^1}^2+\|B_R^{\varepsilon}\|_{H^1}^2\big)+\epsilon_1\|w h^{\varepsilon}\|^2_{H^{1}_{\bf D}}\Big].
\end{align*}
Regarding the 6th term on the R.H.S. of \eqref{H1h1 VML}, one can use \eqref{wGGL} and Sobolev's inequalities to bound it by
\begin{align*}
C\varepsilon^{k-1}\Big(\big\|{\lag v\rag}^{\ell-1}h^{\varepsilon}\big\|_{H^2}\|w_1\nabla_xh^{\varepsilon}\|_{\bf D}^2
    +\|w_1\nabla_xh^{\varepsilon}\|\big\|{\lag v\rag}^{\ell-1}h^{\varepsilon}\big\|_{H^2_{\bf D}}\|w_1\nabla_xh^{\varepsilon}\|_{\bf D}\Big).
\end{align*}
On the other hand, similar to \eqref{zh1} and \eqref{zh2}, one has
\begin{align*}
\big\|{\lag v\rag}^{\ell-1}h^{\varepsilon}\big\|_{H^2}
\lesssim \frac{\sqrt{\ln(\mathrm{e}+t)}}{\sqrt{\epsilon_1}} \|wh^{\varepsilon}\|_{H^2}, \quad
\big\|{\lag v\rag}^{\ell-1}h^{\varepsilon}\big\|_{H^2_{\bf D}}
\lesssim \frac{\sqrt{\ln(\mathrm{e}+t)}}{\sqrt{\epsilon_1}} \|wh^{\varepsilon}\|_{H^2_{\bf D}}.
\end{align*}
Then we can further bound this term by
$ o(1)\varepsilon^{-1}\|w_1\nabla_xh^{\varepsilon}\|_{\bf D}^2
    +C\varepsilon\|wh^{\varepsilon}\|_{H^2_{\bf D}}^2$.

For the 7th term on the R.H.S. of \eqref{H1h1 VML}, considering \eqref{em-Fn-es} and Lemma \ref{wggmL}, we can establish its upper bound as
$o(1)\varepsilon^{-1}\|w_1\nabla_xh^{\varepsilon}\|_{\bf D}^2+C\varepsilon^{\kappa}\big(\|wh^{\varepsilon}\|^2_{H^1}
    +\|w_0h^{\varepsilon}\|^2_{\bf D}\big)$
     for $t\leq \varepsilon^{-\kappa}$.

Similarly, for the last term on the R.H.S. of \eqref{H1h1 VML}, it can be estimated as
\begin{align*}
\big|\big\langle \nabla_x\CQ_1, w^2_1\nabla_x h^{\varepsilon}\big\rangle\big|\lesssim \frac{\epsilon_1}{\varepsilon} \|w_1\nabla_xh^{\varepsilon}]\|^2_{\bf D} +\varepsilon^{2k+1}(1+t)^{4k+2}+\varepsilon^{k}(1+t)^{2k}\|w_1\nabla_xh^{\varepsilon}\|.
\end{align*}
Plugging these estimates into \eqref{H1h1 VML} and multiplying the resulting inequality by $\varepsilon^{2+\kappa}$, we get
\begin{align}\label{H1h VML}
&\frac{\mathrm{d}}{\mathrm{d} t}\big(\varepsilon^{2+\kappa}\|w_1\nabla_xh^{\varepsilon}\|^2\big)+\varepsilon^{2+\kappa}Y\big\|{\lag v\rag}w_1\nabla_xh^{\varepsilon}\big\|^2
    +\delta\varepsilon^{1+\kappa}\|w_1\nabla_xh^{\varepsilon}\|_{\bf D}^2 \\
    \lesssim&\,\varepsilon^{1+\kappa}\|f^{\varepsilon}\|^2_{H^1}+\varepsilon^{2+2\kappa}\Big(\|E_R^{\varepsilon}\|^2_{H^1}+\|B_R^{\varepsilon}\|^2_{H^1}
    +\|wh^{\varepsilon}\|^2_{H^1}\Big)+\varepsilon^{3+\kappa}\|wh^{\varepsilon}\|_{H^2_{\bf D}}^2\nonumber\\
    &+\epsilon_1\varepsilon^{2+\kappa}\Big((1+t)^{-p_0}\|w_0 h^{\varepsilon}\|^2+\|w_0 h^{\varepsilon}\|^2_{\bf D}\Big)\nonumber\\
    &+\varepsilon^{2k+3+\kappa}(1+t)^{4k+2}+\varepsilon^{k+2+\kappa}(1+t)^{2k}\|w_1\nabla_xh^{\varepsilon}\|.\nonumber
\end{align}

\vskip 0.2cm
\noindent\underline{{\it Step 3. Second order derivative estimates of $h^{\varepsilon}$ with weight.} }
In this final step, we proceed to derive the following estimate for $\|w_2\nabla_x^2h^{\varepsilon}\|$:
\begin{align}\label{H2h VML}
\frac{\mathrm{d}}{\mathrm{d} t}\big(\varepsilon^{3+\kappa}&\|w_2\nabla_x^2h^{\varepsilon}\|^2\big)+\varepsilon^{3+\kappa}Y\big\|{\lag v\rag}w_2\nabla_x^2h^{\varepsilon}\big\|^2
    +\delta\varepsilon^{2+\kappa}\|w_2\nabla_x^2h^{\varepsilon}\|_{\bf D}^2 \\
    \lesssim&\,\varepsilon^{2+\kappa}\|f^{\varepsilon}\|^2_{H^2}+\varepsilon^{3+2\kappa}\Big(\|E_R^{\varepsilon}\|^2_{H^2}+\|B_R^{\varepsilon}\|^2_{H^2}
    +\|wh^{\varepsilon}\|^2_{H^2}\Big)\nonumber\\
    &+\epsilon_1\varepsilon^{3+\kappa}\Big[(1+t)^{-p_0}\|w h^{\varepsilon}\|^2_{H^1}+\|wh^{\varepsilon}\|_{H^1_{\bf D}}^2\Big]\nonumber\\
    &+\varepsilon^{2k+4+\kappa}(1+t)^{4k+2}+\varepsilon^{k+3+\kappa}(1+t)^{2k}\|w_2\nabla_x^2h^{\varepsilon}\|.\nonumber
\end{align}
To demonstrate this, we first take the $L^2$ inner product of $w^2_2 \partial^{\alpha} h^{\varepsilon}$ and \eqref{mh1x VML} with $|\alpha|=2$ and use \eqref{wLLL} to obtain
\begin{align}\label{H2h1 VML}
&\frac{1}{2}\frac{\mathrm{d}}{\mathrm{d} t}\|w_2 \nabla_x^2h^{\varepsilon}\|^2+Y(t)\big\|{\lag v\rag}w_2\nabla_x^2h^{\varepsilon}\big\|^2
    +\frac{\delta}{\varepsilon}\|w_2 \nabla_x^2h^{\varepsilon}]\|^2_{\bf D} \\
    \leq&\;\frac{C}{\varepsilon}\|f^{\varepsilon}\|^2_{H^2}+\frac{1}{\varepsilon}\big|\big\langle \nabla_x^2\mathcal{L}_d[ h^{\varepsilon}], w^2_2 \nabla_x^2h^{\varepsilon}\big\rangle\big|\nonumber\\
    &+\Big|\Big\langle \nabla_x^2\Big[\big(E_R^{\varepsilon}+v \times B_R^{\varepsilon} \big) \cdot \frac{v-u}{ RT}\mu^{-\frac{1}{2}}\mathbf{M}\Big], w^2_2 \nabla_x^2h^{\varepsilon}\Big\rangle\Big|+\Big|\Big\langle \nabla_x^2\Big[\frac{E\cdot v }{ 2RT_c}h^{\varepsilon}\Big], w^2_2 \nabla_x^2h^{\varepsilon}\Big\rangle\Big|\nonumber\\
    &  +\big|\big\langle \nabla_x^2\big[\big(E+v \times B \big)\cdot\nabla_vh^{\varepsilon}\big],w^2_2 \nabla_x^2 h^{\varepsilon}\big\rangle\big|  +\varepsilon^{k-1}\big|\big\langle \nabla_x^2\Gamma ( h^{\varepsilon},
    h^{\varepsilon} ), w^2_2 \nabla_x^2 h^{\varepsilon}\big\rangle\big|\nonumber\\
    &+\sum_{n=1}^{2k-1}\varepsilon^{n-1}\big|\big\langle[\nabla_x^2\Gamma(\mu^{-\frac{1}{2}} F_n,h^{\varepsilon})
    +\nabla_x^2\Gamma(
 h^{\varepsilon}, \mu^{-\frac{1}{2}} F_n)], w^2_2 \nabla_x^2 h^{\varepsilon}\big\rangle\big|\nonumber\\
 &+\varepsilon^k \big|\big\langle \nabla_x^2\Big[\Big(E_R^{\varepsilon}+v \times B_R^{\varepsilon} \Big)\cdot\nabla_vh^{\varepsilon}\Big], w^2_2 \nabla_x^2 h^{\varepsilon}\big\rangle\big|+\varepsilon^k \Big|\Big\langle \nabla_x^2\Big[\frac{E_R^{\varepsilon} \cdot v}{ 2RT_c}h^{\varepsilon}\Big], w^2_2 \nabla_x^2 h^{\varepsilon}\Big\rangle\Big|\nonumber\\
 &+\sum_{n=1}^{2k-1}\varepsilon^n\Big|\Big\langle \nabla_x^2\Big[\big(E_n+v \times B_n \big)\cdot\nabla_vh^{\varepsilon}+\big(E_R^{\varepsilon}+v \times B_R^{\varepsilon} \big)\cdot\nabla_v F_n\Big], w^2_2 \nabla_x^2 h^{\varepsilon}\Big\rangle\Big|\nonumber\\
 &+\sum_{n=1}^{2k-1}\varepsilon^n\Big|\Big\langle \nabla_x^2\Big(\frac{E_n\cdot v}{ 2RT_c}h^{\varepsilon}\Big), w^2_2 \nabla_x^2 h^{\varepsilon}\Big\rangle\Big|
  +\big|\big\langle \nabla_x^2\CQ_1, w^2_2 \nabla_x^2 h^{\varepsilon}\big\rangle\big|.\nonumber
\end{align}
Now, let's estimate the R.H.S. of \eqref{H2h1 VML} term by term.

For the 2nd and 3rd terms on the R.H.S. of \eqref{H2h1 VML}, their upper bounds are
$C\epsilon_1\varepsilon^{-1}\|w h^{\varepsilon}\|^2_{H^2_{\bf D}}+C\varepsilon\Big(\|E_R^{\varepsilon}\|_{H^2}^2+\|B_R^{\varepsilon}\|_{H^2}^2\Big).$
For the 4th, 5th, 8th-11th terms on the R.H.S. of \eqref{H2h1 VML}, based on Lemma \ref{dxlmh VML} and Lemma \ref{key-VML-2}, we can bound them as follows:
\begin{align*}
C\Big[\epsilon_1Y(t)\|{\lag v\rag}w_2\nabla_x^2h^{\varepsilon}\|^2+\big[(1+t)^{-p_0}+\varepsilon^{2\kappa}\big]
\big(\|w h^{\varepsilon}\|^2_{H^1}+\|E_R^{\varepsilon}\|_{H^2}^2+\|B_R^{\varepsilon}\|_{H^2}^2\big)+\epsilon_1\|w h^{\varepsilon}\|^2_{H^{2}_{\bf D}}\Big].
\end{align*}
Regarding the 6th and 7th terms on the R.H.S. of \eqref{H2h1 VML},
from \eqref{wGGL}, \eqref{em-Fn-es}, \eqref{hf2b VML},  and Sobolev's inequalities, we can control them by
\begin{align*}
C \|wh^{\varepsilon}\|_{H^2_{\bf D}}^2+\frac{o(1)}{\varepsilon}\|w_2\nabla_x^2h^{\varepsilon}\|_{\bf D}^2 +C\varepsilon^{\kappa}\Big(\|wh^{\varepsilon}\|^2_{H^2}
    +\|wh^{\varepsilon}\|^2_{H^1_{\bf D}}\Big)\nonumber.
\end{align*}

Similarly, for the last term on the R.H.S. of \eqref{H2h1 VML}, we can express it as
\begin{align*}
\big|\big\langle \nabla_x^2\CQ_1, w^2_2\nabla_x^2 h^{\varepsilon}\big\rangle\big|\lesssim \frac{o(1)}{\varepsilon} \|w_2\nabla_x^2h^{\varepsilon}]\|^2_{\bf D} +\varepsilon^{2k+1}(1+t)^{4k+2}+\varepsilon^{k}(1+t)^{2k}\|w_2\nabla_x^2h^{\varepsilon}\|.
\end{align*}

By substituting these estimates into \eqref{H2h1 VML} and multiplying the resulting inequality by $\varepsilon^{3+\kappa}$, we obtain \eqref{H2h VML}.

Finally, \eqref{h-vml-eng} follows from \eqref{L2h VML}, \eqref{H1h VML}, and \eqref{H2h VML}.
\end{proof}

\subsection{Proof of the Theorem \ref{resultVML}} 
We are now prepared to conclude the proof of Theorem~\ref{resultVML}.  For simplicity, we will exclusively derive the energy estimates \eqref{TVML1} and omit details of non-negativity proof of $F^{\varepsilon}$ as it closely resembles the proof provided in \cite{Ouyang-Wu-Xiao-arxiv-2022-rLM}.

For this purpose, combing Propositions \ref{h-vml-eng-prop} and \ref{f-eng-vml}, one has
\begin{align*}
\frac{\mathrm{d}}{\mathrm{d} t}\mathcal{E}(t)
&+\frac{\delta}{2}\sum_{i=0}^2 \varepsilon^i\Big[\frac{1}{\varepsilon}\|\nabla_x^i({\bf I}-{\bf P}_{\mathbf{M}})[f^{\varepsilon}]\|^2_{\bf D}+\varepsilon^{\kappa}\|w_i\nabla_x^i h^{\varepsilon}\|^2_{\bf D}+\varepsilon^{1+\kappa}Y\|{\lag v\rag}w_i\nabla_x^i h^{\varepsilon}\|^2\Big]\\
\lesssim & \sum_{i=0}^2\varepsilon^{i+\kappa}\Big[\|\nabla_x^if^{\varepsilon}\|^2+\Big(\varepsilon^{1+\kappa}
+\varepsilon(1+t)^{-p_0}\Big)\|w_i\nabla_x^ih^{\varepsilon}\|^2+\epsilon_1\varepsilon\|w_i\nabla_x^ih^{\varepsilon}\|^2_{\bf D}\Big)\\
&+ \varepsilon^{3+\kappa}\|w h^{\varepsilon}\|^2_{H^2_{\bf D}}+\sum_{i=0}^2\varepsilon^i\Big((1+t)^{-p_0}+\varepsilon^{\kappa}\Big)
\Big[\|\nabla_x^if^{\varepsilon}\|^2
     +\|E_R^{\varepsilon}\|^2_{H^i}+\|B_R^{\varepsilon}\|^2_{H^i}
     \nonumber\\
     &+ C_{\epsilon_1}\exp\left(-\frac{\epsilon_1}{8C_0RT^2_c\sqrt{\varepsilon}}\right)
    \Big(\|h^{\varepsilon}\|^2_{H^i}+\|h^{\varepsilon}\|^2_{H^i_{\bf D}}\Big)\Big]\\
      &+\varepsilon^{2k+1}(1+t)^{4k+2}+\varepsilon^{k}(1+t)^{2k}\sum_{i=0}^2\varepsilon^i\Big(\|\nabla_x^if^{\varepsilon}\|+\varepsilon^{1+\kappa}\|w_i\nabla_x^i h^{\varepsilon}\|\Big),
\end{align*}
where $\mathcal{E}(t)$ is defined in \eqref{eg-vml}.

Next, choosing $\varepsilon_0\geq \varepsilon>0$ sufficiently small such that
$$C_{\epsilon_1}\exp\left(\frac{-\epsilon_1}{8C_0RT^2_c\sqrt{\varepsilon}}\right)\leq \varepsilon_0^{2+\kappa}, \qquad \kappa=\frac13,$$
we have
\begin{align*}
&\frac{\mathrm{d}}{\mathrm{d} t}\mathcal{E}(t)+\mathcal{D}(t)\lesssim \Big[\varepsilon^{\kappa}+(1+t)^{-p_0}+\varepsilon^{k}(1+t)^{2k}\Big] \mathcal{E}(t)+\varepsilon^{2k+1}(1+t)^{4k+2},
\end{align*}
where $\mathcal{D}(t)$ is given by \eqref{dn-vml}. Then for $0<\varepsilon\leq \varepsilon_0$, we apply Gronwall's inequality over $t\in [0, \varepsilon^{-\kappa}]$ with $\kappa=\frac13$ to obtain
\begin{align*}
&\mathcal{E}(t)+\int_0^t\mathcal{D}(s)\, d s\lesssim \mathcal{E}(0)+1.
\end{align*}
This completes the proof of Theorem~\ref{resultVML}.


\setcounter{equation}{0}



\section{Global Hilbert expansion for the non-cutoff VMB system}\label{sec-vmb}
To justify the validity of the Hilbert expansion \eqref{expan} for the non-cutoff VMB system, as discussed in Section \ref{h-VML},
our first task is to determine the coefficients $[F_n,E_n,B_n] (1\leq n\leq 2k-1)$, which is given by the iteration equations \eqref{expan2}. Subsequently, we need to prove the wellposedness of the remainders $[f^\vps,E^\vps_R,B^\vps_R]$ and $h^\vps$ via the equations \eqref{VMLf} and \eqref{h-equation} with prescribed initial data \eqref{fEB-initial} and \eqref{hvps-initial}, respectively.

\subsection{Preliminaries}
In this subsection, we gather various estimates related to the collision operators of the non-cutoff VMB system. To facilitate later analysis, we will derive estimates for the linearized collision operators $\mathcal{L}_{\mathbf{M}}$, $\mathcal{L}_d$, and the nonlinear collision operator $\Gamma_{\mathbf{M}}$. These estimates are adaptations of those corresponding to the global Maxwellian $\mu$, as established in prior works like \cite{AMUXY-JFA-2012, Gressman-Strain-2011}. To achieve this, we must overcome difficulties arising from the local Maxwellian, such as the non-commutative property of the collision operators $\mathcal{L}_{\mathbf{M}}$, $\Gamma_{\mathbf{M}}$, and the differential operators $\partial^{\alpha}$. Simultaneously, we aim to improve the exponentially weighted estimate for $\Gamma$, previously proven in \cite{Duan-Liu-Yang-Zhao-2013-KRM, FLLZ-SCM-2018}. The improved estimate for $\Gamma$ holds significance not only for the weighted estimates associated with the coefficient $F_1$ but also of its own interest.

\subsubsection{ The linearized collision operators} In contrast to the VML case, the expression of $\CL_\FM$ is now less complicated. Similar to Subsection \ref{1cl}, we can also establish corresponding estimates on the operators $\CL$ and $\CL_\FM$ in the non-cutoff VMB case.
\begin{lemma}
 The linearized collision operators $\CL$ and $\CL_\FM$ satisfy the following coercivity estimates:
	\begin{itemize}
		\item [(i)] For the operator $\CL$ corresponding to the global Maxwellian $\mu$, it holds that
		\begin{equation}
			\big\lag \mathcal{L}h, h\big\rag\gtrsim |({\bf I}-{\bf P})h|_{\bf D}^2.\label{coL-Non}
		\end{equation}
		Moreover, let $|\beta|>0$, $\ell\geq0$  and $q\in [0, \infty)$, for $\eta>0$ small enough, there exists $C_{\eta}>0$ such that
		\begin{align}\label{L_v-Non}
			\left\langle \partial_{\beta}\mathcal{L}h, \lag v\rag^{2\ell}e^{2q\langle v\rangle}\partial_{\beta}h\right\rangle\geq& \left|\lag v\rag^{\ell}e^{q\langle v\rangle}\partial_{\beta}f\right|^2_{\bf D}
			-\eta\sum_{|\beta'|<|\beta|}\left|w_{\ell}\partial_{\beta'}({\bf I}-{\bf P})h\right|_{\bf D}^2\\
&-C_{\eta}\left|\chi_{\{| v|\leq2C_{\eta}\}}h\right|^2_{L^2}.\notag
		\end{align}
		\item [(ii)] Let $\FM=\FM_{[\rho,u,T]},$ where $[\rho,u,T]$ is determined by  Proposition \ref{em-ex-lem}, then there exists a constant $\de>0$ such that
		\begin{align}
			\big(\mathcal{L}_{\mathbf{M}}f, f\big)\geq\de \big| ({\bf I}-{\bf P}_{\mathbf{M}})f\big|_{\bf D}^2. \label{coL0-Non}
		\end{align}
		Moreover, for $\ell\geq 0$, $q\in [0, \infty)$, $|\beta|>0$, and some small constant $\eta>0$, it holds that
		\begin{align}\label{wLLM-Non}
			&\big(\partial^{\alpha}_{\beta}\mathcal{L}_{\mathbf{M}}f,   \lag v\rag^{2\ell}e^{2q\langle v\rangle}\partial_\beta^{\alpha}f\big)\nonumber\\
			\geq& \de\big| \lag v\rag^{\ell}e^{q\langle v\rangle}\partial^{\alpha}_{\beta}f\big|_{\bf D}^2-C(\eta+\epsilon_1)\sum_{\al'\leq \alpha}\sum_{|\beta'|\leq|\beta|}\big| \lag v\rag^{\ell}e^{q\langle v\rangle}\partial^{\al'}_{\beta'}f\big|_{\bf D}^2\\
			&-C(\eta)\sum_{\al'\leq \alpha}\sum_{|\beta'|<|\beta|}\big| \partial^{\al'}_{\beta'}f\big|^2_{L^2(B_{C(\eta)})},\nonumber
		\end{align}
			and for $|\beta|=0$, it holds that
		\begin{align}\label{wLLM0-Non}
			&\big(\partial^{\alpha}\mathcal{L}_{\mathbf{M}}f, \lag v\rag^{2\ell}e^{2q\langle v\rangle} \partial^{\alpha}f\big)\nonumber\\
			\geq& \de\big|\lag v\rag^{\ell}e^{q\langle v\rangle}\partial^{\alpha}f\big|_{\bf D}^2
			-C\epsilon_1\sum_{\al'< \alpha}{\bf 1}_{|\alpha|\geq1}\big|\lag v\rag^{\ell}e^{q\langle v\rangle}\partial^{\al'}f\big|_{\bf D}^2\\
			&-C(\eta)\sum_{\al'\leq \alpha}\big|\lag v\rag^{\ell}e^{q\langle v\rangle}\partial^{\al'}f\big|^2_{L^2(B_{C(\eta)})}.\nonumber
		\end{align}
				Furthermore, for $|\alpha|>0$, there exists a polynomial $\CU(\cdot)$ with $\CU(0)=0$ such that
		\begin{align}\label{coLh-Non}
			\big\langle\partial^{\alpha}\mathcal{L}_{\mathbf{M}}f, \partial^{\alpha}f\big\rangle\geq&\frac{3\de}{4}
\|\partial^{\alpha}({\bf I}-{\bf P}_{\mathbf{M}})f\|_{\bf D}^2-C\CU(\|\nabla_x[\rho,u,T]\|_{W^{|\alpha|-1,\infty}})\\
			&\times\Big( \frac{1}{\varepsilon}\|({\bf I}-{\bf P}_{\mathbf{M}})f\|_{H^{|\alpha|-1}_{\bf D}}^2+\varepsilon\|f\|^2_{H^{|\alpha|}}\Big).\nonumber
		\end{align}
	\end{itemize}
\end{lemma}
\begin{proof}
	\eqref{coL-Non} is derived from \cite[Theorem 8.1, p. 829]{Gressman-Strain-2011} and \cite[Theorem 1.1, p. 920]{AMUXY-JFA-2012}. Equation \eqref{L_v-Non} is a direct consequence of Lemmas 2.6 and 2.7 in \cite{Duan-Liu-Yang-Zhao-2013-KRM}. \eqref{coL0-Non} and \eqref{wLLM-Non} can be established following the same method as for \eqref{coL-Non} and \eqref{L_v-Non}, respectively. It's important to note that, for any vector indices $\alpha$ and $\beta$, the inequality
\begin{align*}
 |\partial^{\alpha}_{\beta}\mathbf{M}|\lesssim \mu^{5/6}
 \end{align*}
is valid due to \eqref{tt1} and the smallness of $\epsilon_1$.

In contrast, \eqref{wLLM0-Non} and \eqref{coLh-Non} are consequences of \eqref{tt1}, \eqref{em-decay}, and \eqref{wLLM-Non}.
\end{proof}

\subsubsection{Nonlinear collision operators}
Recalling that the operators $\Gamma$ and $\Gamma_{\mathbf{M}}$ are defined as
\begin{align*}
\Gamma(f,g)=\mu^{-\frac{1}{2}} \mathcal{C}( \mu^{\frac{1}{2}} f, \mu^{\frac{1}{2}} g),\quad \Gamma_{\mathbf{M}}(f,g)=\mathbf{M}^{-\frac{1}{2}} \mathcal{C}( \mathbf{M}^{\frac{1}{2}} f, \mathbf{M}^{\frac{1}{2}} g).
\end{align*}
For later use, we will derive trilinear estimates for the two nonlinear collision operators. The first lemma concerns trilinear estimates of $\Gamma$ and $\Gamma_{\mathbf{M}}$ without velocity weight.
\begin{lemma}
Assume $\max\big\{-3, -\frac{3}{2}-2s\big\}<\gamma<0$. It holds that
\begin{align}\label{GL-Non}
\big|\big(\partial^{\alpha}_{\beta}\Gamma(f,g),\partial^{\alpha}_{\beta} h\big)\big|
\lesssim& \sum_{\al_1+\al_2= \alpha, \beta_1+\beta_2\leq\beta}\Big(\big|\partial^{\al_1}_{\beta_1}f\big|_{L^2}\big|\partial^{\al_2}_{\beta_2}g\big|_{\bf D}+\big|\partial^{\al_1}_{\beta_1}f\big|_{\bf D}
\big|\partial^{\al_2}_{\beta_2}g\big|_{L^2}\Big)\big|\partial^{\alpha}_{\beta}h\big|_{\bf D}
\end{align}
and
\begin{align}\label{GLM-Non}
&\big|\big(\partial^{\alpha}_{\beta}\Gamma_{\mathbf{M}}(f,g), \partial^{\alpha}_{\beta}h\big)\big|\nonumber\\
\lesssim& \sum_{\al_1+\al_2= \alpha, \beta_1+\beta_2\leq\beta}\Big(\big|\partial^{\al_1}_{\beta_1}f\big|_{L^2}\big|\partial^{\al_2}_{\beta_2}g\big|_{\bf D}+\big|\partial^{\al_1}_{\beta_1}f\big|_{\bf D}
\big|\partial^{\al_2}_{\beta_2}g\big|_{L^2}\Big)\big|\partial^{\alpha}_{\beta}h\big|_{\bf D}\\
&+ \sum_{\al_1+\al_2+\al_3= \alpha, \beta_1+\beta_2\leq\beta}{\bf 1}_{|\alpha_3|\geq1}\CU(\|\nabla_x[\rho,u,T]\|_{W^{|\alpha_3|-1,\infty}})\nonumber\\
&\times\Big(\big|\partial^{\al_1}_{\beta_1}f\big|_{L^2}\big|\partial^{\al_2}_{\beta_2}g\big|_{\bf D}+\big|\partial^{\al_1}_{\beta_1}f\big|_{\bf D}
\big|\partial^{\al_2}_{\beta_2}g\big|_{L^2}\Big)\big|\partial^{\alpha}_{\beta}h\big|_{\bf D}.\nonumber
\end{align}
\end{lemma}
\begin{proof}
	\eqref{GL-Non} follows from \cite[Theorem 1.2, p. 921]{AMUXY-JFA-2012} and \cite[Proposition 4.4, p. 202]{Duan-Liu-Yang-Zhao-2013-KRM}. with slight modifications. Since the local Maxwellian doesn't pose significant difficulties, \eqref{GLM-Non} can be proven similarly.
\end{proof}

We now prove exponentially weighted estimates for the nonlinear collision operator $\Gamma$ in the following lemma. These improved estimates address the undesirable terms that arise in previous results \cite[Lemma 2.3, p. 176]{Duan-Liu-Yang-Zhao-2013-KRM} for $0 < s < 1$ and \cite[Lemma 2.4, p. 121]{FLLZ-SCM-2018} for $0 < s < \frac{1}{2}$, such as $\left|\overline{w}g_1\right|_{1+\gamma/2}\big|\mu^{\frac{\lambda_0}{128}}g_2\big|_{L^2}
		\left|\overline{w}g_3\right|_{1+\gamma/2}$. This adjustment allows us to avoid the unmanageable term \eqref{unp-term} when dealing with the following inner product
 $$\big|\big\langle[\Gamma(\mu^{-\frac{1}{2}}F_1,h^{\varepsilon})+\Gamma(
 h^{\varepsilon}, \mu^{-\frac{1}{2}} F_1)], \overline{w}h^{\varepsilon}\big\rangle\big|.$$

  Our proof relies on a precise angular integration estimate. The estimates derived below are of their own interests and can be applied in further study of the non-cutoff Boltzmann-type equations.
\begin{lemma}\label{non-nonlinear}
	Assume $\max\{-3, -\frac{3}{2}-2s\}<\gamma<-2s$ for all
	$0<s<1$, and $\lambda_0>0$.
\begin{itemize}
		\item [(i)] It holds that	
	\begin{eqnarray}\label{Gamma-noncut-1}
			&&\left|\left( \partial_\beta^\alpha\Gamma(f,g), \overline{w}^2\partial_\beta^\alpha h\right)\right|\nonumber\\
			&\lesssim&\sum\left\{\left|\overline{w}\partial^{\alpha_1}_{\beta_1}f\right|_{L^2_
				{\frac\gamma2+s}}
			\left|\partial^{\alpha_2}_{\beta_2}g\right|_{{\bf D}}+\left|\partial^{\alpha_2}_{\beta_2}g\right|_{L^2_{\frac\gamma2+s}}
			\left|\overline{w}\partial^{\alpha_1}_{\beta_1}f\right|_{\bf D}\right\}
			\left|\overline{w}\partial^{\alpha}_{\beta}h\right|_{\bf D}\nonumber\\
			&&+\min\left\{\left|\overline{w}\partial^{\alpha_1}_{\beta_1}f\right|_{L^2}
			\left|\partial^{\alpha_2}_{\beta_2}g\right|_{L^2_{\frac\gamma2+s}},\left|\partial^{\alpha_2}_{\beta_2}g\right|_{L^2_v}
			\left|\overline{w}\partial^{\alpha_1}_{\beta_1}f\right|_{L^2_{\frac\gamma2+s}}\right\}
			\left|\overline{w}\partial^{\alpha}_{\beta}h\right|_{\bf D}\nonumber\\
			&&+\sum\left|e^{\frac{\langle v\rangle}{8RT_c\ln(e+t)}}\partial^{\alpha_2}_{\beta_2}g\right|_{L^2}
			\left|\overline{w}\partial^{\alpha_1}_{\beta_1}f\right|_{L^2_{\frac\gamma2+s}}
			\left|\overline{w}\partial^{\alpha}_{\beta}h\right|_{\bf D},
		\end{eqnarray}
		where the summation $\sum$ is taken over $\alpha_1+\alpha_2\leq \alpha$ and $\beta_1+\beta_2\leq\beta$.
\item [(ii)]For $\ell\geq0$ and $q\in [0,\infty)$, we have
\begin{eqnarray}\label{Gamma-noncut-2}
			&&\left|\left( \partial_\beta^\alpha\Gamma(f,g), \lag v\rag^{2\ell}e^{2q\langle v\rangle}\partial_\beta^\alpha h\right)\right|\nonumber\\
			&\lesssim&\sum\left\{\left|\lag v\rag^{\ell}e^{q\langle v\rangle}\partial^{\alpha_1}_{\beta_1}f\right|_{L^2_
				{\frac\gamma2+s}}
			\left|\partial^{\alpha_2}_{\beta_2}g\right|_{\bf D}+\left|\partial^{\alpha_2}_{\beta_2}g\right|_{L^2_{\frac\gamma2+s}}
			\left|\lag v\rag^{\ell}e^{q\langle v\rangle}\partial^{\alpha_1}_{\beta_1}f\right|_{\bf D}\right\}
			\left|\lag v\rag^{\ell}e^{q\langle v\rangle}\partial^{\alpha}_{\beta}h\right|_{\bf D}\nonumber\\
			&&+\min\left\{\left|\lag v\rag^{\ell}e^{q\langle v\rangle}\partial^{\alpha_1}_{\beta_1}f\right|_{L^2}
			\left|\partial^{\alpha_2}_{\beta_2}g\right|_{L^2_{\frac\gamma2+s}},\left|\partial^{\alpha_2}_{\beta_2}g\right|_{L^2}
			\left|\lag v\rag^{\ell}e^{q\langle v\rangle}\partial^{\alpha_1}_{\beta_1}f\right|_{L^2_{\frac\gamma2+s}}\right\}
			\left|\lag v\rag^{\ell}e^{q\langle v\rangle}\partial^{\alpha}_{\beta}h\right|_{\bf D}\nonumber\\
			&&+\sum\left|e^{{q}\langle v\rangle}\partial^{\alpha_2}_{\beta_2}g\right|_{L^2}
			\left|\lag v\rag^{\ell}e^{q\langle v\rangle}\partial^{\alpha_1}_{\beta_1}f\right|_{L^2_{\frac\gamma2+s}}
			\left|\lag v\rag^{\ell}e^{q\langle v\rangle}\partial^{\alpha}_{\beta}h\right|_{\bf D},
		\end{eqnarray}
		where the summation $\sum$ is taken over $\alpha_1+\alpha_2\leq \alpha$ and $\beta_1+\beta_2\leq\beta$.
	\end{itemize}	
\end{lemma}
\begin{proof}
	We only prove \eqref{Gamma-noncut-1} since \eqref{Gamma-noncut-2} can be proved in the same way. Note that \cite[Lemma 2.3, p. 176]{Duan-Liu-Yang-Zhao-2013-KRM} is a direct conclusion of \cite[Lemma 2.1, p. 169]{Duan-Liu-Yang-Zhao-2013-KRM}. We mainly focus on refining the estimate on $I_3$ in \cite[Lemma 2.1]{Duan-Liu-Yang-Zhao-2013-KRM} and omit the estimation of the term $I_{2,2}$, which can be treated in the same way. Other parts of the proof there can remain unchanged.
	
	Noting that
\begin{eqnarray*}
	v'-v=\left(\frac{v-v_\ast}{|v-v_\ast|}\cdot\sigma-1\right)\frac{v-v_\ast}2=\left(\cos \theta-1\right)\frac{v-v_\ast}2,
\end{eqnarray*}
we have
\begin{eqnarray*}
	&&\overline{w}(v')-\overline{w}(v)\nonumber\\
	&=&(v-v')\cdot \nabla_v \overline{w}(v)+\frac12(v'-v)\otimes(v'-v):\nabla^2_v \overline{w}|_{v=v(\tau)}\nonumber\\
	&=&\left(\cos \theta-1\right)\frac{v-v_\ast}2\nabla_v \overline{w}(v)+\frac12(v'-v)\otimes(v'-v):\nabla^2_v \overline{w}|_{v=v(\tau)}.
\end{eqnarray*}
This implies
\begin{equation*}
\overline{w}(v')-\overline{w}(v)\lesssim\left(\left|\cos \theta-1\right| |v-v_\ast|+\theta^2 |v-v_\ast|^2\right)\overline{w}(v)\overline{w}(v_\ast)\langle v_\ast\rangle^2.
\end{equation*}
On the other hand, it is straightforward to compute that
\begin{eqnarray*}
&&|\overline{w}(v')-\overline{w}(v)|\lesssim \langle v_\ast\rangle^{\ell}\langle v\rangle^{\ell}\exp{\frac{(\langle v_\ast\rangle+\langle v\rangle)}{8RT_c\ln(e+t)}}\lesssim \overline{w}(v)\overline{w}(v_\ast).
\end{eqnarray*}
Then we can deduce that
\begin{eqnarray*}
	&&\overline{w}(v')-\overline{w}(v)\nonumber\\
	&\lesssim&\min\{\left(\left|\cos \theta-1\right| |v-v_\ast|+\theta^2 |v-v_\ast|^2\right),1\}\overline{w}(v)\overline{w}(v_\ast)\langle v_\ast\rangle^2.
\end{eqnarray*}
Now we re-estimate the term $I_3$ in the proof of Lemma 2.1 in \cite{Duan-Liu-Yang-Zhao-2013-KRM}.
\begin{eqnarray*}
	&&\int_{\mathbb{R}^3\times \mathbb{S}^2}{\bf B}(v-v_\ast,\sigma)\mu(v,v_\ast)f(v)g(v_\ast)h(v)(\overline{w}(v')-\overline{w}(v))dvdv_\ast d\sigma\nonumber\\
	&\lesssim&\int_{\mathbb{R}^3\times \mathbb{S}^2}{\bf B}(v-u,\sigma)\mu(v,v_\ast)f(v)g(v_\ast)h(v)\nonumber\\
	&&\times\min\{\left(\left|\cos \theta-1\right| |v-v_\ast|+\theta^2 |v-v_\ast|^2\right),1\}\overline{w}(v)\overline{w}(v_\ast)\langle u\rangle^2dvdv_\ast d\sigma\nonumber\\
		&\lesssim&\int_{\mathbb{R}^3\times \mathbb{S}^2}{\bf B}(v-v_\ast,\sigma)(\mu^{\tilde{\lambda}}(v)-\mu^{\tilde{\lambda}}(v_\ast))^{k}\mu^{\bar{\lambda}}(v_\ast)f(v)g(v_\ast)h(v)\nonumber\\
	&&\times\min\{\left(\left|\cos \theta-1\right| |v-v_\ast|+\theta^2 |v-v_\ast|^2\right),1\}\overline{w}(v)\overline{w}(v_\ast)\langle v_\ast\rangle^2dvdv_\ast d\sigma\nonumber\\
			&\lesssim&\int_{\mathbb{R}^3\times \mathbb{S}^2}{\bf B}(v-v_\ast,\sigma)(\mu^{\tilde{\lambda}}(v)-\mu^{\tilde{\lambda}}(v_\ast))^{k}\mu^{\frac{\bar{\lambda}}{16}}(v_\ast)f(v)g(v_\ast)h(v)\nonumber\\
	&&\times\min\{\left(\left|\cos \theta-1\right| |v-v_\ast|+\theta^2 |v-v_\ast|^2\right),1\}\overline{w}(v)dvdv_\ast d\sigma\nonumber\\
		&\lesssim&\int_{\mathbb{R}^6}\int^{\frac\pi2}_0\langle v-v_\ast\rangle^{\gamma}{\bf b}(\cos\theta)(\mu^{\tilde{\lambda}}(v)-\mu^{\tilde{\lambda}}(v_\ast))^{k}\mu^{\frac{\bar{\lambda}}{16}}(v_\ast)f(v)g(v_\ast)h(v)\nonumber\\
	&&\times\min\{\left(\left|\cos \theta-1\right| |v-v_\ast|+\theta^2 |v-v_\ast|^2\right),1\}\overline{w}(v)\sin \theta dvdv_\ast d\theta\nonumber\\
			&\lesssim&\int_{\mathbb{R}^6}\langle v-v_\ast\rangle^{\gamma}(\mu^{\tilde{\lambda}}(v)-\mu^{\tilde{\lambda}}(v_\ast))^{k}\mu^{\frac{\bar{\lambda}}{16}}(v_\ast)f(v)g(v_\ast)h(v)\overline{w}(v)\nonumber\\
	&&\times\int^{\min\{\frac\pi2,|v-v_\ast|^{-1}\}}_0{\bf b}(\cos\theta)\min\{\left(\left|\cos \theta-1\right| |v-v_\ast|+\theta^2 |v-v_\ast|^2\right),1\}\sin \theta d\theta dvdv_\ast\nonumber\\
	&&+\int_{\mathbb{R}^6}\langle v-v_\ast\rangle^{\gamma}(\mu^{\tilde{\lambda}}(v)-\mu^{\tilde{\lambda}}(v_\ast))^{k}\mu^{\frac{\bar{\lambda}}{16}}(v_\ast)f(v)g(v_\ast)h(v)\overline{w}(v)\nonumber\\
	&&\times\int_{\min\{\frac\pi2,|v-v_\ast|^{-1}\}}^{\frac\pi2}{\bf b}(\cos\theta)\min\{\left(\left|\cos \theta-1\right| |v-v_\ast|+\theta^2 |v-v_\ast|^2\right),1\}\sin \theta d\theta dvdv_\ast\nonumber\\
	&=&\int_{\mathbb{R}^6}\langle v-v_\ast\rangle^{\gamma}(\mu^{\tilde{\lambda}}(v)-\mu^{\tilde{\lambda}}(v_\ast))^{k}\mu^{\frac{\bar{\lambda}}{16}}(v_\ast)f(v)g(v_\ast)h(v)\overline{w}(v)\nonumber\\
	&&\times\int^{\min\{\frac\pi2,|v-v_\ast|^{-1}\}}_0{\bf b}(\cos\theta)\left(\left|\cos \theta-1\right| |v-v_\ast|+\theta^2 |v-v_\ast|^2\right)\sin \theta d\theta dvdv_\ast\nonumber\\
	&&+\int_{\mathbb{R}^6}\langle v-v_\ast\rangle^{\gamma}(\mu^{\tilde{\lambda}}(v)-\mu^{\tilde{\lambda}}(v_\ast))^{k}\mu^{\frac{\bar{\lambda}}{16}}(v_\ast)f(v)g(v_\ast)h(v)\overline{w}(v)\nonumber\\
	&&\times\int_{\min\{\frac\pi2,|v-v_\ast|^{-1}\}}^{\frac\pi2}{\bf b}(\cos\theta)\sin \theta d\theta dvdv_\ast\nonumber\\
		&\lesssim&\int_{\mathbb{R}^6}\langle v-v_\ast\rangle^{\gamma}(\mu^{\tilde{\lambda}}(v)-\mu^{\tilde{\lambda}}(v_\ast))^{k}\mu^{\frac{\bar{\lambda}}{16}}(v_\ast)f(v)g(v_\ast)h(v)\overline{w}(v)\nonumber\\
	&&\times\int^{\min\{\frac\pi2,|v-v_\ast|^{-1}\}}_0
	\theta^{-1-2s}\left(\theta^2 |v-v_\ast|+\theta^2 |v-v_\ast|^2\right) d\theta dvdv_\ast\nonumber\\
	&&+\int_{\mathbb{R}^6}\langle v-v_\ast\rangle^{\gamma}(\mu^{\tilde{\lambda}}(v)-\mu^{\tilde{\lambda}}(v_\ast))^{k}\mu^{\frac{\bar{\lambda}}{16}}(v_\ast)f(v)g(v_\ast)h(v)\overline{w}(v)\nonumber\\
	&&\times\int_{\min\{\frac\pi2,|v-v_\ast|^{-1}\}}^{\frac\pi2}\theta^{-1-2s} d\theta dvdv_\ast\nonumber\\
	&\lesssim&\left|\mu^{\frac{\bar{\lambda}}{32}}g\right|\left|\langle v\rangle^{\frac{\gamma+2s}2} \overline{w}f\right|\left|\langle v\rangle^{\frac{\gamma+2s}2}\overline{w}h\right|\nonumber\\
	&\lesssim&\left|\mu^{\frac{\bar{\lambda}}{32}}g\right|\left| \overline{w}f\right|_{L^2_{\frac\gamma2+s}}\left|\overline{w}h\right|_{L^2_{\frac\gamma2+s}}\nonumber,
\end{eqnarray*}
where we used the fact that there exist positive constants
$ \tilde{\lambda}, k,\bar{\lambda} $ such that
\[\mu(v,v_\ast)=(\mu^{\tilde{\lambda}}(v)-\mu^{\tilde{\lambda}}(v_\ast))^{k}\mu^{\bar{\lambda}}(v_\ast).\]
\end{proof}

\subsubsection{ The difference operator $\CL_d$}
Our last lemma in this subsection is concerned with some estimates related to the difference operator $\mathcal{L}_d$. The proof is omitted as it closely resembles that of Lemma \ref{WGLLL0} for the VML case.
\begin{lemma}
Let $\rho(t,x), u(t,x)$ and $T(t,x)$ satisfy  \eqref{em-decay}, and suppose  \eqref{tt1} is valid. For $\overline{w}(t,v)=\exp\big(\frac{\lag v\rag}{8RT_c\ln(\mathrm{e}+t)}\big)$, it holds that
	\begin{align}
		\big|\big\langle\partial^{\alpha}_{\beta}\mathcal{L}_d[h], \overline{w}^2 \partial^{\alpha}_{\beta}h\big\rangle\big|\lesssim& \epsilon_1\big\|\overline{w}h\big\|_{H^{|\alpha|}_xH^{|\beta|}_{\bf D}}^2,\label{wGLd-Non}
	\end{align}
and for $\ell\geq0$, $q\in [0,\infty)$,
\begin{align}
		\big|\big\langle\partial^{\alpha}_{\beta}\mathcal{L}_d[g], {\lag v\rag}^{2\ell} e^{2q\lag v\rag} \partial^{\alpha}_{\beta}h\big\rangle\big|\lesssim& \epsilon_1\big\|{\lag v\rag}^{\ell} e^{q\lag v\rag}h\big\|_{H^{|\alpha|}_xH^{|\beta|}_{\bf D}}\big\|{\lag v\rag}^{\ell} e^{q\lag v\rag}h\big\|_{H^{|\alpha|}_xH^{|\beta|}_{\bf D}}.\label{wGLd-Nong}
	\end{align}
Moreover, if we set $\sqrt{\FM}f=\sqrt{\mu}h$, then it holds that
	\begin{align*}
		\big|\big\langle\partial^{\alpha}_{\beta}\mathcal{L}_d[h], \partial^{\alpha}_{\beta}h\big\rangle\big|\lesssim& \epsilon_1\Big(\big\|({\bf I}-{\bf P})h\|_{H^{|\alpha|}_xH^{|\beta|}_{\bf D}}^2+\big\|f\|_{H^{|\alpha|}_xH^{|\beta|}_{\bf D}}^2\Big)
	\end{align*}
	and
	\begin{align}
		\big(\partial^{\alpha}_{\beta}\mathcal{L}[h], \overline{w}^2 \partial^{\alpha}_{\beta}h\big)\geq& \de\big|\overline{w}\partial^{\alpha}_{\beta}h\big|_{\bf D}^2-\eta\sum_{|\beta'|\leq|\beta|}\big|\overline{w}\partial^{\alpha}_{\beta'}h\big|_{\bf D}^2-C(\eta)\sum_{\al'\leq \alpha}|\partial^{\al'}f|^2_{L^2},\label{wLL-Non}
	\end{align}
where $\eta>0$ is a small constant and $C(\eta)$ is some large constant depending on $\eta$.
\end{lemma}

\subsection{Estimates of $f^{\varepsilon}, E_R^{\varepsilon}, B_R^{\varepsilon}$}
In this subsection, we aim to derive estimates for the remainders $f^{\varepsilon}$, $E_R^{\varepsilon}$, and $B_R^{\varepsilon}$, which satisfy the equations \eqref{VMLf} and \eqref{fM-2} with the initial datum \eqref{fEB-initial} for the VMB case.

As in Section \ref{h-VML} for the VML case, we focus solely on the {\it a priori} energy estimate \eqref{TVML1}, made under the {\it a priori} assumption with the same form as \eqref{aps-vml}.
The following two lemmas deal with the terms related to the Lorentz force, and their proofs closely resemble those in Lemma \ref{dxlm VML} and Lemma \ref{dxvlm VML}. The primary difference is that, unlike the VML case, we can now bound the terms involving $\nabla_v f$ using the energy functional $\mathcal{E}(t)$ of $h^{\varepsilon}$ when the velocity $v$ is large. As such, we provide the estimates and omit the proof details for brevity.

\begin{lemma}\label{dxlm}
	 Assume that $f^{\varepsilon}, E_R^{\varepsilon}, B_R^{\varepsilon}$ and $h^{\varepsilon}$ is the smooth solution to the Cauchy problem \eqref{VMLf}, \eqref{fM-2}, \eqref{h-equation}, \eqref{VML-id-pt}, and \eqref{aps-vml} corresponding to the VMB case holds. Then for $|\alpha|+|\beta|\leq 4$, it holds that
\begin{align*}
	\Big|\Big\langle \partial^\alpha_{\beta}\Big(f^{\varepsilon}\mathbf{M}^{-\frac{1}{2}}\big(\partial_t+v\cdot\nabla_x\big)\mathbf{M}^{\frac{1}{2}}\Big), 4\pi RT\partial^\alpha_{\beta}f^{\varepsilon}\Big\rangle\Big|
	\lesssim&\,\epsilon_1(1+t)^{-p_0}Z_{3,\alpha,\beta}(t),
\end{align*}
\begin{align*}
   &\Big|\Big\langle\partial^\alpha_{\beta}\Big[\Big(E+v \times B \Big)\cdot\frac{v-u }{ T}f^{\varepsilon}\Big], 2\pi T \partial^\alpha_{\beta}f^{\varepsilon}\Big\rangle\Big|
    \lesssim\,\epsilon_1(1+t)^{-p_0}\FZ_{3,\alpha,\beta}(t),
   \end{align*}
  \begin{align*}
  &\varepsilon^k\Big|\Big\langle\partial^\alpha_{\beta}\Big[\big(E_R^{\varepsilon}+v \times B_R^{\varepsilon}\big) \cdot\frac{(v-u)}{ T}f^{\varepsilon}\Big], 2\pi T \partial^\alpha_{\beta}f^{\varepsilon}\Big\rangle\Big|
    \lesssim\,\varepsilon \FZ_{3,\alpha,\beta}(t),
   \end{align*}
  and
   \begin{align*}
  &\sum_{n=1}^{2k-1}\varepsilon^n\Big|\Big\langle \partial^\alpha_{\beta}\Big[\Big(E_n+v \times B_n \Big)\cdot\frac{(v-u)}{ T}f^{\varepsilon}\Big], 2\pi T\partial^\alpha_{\beta}f^{\varepsilon}\Big\rangle\Big|
    \lesssim\,\varepsilon^{1-\kappa}\FZ_{3,\alpha,\beta}(t),
   \end{align*}
 where
   \begin{align*}
   \FZ_{3,\alpha,\beta}(t):=&\|f^{\varepsilon}\|^2_{H^{|\alpha|}_xH^{|\beta|}_v} +\sum_{\alpha'\leq\alpha,\beta'\leq\beta}\frac{1}{\varepsilon}\left\|\partial^{\alpha'}_{\beta'}({\bf I}-{\bf P}_{\mathbf{M}})[f^{\varepsilon}]\right\|_{\bf D}^2\nonumber\\
   &+C_{\epsilon_1}\exp\Big(-\frac{\epsilon_1}{8C_0RT^2_c\sqrt{\varepsilon}}\Big)
   \Big(\sum_{\alpha'\leq\alpha,\beta'\leq\beta}\|\partial^{\alpha'}_{\beta'}h^{\varepsilon}\|^2\Big).
   \end{align*}
\end{lemma}

\begin{lemma}\label{dxvlm VMB} Under the same assumptions in Lemma \ref{dxlm}, for $|\alpha|+|\beta|\leq 4$ and $t\in [0, \varepsilon^{-\kappa}]$ with $\kappa=\frac{1}{3}$, it holds that
  \begin{align*}
   &\Big|\Big\langle \partial^\alpha_{\beta}\Big[\Big(E+v \times B \Big)\cdot\nabla_vf^{\varepsilon}\Big], 4\pi R T \partial^\alpha_{\beta}f^{\varepsilon}\Big\rangle\Big|
    \lesssim\,\epsilon_1(1+t)^{-p_0}{\bf 1}_{|\alpha|+|\beta|\geq1}Z_{4,\alpha,\beta}(t),
   \end{align*}
   \begin{align*}
   &\varepsilon^k\Big|\Big\langle \partial^\alpha_{\beta}\Big[\big(E_R^{\varepsilon}+v \times B_R^{\varepsilon}\big) \cdot\nabla_vf^{\varepsilon}\Big], 4\pi RT \partial^\alpha_{\beta}f^{\varepsilon}\Big\rangle\Big|
    \lesssim\,\varepsilon{\bf 1}_{|\alpha|+|\beta|\geq1}\FZ_{4,\alpha,\beta}(t),
   \end{align*}
  and
   \begin{align*}
   &\sum_{n=1}^{2k-1}\varepsilon^n\Big|\Big\langle \partial^\alpha_{\beta}\Big[\Big(E_n+v \times B_n \Big)\cdot\nabla_vf^{\varepsilon}+\Big(E_R^{\varepsilon}+v \times B_R^{\varepsilon} \Big)\cdot \frac{\nabla_vF_n}{\sqrt{\mathbf{M}}}\Big], 4\pi RT \partial^\alpha_{\beta}f^{\varepsilon}\Big\rangle\Big|\\
    \lesssim&\,\varepsilon^{2\kappa}\Big[\|f^{\varepsilon}\|^2_{H^{|\alpha|}_xH^{|\beta|}_v}+\|E_R^{\varepsilon}\|^2_{H^{|\alpha|}}+\|B_R^{\varepsilon}\|^2_{H^{|\alpha|}}+{\bf 1}_{|\alpha|+|\beta|\geq1}\FZ_{4,\alpha,\beta}(t)\Big],\nonumber
   \end{align*}
   where
   \begin{align*}
  \FZ_{4,\alpha,\beta}(t)&:=\|f^{\varepsilon}\|^2_{H^{|\alpha|}_xH^{|\beta|}_v} +\sum_{|\al'|\leq|\al|,|\beta'|\leq |\beta|+1,\atop{|\alpha'|+|\beta'|\leq |\alpha|+|\beta|}}\frac{1}{\varepsilon}\left\|\partial^{\alpha'}_{\beta'}({\bf I}-{\bf P}_{\mathbf{M}})[f^{\varepsilon}]\right\|_{\bf D}^2\\
&+C_{\epsilon_1}\exp\Big(-\frac{\epsilon_1}{8C_0RT^2_c\sqrt{\varepsilon}}\Big)
    \Big(\sum_{|\al'|\leq|\al|,|\beta'|\leq |\beta|+1,\atop{|\alpha'|+|\beta'|\leq |\alpha|+|\beta|}}\|\partial^{\alpha'}_{\beta'}h^{\varepsilon}\|^2\Big).
   \end{align*}

   \end{lemma}

Now, we are prepared to derive energy estimates for $f^{\varepsilon}$, $E_R^{\varepsilon}$, and $B_R^{\varepsilon}$. Let's begin with the energy estimate that does not involve velocity derivatives.
\begin{proposition}\label{f-eng-vmb}
Under the same assumptions in Lemma \ref{dxlm}, we have
\begin{align}\label{f-eng-vmb-sum}
&\frac{\mathrm{d}}{\mathrm{d} t}\sum_{|\alpha|\leq 4}\Big[\varepsilon^{|\alpha|}\Big(\|\sqrt{4\pi RT}\partial^{\alpha}f^{\varepsilon}\|^2+\|\partial^{\alpha}E_R^{\varepsilon}\|^2+\|\partial^{\alpha}B_R^{\varepsilon}\|^2\Big)\Big]\nonumber\\
&
+\delta\varepsilon^{|\alpha|-1}\sum_{|\alpha|\leq 4}\|\partial^{\alpha}({\bf I}-{\bf P}_{\mathbf{M}})[f^{\varepsilon}]\|^2_{\bf D}   \nonumber\\
\lesssim&\sum\limits_{|\alpha|\leq 4}\vps^{|\alpha|}\Big[\epsilon_1(1+t)^{-p_0}+\varepsilon^{\kappa}
\Big]\Big[\|f^{\varepsilon}\|^2_{H^{|\alpha|}_xL^2_v}+\|E_R^{\varepsilon}\|^2_{H^{|\alpha|}}+\|B_R^{\varepsilon}\|^2_{H^{|\alpha|}} \\
&+\sum_{|\al'|\leq|\al|-1,\atop{|\beta'|= 1,|\alpha|\geq1}}\Big(\frac{1}{\varepsilon}\Big\|\partial^{\alpha'}_{\beta'}({\bf I}-{\bf P}_{\mathbf{M}})[f^{\varepsilon}]\Big\|_{\bf D}^2+C_{\epsilon_1}\exp\Big(-\frac{\epsilon_1}{8C_0RT^2_c\sqrt{\varepsilon}}\Big)\|\partial^{\alpha'}_{\beta'}h^{\varepsilon}\|^2\Big)\Big]\nonumber\\
&+\sum_{|\alpha|\leq 4}{\bf 1}_{|\alpha|\geq1}\epsilon_1(1+t)^{-p_0}\Big( \varepsilon^{|\alpha|-2}\|({\bf I}-{\bf P}_{\mathbf{M}})f\|_{H^{|\alpha|-1}_{\bf D}}^2+\varepsilon^{|\alpha|}\|f\|^2_{H^{|\alpha|}_xL^2_v}\Big)\nonumber\\
     &+\varepsilon^{2k+1}(1+t)^{4k+2}+\sum\limits_{|\alpha|\leq 4}\varepsilon^{k+|\alpha|}(1+t)^{2k}\|\partial^{\alpha}f^{\varepsilon}\|.\nonumber
\end{align}
\end{proposition}
\begin{proof}The proof is divided into two steps.
\vskip 0.2cm

\noindent\underline{{\it Step 1. Basic energy estimate of the remainders.}} In this step, we derive the $L^2$ estimate on $[f^{\varepsilon}, E_R^{\varepsilon}, B_R^{\varepsilon}]$. Corresponding to \eqref{L2f1 VML}, in the VMB case, we can apply \eqref{coL0-Non} to have
\begin{align}\label{L2f1 VMB}
&\frac{1}{2}\frac{\mathrm{d}}{\mathrm{d} t}\Big(\|\sqrt{4\pi RT}f^{\varepsilon}\|^2+\|E_R^{\varepsilon}\|^2+\|B_R^{\varepsilon}\|^2\Big)
    +\frac{\delta}{\varepsilon}\|({\bf I}-{\bf P}_{\mathbf{M}})[f^{\varepsilon}]\|^2_{\bf D} \\
    \leq&\;\frac{1}{2}\big|\big\langle \partial_tT f^{\varepsilon}, 4\pi R f^{\varepsilon}\big\rangle\big|+\Big|\Big\langle \big(E_R^{\varepsilon}+v \times B_R^{\varepsilon} \big) \cdot u\mathbf{M}^{\frac{1}{2}},4\pi  f^{\varepsilon}\Big\rangle\Big|\nonumber\\
    &+\Big|\Big\langle \Big(E+v \times B \Big)\cdot(v-u )f^{\varepsilon}, 2\pi  f^{\varepsilon}\Big\rangle\Big|\nonumber\\
    &+    \Big|\Big\langle f^{\varepsilon}\mathbf{M}^{-\frac{1}{2}}\big(\partial_t+v\cdot\nabla_x\big)\mathbf{M}^{\frac{1}{2}}, 4\pi RT f^{\varepsilon}\Big\rangle\Big|
    +\varepsilon^{k-1}\big|\big\langle\Gamma_{\mathbf{M}} ( f^{\varepsilon},
    f^{\varepsilon} ), 4\pi RT f^{\varepsilon}\big\rangle\big|\nonumber\\
    &+\sum_{n=1}^{2k-1}\varepsilon^{n-1}\big|\big\langle[\Gamma_{\mathbf{M}}(\mathbf{M}^{-\frac{1}{2}}F_n,f^{\varepsilon})+\Gamma_{\mathbf{M}}(
 f^{\varepsilon}, \mathbf{M}^{-\frac{1}{2}}F_n)\big], 4\pi RT f^{\varepsilon}\big\rangle\big|\nonumber\\
 &+\varepsilon^k \big|\big\langle \big(E_R^{\varepsilon}+v \times B_R^{\varepsilon}\big) \cdot(v-u)f^{\varepsilon},2\pi  f^{\varepsilon}\big\rangle\big|\nonumber\\
 &+\sum_{n=1}^{2k-1}\varepsilon^n\Big|\Big\langle \Big(E_R^{\varepsilon}+v \times B_R^{\varepsilon} \Big)\cdot\mathbf{M}^{-\frac{1}{2}}\nabla_v F_n,4\pi RT f^{\varepsilon}\Big\rangle\Big|\nonumber\\
 &+\sum_{n=1}^{2k-1}\varepsilon^n\Big|\Big\langle \Big(E_n+v \times B_n \Big)\cdot(v-u)f^{\varepsilon},2\pi f^{\varepsilon}\Big\rangle\Big|
  +\varepsilon^k\big|\big\langle \CQ_0,4\pi RT f^{\varepsilon}\big\rangle\big|.\nonumber
\end{align}

Now we estimate the terms on the R.H.S. of \eqref{L2f1 VMB} individually.\\
In view of \eqref{em-decay}, the upper bound of the first two terms on the R.H.S. of \eqref{L2f1 VMB} is
$ (1+t)^{-p_0}\epsilon_1 \Big(\|f^{\varepsilon}\|^2+\|E_R^{\varepsilon}\|^2
+\|B_R^{\varepsilon}\|^2\Big).$  For the 3rd, 4th, 7th-9th terms on the R.H.S. of \eqref{L2f1 VMB},
we use Lemma \ref{dxlm} and Lemma \ref{dxvlm VMB} with $\beta=0$ to bound them by
$$
 C\Big[\epsilon_1(1+t)^{-p_0}+\varepsilon^{2\kappa}\Big]\Big[\|E_R^{\varepsilon}\|^2+\|B_R^{\varepsilon}\|^2+Z_{4,0,0}(t)\Big].
$$

Noting that corresponding to \eqref{fbd-vml} and \eqref{hf2b VML},
\begin{align}\label{vmbf4}
\varepsilon^{k-1}\Big(\|f^{\varepsilon}\|_{H^4}+\|E_R^{\varepsilon}\|_{H^4}+\|B_R^{\varepsilon}\|_{H^4}\Big)\lesssim \varepsilon^{\frac{1}{2}},\qquad \varepsilon^{k-1}\|\overline{w}h^{\varepsilon}\|_{H^4}\lesssim 1, \qquad k\geq 4,
\end{align}
 by the {\it a priori} assumption \eqref{aps-vml} in the VMB case. Then for the 5th and 6th terms on the R.H.S. of \eqref{L2f1 VMB},
similar to the estimates  \eqref{f0L} and \eqref{fn0L}, we can  use \eqref{GLM-Non}  and \eqref{vmb-Fn-es} to bound them  by
$o(1)\varepsilon^{-1}\|({\bf I}-{\bf P}_{\mathbf{M}})[f^{\varepsilon}]\|_{\bf D}^2+C\varepsilon^{\kappa}\|f^{\varepsilon}\|^2.$


For the last term on the R.H.S. of \eqref{L2f1 VMB}, similar to the VML case, we can
use \eqref{vmb-Fn-es}, \eqref{vmb-EB-es}, and  Cauchy's inequality to have
\begin{equation}\label{nonlinear-f-i-0-1}
    \varepsilon^k\big\langle \CQ_0, 4\pi RT f^{\varepsilon}\big\rangle
    \lesssim \frac{o(1)}{\varepsilon}\|({\bf I}-{\bf P}_{\mathbf{M}})[f^{\varepsilon}]\|_{\bf D}^2+\varepsilon^{2k+1}(1+t)^{4k+2}+\varepsilon^{k}(1+t)^{2k}\|f^{\varepsilon}\|.
\end{equation}

Substituting the above estimates into \eqref{L2f1 VMB} leads to
\begin{align}\label{L2f VMB}
&\frac{\mathrm{d}}{\mathrm{d} t}\Big(\|\sqrt{4\pi RT}f^{\varepsilon}\|^2+\|E_R^{\varepsilon}\|^2+\|B_R^{\varepsilon}\|^2\Big)
    +\frac{\delta}{\varepsilon}\|({\bf I}-{\bf P}_{\mathbf{M}})[f^{\varepsilon}]\|^2_{\bf D}\nonumber\\
     \lesssim&\,\Big[(1+t)^{-p_0}+\varepsilon^{\kappa}\Big]\Big(\|f^{\varepsilon}\|^2+\|E_R^{\varepsilon}\|^2+\|B_R^{\varepsilon}\|^2
     +C_{\epsilon_1}\exp\Big(-\frac{\epsilon_1}{8C_0RT^2_c\sqrt{\varepsilon}}\Big)\|h^{\varepsilon}\|^2\Big)\nonumber\\
     &+\varepsilon^{2k+1}(1+t)^{4k+2}+\varepsilon^{k}(1+t)^{2k}\|f^{\varepsilon}\|.
\end{align}
\vskip 0.2cm
\noindent\underline{{\it Step 2. Estimates on the pure spatial derivatives of the remainders.}}
In this step, we proceed to derive estimates for $\partial^{\alpha}f^{\varepsilon}$, $\partial^{\alpha} E_R^{\varepsilon}$, and $\partial^{\alpha} B_R^{\varepsilon}$ for $1\leq |\alpha|\leq N$.

Proceeding similar derivation as that in \eqref{H1f1 VML}, one gets
\begin{align}\label{H1f1 VMB}
&\frac{1}{2}\frac{\mathrm{d}}{\mathrm{d} t}\Big(\|\sqrt{4\pi RT}\partial^{\alpha}f^{\varepsilon}\|^2+\|[\partial^{\alpha}E_R^{\varepsilon},\partial^{\alpha}B_R^{\varepsilon}]\|^2\Big)
    +\frac{1}{\varepsilon}\big\langle \partial^{\alpha}\mathcal{L}_{\mathbf{M}}[f^{\varepsilon}],4\pi RT\partial^{\alpha}f^{\varepsilon}\big\rangle \\
    \leq&\;\big|\big\langle \partial_tT \partial^{\alpha}f^{\varepsilon}+\pa^{\alpha}\big(E_R^{\varepsilon}+v\times B_R^{\varepsilon}\big)\cdot \frac{u}{T},  4\pi R \partial^{\alpha}f^{\varepsilon}\big\rangle\big|\nonumber\\
    &+\sum_{\al'<\al}\Big|\Big\langle \partial^{\alpha'} \big(E_R^{\varepsilon}+v \times B_R^{\varepsilon} \big) \cdot \partial^{\alpha-\alpha'}_x\Big[\frac{v-u}{ T}\mathbf{M}^{\frac{1}{2}}\Big], 2\pi T \partial^{\alpha}f^{\varepsilon}\Big\rangle\Big|\nonumber\\
    &+\sum\limits_{\alpha'<\alpha}\big|\big\langle v\pa^{\al-\al'}\big(\mathbf{M}^{\frac{1}{2}}\big)\pa^{\al'}f^{\varepsilon},  \partial^{\alpha} E_R^{\varepsilon}\big\rangle\big|+\Big|\Big\langle \partial^{\alpha}\left[\big(E+v \times B \big) \cdot \nabla_v f^{\varepsilon}\right], 4\pi R T \partial^{\alpha}f^{\varepsilon}\Big\rangle\Big|\nonumber\\
    &+\Big|\Big\langle \partial^{\alpha}\Big[\Big(E+v \times B \Big)\cdot\frac{v-u }{ 2T}f^{\varepsilon}\Big],4\pi T \partial^{\alpha}f^{\varepsilon}\Big\rangle\Big|\nonumber\\
    &+    \Big|\Big\langle \partial^{\alpha}\Big[f^{\varepsilon}\mathbf{M}^{-\frac{1}{2}}\big(\partial_t+v\cdot\nabla_x\big)\mathbf{M}^{\frac{1}{2}}\Big], 4\pi RT \partial^{\alpha}f^{\varepsilon}\Big\rangle\Big|
    +\varepsilon^{k-1}\big|\big\langle\partial^{\alpha}\Gamma_{\mathbf{M}} ( f^{\varepsilon},
    f^{\varepsilon} ), 4\pi RT \partial^{\alpha}f^{\varepsilon}\big\rangle\big|\nonumber\\
    &+\sum_{n=1}^{2k-1}\varepsilon^{n-1}\big|\big\langle[\partial^{\alpha}\Gamma_{\mathbf{M}}(\mathbf{M}^{-\frac{1}{2}}F_n, f^{\varepsilon})+\partial^{\alpha}\Gamma_{\mathbf{M}}(
 f^{\varepsilon}, \mathbf{M}^{-\frac{1}{2}}F_n)\big], 4\pi RT \partial^{\alpha}f^{\varepsilon}\big\rangle\big|\nonumber\\
 &+\varepsilon^k \big|\big\langle \partial^{\alpha}\Big[\big(E_R^{\varepsilon}+v \times B_R^{\varepsilon}\big) \cdot\nabla_vf^{\varepsilon}\Big], 4\pi RT  \partial^{\alpha}f^{\varepsilon}\big\rangle\big|\nonumber\\
 &+\varepsilon^k \big|\big\langle \partial^{\alpha}\Big[\big(E_R^{\varepsilon}+v \times B_R^{\varepsilon}\big) \cdot\frac{(v-u)}{ 2T}f^{\varepsilon}\Big], 4\pi T \partial^{\alpha}f^{\varepsilon}\big\rangle\big|\nonumber\\
 &+\sum_{n=1}^{2k-1}\varepsilon^n\Big|\Big\langle \partial^{\alpha}\Big[\big(E_n+v \times B_n\big) \cdot\nabla_vf^{\varepsilon}+\Big(E_R^{\varepsilon}+v \times B_R^{\varepsilon} \Big)\cdot \mathbf{M}^{-\frac{1}{2}}\nabla_v F_n\Big], 4\pi RT \partial^{\alpha}f^{\varepsilon}\Big\rangle\Big|\nonumber\\
 &+\sum_{n=1}^{2k-1}\varepsilon^n\Big|\Big\langle \partial^{\alpha}\Big[\Big(E_n+v \times B_n \Big)\cdot\frac{(v-u)}{2 T}f^{\varepsilon}\Big], 4\pi T\partial^{\alpha}f^{\varepsilon}\Big\rangle\Big|
  +\big|\big\langle \partial^{\alpha}\CQ_0, 4\pi RT \partial^{\alpha}f^{\varepsilon}\big\rangle\big|.\nonumber
\end{align}

For the 2nd term on the L.H.S. of \eqref{H1f1 VMB}, applying \eqref{em-decay} and \eqref{coLh-Non}, we have
\begin{align*}
\frac{1}{\varepsilon}\big\langle\partial^{\alpha}\mathcal{L}_{\mathbf{M}}f^{\varepsilon}, 4\pi RT\partial^{\alpha}f^{\varepsilon}\big\rangle\geq&\frac{3\de}{4\varepsilon}\|\partial^{\alpha}({\bf I}-{\bf P}_{\mathbf{M}})f^{\varepsilon}\|_{\bf D}^2-\frac{\epsilon_1}{\varepsilon}(1+t)^{-p_0}\\
&\times\Big( \frac{1}{\varepsilon}\|({\bf I}-{\bf P}_{\mathbf{M}})f^{\varepsilon}\|_{H^{|\alpha|-1}_xL^2_{\bf D}}^2+\varepsilon\|f^{\varepsilon}\|^2_{H^{|\alpha|}_xL^2_v}\Big).\nonumber
\end{align*}
We now turn to estimate terms on the R.H.S. of \eqref{H1f1 VMB} separately.

From \eqref{em-decay}, we have the upper bound of
the first three terms on the R.H.S. of \eqref{H1f1 VMB} as
$C\epsilon_1(1+t)^{-p_0}\big(\|f^{\varepsilon}\|^2_{H^{|\alpha|}_xL^2_v}+\|E_R^{\varepsilon}\|^2_{H^{|\alpha|}}
	+\|B_R^{\varepsilon}\|^2_{H^{|\alpha|}}\big)$.

It follows from Lemma \ref{dxlm} and Lemma \ref{dxvlm VMB} for $\beta=0$ that the 4th-6th, 9th-12th terms on the R.H.S. of \eqref{H1f1 VMB} can be dominated by
\begin{align*}
&C\Big[(1+t)^{-p_0}\epsilon_1+\varepsilon^{2\kappa}\Big]\Big[\|f^{\varepsilon}\|^2_{H^{|\alpha|}_xL^2_v}+\|E_R^{\varepsilon}\|^2_{H^{|\alpha|}}+\|B_R^{\varepsilon}\|^2_{H^{|\alpha|}} \\
&+\sum_{|\al'|\leq|\al|,|\beta'|\leq 1,\atop{|\alpha'|+|\beta'|\leq |\alpha|}}\frac{1}{\varepsilon}\big\|\partial^{\alpha'}_{\beta'}({\bf I}-{\bf P}_{\mathbf{M}})[f^{\varepsilon}]\big\|_{\bf D}^2+C_{\epsilon_1}\exp\Big(-\frac{\epsilon_1}{8C_0RT^2_c\sqrt{\varepsilon}}\Big)\sum_{|\al'|\leq|\al|,|\beta'|\leq 1,\atop{|\alpha'|+|\beta'|\leq |\alpha|}}\|\partial^{\alpha'}_{\beta'}h^{\varepsilon}\|^2\Big].
\end{align*}

For the 7th term on the R.H.S. of \eqref{H1f1 VMB},  we
use \eqref{em-decay}, \eqref{GLM-Non}, \eqref{vmbf4}  and Sobolev's inequalities to have
\begin{align*}
    &\varepsilon^{k-1}\big|\big\langle \partial^{\alpha}\Gamma_{\mathbf{M}} ( f^{\varepsilon},  f^{\varepsilon} ), 4\pi RT \partial^{\alpha}f^{\varepsilon}\big\rangle\big|\\
    \lesssim &\,\varepsilon^{k-1}\|f^{\varepsilon}\|_{H^{4}_xL^2_v}\Big(\sum_{\alpha'\leq\alpha}\|\partial^{\alpha'}f^{\varepsilon}\|_{{\bf D}}\|({\bf I}-{\bf P}_{\mathbf{M}})\partial^{\alpha}f^{\varepsilon}\|_{{\bf D}}+ \epsilon_1(1+t)^{-p_0}\sum_{\alpha'<\alpha}\|\partial^{\alpha'}f^{\varepsilon}\|_{{\bf D}}\|\partial^{\alpha}f^{\varepsilon}\|_{{\bf D}}\Big)\\
    \lesssim &\,\|({\bf I}-{\bf P}_{\mathbf{M}})[f^{\varepsilon}]\|_{H^{|\alpha|}_xL^2_{\bf D}}^2+\Big[\epsilon_1(1+t)^{-p_0}\varepsilon^{\frac{1}{2}}+\varepsilon\Big]\|f^{\varepsilon}\|_{H^{|\alpha|}_xL^2_v}^2.\nonumber
\end{align*}

Similarly, for the 8th term on the R.H.S. of \eqref{H1f1 VMB}, we further use \eqref{vmb-Fn-es} to get
\begin{align*}
&\sum_{n=1}^{2k-1}\varepsilon^{n-1}\big|\big\langle[\partial^{\alpha}\Gamma_{\mathbf{M}}(\mathbf{M}^{-\frac{1}{2}}F_n, f^{\varepsilon})+\partial^{\alpha}\Gamma_{\mathbf{M}}(
 f^{\varepsilon}, \mathbf{M}^{-\frac{1}{2}}F_n)\big], 4\pi RT \partial^{\alpha}f^{\varepsilon}\big\rangle\big|\\
  \lesssim&(1+t)\big(\|{f^{\varepsilon}}\|_{H^{|\alpha|}_xL^2_v}+\|{f^{\varepsilon}}\|_{H^{|\alpha|}_xL^2_{\bf D}}\big)\Big(\|\partial^{\alpha}({\bf I}-{\bf P}_{\mathbf{M}})[f^{\varepsilon}]\|_{\bf D}+\|\Lbrack\partial^{\alpha},{{\bf P}_{\mathbf{M}}}\Rbrack[f^{\varepsilon}]\|\Big)\\
 &+\epsilon_1(1+t)^{1-p_0}\big(\|{f^{\varepsilon}}\|_{H^{|\alpha|-1}_xL^2_v}+\|{f^{\varepsilon}}\|_{H^{|\alpha|-1}_xL^2_{\bf D}}\big)\Big(\|\partial^{\alpha}({\bf I}-{\bf P}_{\mathbf{M}})[f^{\varepsilon}]\|_{\bf D}+\|f^{\varepsilon}\|_{H^{|\alpha|}L^2_v}\Big)\\
    \lesssim&\, \frac{o(1)}{\varepsilon}\|\partial^{\alpha}({\bf I}-{\bf P}_{\mathbf{M}})[f^{\varepsilon}]\|_{\bf D}^2 +\frac{\epsilon_1}{\varepsilon}(1+t)^{-p_0}\|f^{\varepsilon}\|_{H^{|\alpha|-1}_xL^2_v}^2+\varepsilon^{1-2\kappa}\|f^{\varepsilon}\|^2_{H^{|\alpha|}_xL^2_v}\nonumber.
\end{align*}

Similar to \eqref{nonlinear-f-i-0-1},
we can deduce from \eqref{vmb-Fn-es} and \eqref{vmb-EB-es} that the last term on the R.H.S. of \eqref{H1f1 VMB} can be dominated by
\begin{align*}
   & \frac{o(1)}{\varepsilon}\|\partial^{\alpha}({\bf I}-{\bf P}_{\mathbf{M}})[f^{\varepsilon}]\|_{\bf D}^2+\frac{C\epsilon_1}{\varepsilon}\|({\bf I}-{\bf P}_{\mathbf{M}})[f^{\varepsilon}]\|_{H^{|\alpha|-1}_xL^2_{\bf D}}^2\\
    &+C\Big[\frac{\epsilon_1}{\varepsilon}(1+t)^{-p_0}\|f^{\varepsilon}\|_{H^{|\alpha|-1}_xL^2_v} +\varepsilon^{2k+1}(1+t)^{4k+2}+\varepsilon^{k}(1+t)^{2k}\|\partial^{\alpha}f^{\varepsilon}\|\Big].\nonumber
\end{align*}

Plugging the above estimates into \eqref{H1f1 VMB}, and multiplying it by $\varepsilon^{|\alpha|}$ , we arrive at
\begin{align}\label{H1f VMB}
&\frac{\mathrm{d}}{\mathrm{d} t}\sum_{1\leq|\alpha|\leq 4}\Big[\varepsilon^{|\alpha|}\Big(\|\sqrt{4\pi RT}\partial^{\alpha}f^{\varepsilon}\|^2+\|\partial^{\alpha}E_R^{\varepsilon}\|^2+\|\partial^{\alpha}B_R^{\varepsilon}\|^2\Big)\Big]\nonumber\\
&
    +\delta\varepsilon^{|\alpha|-1}\sum_{1\leq|\alpha|\leq 4}\|\partial^{\alpha}({\bf I}-{\bf P}_{\mathbf{M}})[f^{\varepsilon}]\|^2_{\bf D}   \nonumber\\
     \lesssim&\sum\limits_{1\leq|\alpha|\leq 4}\vps^{|\alpha|}\Big[\epsilon_1(1+t)^{-p_0}+\varepsilon^{\kappa}\Big]
     \Big[\|f^{\varepsilon}\|^2_{H^{|\alpha|}_xL^2_v}+\|E_R^{\varepsilon}\|^2_{H^{|\alpha|}}+\|B_R^{\varepsilon}\|^2_{H^{|\alpha|}} \\
     &+\sum_{|\al'|\leq|\al|,|\beta'|\leq 1,\atop{|\alpha'|+|\beta'|\leq |\alpha|}}\Big(\frac{1}{\varepsilon}\left\|\partial^{\alpha'}_{\beta'}({\bf I}-{\bf P}_{\mathbf{M}})[f^{\varepsilon}]\right\|_{\bf D}^2+C_{\epsilon_1}\exp\Big(-\frac{\epsilon_1}{8C_0RT^2_c\sqrt{\varepsilon}}\Big)\|\partial^{\alpha'}_{\beta'}h^{\varepsilon}\|^2\Big)\Big]\nonumber\\
     &+\sum_{1\leq|\alpha|\leq 4}\epsilon_1(1+t)^{-p_0}\Big( \varepsilon^{|\alpha|-2}\|({\bf I}-{\bf P}_{\mathbf{M}})f\|_{H^{|\alpha|-1}_xL^2_{\bf D}}^2+\varepsilon^{|\alpha|}\|f\|^2_{H^{|\alpha|}_xL^2_v}\Big),\nonumber\\
    & +\varepsilon^{2k+2}(1+t)^{4k+2}+\sum\limits_{1\leq|\alpha|\leq 4}\varepsilon^{k+|\alpha|}(1+t)^{2k}\|\partial^{\alpha}f^{\varepsilon}\|\nonumber.
\end{align}
\vskip 0.2cm

Finally, \eqref{f-eng-vmb-sum} follows from \eqref{L2f VMB}, \eqref{H1f VMB}, this ends the proof of Proposition \ref{f-eng-vmb}.
\end{proof}

Next we will estimate the mixed space-velocity derivatives of $f^{\varepsilon}$. To do this, we first
 apply the micro-projection operator into \eqref{VMLf} to have
	\begin{align}\label{VMLf-2-micro}
		&\Big[\partial_t+v\cdot\nabla_x-\Big(E+v \times B \Big)\cdot\nabla_v\Big]({\bf I}-{\bf P}_{\mathbf{M}})[f^{\varepsilon}]\nonumber\\
		&+\Big(E+v \times B \Big)\cdot\frac{v-u }{ 2RT}({\bf I}-{\bf P}_{\mathbf{M}})f^{\varepsilon}+\frac{\mathcal{L}_{\mathbf{M}}[f^{\varepsilon}]}{\varepsilon}\nonumber\\
		=&-({\bf I}-{\bf P}_{\mathbf{M}})[f^{\varepsilon}]\mathbf{M}^{-\frac{1}{2}}\Big[\partial_t+v\cdot\nabla_x-\Big(E+v \times B \Big)\cdot\nabla_v\Big]\mathbf{M}^{\frac{1}{2}}\nonumber\\
				&+\varepsilon^{k-1}\Gamma_{\mathbf{M}}(f^{\varepsilon},f^{\varepsilon})+\sum_{i=1}^{2k-1}
		\varepsilon^{i-1}\Big[\Gamma_{\mathbf{M}}(\mathbf{M}^{-\frac{1}{2}}F_n, f^{\varepsilon})+\Gamma_{\mathbf{M}}(f^{\varepsilon}, \mathbf{M}^{-\frac{1}{2}}F_n)\Big]\\
		&+\varepsilon^k\left(E_R^{\varepsilon}+v \times B_R^{\varepsilon}\right)\cdot\nabla_v({\bf I}-{\bf P}_{\mathbf{M}})[f^{\varepsilon}]- \varepsilon^k\left(E_R^{\varepsilon}+v \times B_R^{\varepsilon}\right) \cdot\frac{v-u }{ 2RT}({\bf I}-{\bf P}_{\mathbf{M}})[f^{\varepsilon}]\nonumber\\
		&+\sum_{n=1}^{2k-1}\varepsilon^n\Big[\left(E_n+v \times B_n \right)\cdot\nabla_v({\bf I}-{\bf P}_{\mathbf{M}})[f^{\varepsilon}]+({\bf I}-{\bf P}_{\mathbf{M}})\Big[\left(E_R^{\varepsilon}+v \times B_R^{\varepsilon} \right)\cdot\frac{\nabla_v F_n}{\sqrt{\mathbf{M}}}\Big]\Big]\nonumber\\
		&-\sum_{n=1}^{2k-1}\varepsilon^n\Big[\left(E_n+v \times B_n \right)\cdot\frac{v-u}{ 2RT}({\bf I}-{\bf P}_{\mathbf{M}})[f^{\varepsilon}]\Big]+\varepsilon^{k}({\bf I}-{\bf P}_{\mathbf{M}})\CQ_0+\Lbrack {\bf P},\tau_{EB}\Rbrack f^{\varepsilon},
		\nonumber
	\end{align}
where $\Lbrack{\bf P}_\FM,\tau_{EB}\Rbrack$ is a commutative operator with the following form:
\begin{align*}
\Lbrack{\bf P}_\FM,\tau_{EB}\Rbrack=&\partial_t+v\cdot\nabla_x-\Big(E+v \times B \Big)\cdot\nabla_v+\Big(E+v \times B \Big)\cdot\frac{v-u }{ 2RT}\\
&+\mathbf{M}^{-\frac{1}{2}}\Big[\partial_t+v\cdot\nabla_x-\Big(E+v \times B \Big)\cdot\nabla_v\Big]\mathbf{M}^{\frac{1}{2}}\\
&-\varepsilon^k\left[ \left(E_R^{\varepsilon}+v \times B_R^{\varepsilon}\right)\cdot\nabla_v- \left(E_R^{\varepsilon}+v \times B_R^{\varepsilon}\right) \cdot\frac{v-u }{ 2RT}\right]\\
&-\sum_{n=1}^{2k-1}\varepsilon^n\left[\left(E_n+v \times B_n \right)\cdot\nabla_v-\left(E_n^{\varepsilon}+v \times B_n^{\varepsilon} \right)\cdot \frac{v-u }{ 2RT}\right].
\end{align*}
\begin{proposition}\label{f-eng-vml-1}
	Under the same assumptions in Lemma \ref{dxlm}, for $t\leq \varepsilon^{-\kappa}$ and $|\alpha|+|\beta|\leq4$ with $|\beta|\geq1$, we have
	\begin{align}\label{f-eng-vmlsum}
		&\frac{\mathrm{d}}{\mathrm{d} t}\sum\limits_{|\alpha|+|\beta|\leq 4,|\beta|\geq 1}\vps^{|\alpha|+|\beta|}\|\pa^\alpha _\beta ({\bf I}-{\bf P}_{\mathbf{M}})[f^{\varepsilon}]\|^2+\delta\sum\limits_{|\alpha|+|\beta|\leq 4,|\beta|\geq 1}\vps^{|\alpha|+|\beta|-1}\|\pa^\alpha _\beta({\bf I}-{\bf P}_{\mathbf{M}})[f^{\varepsilon}]\|^2_{\bf D}  \nonumber \\
		\lesssim&\sum\limits_{|\alpha|\leq 4}\vps^{|\alpha|}\Big[(1+t)^{-p_0}\epsilon_1+\varepsilon^{\kappa}\Big]\Big(\frac{1}{\vps}\|\partial^{\alpha}({\bf I}-{\bf P}_{\mathbf{M}})[f^{\varepsilon}]\|^2_{\bf D}+\|\partial^{\alpha}f^\vps\|^2+\|\partial^{\alpha}E_R^{\varepsilon}\|^2+\|\partial^{\alpha}B_R^{\varepsilon}\|^2\Big)\nonumber\\
		&+\sum\limits_{|\alpha|+|\beta|\leq4,|\beta|\geq 1}\vps^{|\alpha|+|\beta|+1}C_{\epsilon_1}\exp\Big(-\frac{\epsilon_1}{8C_0RT^2_c\sqrt{\varepsilon}}\Big)
		\Big(\sum_{|\al'|\leq|\al|,|\beta'|\leq |\beta|+1,\atop{|\alpha'|+|\beta'|\leq |\alpha|+|\beta|}}\|\partial^{\alpha'}_{\beta'}h^{\varepsilon}\|^2\Big)\\
		&+\sum\limits_{|\alpha|+|\beta|\leq 4,|\beta|\geq 1}\varepsilon^{2k+1+|\alpha|+|\beta|}(1+t)^{4k+2}.\nonumber
	\end{align}
\end{proposition}
\begin{proof}
By applying $\pa^\al_\beta$ to \eqref{VMLf-2-micro} and multiplying the result by $\pa^\alpha_\beta ({\bf I}-{\bf P}_{\mathbf{M}})[f^{\varepsilon}]$, we can make use of \eqref{wLLM-Non} without the weight function to obtain	
	\begin{align}\label{mix-h-es}
	&\frac{1}{2}\frac{\mathrm{d}}{\mathrm{d} t}\|\pa^\alpha _\beta ({\bf I}-{\bf P}_{\mathbf{M}})[f^{\varepsilon}]\|^2+\frac{\delta}{\varepsilon}\|\pa^\alpha_\beta ({\bf I}-{\bf P}_{\mathbf{M}})[f^{\varepsilon}]\|_{\bf D}^2 \\
	\lesssim& \frac{1}{\varepsilon}\sum_{\al'\leq \alpha}\sum_{|\beta'|<|\beta|}\big|\partial^{\al'}_{\beta'} ({\bf I}-{\bf P}_{\mathbf{M}})[f^{\varepsilon}]\big|_{\bf D}^2+\big|\big\langle \pa^\alpha_{\beta}\big(v\cdot\nabla_x({\bf I}-{\bf P}_{\mathbf{M}})[f^{\varepsilon}]\big) ,\pa^\alpha_\beta ({\bf I}-{\bf P}_{\mathbf{M}})[f^{\varepsilon}]\big\rangle\big|\nonumber\\
&+\Big|\Big\langle \pa^\alpha_{\beta}\Big[\Big(E+v \times B \Big)\cdot\nabla_v({\bf I}-{\bf P}_{\mathbf{M}})[f^{\varepsilon}]\Big] ,\pa^\alpha_\beta ({\bf I}-{\bf P}_{\mathbf{M}})[f^{\varepsilon}]\Big\rangle\Big|\nonumber\\
	&+\Big|\Big\langle \pa^\alpha_{\beta}\Big[\Big(E+v \times B \Big)\cdot\frac{v-u }{ 2RT}({\bf I}-{\bf P}_{\mathbf{M}})[f^{\varepsilon}]\Big] ,\pa^\alpha_\beta ({\bf I}-{\bf P}_{\mathbf{M}})[f^{\varepsilon}]\Big\rangle\Big|\nonumber\\
&+\Big|\Big\langle \pa^\alpha_{\beta}\Big[\mathbf{M}^{-\frac{1}{2}}({\bf I}-{\bf P}_{\mathbf{M}})f^{\varepsilon}\Big[\partial_t+v\cdot\nabla_x-\Big(E+v \times B \Big)\cdot\nabla_v\Big]\mathbf{M}^{\frac{1}{2}}\Big] ,\pa^\alpha_\beta ({\bf I}-{\bf P}_{\mathbf{M}})[f^{\varepsilon}]\Big\rangle\Big|\nonumber\\
&+\varepsilon^{k-1}\Big|\Big\langle \pa^\alpha_{\beta}\Gamma_{\mathbf{M}}(f^{\varepsilon},f^{\varepsilon}) ,\pa^\alpha_\beta ({\bf I}-{\bf P}_{\mathbf{M}})[f^{\varepsilon}]\Big\rangle\Big|\nonumber\\
&+\sum_{n=1}^{2k-1}
		\varepsilon^{n-1}\Big|\Big\langle \pa^\alpha_{\beta}\Big[\Gamma_{\mathbf{M}}(\mathbf{M}^{-\frac{1}{2}}F_n, f^{\varepsilon})+\Gamma_{\mathbf{M}}(f^{\varepsilon}, \mathbf{M}^{-\frac{1}{2}}F_n)\Big] ,\pa^\alpha_\beta ({\bf I}-{\bf P}_{\mathbf{M}})[f^{\varepsilon}]\Big\rangle\Big|\nonumber\\
&+\varepsilon^k\Big|\Big\langle \pa^\alpha_{\beta}\Big[\left(E_R^{\varepsilon}+v \times B_R^{\varepsilon}\right)\cdot\nabla_v({\bf I}-{\bf P}_{\mathbf{M}})[f^{\varepsilon}]\Big] ,\pa^\alpha_\beta ({\bf I}-{\bf P}_{\mathbf{M}})[f^{\varepsilon}]\Big\rangle\Big|\nonumber\\
&+\varepsilon^k\Big|\Big\langle \pa^\alpha_{\beta}\Big[\left(E_R^{\varepsilon}+v \times B_R^{\varepsilon}\right)\cdot\frac{v-u }{ 2RT}({\bf I}-{\bf P}_{\mathbf{M}})[f^{\varepsilon}]\Big] ,\pa^\alpha_\beta ({\bf I}-{\bf P}_{\mathbf{M}})[f^{\varepsilon}]\Big\rangle\Big|\nonumber\\
&+\sum_{n=1}^{2k-1}\varepsilon^n\Big|\Big\langle \pa^\alpha_{\beta}\Big[\left(E_n+v \times B_n \right)\cdot\nabla_v({\bf I}-{\bf P}_{\mathbf{M}})[f^{\varepsilon}]\Big] ,\pa^\alpha_\beta ({\bf I}-{\bf P}_{\mathbf{M}})[f^{\varepsilon}]\Big\rangle\Big|\nonumber\\
&+\sum_{n=1}^{2k-1}\varepsilon^n\Big|\Big\langle \pa^\alpha_{\beta}\Big[({\bf I}-{\bf P}_{\mathbf{M}})\Big[\left(E_R^{\varepsilon}+v \times B_R^{\varepsilon} \right)\cdot\mathbf{M}^{-\frac{1}{2}}\nabla_v F_n\Big]\Big] ,\pa^\alpha_\beta ({\bf I}-{\bf P}_{\mathbf{M}})[f^{\varepsilon}]\Big\rangle\Big|\nonumber\\
&+\sum_{n=1}^{2k-1}\varepsilon^n\Big|\Big\langle \pa^\alpha_{\beta}\Big[\left(E_n+v \times B_n \right)\cdot\frac{v-u}{ 2RT}({\bf I}-{\bf P}_{\mathbf{M}})[f^{\varepsilon}]\Big] ,\pa^\alpha_\beta ({\bf I}-{\bf P}_{\mathbf{M}})[f^{\varepsilon}]\Big\rangle\Big|\nonumber\\
&+\Big|\Big\langle \pa^\alpha_{\beta}\Big[\varepsilon^{k}({\bf I}-{\bf P}_{\mathbf{M}})\CQ_0+\Lbrack {\bf P},\tau_{EB} \Rbrack f^{\varepsilon}\Big] ,\pa^\alpha_\beta ({\bf I}-{\bf P}_{\mathbf{M}})[f^{\varepsilon}]\Big\rangle\Big|.\nonumber
\end{align}

Now we estimate terms on the R.H.S. of \eqref{mix-h-es}.

The 2nd term on the R.H.S. of \eqref{mix-h-es} can be bounded by
\begin{align*}
&C\|({\bf I}-{\bf P}_{\mathbf{M}})[f^{\varepsilon}]\|_{H^{|\alpha|+1}_xH^{|\beta|-1}_{\bf D}}\big\|\lag v\rag^{\frac{|\gamma+2s|}{2}}\pa^\alpha_{\beta}({\bf I}-{\bf P}_{\mathbf{M}})[f^{\varepsilon}]\big\|\\
\lesssim&\|({\bf I}-{\bf P}_{\mathbf{M}})[f^{\varepsilon}]\|_{H^{|\alpha|+1}_xH^{|\beta|-1}_{\bf D}}^2\nonumber\\
&+o(1)\int_{{\mathbb R}^3}\Big(\int_{{\lag v\rag}^{2|\gamma+2s|}\leq\varepsilon^{-|\gamma+2s|}}+\int_{{\lag v\rag}^{2|\gamma+2s|}\geq\varepsilon^{-|\gamma+2s|}}\Big)\left|\pa^\alpha_\beta({\bf I}-{\bf P}_{\mathbf{M}})[f^{\varepsilon}]\right|^2\, dv dx\nonumber\\
\lesssim&\|({\bf I}-{\bf P}_{\mathbf{M}})[f^{\varepsilon}]\|_{H^{|\alpha|+1}_xH^{|\beta|-1}_{\bf D}}^2+\frac{o(1)}{\varepsilon}\|\pa^\alpha_\beta ({\bf I}-{\bf P}_{\mathbf{M}})[f^{\varepsilon}]\|_{\bf D}^2\nonumber\\
&+C_{\epsilon_1}\exp\Big(-\frac{\epsilon_1}{8C_0RT^2_c\sqrt{\varepsilon}}\Big)\sum_{\alpha'\leq\alpha,\beta'\leq \beta}\big\|\pa^{\alpha'}_{\beta'} h^{\varepsilon}\big\|^2 ,\nonumber
\end{align*}
where we used similar estimates to those in  \eqref{dxihfL} in the last inequality.

It follows from Lemma \ref{dxlm} and Lemma \ref{dxvlm VMB} that the 3rd-5th, 8th-12th terms on the R.H.S. of \eqref{mix-h-es} can be dominated by
\begin{align*}
&C\Big[\epsilon_1(1+t)^{-p_0}+\varepsilon^{2\kappa}\Big]
	\Big[\|E_R^{\varepsilon}\|^2_{H^{|\alpha|}}+\|B_R^{\varepsilon}\|^2_{H^{|\alpha|}} \nonumber\\
	&+\sum_{|\al'|\leq|\al|,|\beta'|\leq |\beta|+1,\atop{|\alpha'|+|\beta'|\leq |\alpha|+|\beta|}}\Big(\frac{1}{\varepsilon}\left\|\partial^{\alpha'}_{\beta'}({\bf I}-{\bf P}_{\mathbf{M}})[f^{\varepsilon}]\right\|_{\bf D}^2+C_{\epsilon_1}\exp\Big(-\frac{\epsilon_1}{8C_0RT^2_c\sqrt{\varepsilon}}\Big)
	\|\partial^{\alpha'}_{\beta'}h^{\varepsilon}\|^2\Big)\Big].
\end{align*}		

For the 6th term on the R.H.S. of \eqref{mix-h-es}, we can use \eqref{aps-vml} for the VMB case, \eqref{GLM-Non}, and Sobolev's inequalities to directly bound it by
\begin{align*}
\frac{o(1)}{\varepsilon}\|\pa^\alpha_\beta ({\bf I}-{\bf P}_{\mathbf{M}})[f^{\varepsilon}]\|_{\bf D}^2+C\varepsilon\Big(\|({\bf I}-{\bf P}_{\mathbf{M}})[f^{\varepsilon}]\|_{H^{|\alpha|}_xH^{|\beta|}_{\bf D}}^2+\|f^{\varepsilon}\|_{H^{|\alpha|}_xL^2_v}^2\Big).\nonumber
\end{align*}

Via \eqref{vmb-Fn-es}, the 7th term on the R.H.S. of \eqref{mix-h-es} can be controlled similarly by
\begin{align*}
  \frac{o(1)}{\varepsilon}\|\pa^\alpha_{\beta}({\bf I}-{\bf P}_{\mathbf{M}})[f^{\varepsilon}]\|_{\bf D}^2 + C\varepsilon^{\kappa}\Big(\|{f^{\varepsilon}}\|^2_{H^{|\alpha|}_xL^{2}_v}+\|({\bf I}-{\bf P}_{\mathbf{M}})[f^{\varepsilon}]\|^2_{H^{|\alpha|}_xH^{|\beta|}_{\bf D}}\Big) .\nonumber
\end{align*}

As in \eqref{nonlinear-f-i-0-1},
the upper bound of the last term on the R.H.S. of \eqref{mix-h-es} is
\begin{align*}
   & \frac{o(1)}{\varepsilon}\|\pa^\alpha_{\beta}({\bf I}-{\bf P}_{\mathbf{M}})[f^{\varepsilon}]\|_{\bf D}^2 +\varepsilon^{2k+1}(1+t)^{4k+2}+\varepsilon\Big(\|f^{\varepsilon}\|^2_{H^{|\alpha|+1}_xL^2_v}+\|f^{\varepsilon}\|^2_{H^{|\alpha|}_xH^{|\beta|}_v}\Big) .\nonumber
\end{align*}
Finally, plugging the above estimates into \eqref{H1f1 VMB}, we multiply the result by $\varepsilon^{|\alpha|+|\beta|}$, sum over $|\alpha|+|\beta|\leq 4, |\beta|\geq 1$, and arrive at \eqref{f-eng-vmlsum} by the smallness of $\epsilon_1$.
\end{proof}	

Combining Proposition \ref{f-eng-vml} and Proposition \ref{f-eng-vml-1}, one can deduce the following total energy estimate of $f^{\varepsilon}$ by leveraging the smallness of $\epsilon_1$ and $\varepsilon$.
\begin{proposition}\label{f-end}
	Assume that the assumptions in Lemma \ref{dxlm} hold. Then for $t\leq \varepsilon^{-\kappa}$ with $\kappa=\frac{1}{3}$, it holds that
	\begin{align}\label{f-eng-vmb-sum}
		&\frac{\mathrm{d}}{\mathrm{d} t}\Big(\sum\limits_{|\alpha|\leq 4}\vps^{|\alpha|}\|\sqrt{4\pi RT}\pa^\alpha[f^{\varepsilon}]\|^2+\sum\limits_{|\alpha|+|\beta|\leq 4,|\beta|\geq1}\vps^{|\alpha|+|\beta|}\|\pa^\alpha _\beta ({\bf I}-{\bf P}_{\mathbf{M}})[f^{\varepsilon}]\|^2\Big)\\
		&+\delta\sum\limits_{|\alpha|+|\beta|\leq 4}\vps^{|\alpha|+|\beta|-1}\|\pa^\alpha _\beta({\bf I}-{\bf P}_{\mathbf{M}})[f^{\varepsilon}]\|^2_{\bf D}\nonumber\\
				\lesssim&\Big[\epsilon_1(1+t)^{-p_0}+\varepsilon^{\kappa}\Big]\mathcal{E}(t)+\varepsilon^{2k+1}(1+t)^{4k+2}+\varepsilon^{k}(1+t)^{2k}\sqrt{\mathcal{E}(t)},\nonumber
	\end{align}
where $\mathcal{E}(t)$ is defined in Lemma \ref{dxlmh VML}.
\end{proposition}


\subsection{Estimates of $h^{\varepsilon}$}
In this subsection, we focus on the weighted energy estimates of $h^\varepsilon$, which are divided into the spatial derivatives estimates part and the space-velocity derivatives estimates part.

\subsubsection{The spatial derivatives Estimates of $h^{\varepsilon}$}

Compared with the time-dependent exponential weight function $\exp\left(\frac{\langle v \rangle^2}{8RT_c\ln(\mathrm{e}+t)}\right)$ in the VML case, the weight function available for the non-cutoff VMB case is $\overline{w}=\exp\left(\frac{\lag v\rag}{8RT_c\ln(\mathrm{e}+t)}\right)$, which results in weaker additional dissipation terms $Y(t)\|{\lag v\rag}^{\frac{1}{2}}\overline{w} \pa^{\alpha}_{\beta}h^{\varepsilon}\|^2, |\alpha|+|\beta|\leq 4$.

Our first lemma concerns terms involving the velocity growth $|v|$ in the presence of the electric field. Since the velocity growth in our additional dissipation terms $Y(t)\|\langle v \rangle^{\frac{1}{2}}\overline{w} \partial^{\alpha}_{\beta}h^{\varepsilon}\|^2$ ($|\alpha|+|\beta|\leq 4$) is also of order 1, we need to bound these terms by the additional dissipation terms only, instead of controlling them through interpolation between the additional dissipation norms and the weighted norms of $h^{\varepsilon}$ as in \eqref{EBLndxih}.

\begin{lemma}\label{dxlmh VMB}
Assuming the conditions in Lemma \ref{dxlm} are satisfied, then for $|\alpha|+|\beta|\leq4$ and $t\leq \varepsilon^{-\kappa}$, it holds that
  \begin{align*}
   &\Big|\Big\langle \pa^\alpha_{\beta}\Big[\frac{E\cdot v }{ 2RT_c}h^{\varepsilon}\Big], \overline{w}^2 \pa^\alpha_{\beta}h^{\varepsilon}\Big\rangle\Big|
   \lesssim\epsilon_1Y(t)\sum_{\alpha'\leq\alpha,\beta'\leq\beta}\|{\lag v\rag}^{\frac{1}{2}}\overline{w} \pa^{\alpha'}_{\beta'}h^{\varepsilon}\|^2,
\end{align*}
   \begin{align*}
   \varepsilon^k\Big|\Big\langle \pa^\alpha_{\beta}\Big[\frac{E_R^{\varepsilon}\cdot v }{ 2RT_c}h^{\varepsilon}\Big],\overline{w}^2 \pa^\alpha_{\beta}h^{\varepsilon}\Big\rangle\Big|
    \lesssim\,\varepsilon Y(t)\sum_{\alpha'\leq\alpha,\beta'\leq\beta}\|{\lag v\rag}^{\frac{1}{2}}\overline{w} \pa^{\alpha'}_{\beta'}h^{\varepsilon}\|^2,
   \end{align*}
  and
   \begin{align}\label{EBndxih}
   &\sum_{n=1}^{2k-1}\varepsilon^n\Big|\Big\langle \pa^\alpha_{\beta}\Big[\frac{E_n\cdot v }{ 2RT_c}h^{\varepsilon}\Big], \overline{w}^2 \pa^\alpha_{\beta}h^{\varepsilon}\Big\rangle\Big|
    \lesssim\,\varepsilon^{\kappa^-}Y(t)\sum_{\alpha'\leq\alpha,\beta'\leq\beta}\|{\lag v\rag}^{\frac{1}{2}}\overline{w} \pa^{\alpha'}_{\beta'}h^{\varepsilon}\|^2.
   \end{align}
where $\kappa^-$ denote any postive number strictly smaller than $\kappa=\frac{1}{3}$.
\end{lemma}
\begin{proof}For brevity, let's verify \eqref{EBndxih} since the other two inequalities can be treated similarly. To do this, we use \eqref{em-decay} and \eqref{vmb-EB-es} to have
\begin{align*}
&\sum_{n=1}^{2k-1}\varepsilon^n\Big|\Big\langle \pa^\alpha_{\beta}\Big[\frac{E_n\cdot v }{ 2RT_c}h^{\varepsilon}\Big], \overline{w}^2 \pa^\alpha_{\beta}h^{\varepsilon}\Big\rangle\Big|
\lesssim\sum_{n=1}^{2k-1}\varepsilon^n(1+t)^{n}\sum_{\alpha'\leq\alpha,\beta'\leq\beta}\big\|{\lag v\rag}^{\frac{1}{2}}\overline{w} \pa^{\alpha'}_{\beta'} h^{\varepsilon}\big\|^2\\
\lesssim&\,\varepsilon(1+t)\sum_{\alpha'\leq\alpha,\beta'\leq\beta}\big\|{\lag v\rag}^{\frac{1}{2}}\overline{w} \pa^{\alpha'}_{\beta'} h^{\varepsilon}\big\|^2
\lesssim\,\varepsilon(1+t)^2\ln(1+t)Y(t)\sum_{\alpha'\leq\alpha,\beta'\leq\beta}\big\|{\lag v\rag}^{\frac{1}{2}}\overline{w} \pa^{\alpha'}_{\beta'} h^{\varepsilon}\big\|^2\\
\lesssim&\varepsilon^{\kappa^-
}Y(t)\sum_{\alpha'\leq\alpha,\beta'\leq\beta}\big\|{\lag v\rag}^{\frac{1}{2}}\overline{w} \pa^{\alpha'}_{\beta'} h^{\varepsilon}\big\|^2
\end{align*}
for $t\leq \varepsilon^{-\kappa}$ with $\kappa=\frac{1}{3}$. Here, $\kappa^-$ is defined earlier in the text.
\end{proof}

The next lemma focuses on controlling the velocity derivative terms caused by the Lorentz force. The most complicated case involves estimating terms such as $\big(v \times B \big)\cdot\nabla_vh^{\varepsilon}$, where both velocity growth and velocity derivatives are present simultaneously. Similar to Lemma \ref{dxlmh VMB}, we are able to bound these terms using the additional dissipation terms $Y(t)\|{\lag v\rag}^{\frac{1}{2}}\overline{w} \pa^{\alpha}_{\beta}h^{\varepsilon}\|^2 (|\alpha|+|\beta|\leq 4)$. These additional dissipation terms are induced by our weight function $\overline{w}$.
\begin{lemma}\label{key-VMB-2} Assume that the assumptions in Lemma \ref{dxlm} hold. Then for $|\alpha|+|\beta|\leq4$, $t\leq \varepsilon^{-\kappa}$, one has
 \begin{align}\label{EBvdxih-1}
		&\Big|\Big\langle \pa^\alpha_\beta\Big[\Big(E+v \times B \Big)\cdot\nabla_vh^{\varepsilon}\Big],  \overline{w}^2 \pa^\alpha_\beta h^{\varepsilon}\Big\rangle\Big|		\lesssim\epsilon_1 {\bf Z}_{5,\alpha,\beta}(t),
	\end{align}
	\begin{align}\label{EBRvdxih-1}
		&\varepsilon^k\Big|\Big\langle \pa^\alpha_\beta\Big[\big(E_R^{\varepsilon}+v \times B_R^{\varepsilon}\big) \cdot\nabla_vh^{\varepsilon}\Big], \overline{w}^2 \pa^\alpha_\beta h^{\varepsilon}\Big\rangle\Big|
		\lesssim\,\varepsilon {\bf Z}_{5,\alpha,\beta}(t),
	\end{align}
	and
	\begin{align}\label{EBnvdxih-1}
		&\sum_{n=1}^{2k-1}\varepsilon^n\Big|\Big\langle \pa^\alpha_\beta\big[\big(E_n+v \times B_n \big)\cdot\nabla_vh^{\varepsilon}+\big(E_R^{\varepsilon}+v \times B_R^{\varepsilon} \big)\cdot\mu^{-\frac{1}{2}}\nabla_v F_n\big],  \overline{w}^2 \pa^\alpha_\beta h^{\varepsilon}\Big\rangle\Big|\\
		\lesssim&\,\frac{o(1)}{\varepsilon}\|\overline{w}\pa^\alpha_\beta h^{\varepsilon}\|^2_{\bf D}+\varepsilon^{3-2\kappa}\Big(\|E_R^{\varepsilon}\|_{H^{|\alpha|}}^2+\|B_R^{\varepsilon}\|_{H^{|\alpha|}}^2\Big)+\varepsilon^{\kappa^-}{\bf Z}_{5,\alpha,\beta}(t),\nonumber
	\end{align}
where
$${\bf Z}_{5,\alpha,\beta}(t):=Y(t) \Big(\|{\lag v\rag}^{\frac{1}{2}}\overline{w} \pa^\alpha_\beta h^{\varepsilon}\|^2+{\bf I}_{|\alpha|+|\beta|\geq1}\sum_{|\alpha'|+|\beta'|<|\alpha|+|\beta|}\|\lag v\rag^{\frac{1}{2}}\overline{w} \pa^{\alpha'}_{\beta'}\nabla_v h^{\varepsilon}\|^2\Big).$$
  \end{lemma}
\begin{proof}
	For brevity, we will prove  \eqref{EBRvdxih-1} and omit the proof of \eqref{EBvdxih-1} and \eqref{EBnvdxih-1}.	

	From \eqref{vmbf4} and Sobolev's inequalities, we can obtain that for $t\leq \varepsilon^{-\kappa}$,
	\begin{align*}
		&\varepsilon^k \Big|\Big\langle \pa^\alpha_\beta\Big[\Big(E_R^{\varepsilon}+v \times B_R^{\varepsilon} \Big)\cdot\nabla_vh^{\varepsilon}\Big], \overline{w}^2 \pa^\alpha_\beta h^{\varepsilon}\Big\rangle\Big|\\
\leq&\,\varepsilon^k\sum_{\alpha'\leq\alpha, \beta'\leq \beta }\big|\big\langle\partial^{\alpha'}_{\beta'}\big(E_R^{\varepsilon}+ v \times B_R^{\varepsilon} \big)\cdot\nabla_v\partial^{\alpha-\alpha'}_{\beta-\beta'} h^{\varepsilon}, \overline{w}^2\pa^\alpha_\beta h^{\varepsilon}\big\rangle\big|\\
		\lesssim&\,\varepsilon^k\big|\big\langle E_R^{\varepsilon} \cdot\nabla_v\big(\overline{w}^2\big)  \pa^\alpha_\beta h^{\varepsilon},  \pa^\alpha_\beta h^{\varepsilon}\big\rangle\big|\\
		&+\varepsilon^k\sum_{\alpha'\leq\alpha, \beta'\leq \beta \atop{|\alpha'|+|\beta'|>0}}
{\bf 1}_{|\alpha|+|\beta|\geq1}\big|\big\langle\partial^{\alpha'}_{\beta'}\big(E_R^{\varepsilon}+ v \times B_R^{\varepsilon} \big)\cdot\nabla_v\partial^{\alpha-\alpha'}_{\beta-\beta'} h^{\varepsilon}, \overline{w}^2\pa^\alpha_\beta h^{\varepsilon}\big\rangle\big|\\
				\lesssim &\varepsilon^k\Big(\|E_R^\varepsilon\|_{H^4}+\|B_R^\varepsilon\|_{H^4}\Big){\bf Z}_{5,\alpha,\beta}(t)
						\lesssim\varepsilon {\bf Z}_{5,\alpha,\beta}(t).
	\end{align*}
Here in the third line of the above estimate, we used the integration by parts w.r.t. $v$ when $|\alpha'|+|\beta'|=0$ and the fact $\nabla_v\cdot\big(v \times B_R^{\varepsilon} \big)=0$.
\end{proof}

With the above preparations in hand,
now we turn to the weighted energy estimates on $h^\varepsilon$. The following proposition is about the estimate of the pure spatial derivatives of $h^\varepsilon$.
\begin{proposition}\label{h-vmb-eng-prop}
Assume that the assumptions in Lemma \ref{dxlm} hold. For $t\leq \varepsilon^{-\kappa}$, it holds that
\begin{align}\label{h-vmb-eng}
	&\sum_{|\alpha|\leq 4}\varepsilon^{|\alpha|+1+\kappa}\Big(\frac{\mathrm{d}}{\mathrm{d} t}\|\overline{w} \partial^{\alpha} h^{\varepsilon}\|^2+Y(t)\big\|{\lag v\rag}^{\frac{1}{2}}\overline{w} \partial^{\alpha} h^{\varepsilon}\big\|^2
	+\frac{\delta}{\varepsilon}\|\overline{w} \partial^{\alpha} h^{\varepsilon}\|^2_{\bf D}\Big) \\
	\lesssim&\;\varepsilon^{\kappa}\mathcal{E}(t)+\left(\epsilon_1+\varepsilon^{\kappa^-}\right)\mathcal{D}(t)+\varepsilon^{2k+1}(1+t)^{4k+2},\nonumber
\end{align}
where $\mathcal{E}(t)$ is defined in \eqref{eg-vmb}, and $\mathcal{D}(t)$ is given in \eqref{dn-vmb}.
\end{proposition}

\begin{proof}
The proof is divided into two steps.
\vskip 0.2cm
\noindent{\underline{\it Step 1: Basic Weighted Energy Estimate of $h^{\varepsilon}$.}} In this step, we aim to establish the following weighted $L^2$ estimate for $h^{\varepsilon}$:
\begin{align}\label{L2h VMB}
	&\frac{\mathrm{d}}{\mathrm{d} t}\|\overline{w}h^{\varepsilon}\|^2+Y(t)\|{\lag v\rag}^{\frac{1}{2}}\overline{w}h^{\varepsilon}\|^2+\frac{\delta}{\varepsilon}\|\overline{w}h^{\varepsilon}\|^2_{\bf D}
	\nonumber\\
	\lesssim&\frac{1}{\varepsilon}\|f^{\varepsilon}\|^2+\varepsilon^{\kappa}\Big(\|\overline{w}h^{\varepsilon}\|^2+\|E_R^{\varepsilon}\|^2+\|B_R^{\varepsilon}\|^2\Big)+\varepsilon^{2k+1}(1+t)^{4k+2}.
\end{align}
To prove this, we first take the $L^2$ inner product of $\overline{w}^2 h^{\varepsilon}$ with \eqref{h-equation} and use \eqref{wLL-Non} with $|\beta|=0$ to obtain
\begin{align}\label{L2h1 VMB}
	&\frac{1}{2}\frac{\mathrm{d}}{\mathrm{d} t}\|\overline{w}h^{\varepsilon}\|^2+Y(t)\|{\lag v\rag}^{\frac{1}{2}}\overline{w}h^{\varepsilon}\|^2
	+\frac{\delta}{\varepsilon}\|\overline{w}h^{\varepsilon}\|^2_{\bf D} \\
	\leq&\;\frac{C}{\varepsilon}\|f^{\varepsilon}\|^2+\frac{1}{\varepsilon}\big|\big\langle \mathcal{L}_d[ h^{\varepsilon}], \overline{w}^2h^{\varepsilon}\big\rangle\big|+\Big|\Big\langle \big(E_R^{\varepsilon}+v \times B_R^{\varepsilon} \big) \cdot \frac{v-u}{ RT}\mu^{-\frac{1}{2}}\mathbf{M},\overline{w}^2 h^{\varepsilon}\Big\rangle\Big| \nonumber\\
	& +\Big|\Big\langle \big(E+v \times B \big) \cdot \nabla_v h^{\varepsilon}, \overline{w}^2 h^{\varepsilon}\Big\rangle\Big|  +\Big|\Big\langle \frac{E\cdot v }{ 2RT_c}h^{\varepsilon}, \overline{w}^2 h^{\varepsilon}\Big\rangle\Big|\nonumber\\
	&+\varepsilon^{k-1}\big|\big\langle\Gamma ( h^{\varepsilon},
	h^{\varepsilon} ), \overline{w}^2 h^{\varepsilon}\big\rangle\big|+\sum_{n=1}^{2k-1}\varepsilon^{n-1}\big|\big\langle[\Gamma(\mu^{-\frac{1}{2}}F_n, h^{\varepsilon})    +\Gamma( h^{\varepsilon}, \mu^{-\frac{1}{2}} F_n)], \overline{w}^2 h^{\varepsilon}\big\rangle\big|\nonumber\\
	&+\varepsilon^k \big|\big\langle \Big(E_R^{\varepsilon}+v \times B_R^{\varepsilon} \Big)\cdot\nabla_vh^{\varepsilon},\overline{w}^2 h^{\varepsilon}\big\rangle\big|+\varepsilon^k \big|\big\langle \frac{E_R^{\varepsilon} \cdot v}{ 2RT_c}h^{\varepsilon},\overline{w}^2 h^{\varepsilon}\big\rangle\big|\nonumber\\
	&+\sum_{n=1}^{2k-1}\varepsilon^n\Big|\Big\langle \big(E_n+v \times B_n \Big)\cdot\nabla_vh^{\varepsilon}+\big(E_R^{\varepsilon}+v \times B_R^{\varepsilon} \big)\cdot\frac{\nabla_v F_n}{\sqrt{\mu}},\overline{w}^2 h^{\varepsilon}\Big\rangle\Big|\nonumber\\
	&+\sum_{n=1}^{2k-1}\varepsilon^n\Big|\Big\langle \frac{E_n\cdot v}{ 2RT_c}h^{\varepsilon},\overline{w}^2 h^{\varepsilon}\Big\rangle\Big|
	+\varepsilon^k\big|\big\langle \CQ_1,\overline{w}^2 h^{\varepsilon}\big\rangle\big|.\nonumber
\end{align}

We can now proceed to estimate each term on the R.H.S. of \eqref{L2h1 VMB}.
For the 2nd term on the R.H.S. of \eqref{L2h1 VMB}, we use \eqref{wGLd-Non} to bound it by
$ C\epsilon_1\varepsilon^{-1}\|\overline{w}h^{\varepsilon}\|^2_{\bf D}$.
The upper bound of the 3rd term on the R.H.S. of \eqref{L2h1 VMB} is
$ o(1)\varepsilon^{-1}\|\overline{w}h^{\varepsilon}\|^2_{\bf D} +C\varepsilon\big(\|E_R^{\varepsilon}\|^2+\|B_R^{\varepsilon}\|^2\big)$.
From Lemma \ref{dxlmh VMB} and Lemma \ref{key-VMB-2}, we can see that the 4th, 5th, 8th-11th terms on the R.H.S. of \eqref{L2h1 VMB} can be bounded by
\begin{align*}
	&C\left(\epsilon_1+\varepsilon^{\kappa^-}\right) Y(t)\|{\lag v\rag}^{\frac{1}{2}}\overline{w} h^{\varepsilon}\|^2+\frac{o(1)}{\vps}\|\overline{w}h^{\varepsilon}\|^2_{\bf D}+C\varepsilon^{3-2\kappa}\Big(\|E_R^{\varepsilon}\|^2
	+\|B_R^{\varepsilon}\|^2\Big).\nonumber
\end{align*}

For the 6th term on the R.H.S. of \eqref{L2h1 VMB}, we can use \eqref{Gamma-noncut-1}, \eqref{vmbf4}, and Sobolev's inequalities to bound it by
\begin{align*}
	C\varepsilon^{k-1}\|{\lag v\rag}^{\frac{1}{2}}\overline{w}h^{\varepsilon}\|_{L^{\infty}_{x}L^2_v}
	\|\overline{w}h^{\varepsilon}\|^2_{{\bf D}}
	\lesssim\varepsilon^{k-1}\|h^{\varepsilon}\|_{H^2_{x}L^2_v}
	\|\overline{w}h^{\varepsilon}\|^2_{{\bf D}}
	\lesssim
	\|\overline{w}h^{\varepsilon}\|^2_{{\bf D}}.
\end{align*}


For the 7th term on the R.H.S. of \eqref{L2h1 VMB}, we can use \eqref{vmb-Fn-es} to bound it similarly by
\begin{align*}
    \frac{o(1)}{\varepsilon}\|\overline{w}h^{\varepsilon}\|_{\bf D}^2+C(1+t)^2\varepsilon\|\overline{w}h^{\varepsilon}\|^2
\lesssim\frac{o(1)}{\varepsilon}\|\overline{w}h^{\varepsilon}\|_{\bf D}^2+\varepsilon^{\kappa}\|\overline{w}h^{\varepsilon}\|^2.
\end{align*}

As in \eqref{nonlinear-f-i-0-1}, the last term on the R.H.S. of \eqref{L2h1 VMB} can be bounded by
\begin{align*}
\varepsilon^k\big|\big\langle \CQ_1, \overline{w}^2 h^{\varepsilon}\big\rangle\big|\lesssim \frac{\epsilon_1}{\varepsilon} \|\overline{w}h^{\varepsilon}]\|^2_{\bf D} +\varepsilon^{2k+1}(1+t)^{4k+2}.
\end{align*}

Collecting the above estimates in \eqref{L2h1 VMB} gives \eqref{L2h VMB}.

\vskip 0.2cm
\noindent\underline{{\it Step 2. High-order derivative estimates of $h^{\varepsilon}$ with weight.} }
In this step, we continue to deduce the estimate of $\|\overline{w}\pa^\alpha h^{\varepsilon}\|$ with  $1\leq |\alpha|\leq N$.

Next, taking the $L^2$ inner product of $\overline{w}^2 \partial^{\alpha} h^{\varepsilon}$ with \eqref{mh1x VML} in the non-cutoff VMB case and applying \eqref{wLL-Non} with $|\beta|=0$, one obtains
\begin{align}\label{H1h1 VMB}
&\frac{1}{2}\frac{\mathrm{d}}{\mathrm{d} t}\|\overline{w} \partial^{\alpha} h^{\varepsilon}\|^2+Y(t)\big\|{\lag v\rag}^{\frac{1}{2}}\overline{w} \partial^{\alpha} h^{\varepsilon}\big\|^2
    +\frac{\delta}{\varepsilon}\|\overline{w} \partial^{\alpha} h^{\varepsilon}]\|^2_{\bf D} \\
    \leq&\;\frac{C}{\varepsilon}\|f^{\varepsilon}\|^2_{H^{|\alpha|}_xL^2_v}+\frac{1}{\varepsilon}\big|\big\langle \partial^{\alpha}\mathcal{L}_d[ h^{\varepsilon}], \overline{w}^2 \partial^{\alpha}h^{\varepsilon}\big\rangle\big|\nonumber\\
    &+\Big|\Big\langle \partial^{\alpha}\Big[\big(E_R^{\varepsilon}+v \times B_R^{\varepsilon} \big) \cdot \frac{v-u}{ RT}\mu^{-\frac{1}{2}}\mathbf{M}\Big], \overline{w}^2 \partial^{\alpha} h^{\varepsilon}\Big\rangle\Big|
    +\Big|\Big\langle \partial^{\alpha}\Big[\frac{E\cdot v }{ 2RT_c}h^{\varepsilon}\Big], \overline{w}^2 \partial^{\alpha}h^{\varepsilon}\Big\rangle\Big|\nonumber\\
    &  +\big|\big\langle \partial^{\alpha}\big[\big(E+v \times B \big)\cdot\nabla_vh^{\varepsilon}\big], \overline{w}^2 \partial^{\alpha} h^{\varepsilon}\big\rangle\big|  +\varepsilon^{k-1}\big|\big\langle \partial^{\alpha}\Gamma ( h^{\varepsilon},
    h^{\varepsilon} ), \overline{w}^2 \partial^{\alpha} h^{\varepsilon}\big\rangle\big|\nonumber\\
    &+\sum_{n=1}^{2k-1}\varepsilon^{n-1}\big|\big\langle[\partial^{\alpha}\Gamma(\mu^{-\frac{1}{2}}F_n,h^{\varepsilon})
    +\partial^{\alpha}\Gamma(
 h^{\varepsilon}, \mu^{-\frac{1}{2}} F_n)], \overline{w}^2 \partial^{\alpha} h^{\varepsilon}\big\rangle\big|\nonumber\\
 &+\varepsilon^k \big|\big\langle \partial^{\alpha}\Big[\Big(E_R^{\varepsilon}+v \times B_R^{\varepsilon} \Big)\cdot\nabla_vh^{\varepsilon}\Big], \overline{w}^2 \partial^{\alpha} h^{\varepsilon}\big\rangle\big|+\varepsilon^k \Big|\Big\langle \partial^{\alpha}\Big[\frac{E_R^{\varepsilon} \cdot v}{ 2RT_c}h^{\varepsilon}\Big], \overline{w}^2 \partial^{\alpha} h^{\varepsilon}\Big\rangle\Big|\nonumber\\
 &+\sum_{n=1}^{2k-1}\varepsilon^i\Big|\Big\langle \partial^{\alpha}\Big[\big(E_n+v \times B_n \big)\cdot\nabla_vh^{\varepsilon}+\big(E_R^{\varepsilon}+v \times B_R^{\varepsilon} \big)\cdot\mu^{-\frac{1}{2}}\nabla_v F_n\Big], \overline{w}^2 \partial^{\alpha} h^{\varepsilon}\Big\rangle\Big|\nonumber\\
 &+\sum_{n=1}^{2k-1}\varepsilon^n\Big|\Big\langle \partial^{\alpha}\Big(\frac{E_n\cdot v}{ 2RT_c}h^{\varepsilon}\Big), \overline{w}^2 \partial^{\alpha} h^{\varepsilon}\Big\rangle\Big|
  +\varepsilon^k\big|\big\langle \partial^{\alpha}\CQ_1, \overline{w}^2 \partial^{\alpha} h^{\varepsilon}\big\rangle\big|.\nonumber
\end{align}
We can now estimate the R.H.S. of \eqref{H1h1 VMB} individually.

For the 2nd term on the R.H.S. of \eqref{H1h1 VMB}, from \eqref{wGLd-Non}, its upper bound is given by
$C\epsilon_1 \varepsilon^{-1}\big(\|\overline{w} \partial^{\alpha}h^{\varepsilon}\|^2_{\bf D}+\sum_{\alpha'<\alpha}\|\overline{w} \partial^{\alpha'}h^{\varepsilon}\|^2_{\bf D}\big).$
The 3rd term on the R.H.S. of \eqref{H1h1 VMB} can be dominated by
\begin{align*}
 \frac{o(1)}{\varepsilon}\|\overline{w} \partial^{\alpha}h^{\varepsilon}\|^2_{\bf D}+C\varepsilon\Big(\|E_R^{\varepsilon}\|_{H^{|\alpha|}}^2+\|B_R^{\varepsilon}\|_{H^{|\alpha|}}^2\Big).
\end{align*}

By applying Lemma \ref{dxlmh VMB} and Lemma \ref{key-VMB-2}, we can bound the 4th, 5th, 8th-11th terms on the R.H.S. of \eqref{H1h1 VMB} by
\begin{align*}
&C\left(\epsilon_1+\varepsilon^{\kappa^-}\right)Y(t)\sum_{|\alpha'|+|\beta'|\leq |\alpha|,\atop |\beta'|\leq1}\|\lag v\rag^{\frac{1}{2}} \overline{w} \partial^{\alpha'}_{\beta'}h^{\varepsilon}\|^2
+\frac{o(1)}{\vps}\|\overline{w} \partial^{\alpha}h^{\varepsilon}\|^2_{\bf D}\nonumber\\
&+C\varepsilon^{3-2\kappa}\Big(\|E_R^{\varepsilon}\|_{H^{|\alpha|}}^2
+\|B_R^{\varepsilon}\|_{H^{|\alpha|}}^2\Big).
\end{align*}

For the 6th term on the R.H.S. of \eqref{H1h1 VMB}, we use \eqref{Gamma-noncut-1}, \eqref{vmbf4}, and Sobolev's inequalities to obtain
\begin{align}\label{gamh-1-Non}
    &\varepsilon^{k-1}\big|\big\langle\partial^{\alpha}\Gamma ( h^{\varepsilon},
    h^{\varepsilon} ), \overline{w}^2  \partial^{\alpha}h^{\varepsilon}\big\rangle\big|\nonumber\\
    \lesssim&\varepsilon^{k-1}\sum_{\alpha_1+\alpha_2=\alpha}{\bf I}_{|\alpha_1|\leq 2}\left\|\overline{w}\pa^{\alpha_1}h^{\varepsilon}\right\|_{L^{\infty}_{x}L^2_v}\left\|\overline{w}\pa^{\alpha_2}h^{\varepsilon}\right\|_{{\bf D}}\left\|\overline{w}\partial^{\alpha}h^{\varepsilon}\right\|_{{\bf D}}\nonumber\\
    &+\varepsilon^{k-1}\sum_{\alpha_1+\alpha_2=\alpha}{\bf I}_{|\alpha_1|> 2}\left\|\overline{w}\pa^{\alpha_1}h^{\varepsilon}\right\|
    \left\|\overline{w}\pa^{\alpha_2}h^{\varepsilon}\right\|_{L^{\infty}_{x}L^2_{\bf D}}\left\|\overline{w}\partial^{\alpha}h^{\varepsilon}\right\|_{{\bf D}}\\
    \lesssim&\varepsilon^{k-1}\left\|\overline{w}h^{\varepsilon}\right\|_{H^4}\sum_{|\alpha'|\leq|\alpha|}
    \left\|\overline{w}\partial^{\alpha'}h^{\varepsilon}\right\|^2_{{\bf D}}    \lesssim \sum_{|\alpha'|\leq|\alpha|}
    \big\|\overline{w}\partial^{\alpha'}h^{\varepsilon}\big\|^2_{{\bf D}}.\nonumber
\end{align}

For the 7th term on the R.H.S. of \eqref{H1h1 VMB},
we use \eqref{em-Fn-es} to similarly obtain
	\begin{align*}
	&\sum_{n=1}^{2k-1}\varepsilon^{n-1}\big|\big\langle[\partial^{\alpha}\Gamma(\mu^{-\frac{1}{2}}F_n,h^{\varepsilon})+\partial^{\alpha}\Gamma(
	h^{\varepsilon}, \mu^{-\frac{1}{2}} F_n)], \overline{w}^2\partial^{\alpha} h^{\varepsilon}\big\rangle\big|\\
	\lesssim&\,\sum_{n=1}^{2k-1}[\varepsilon(1+t)]^{n-1}(1+t)\sum_{\alpha'\leq\alpha}\Big(
\big\|\overline{w}\partial^{\alpha'}h^{\varepsilon}\big\|_{{\bf D}}+
\big\|\overline{w}\partial^{\alpha'}h^{\varepsilon}\big\|\Big)\big\|\overline{w}\pa^{\alpha}h^{\varepsilon}\big\|_{{\bf D}}\\
	\lesssim&\,\frac{o(1)}{\varepsilon} \|\overline{w}\pa^{\alpha}h^{\varepsilon}\|^2_{{\bf D}} + \varepsilon^{\kappa}\sum_{\alpha'\leq\alpha}\Big(
\big\|\overline{w}\partial^{\alpha'}h^{\varepsilon}\big\|_{{\bf D}}^2+
\big\|\overline{w}\partial^{\alpha'}h^{\varepsilon}\big\|^2\Big),
\end{align*}
for $t\leq \varepsilon^{-\kappa}$ with $\kappa=\frac{1}{3}$. By a similar way as that in \eqref{nonlinear-f-i-0-1}, one gets
\begin{align*}
\big|\big\langle \pa^{\alpha}\CQ_1, \overline{w}^2\pa^{\alpha} h^{\varepsilon}\big\rangle\big|\lesssim \frac{\epsilon_1}{\varepsilon} \|\overline{w}\pa^{\alpha}h^{\varepsilon}\|^2_{{\bf D}} +\varepsilon^{2k+1}(1+t)^{4k+2}.
\end{align*}

Putting the above estimates into \eqref{H1h1 VMB} with $\vps$ sufficiently small, we arrive at
\begin{align}\label{H1h VMB-al-w}
	&\frac{\mathrm{d}}{\mathrm{d} t}\|\overline{w} \partial^{\alpha} h^{\varepsilon}\|^2+Y(t)\big\|{\lag v\rag}^{\frac{1}{2}}\overline{w} \partial^{\alpha} h^{\varepsilon}\big\|^2
	+\frac{\delta}{\varepsilon}\|\overline{w} \partial^{\alpha} h^{\varepsilon}]\|^2_{{\bf D}} \\
	\lesssim&\;\frac{1}{\varepsilon}\|f^{\varepsilon}\|^2_{H^{|\alpha|}_xL^2_v}+\varepsilon^{\kappa}\Big(\sum_{\alpha'\leq\alpha}\big\|\overline{w}\partial^{\alpha'}h^{\varepsilon}\big\|^2+\|E_R^{\varepsilon}\|_{H^{|\alpha|}}^2+\|B_R^{\varepsilon}\|_{H^{|\alpha|}}^2\Big)\nonumber\\
	&+\frac{\epsilon_1 }{\varepsilon}\sum_{\alpha'<\alpha}\|\overline{w} \partial^{\alpha'}h^{\varepsilon}\|^2_{\bf D}+\epsilon_1Y(t)\sum_{|\alpha'|+|\beta'|\leq |\alpha|,\atop |\beta'|\leq1}\|\lag v\rag^{\frac{1}{2}} \overline{w} \partial^{\alpha'}_{\beta'}h^{\varepsilon}\|^2+\varepsilon^{2k+1}(1+t)^{4k+2}.\nonumber
\end{align}

Multiplying  \eqref{H1h VMB-al-w} by $\varepsilon^{|\alpha|+1+\kappa}$, taking summation over $1\leq |\alpha|\leq 4$, and combing with $\varepsilon^{1+\kappa}\times \eqref{L2h1 VMB}$ give \eqref{h-vmb-eng}.
%
 This ends the proof of Proposition \ref{h-vmb-eng-prop}.
\end{proof}
\subsubsection{The space-velocity derivatives Estimates of $h^{\varepsilon}$}
In this part, we are devoted to the space-velocity derivatives estimates of $h^{\varepsilon}$, with weight. As a preparation, we first estimate the delicate transport term. Thanks to the additional dissipation induced by our weight functions $\overline{w}$ and the coercivity of the linearized Boltzmann collision operator, we can control this term  via interpolation inequalities.
\begin{lemma}\label{transp}
	For $|\beta|\geq1$, $t\leq \eps^{-\kappa}$ with $\kappa=\frac13$,  it holds that
	\begin{eqnarray}\label{transport-h}
		&&|\lag\pa^\alpha_\beta (v\cdot \nabla_x h^\eps),\overline{w}^2\pa^\alpha_\beta h^\varepsilon\rag|\nonumber\\
		&\lesssim&\varepsilon^{\frac{1}{5}} \Big(\frac{1}{\varepsilon}\big\|\overline{w}\pa^\alpha_\beta h^\varepsilon\big\|_{\bf D}^2+Y(t)\big\|\lag v\rag^{\frac{1}{2}}\overline{w}\pa^\alpha_\beta h^\varepsilon\big\|^2\\
&&+\frac{1}{\varepsilon}\big\|\overline{w} \nabla_x^{|\alpha|+1}\nabla^{|\beta|-1}_vh^\varepsilon\big\|^2_{\bf D}+Y(t)\big\|\lag v\rag^{\frac{1}{2}}\overline{w} \nabla_x^{|\alpha|+1}\nabla^{|\beta|-1}_vh^\varepsilon\big\|^2\Big).\nonumber
	\end{eqnarray}
\end{lemma}
\begin{proof} It is straightforward to see that
\begin{align*}
|\lag\pa^\alpha_\beta (v\cdot \nabla_x h^\eps),\overline{w}^2\pa^\alpha_\beta h^\varepsilon\rag|\lesssim \big\|\overline{w} \nabla_x^{|\alpha|+1}\nabla^{|\beta|-1}_vh^\varepsilon\big\| \left\|\overline{w}\pa^\alpha_\beta h^\varepsilon\right\|.
\end{align*}
We further bound this term  via interpolation inequalities as follows:
	\begin{eqnarray*}
		&& |\lag\pa^\alpha_\beta (v\cdot \nabla_x h^\eps),\overline{w}^2\pa^\alpha_\beta h^\varepsilon\rag|\nonumber\\
		&\lesssim& \Big(\big\|\overline{w} \nabla_x^{|\alpha|+1}\nabla^{|\beta|-1}_vh^\varepsilon\big\|_{\bf D} \Big)^{\frac{1}{1+|\gamma+2s|}}
\Big(\big\|\lag v\rag^{\frac12}\overline{w} \nabla_x^{|\alpha|+1}\nabla^{|\beta|-1}_vh^\varepsilon\big\|\Big)^{\frac{|\gamma+2s|}{1+|\gamma+2s|}}\nonumber\\
		&&\times\Big(\big\|\overline{w}\pa^\alpha_\beta h^\varepsilon\big\|_{\bf D}\Big)^{\frac{1}{1+|\gamma+2s|}}
\Big(\big\|\lag v\rag^{\frac12}\overline{w}\pa^\alpha_\beta h^\varepsilon\big\|\Big)^{\frac{|\gamma+2s|}{1+|\gamma+2s|}}\nonumber\\
		&\lesssim& \varepsilon^{\frac{1}{1+|\gamma+2s|}}Y(t)^{\frac{-|\gamma+2s|}{1+|\gamma+2s|}} \Big(\frac{1}{\varepsilon}\big\|\overline{w}\pa^\alpha_\beta h^\varepsilon\big\|_{\bf D}^2+Y(t)\big\|\lag v\rag^{\frac12}\overline{w}\pa^\alpha_\beta h^\varepsilon\big\|^2\nonumber\\
&&+\frac{1}{\varepsilon}\big\|\overline{w} \nabla_x^{|\alpha|+1}\nabla^{|\beta|-1}_vh^\varepsilon\big\|^2_{\bf D}+Y(t)\big\|\lag v\rag^{\frac12}\overline{w} \nabla_x^{|\alpha|+1}\nabla^{|\beta|-1}_vh^\varepsilon\big\|^2\Big).
	\end{eqnarray*}
Noting that for $t\leq \eps^{-\kappa}$ with $\kappa=\frac{1}{3}$ and $|\gamma+2s|<\frac{3}{2}$,
\begin{align*}
\varepsilon^{\frac{1}{1+|\gamma+2s|}}Y(t)^{\frac{-|\gamma+2s|}{1+|\gamma+2s|}} \lesssim \varepsilon^{\frac{1}{1+|\gamma+2s|}}\varepsilon^{\frac{-|\gamma+2s|\kappa^-}{1+|\gamma+2s|}}\lesssim \varepsilon^{\frac{1}{1+3/2}-\frac{3/2}{3(1+3/2)}}\lesssim \varepsilon^{\frac{1}{5}},
\end{align*}
we further obtain \eqref{transport-h}.

\end{proof}

Now we are ready to give the energy estimates on $\overline{w}^2\pa^\alpha_\beta h^\varepsilon$ with $|\beta|\geq 1$ in the following:

\begin{proposition}\label{h-vml-eng-prop-y}
	Assume that the assumptions in Lemma \ref{dxlm} hold. For $t\leq \varepsilon^{-\kappa}$, it holds that
	\begin{align}\label{h-vml-eng-y}
	\sum_{|\alpha|+|\beta|\leq 4, |\beta|\geq 1}\varepsilon^{|\alpha|+|\beta|+1+\kappa}&\Big(\frac{\mathrm{d}}{\mathrm{d} t}\|\overline{w} \partial^\alpha_\beta h^{\varepsilon}\|^2+Y(t)\big\|{\lag v\rag}^{\frac{1}{2}}\overline{w} \partial^\alpha_\beta h^{\varepsilon}\big\|^2+\frac{\delta}{\varepsilon}\|\overline{w} \partial^\alpha_\beta   h^{\varepsilon}]\|^2_{\bf D}\Big) \nonumber\\
		\lesssim&\varepsilon^{\kappa}\mathcal{E}(t)+\epsilon_1\mathcal{D}(t)+\varepsilon^{2k+1}(1+t)^{4k+2}.
	\end{align}
\end{proposition}

\begin{proof} For $|\alpha|+|\beta|\leq 4,|\beta|\geq 1$, applying $\partial^{\alpha}_\beta $ to \eqref{h-equation}, one has
	\begin{align}\label{mh1x VML-y}
		\partial_t\partial^{\alpha}_\beta h^{\varepsilon}&+\partial^{\alpha}_\beta [v\cdot\nabla_xh^{\varepsilon}]+\partial^{\alpha}_\beta\Big[\frac{\big(E_R^{\varepsilon}+v \times B_R^{\varepsilon} \big) }{ RT}\cdot \big(v-u\big)\mu^{-\frac{1}{2}}\mathbf{M}\Big]\nonumber\\
		&+\partial^{\alpha}_\beta\Big[\frac{E\cdot v }{ 2RT_c}h^{\varepsilon}\Big]-\partial^{\alpha}_\beta\big[\big(E+v \times B \big)\cdot\nabla_vh^{\varepsilon}\big]+\frac{\partial^{\alpha}_\beta\mathcal{L}[h^{\varepsilon}]}{\varepsilon}\nonumber\\
		=&-\frac{\partial^{\alpha}_\beta\mathcal{L}_d[h^{\varepsilon}]}{\varepsilon}+\varepsilon^{k-1}\partial^{\alpha}_\beta\Gamma(h^{\varepsilon},h^{\varepsilon})
		+\sum_{n=1}^{2k-1}\varepsilon^{n-1}[\partial^{\alpha}_\beta\Gamma(\mu^{-\frac{1}{2}}F_n, h^{\varepsilon})+\partial^{\alpha}_\beta\Gamma(h^{\varepsilon}, \mu^{-\frac{1}{2}}F_n)]\nonumber\\
		&+\varepsilon^k \partial^{\alpha}_\beta\Big[\Big(E_R^{\varepsilon}+v \times B_R^{\varepsilon}\Big)\cdot\nabla_vh^{\varepsilon}\Big]-\varepsilon^k  \partial^{\alpha}_\beta\Big[\frac{E_R^{\varepsilon}\cdot v}{ 2RT_c}h^{\varepsilon}\Big]\\
		&+\sum_{n=1}^{2k-1}\varepsilon^n\partial^{\alpha}_\beta\Big[\Big(E_n+v \times B_n \Big)\cdot\nabla_vh^{\varepsilon}+\Big(E_R^{\varepsilon}+v \times B_R^{\varepsilon} \Big)\cdot \mu^{-\frac{1}{2}}\nabla_v F_n\Big]\nonumber\\
		&-\sum_{n=1}^{2k-1}\varepsilon^n\partial^{\alpha}_\beta\Big(\frac{E_n\cdot v}{ 2RT_c}h^{\varepsilon}\Big)+\varepsilon^{k}\partial^{\alpha}_\beta\CQ_1.
		\nonumber
	\end{align}
	Next, taking the $L^2$ inner product of $\overline{w}^2\partial^\alpha_\beta h^{\varepsilon}$ with \eqref{mh1x VML-y}  and applying \eqref{wLL-Non}, one has
	\begin{align}\label{H1h1 VML-y}
	&\frac{1}{2}\frac{\mathrm{d}}{\mathrm{d} t}\|\overline{w} \partial^\alpha_\beta h^{\varepsilon}\|^2+Y(t)\big\|{\lag v\rag}^{\frac{1}{2}}\overline{w} \partial^\alpha_\beta h^{\varepsilon}\big\|^2+\frac{\delta}{\varepsilon}\|\overline{w} \partial^\alpha_\beta   h^{\varepsilon}]\|^2_{\bf D} \\ \lesssim&\;\frac{1}{\varepsilon}\Big(\|f^{\varepsilon}\|^2_{H^{|\alpha|}_xH^{|\beta|}_v}+o(1)\sum_{|\beta'|\leq|\beta|}\big\|\overline{w}\partial^{\alpha}_{\beta'}h^{\varepsilon}\big\|_{\bf D}^2\Big)\nonumber\\
	&+|\lag\pa^\alpha_\beta (v\cdot \nabla_x h^\eps),\overline{w}^2\pa^\alpha_\beta h^\varepsilon\rag|+\frac{1}{\varepsilon}\big|\big\langle \partial^\alpha_\beta\mathcal{L}_d[ h^{\varepsilon}], \overline{w}^2  \partial^\alpha_\beta h^{\varepsilon}\big\rangle\big|\nonumber\\
	&+\Big|\Big\langle \pa^{\alpha}_\beta\Big[\big(E_R^{\varepsilon}+v \times B_R^{\varepsilon} \big) \cdot \frac{v-u}{ RT}\mu^{-\frac{1}{2}}\mathbf{M}\Big], \overline{w}^2 \partial^\alpha_\beta h^{\varepsilon}\Big\rangle\Big|
	+\Big|\Big\langle \partial^\alpha_\beta\Big[\frac{E\cdot v }{ 2RT_c}h^{\varepsilon}\Big], \overline{w}^2 \partial^\alpha_\beta h^{\varepsilon}\Big\rangle\Big|\nonumber\\
	&  +\big|\big\langle \partial^\alpha_\beta\big[\big(E+v \times B \big)\cdot\nabla_vh^{\varepsilon}\big],\overline{w}^2 \partial^\alpha_\beta h^{\varepsilon}\big\rangle\big|  +\varepsilon^{k-1}\big|\big\langle \partial^\alpha_\beta\Gamma ( h^{\varepsilon},
	h^{\varepsilon} ), \overline{w}^2 \partial^\alpha_\beta h^{\varepsilon}\big\rangle\big|\nonumber\\
	&+\sum_{n=1}^{2k-1}\varepsilon^{n-1}\big|\big\langle[\pa^{\alpha}_\beta\Gamma(\mu^{-\frac{1}{2}}F_n,h^{\varepsilon})
	+\pa^{\alpha}_\beta\Gamma(
	h^{\varepsilon}, \mu^{-\frac{1}{2}} F_n)], \overline{w}^2 \partial^\alpha_\beta h^{\varepsilon}\big\rangle\big|\nonumber\\
	&+\varepsilon^k \big|\big\langle \pa^{\alpha}_\beta\Big[\Big(E_R^{\varepsilon}+v \times B_R^{\varepsilon} \Big)\cdot\nabla_vh^{\varepsilon}\Big], \overline{w}^2 \pa^{\alpha}_\beta h^{\varepsilon}\big\rangle\big|+\varepsilon^k \Big|\Big\langle \pa^{\alpha}_\beta\Big[\frac{E_R^{\varepsilon} \cdot v}{ 2RT_c}h^{\varepsilon}\Big], \overline{w}^2 \partial^\alpha_\beta h^{\varepsilon}\Big\rangle\Big|\nonumber\\
	&+\sum_{n=1}^{2k-1}\varepsilon^n\Big|\Big\langle \pa^{\alpha}_\beta\Big[\big(E_n+v \times B_n \big)\cdot\nabla_vh^{\varepsilon}+\big(E_R^{\varepsilon}+v \times B_R^{\varepsilon} \big)\cdot\mu^{-\frac{1}{2}}\nabla_v F_n\Big], \overline{w}^2 \partial^\alpha_\beta h^{\varepsilon}\Big\rangle\Big|\nonumber\\
	&+\sum_{n=1}^{2k-1}\varepsilon^n\Big|\Big\langle \pa^{\alpha}_\beta\Big(\frac{E_n\cdot v}{ 2RT_c}h^{\varepsilon}\Big), \overline{w}^2 \partial^\alpha_\beta h^{\varepsilon}\Big\rangle\Big|
	+\big|\big\langle \pa^{\alpha}_\beta\CQ_1, \overline{w}^2 \partial^\alpha_\beta h^{\varepsilon}\big\rangle\big|.\nonumber
\end{align}

	We now estimate the R.H.S. of \eqref{H1h1 VML-y} individually.

 In view of Lemma \ref{transp}, the 2nd term on the R.H.S. of \eqref{H1h1 VML-y} can be bounded by
	\begin{align*}
&C\varepsilon^{\frac{1}{5}} \Big(\frac{1}{\varepsilon}\big\|\overline{w}\pa^\alpha_\beta h^\varepsilon\big\|_{\bf D}^2+Y(t)\big\|\lag v\rag^{\frac12}\overline{w}\pa^\alpha_\beta h^\varepsilon\big\|^2\\
&+\frac{1}{\varepsilon}\big\|\overline{w} \nabla_x^{|\alpha|+1}\nabla^{|\beta|-1}_vh^\varepsilon\big\|^2_{\bf D}+Y(t)\big\|\lag v\rag^{\frac12}\overline{w} \nabla_x^{|\alpha|+1}\nabla^{|\beta|-1}_vh^\varepsilon\big\|^2\Big).
\end{align*}

		The upper bound of the 3rd term on the R.H.S. of \eqref{H1h1 VML-y} is
	$C\epsilon_1 \sum_{|\al'|\leq|\al|,\beta'\leq \beta}\|\overline{w} \pa^{\alpha'}_{\beta'}h^{\varepsilon}\|^2_{\bf D}$  by \eqref{wGLd-Non} and
 the 4th term on the R.H.S. of \eqref{H1h1 VML-y} is dominated by
	\begin{align*}
		\frac{o(1)}{\varepsilon}\|\overline{w} \pa^\alpha_\beta h^{\varepsilon}\|^2_{\bf D}+C\varepsilon\Big(\|E_R^{\varepsilon}\|_{H^{|\alpha|}}^2+\|B_R^{\varepsilon}\|_{H^{|\alpha|}}^2\Big).
	\end{align*}

	Applying Lemma \ref{dxlmh VMB} and  Lemma \ref{key-VMB-2}, we can bound the 5th, 6th, 9th-12th terms on the R.H.S. of \eqref{H1h1 VML-y} by
		\begin{eqnarray*}
&&\,C\left(\epsilon_1+\varepsilon^{\kappa^-}\right)Y(t)\sum_{|\alpha'|+|\beta'|\leq |\alpha|+|\beta|,\atop |\beta'|\leq |\beta|+1}\big\|{\lag v\rag}^{\frac{1}{2}}\overline{w} \pa^{\al'}_{\beta'} h^{\varepsilon}\big\|^2\nonumber\\
&&+C\varepsilon^{3-2\kappa}\Big(\|E_R^{\varepsilon}\|_{H^{|\alpha|}}^2+\|B_R^{\varepsilon}\|_{H^{|\alpha|}}^2\Big)+\frac{o(1)}{\varepsilon}\|\overline{w}^2 \pa^\alpha_\beta h^{\varepsilon}\|^2_{\bf D}.
	\end{eqnarray*}

	For the 7th term on the R.H.S. of \eqref{H1h1 VML-y}, similar to \eqref{gamh-1-Non}, we have
		\begin{align*}
    \varepsilon^{k-1}\big|\big\langle\pa^\alpha_\beta\Gamma ( h^{\varepsilon},
    h^{\varepsilon} ), \overline{w}^2  \pa^\alpha_\beta h^{\varepsilon}\big\rangle\big|
        \lesssim \sum_{|\alpha'|+|\beta'|\leq|\alpha|+|\beta|}
   \big\|\overline{w}\pa^{\alpha'}_{\beta'}h^{\varepsilon}\big\|^2_{{\bf D}}.\nonumber
\end{align*}

		For the 8th term on the R.H.S. of \eqref{H1h1 VML-y},
	 we get from \eqref{vmb-Fn-es} that
	\begin{align*} &\sum_{n=1}^{2k-1}\varepsilon^{n-1}\big|\big\langle[\pa^\alpha_\beta\Gamma(\mu^{-\frac{1}{2}}F_n,h^{\varepsilon})+\pa^\alpha_\beta\Gamma(
		h^{\varepsilon}, \mu^{-\frac{1}{2}}F_n)], \overline{w}^2\pa^\alpha_\beta h^{\varepsilon}\big\rangle\big|\\
		\lesssim&\,\sum_{n=1}^{2k-1}[\varepsilon(1+t)]^{n-1}(1+t)\big\|\sum_{\alpha'\leq\alpha,\beta'\leq\beta}\Big(
\big\|\overline{w}\partial^{\alpha'}h^{\varepsilon}\big\|_{{\bf D}}+
\big\|\overline{w}\partial^{\alpha'}h^{\varepsilon}\big\|\Big)\big\|\overline{w}\pa^{\alpha}h^{\varepsilon}\big\|_{{\bf D}}\\
	\lesssim&\,\frac{o(1)}{\varepsilon} \|\overline{w}\pa^{\alpha}_{\beta}h^{\varepsilon}\|^2_{{\bf D}} + \varepsilon^{\kappa}\sum_{\alpha'\leq\alpha,\beta'\leq\beta}\Big(
\big\|\overline{w}\pa^{\alpha'}_{\beta'}h^{\varepsilon}\big\|_{{\bf D}}^2+
\big\|\overline{w}\pa^{\alpha'}_{\beta'}h^{\varepsilon}\big\|^2\Big).
	\end{align*}

	As in \eqref{nonlinear-f-i-0-1}, for the last term on the R.H.S. of \eqref{H1h1 VML-y}, one has
	\begin{align*}
		\big|\big\langle \pa^\alpha_\beta\CQ_1, \overline{w}^2\pa^\alpha_\beta  h^{\varepsilon}\big\rangle\big|\lesssim \frac{o(1)}{\varepsilon} \|\overline{w}\pa^\alpha_\beta h^{\varepsilon}\|^2_{\bf D} +\varepsilon^{2k+1}(1+t)^{4k+2}.
	\end{align*}
	
	Putting the above estimates into \eqref{H1h1 VML-y}, we arrive at
	\begin{align}\label{H1h VMB}
		&\frac{1}{2}\frac{\mathrm{d}}{\mathrm{d} t}\|\overline{w} \partial^\alpha_\beta h^{\varepsilon}\|^2+Y(t)\big\|{\lag v\rag}^{\frac{1}{2}}\overline{w} \partial^\alpha_\beta h^{\varepsilon}\big\|^2+\frac{\delta}{\varepsilon}\|\overline{w} \partial^\alpha_\beta   h^{\varepsilon}]\|^2_{\bf D} \\
		\lesssim&\frac{1}{\varepsilon}\Big(\|f^{\varepsilon}\|^2_{H^{|\alpha|}_xH^{|\beta|}_v}+o(1)\sum_{|\beta'|\leq|\beta|}\big\|\overline{w}\partial^{\alpha}_{\beta'}h^{\varepsilon}\big\|_{\bf D}^2\Big)+\frac{\epsilon_1 }{\varepsilon}\sum_{|\al'|+|\beta'|\leq|\al|+ |\beta|}\|\overline{w} \pa^{\alpha'}_{\beta'}h^{\varepsilon}\|^2_{\bf D}\nonumber\\
		&+\sum_{|\al'|+|\beta'|\leq|\al|+ |\beta|}\varepsilon^{\kappa}\big\|\overline{w}\pa^{\alpha'}_{\beta'}h^{\varepsilon}\big\|^2+\varepsilon\Big(\|E_R^{\varepsilon}\|_{H^{|\alpha|}}^2+\|B_R^{\varepsilon}\|_{H^{|\alpha|}}^2\Big) \nonumber\\
		&+\left(\epsilon_1+\varepsilon^{\kappa^-}\right)Y(t)\sum_{|\alpha'|+|\beta'|\leq |\alpha|+|\beta|}\big\|{\lag v\rag}^{\frac{1}{2}}\overline{w} \pa^{\al'}_{\beta'} h^{\varepsilon}\big\|^2+\varepsilon^{2k+1}(1+t)^{4k+2}.\nonumber
	\end{align}
	
	Now multiplying \eqref{H1h VMB} by $\varepsilon^{|\alpha|+|\beta|+1+\kappa}$ and taking summation over $|\alpha|+|\beta|\leq 4, |\beta|\geq 1$, one deduces \eqref{h-vml-eng-prop-y}.
\end{proof}

Finally, we
collect \eqref{h-vmb-eng} and \eqref{h-vml-eng-y} together to obtain  the following total energy  estimate for $ h^{\varepsilon}$.
\begin{proposition} \label{h-end}
\label{mix-h}Assume that the assumptions in Lemma \ref{dxlm} hold. For $t\leq \varepsilon^{-\kappa}$, it holds that
	\begin{align}\label{mix-h}
		&\sum_{|\alpha|+|\beta|\leq 4}\varepsilon^{|\alpha|+|\beta|+1+\kappa}\Big(\frac{\mathrm{d}}{\mathrm{d} t}\|\overline{w} \partial^\alpha_\beta h^{\varepsilon}\|^2+Y(t)\big\|{\lag v\rag}^{\frac{1}{2}}\overline{w} \partial^\alpha_\beta h^{\varepsilon}\big\|^2+\frac{\delta}{\varepsilon}\|\overline{w} \partial^\alpha_\beta   h^{\varepsilon}]\|^2_{\bf D}\Big) \nonumber\\
		\lesssim&\varepsilon^{\kappa}\mathcal{E}(t)+\varepsilon^{2k+1}(1+t)^{4k+2}.\nonumber
	\end{align}
\end{proposition}
\subsection{Proof of the Theorem \ref{resultVMB}} 

We are now ready to complete the proof of Theorem~\ref{resultVMB}. For brevity, we will derive the energy estimates \eqref{TVML1} and omit the details of the non-negativity proof of $F^{\varepsilon}$.

For this purpose, combining Propositions \ref{h-end} and \ref{f-end}, we have
\begin{align*}
	&\frac{\mathrm{d}}{\mathrm{d} t}\Bigg(\sum\limits_{|\alpha|\leq 4}\varepsilon^{|\alpha|}\big(\|\sqrt{4\pi RT}\partial^{\alpha} f^{\varepsilon}\|^2+\|\partial^{\alpha} E_R^{\varepsilon}\|^2+\|\partial^{\alpha} B^{\varepsilon}_R\|^2\big)\\
	&+\sum\limits_{|\alpha|+|\beta|\leq 4,|\beta|\geq1}\varepsilon^{|\alpha|+|\beta|}\|\pa^\alpha _\beta ({\bf I}-{\bf P}_{\mathbf{M}})[f^{\varepsilon}]\|^2+\sum_{|\alpha|+|\beta|\leq 4}\varepsilon^{|\alpha|+|\beta|+1+\kappa}\|\overline{w} \partial^\alpha_\beta  h^{\varepsilon}\|^2\Bigg)\\
	&+\delta\sum\limits_{|\alpha|+|\beta|\leq 4}\varepsilon^{|\alpha|+|\beta|-1}\|\pa^\alpha _\beta({\bf I}-{\bf P}_{\mathbf{M}})[f^{\varepsilon}]\|^2_{\bf D}\\
	&+\sum_{|\alpha|+|\beta|\leq 4}\varepsilon^{|\alpha|+|\beta|+1+\kappa}\Big(Y(t)\big\|{\lag v\rag}^{\frac{1}{2}}\overline{w} \partial^\alpha_\beta h^{\varepsilon}\big\|^2+\frac{\delta}{\varepsilon}\|\overline{w} \partial^\alpha_\beta   h^{\varepsilon}]\|^2_{\bf D}\Big) \\
	&\lesssim \Big[\epsilon_1(1+t)^{-p_0}+\varepsilon^{\kappa}\Big]\mathcal{E}(t)+\varepsilon^{2k+1}(1+t)^{4k+2}+\varepsilon^{k}(1+t)^{2k}\sqrt{\mathcal{E}(t)},
\end{align*}

Here, $\mathcal{E}(t)$ is defined in \eqref{eg-vmb}. By choosing $\varepsilon_1\geq \varepsilon>0$ sufficiently small such that

$$C_{\epsilon_1}\exp\left(\frac{-\epsilon_1}{8C_0RT^2_c\sqrt{\varepsilon}}\right)\leq \varepsilon_1^{2+\kappa},\qquad \kappa=\frac13,$$

and recalling the definition of $\mathcal{D}(t)$ in \eqref{dn-vmb}, we get

\[\frac{d}{dt}\mathcal{E}(t)+\mathcal{D}(t)\lesssim\Big[\varepsilon^{\kappa}+(1+t)^{-p_0}+\varepsilon^{k}(1+t)^{2k}\Big] \big[\mathcal{E}(t)+1\big]+\varepsilon^{2k+1}(1+t)^{4k+2}.\]

This yields

\begin{align*}
	\mathcal{E}(t)+\int_0^t\mathcal{D}(s)\, d s\lesssim \mathcal{E}(0)+1
\end{align*}

from Gronwall's inequality over $t\in [0,\varepsilon^{-\kappa}]$ with $\kappa=\frac13$. Then, \eqref{TVML1} follows.

\section{Appendix}\label{sec-app}
In this Appendix, we will illustrate the process of determining the coefficients $F_n, E_n, B_n$ for all $1\leq n\leq2k-1$ in the expansion \eqref{expan}.
The construction of coefficients for the non-cutoff VMB case is presented in detail, while the construction of coefficients for the VML system is omitted for brevity.

\subsection{The coefficients of the VML system}
This subsection is concerned with the construction of coefficients $[F_n, E_n, B_n]$ for all $1\leq n\leq2k-1$ in the expansion \eqref{expan}.
It can be stated as in the following lemma, whose proof is omitted since it is similar to that in Lemma \ref{vmb-Fn-lem}.
\begin{lemma}
Under the condition \eqref{Fn-id-vml}, coefficients $[F_n, E_n, B_n]$ for all $1\leq n\leq2k-1$ are well-posed. Moreover, for $N\geq2$, $q\in(0,1)$, and $\ell\geq0$, there exists $C_1>0$ such that
\begin{align}
\sum\limits_{\al_0+|\al|+|\beta|\leq N+2k-2n+2}\left|\pa_t^{\al_0}\pa_\bet^\al \left(\frac{F_n}{\sqrt{\FM}}\right)\right|\leq C_1\lag v\rag^{-\ell+2n}\FM^{\frac{q}{2}}(1+t)^{n}\label{em-Fn-es}
\end{align}
and
\begin{align}
\sum\limits_{\al_0+|\al|\leq N+2k-2n+2}|\pa_t^{\al_0}\pa^\al [E_n,B_n]|\leq C_1(1+t)^{n},\label{em-EB-es}
\end{align}
where $C_1>0$ and depends on $N$, $q$, and initial data of $F_n, E_n, B_n$.
\end{lemma}

\subsection{The coefficients of the non-cutoff VMB system}
Our results for the non-cutoff VMB system are the following
\begin{lemma}\label{vmb-Fn-lem}
Under the condition \eqref{Fn-id},  coefficients $[F_n, E_n, B_n]$ for all $1\leq n\leq2k-1$ are well-posed. Moreover, for $N\geq4$, $q\in(0,1)$, and $\ell\geq0$, there exists $C_2>0$ such that
	\begin{align}
		\sum\limits_{\al_0+|\al|+|\beta|\leq N+4k-2n+2}\left|\pa_t^{\al_0}\pa_\bet^\al \left(\mu^{-\frac{1}{2}}F_n\right)\right|\leq C_2\lag v\rag^{-\ell+n(1+|\gamma+2s|)}e^{-q\langle v\rangle}(1+t)^{n}\label{vmb-Fn-es}
	\end{align}
	and
	\begin{align}
		\sum\limits_{\al_0+|\al|\leq N+4k-2n+2}|\pa_t^{\al_0}\pa^\al [E_n,B_n]|\leq C_2(1+t)^{n},\label{vmb-EB-es}
	\end{align}
	where $C_2>0$ and depends on $N$, $q$, and initial data of $F_n, E_n, B_n$.
\end{lemma}
\begin{proof} Our proof is divided into two steps.
\vskip 0.2cm
\noindent{\underline{\it Step 1. Solvability of the coefficients.}} In this step, we will solve the coefficients $F_n, E_n$ and $B_n$ around the local Maxwellian $\FM$.
	The main argument applied here is originally used in
	\cite{Guo-CPAM-2006} for the case of cutoff Boltzmann equation and Landau equation around global Maxwellian.
	Setting $F_n=\FM^{\frac{1}{2}}f_n$ and recalling the definition \eqref{fn-mac-def}, we write
	\begin{align*}
	{\bf P}_{\mathbf{M}}[f_n]=&\frac{\rho_n}{\sqrt{\rho}} \chi_0+\sum\limits_{i=1}^3\frac{1}{\sqrt{R\rho T}}u_n^i\chi_i+\frac{T_n}{\sqrt{6\rho}}\chi_4,\\
F_n=&\Big[\frac{\rho_n}{\rho} \FM+\sum\limits_{i=1}^3\frac{v^i-u^i}{R\rho T}u_n^i\FM+\frac{T_n}{6\rho}\Big(\frac{|v-u|^2}{RT}-3\Big)\FM\Big]+\FM^{\frac{1}{2}}({\bf I-P}_{\mathbf{M}})[f_n].
	\end{align*}
	From the coefficient equation $\varepsilon^{n-1}$ with $1\leq n\leq 2k-1$ in \eqref{expan2}, it is straightforward to obtain the expression for microscopic part $f_{n}=\FM^{-\frac{1}{2}}F_{n}$:
\begin{align}\label{micro-fn}
({\bf I-P}_{\mathbf{M}})[f_{n}]=&\mathcal{L}_{\FM}^{-1}\Big[-\FM^{-\frac{1}{2}}\Big(\partial_tF_{n-1}+v\cdot \nabla_xF_{n-1}\\
&-\sum_{\substack{i+j=n-1\\i,j\geq0}}\Big(E_n+v \times B_n \Big)\cdot\nabla_vF_j-\sum_{\substack{i+j=n\\i,j\geq1}}\mathcal{C}(F_i,F_j)\Big)\Big], \nonumber
\end{align}
and equations of $E_n, B_n$:
\begin{align*}
&\pa_t E_n-\na_x\times B_n=4\pi\big(u_n+\rho_n u\big),\notag\\
	&		\pa_t B_n+\na_x\times E_n=0,\\
			&\na_x\cdot E_n=-\rho_n,\ \na_x\cdot B_n=0.\notag
\end{align*}
To obtain equations of $[\rho_n, u_n, T_n]$ which corresponds to the macroscopic part of $f_{n}$, we need to project the equation of $f_n$ in \eqref{expan2} onto $1, v, |v|^2$.
Before doing this, we first note that
\begin{align*}
&\int_{\mathbb R^3} F_n dv=\rho_n, \qquad \int_{\mathbb R^3}  v F_n dv=\int_{\mathbb R^3}  u F_n dv+\int_{\mathbb R^3}  (v-u) F_n dv=\rho_n u+u_n, \\
&\int_{\mathbb R^3}  v^iv^j F_n dv=\int_{\mathbb R^3}\Big[(v^i-u^i)(v^j-u^j)-\delta_{ij}\frac{|v-u|^2}{3}\Big] F_n dv\\
&\hspace{2.5cm}+u_n^iu^j+\rho_nu^iu^j+u^iu_n^j+\frac{\delta_{ij}(3RT\rho_n+RTT_n)}{3},
\end{align*}
\begin{align*}
&\int_{\mathbb R^3}  |v|^2 F_n dv=\rho_n|u|^2+2u\cdot u_n+3RTT_n+RTT_n,\\
&\int_{\mathbb R^3} v^i|v|^2F_n dv=\int_{\mathbb R^3}(v^i-u^i)|v|^2 F_n dv+\int_{\mathbb R^3}u^i|v|^2 F_n dv\\
&\hspace{2.3cm}=\int_{\mathbb R^3}(v^i-u^i)\big[|v-u|^2-5RT\big] F_n dv\\
&\hspace{2.5cm}+2u^j\int_{\mathbb R^3}\Big[(v^i-u^i)(v^j-u^j)-\delta_{ij}\frac{|v-u|^2}{3}\Big] F_n dv\\
&\hspace{2.5cm}+(5RT+|u|^2)u_n^i+u^i\Big[(5RT+|u|^2)\rho_n+2u\cdot u_n+\frac{5}{3}RTT_n\Big],
\end{align*}
and
\begin{align*}
&-\int_{\mathbb R^3}v^i(E+v\times B)\cdot\na_v F_n dv\\
&\hspace{1cm}=\int_{\mathbb R^3}\big[E^i+(v\times B)^i\big] F_n dv=\rho_n\big[E^i+\big(u\times B\big)^i\big]+\big(u_n\times B\big)^i,\\
&-\int_{\mathbb R^3}|v|^2(E+v\times B)\cdot\na_v F_n dv=2\int_{\mathbb R^3}v\cdot E F_n dv=2E\cdot(\rho_n u+u_n).
\end{align*}
Now we project the equation of $f_n$ in \eqref{expan2} onto $1, v, |v|^2$ to get
\begin{align*}
&\partial_t\rho_n+\na_x\cdot(\rho_n u+u_n)=0,\\
&\partial_t(\rho_n u^i+u_n^i)+\partial_{x_j}\Big[u_n^iu^j+\rho_nu^iu^j+u^iu_n^j+\frac{\delta_{ij}(3RT\rho_n+RTT_n)}{3}\Big]\\
&+\rho\big[E^i_n+\big(u\times B_n\big)^i\big]+\rho_n\big[E^i+\big(u\times B\big)^i\big]+\big(u_n\times B\big)^i\\
=&-\partial_{x_j}\int_{\mathbb R^3}\Big[(v^i-u^i)(v^j-u^j)-\delta_{ij}\frac{|v-u|^2}{3}\Big] F_n dv\\
&-\sum_{\substack{i_0+j_0=n\\i_0,j_0\geq1}}\Big[\rho_{j_0}\big[E^i_{i_0}+\big(u\times B_{i_0}\big)^i\big]+\big(u_{j_0}\times B_{i_0}\big)^i\Big],\\
&\partial_t\big[\rho_n|u|^2+2u\cdot u_n+3RT\rho_n+RTT_n\big]\\
&+\nabla_x\cdot\Big[(5RT+|u|^2)(u_n+\rho_nu)+2u\cdot u_n+\frac{5}{3}RTT_n u\big)\Big]\\
&+2\rho u\cdot E_n+2E\cdot(\rho_n u+u_n)+2\sum_{\substack{i_0+j_0=n\\i_0,j_0\geq1}}E_{i_0}\cdot(\rho_{j_0} u+u_{j_0})\\
=&-\nabla_x\cdot\int_{\mathbb R^3}(v-u)\big[|v-u|^2-5RT\big] F_n dv\\
&-2\partial_{x_j}\Big[u^i\int_{\mathbb R^3}\Big[(v^i-u^i)(v^j-u^j)-\delta_{ij}\frac{|v-u|^2}{3}\Big] F_n dv\Big].
\end{align*}
By utilizing the compressible Euler-Maxwell system \eqref{EM}, we can express the aforementioned equations in a more concise form
\begin{align}\label{em-abc-eq}
		&\partial_t\rho_n+\na_x\cdot(\rho_n u+u_n)=0,\notag\\
&\partial_tu_n+(u\cdot\nabla_x)u_n+(u_n\cdot\nabla_x)u+(\nabla_x\cdot u)u_n+\frac{1}{3}\nabla_x(3RT\rho_n+RTT_n)-\frac{R\rho_n}{\rho}\nabla_x\big(\rho T\big)\notag\\
&+\rho\big[E_n+\big(u\times B_n\big)\big]+u_n\times B+\sum_{\substack{i+j=n\\i,j\geq1}}\Big[\rho_{j}\big[E_{i}+\big(u\times B_{i}\big)\big]+\big(u_{j}\times B_{i}\big)\Big]\notag\\
=&-\nabla_x\cdot\int_{\mathbb R^3}\Big[(v-u)\otimes(v-u)-\frac{|v-u|^2}{3}\FI_{3\times3}\Big] F_n dv,\notag\\
&RT\big[\partial_tT_n+u\cdot \nabla_x T_n+ (\nabla_x\cdot u) T_n\big]+2RT\nabla_x\cdot u_n+3R\nabla_xT\cdot u_n\notag\\
&-\frac{2RT}{\rho}\nabla_x\rho\cdot u_n-2(u\times B)\cdot u_n\\
=&2u\cdot\sum_{\substack{i+j=n\\i,j\geq1}}\big[\rho_jE_n +u_j\times B_i\big]-\nabla_x\cdot\int_{\mathbb R^3}(v-u)\big[|v-u|^2-5RT\big] F_n dv\notag\\
&+2u^i\partial_{x_j}\Big[\int_{\mathbb R^3}\Big[(v^i-u^i)(v^j-u^j)-\delta_{ij}\frac{|v-u|^2}{3}\Big] F_n dv\Big],\notag\\
&-2\partial_{x_j}\Big[u^i\int_{\mathbb R^3}\Big[(v^i-u^i)(v^j-u^j)-\delta_{ij}\frac{|v-u|^2}{3}\Big] F_n dv\Big],\notag\\
&\pa_t E_n-\na_x\times B_n=u_n+\rho_n u,\notag\\
			&\pa_t B_n+\na_x\times E_n=0,\notag\\
			&\na_x\cdot E_n=-\rho_n,\ \na_x\cdot B_n=0,\notag
			\end{align}
	where $1\leq n\leq2k-1$ and $\FI_{3\times3}$ stands for the unit $3\times 3$ matrix.

Following the approach in \cite{Guo-Jang-CMP-2010} and \cite{Guo-Xiao-CMP-2021}, we can express \eqref{em-abc-eq} as a symmetric hyperbolic system. Consequently, we can construct classical solution $[\rho_n, u_n, T_n, E_n, B_n]$ in the Sobolev space $H^{4k+4-2n}$ with $1\leq n\leq 2k-1$ inductively.
\vskip 0.2cm
\noindent{\underline{\it Step 2. Estimates of the coefficients.}} We aim to establish \eqref{vmb-Fn-es} and \eqref{vmb-EB-es}, indicating the velocity decay estimates of $F_n$ and the temporal growth of $F_n, E_n$, and $B_n$ with $1\leq n\leq 2k-1$, around the global Maxwellian $\mu$.

Letting $F_n=\mu^{\frac{1}{2}}h_n$ for $1\leq n\leq 2k-1$, we decompose $h_n$ as the macroscopic part and microscopic part:
\begin{align*}
h_n={\bf P}[h_n]+({\bf I-P})[h_n].
\end{align*}
	Next, from \eqref{expan2}, one sees that $h_n$ satisfies the following iterative equations
	\begin{align}\notag
		\mathcal{L} [h_1]+\mathcal{L}_d [h_1]=-\mu^{-\frac{1}{2}}\{\pa_t\FM+v\cdot\na_x\FM\}+\mu^{-\frac{1}{2}}(E+v\times B)\cdot\na_v \FM,
	\end{align}
	\begin{align}
		&\pa_th_1+v\cdot\na_xh_1-(E+v\times B)\cdot\na_v h_1+\frac{v}{2RT_c}  \cdot E  h_1\notag\\
&-h_1\mu^{-\frac{1}{2}}(E_1+v\times B_1)\cdot\na_v\FM
		+\CL  [h_2]+\CL_d [h_2]=\Ga (h_1, h_1),\notag
	\end{align}
	\begin{align}
		\cdots\cdots\cdots\cdots\cdots,\notag
	\end{align}
	\begin{align}
		&\pa_th_n+v\cdot\na_xh_n-(E+v\times B)\cdot\na_v h_n+\frac{v}{2RT_c}  \cdot E  h_n\notag\\
  &	-\mu^{-\frac{1}{2}}(E_n+v\times B_n)\cdot\na_v\FM-{\bf 1}_{n>1}\sum\limits_{i+j=n\atop{i,j\geq1}}(E_i+v\times B_i)\cdot\na_v h_j
		\notag\\&
		-{\bf 1}_{n>1}\sum\limits_{i+j=n\atop{i,j\geq1}}\frac{v}{2RT_c}  \cdot E_i  h_j
		+\CL  [h_{n+1}]+\CL_d [h_{n+1}]\notag=\sum\limits_{i+j=n+1}\Ga (h_i,h_j),\notag
	\end{align}
	\begin{align}
		\cdots\cdots\cdots\cdots\cdots,\notag
	\end{align}
	\begin{align*}
		&\pa_th_{2k-2}+v\cdot\na_xh_{2k-2}
		-(E+v\times B)\cdot\na_v h_{2k-2}+\frac{v}{2RT_c}  \cdot E  h_{2k-2}\notag\\
&-\mu^{-\frac{1}{2}}(E_{2k-2}+v\times B_{2k-2})\cdot\na_v\FM-\sum\limits_{i+j={2k-2}\atop{i,j\geq1}}(E_i+v\times B_i)\cdot\na_v h_j\notag\\&
		-\sum\limits_{i+j={2k-2}\atop{i,j\geq1}}\frac{v}{2RT_c}  \cdot E_i  h_j
		+\CL [h_{2k-1}]+\CL_d [h_{2k-1}]
		=\sum\limits_{i+j=2k-1\atop{i,j\geq1}}\Ga (h_i, h_j),
	\end{align*}
	\begin{align}
		\pa_th_{2k-1}&+v\cdot\na_xh_{2k-1}-(E+v\times B)\cdot\na_v h_{2k-1}+\frac{v}{2RT_c}  \cdot E  h_{2k-1}\notag\\&
		-\mu^{-\frac{1}{2}}(E_{2k-1}+v\times B_{2k-1})\cdot\na_v\FM-\sum\limits_{i+j={2k-1}\atop{i,j\geq1}}(E_i+v\times B_i)\cdot\na_v h_j
		\notag\\&
		-\sum\limits_{i+j={2k-1}\atop{i,j\geq1}}\frac{v}{2RT_c}  \cdot E_i  h_j
		=\sum\limits_{i+j=2k}\Ga (h_i,h_j).\notag
	\end{align}

To prove \eqref{vmb-Fn-es} and \eqref{vmb-EB-es}, our current objective is to demonstrate that
	\begin{align}
		\sum\limits_{\al_0+|\al|+|\beta|\leq N+4k-2n+2}&\left|\lag v\rag^{\ell-n(1+|\gamma+2s|}e^{q\lag v\rag}\pa_\bet^\al h_n\right|
		\nonumber\\
		&+\sum\limits_{\al_0+|\al|\leq N+2k-2n+2}\left|\pa_t^{\al_0}\pa^\al [E_n,B_n]\right|\leq C(C_0)(1+t)^{n}\label{em-fn-es}
	\end{align}
	holds for $N\geq4$.
	
In fact, \eqref{em-fn-es} can be deduced from the following estimates:
	\begin{align}
		\sum\limits_{\al_0+|\al|+|\beta|\leq N+4k-2n+2}&\left\|\lag v\rag^{\ell-n-(2n-1)|\gamma+2s|/2}e^{q\lag v\rag}\pa_\bet^\al h_n\right\|_{\bf D}\nonumber\\
		&+\sum\limits_{\al_0+|\al|\leq N+2k-2n+2}\left|\pa_t^{\al_0}\pa^\al [E_n,B_n]\right|
		\leq C(C_0)(1+t)^n.\label{em-w-fn-es}
	\end{align}
Then, using standard energy estimates, we can derive from \eqref{em-abc-eq} and \eqref{em-decay} that
	\begin{align}\label{em-ma-es}
		\frac{d}{dt}\sum\limits_{\al_0+|\al|\leq m}&\|\pa_t^{\al_0}\pa^\al[\rho_n,u_n,T_n,E_n,B_n]\|^2\notag\\
		\leq& C\epsilon_1(1+t)^{-p_0}\sum\limits_{\al_0+|\al|\leq m}\|\pa_t^{\al_0}\pa^\al[\rho_n,u_n,T_n,E_n,B_n]\|^2
		\notag\\&+C\sum\limits_{\al_0+|\al|\leq m+1}\|\pa_t^{\al_0}\pa^\al({\bf I-P}_{\mathbf{M}})[f_{n}]\|_{\bf D}\sum\limits_{\al_0+|\al|\leq m}\|\pa_t^{\al_0}\pa^\al[\rho_n,u_n,T_n,E_n,B_n]\|,
	\end{align}
	where $m>0$ is finite.
\begin{remark}\label{fn-hn}
It's important to note that we use $({\bf I-P}_{\mathbf{M}})[f_{n}]$ instead of $({\bf I-P}_{\mathbf{M}})[h_{n}]$ to establish the above bound. This choice arises from an additional spatial derivative loss when dealing with $({\bf I-P}_{\mathbf{M}})[h_{n}]$. As per \eqref{micro-fn}, this additional spatial derivative loss can be controlled by higher-order norms of the previous $n-1$ items, namely $[\rho_i, u_i, T_i, E_i, B_i]$ for $0\leq i\leq n-1$, through an inductive argument. However, this approach is not applicable for $h_{n}$ because the equation for $h_n$ involves an additional term $\mathcal{L}_{d}[({\bf I-P})[h_n]]$, which contains $[\rho_n, u_n, T_n]$. Consequently, it leads to another loss of one more spatial derivative for $[\rho_n, u_n, T_n]$, resulting in a situation where the estimates enter an endless loop.
\end{remark}

	To close our estimate, we now turn to estimate the microscopic part $({\bf I-P})[h_n]$.
	For $n=1$, in view of the first equation for $\varepsilon^0$ in \eqref{expan2}, one has
	\begin{align}\label{hn-co}
		&\left\lag \pa_t^{\al_0}\pa_\beta^\al \mathcal{L} [({\bf I-P})[h_1]],\lag v\rag^{2(\ell-1)-|\gamma+2s|}e^{2q\lag v\rag}\pa_t^{\al_0}\pa_\beta^\al({\bf I-P})[h_1]\right\rag\notag\\
		=&-\left\lag \pa_t^{\al_0}\pa_\beta^\al \mathcal{L}_{d}[({\bf I-P})[h_1]],\lag v\rag^{2(\ell-1)-|\gamma+2s|}e^{2q\lag v\rag}\pa_t^{\al_0}\pa_\beta^\al({\bf I-P})[h_1]\right\rag\notag\\
& -\left\lag \pa_t^{\al_0}\pa_\beta^\al
		\left[\mu^{-\frac{1}{2}}\{\pa_t\FM+v\cdot\na_x\FM\}\right],\lag v\rag^{2(\ell-1)-|\gamma+2s|}e^{2q\lag v\rag}\pa_t^{\al_0}\pa_\beta^\al({\bf I-P})[h_1]\right\rag\\
& +\left\lag\pa_t^{\al_0}\pa_\beta^\al
		\left[\mu^{-\frac{1}{2}}(E+v\times B)\cdot\na_v \FM\right],\lag v\rag^{2(\ell-1)-|\gamma+2s|}e^{2q\lag v\rag}\pa_t^{\al_0}\pa_\beta^\al({\bf I-P})[h_1]\right\rag.\notag
	\end{align}
For the 1st term in the right hand side of \eqref{hn-co}, by utilizing \eqref{wGLd-Nong}, and using the equivalence of norms for ${\bf P}[h_1]$ and
${\bf P}[f_1]$, along with the following fact:
$$\|[\rho_1,u_1,T_1]\|^2_{H^{\alpha_0}_tH^{|\alpha|}_x}\lesssim\|{\bf P}[f_1]\|^2_{H^{\alpha_0}_tH^{|\alpha|}_xH^{|\beta|}_v}\lesssim \|[\rho_1,u_1,T_1]\|^2_{H^{\alpha_0}_tH^{|\alpha|}_x},$$
one has
\begin{align*}
\Big|&\left\lag \pa_t^{\al_0}\pa_\beta^\al \mathcal{L}_{d}[({\bf I-P})[h_1]],\lag v\rag^{2(\ell-1)-|\gamma+2s|}e^{2q\lag v\rag}\pa_t^{\al_0}\pa_\beta^\al({\bf I-P})[h_1]\right\rag\Big|\\
\leq& C\eps_1 \Big(\sum\limits_{\bar{\al}_0\leq\al_0}\sum_{\bar{\alpha}\leq \alpha,|\bar{\beta}|\leq|\beta|}\left\|\lag v\rag^{\ell-1-|\gamma+2s|/2}e^{q\lag v\rag}\pa_t^{\bar{\al}_0}\partial^{\bar{\alpha}}_{\bar{\beta}}({\bf I-P})[h_1]\right\|_{\bf D}^2+\|[\rho_1,u_1,T_1]\|^2_{H^{\alpha_0}_tH^{|\alpha|}_x}\Big).
\end{align*}
Additionally, considering \eqref{em-decay}, we can bound the 2nd and 3rd terms on the R.H.S. of \eqref{hn-co} as follows:
 \begin{align*}
 &\Big| \left\lag \pa_t^{\al_0}\pa_\beta^\al
		\left[\mu^{-\frac{1}{2}}\{\pa_t\FM+v\cdot\na_x\FM\}\right],\lag v\rag^{2(\ell-1)-|\gamma+2s|}e^{2q\lag v\rag}\pa_t^{\al_0}\pa_\beta^\al({\bf I-P})[h_1]\right\rag\Big|\\
&+\Big| \left\lag\pa_t^{\al_0}\pa_\beta^\al
		\left[\mu^{-\frac{1}{2}}(E+v\times B)\cdot\na_v \FM\right],\lag v\rag^{2(\ell-1)}e^{2q\lag v\rag}\pa_t^{\al_0}\pa_\beta^\al({\bf I-P})[h_1]\right\rag\Big|\\
\leq &C\eps_1\left\|\lag v\rag^{\ell-1-|\gamma+2s|/2}e^{q\lag v\rag}\pa_t^{\al_0}\partial^{\alpha}_{\beta}({\bf I-P})[h_1]\right\|_{\bf D}.
 \end{align*}
Consequently, we use \eqref{wLLM-Non} to get for $\beta>0$ that
	\begin{align}
		&\left\|\lag v\rag^{\ell-1-|\gamma+2s|/2}e^{q\lag v\rag}\pa_t^{\al_0}\partial^{\alpha}_{\beta}({\bf I-P})[h_1]\right\|_{\bf D}^2
		\notag\\
		&-\left(\eta+C\epsilon_1\right)\sum_{\bar{\al}_0+|\bar{\alpha}|\leq \al_0+|\alpha|}\sum_{|\bar{\beta}|=|\beta|}\left\|\lag v\rag^{\ell-1-|\gamma+2s|/2}e^{q\lag v\rag}\pa_t^{\bar{\al}_0}\partial^{\bar{\alpha}}_{\bar{\beta}}({\bf I-P})[h_1]\right\|_{\bf D}^2\notag\\
		&-C(\eta)\sum\limits_{\bar{\al}_0\leq\al_0}\sum_{\bar{\alpha}\leq \alpha}\sum_{|\bar{\beta}|<|\beta|}
		\left\|\lag v\rag^{\ell-1-|\gamma+2s|/2}e^{q\lag v\rag}\pa_t^{\bar{\al}_0}\partial^{\bar{\alpha}}_{\bar{\beta}}({\bf I-P})[h_1]\right\|^2_{\bf D}\notag\\
		\leq& C\Big(\eps_1^2+\|[\rho_1,u_1,T_1]\|^2_{H^{\alpha_0}_tH^{|\alpha|}_x}\Big),\notag
	\end{align}
	and for $\beta=0$
	\begin{align}
		&\left\|\lag v\rag^{\ell-1-|\gamma+2s|/2}e^{q\lag v\rag}\pa_t^{\al_0}\pa^\al({\bf I-P})[h_1]\right\|_{\bf D}^2\notag\\&
		-C\epsilon_1{\bf 1}_{\al_0+|\al|>0}\sum\limits_{\bar{\al}_0\leq\al_0,\bar{\alpha}\leq \alpha,|\bar{\al}_0|+|\bar{\alpha}|<|\al_0|+|\alpha|}\left\|\lag v\rag^{\ell-1-|\gamma+2s|/2}e^{q\lag v\rag}\pa_t^{\bar{\al}_0}\partial^{\bar{\alpha}}({\bf I-P})[h_1]\right\|_{\bf D}^2
		\notag\\
		&-C\left\|\lag v\rag^{\ell-1-|\gamma+2s|/2}e^{q\lag v\rag}\pa_t^{\al_0}\pa^{\al}({\bf I-P})[h_1]\right\|_{L^2(\R^3\times B_C(\eta))}\notag\\
		\leq& C\Big(\eps_1^2+\|[\rho_1,u_1,T_1]\|^2_{H^{\alpha_0}_tH^{|\alpha|}_x}\Big),\notag
	\end{align}
	and moreover
	\begin{align*}
		\left\|\pa_t^{\al_0}\pa^\al({\bf I-P})[h_1]\right\|_{\bf D}^2
		\leq C\Big(\eps_1^2+\|[\rho_1,u_1,T_1]\|^2_{H^{\alpha_0}_tH^{|\alpha|}_x}\Big).
	\end{align*}
	To summarize the above estimates, we can conclude that
	\begin{align}\label{em-cof1-s}
		\sum\limits_{\al_0+|\al|+|\beta|\leq m+4k-4}\left\|\lag v\rag^{\ell-1-|\gamma+2s|/2}e^{q\lag v\rag}\pa_t^{\al_0}\pa_\beta^\al({\bf I-P})[h_1]\right\|_{\bf D}^2
		\leq C\eps_1^2.
	\end{align}
	Thus \eqref{em-cof1-s} and \eqref{em-ma-es} with $n=1$ gives
	\begin{align}\label{em-f1-es}
\sum\limits_{\al_0+|\al|+|\beta|\leq m+4k-4}&\left\|\lag v\rag^{\ell-1-|\gamma+2s|/2}e^{q\lag v\rag}\pa_t^{\al_0}\pa_\beta^\al h_1(t)\right\|_{\bf D}+\sum\limits_{\al_0+|\al|\leq m+4k-4}\left|\pa_t^{\al_0}\pa^\al [E_1,B_1]\right|\notag\\
		\leq& C\sum\limits_{\al_0+|\al|\leq m+4k-4}\left\|\pa_t^{\al_0}\pa^\al[\rho_1,u_1,T_1,E_1,B_1](0,x)\right\|+C\eps_1(1+t).
	\end{align}
	
	Now we assume that
	\begin{align*}
		&\sum\limits_{\al_0+|\al|+|\beta|\leq m+4k-2n-2}\left\|\lag v\rag^{\ell-n-(2n-1)|\gamma+2s|/2}e^{q\lag v\rag}\pa_t^{\al_0}\pa_\beta^\al h_n(t)\right\|_{\bf D}\\
&+\sum\limits_{\al_0+|\al|\leq m+4k-2n-2}\left|\pa_t^{\al_0}\pa^\al [E_n,B_n]\right|\notag\\
		\leq& C\sum\limits_{\al_0+|\al|\leq m+4k-2n-2}\left\|\pa_t^{\al_0}\pa^\al[\rho_n,u_n,T_n,E_n,B_n](0,x)\right\|+C(C_0)(1+t)^n\notag
	\end{align*}
	holds for $1\leq n\leq2k-2 $, then by repeating the argument used to deduce \eqref{em-cof1-s}, one has
	\begin{align}\label{em-cohn-s}
		\sum\limits_{\al_0+|\al|+|\beta|\leq m+4k-2n-3}\left\|\lag v\rag^{\ell-n-|\gamma+2s|/2}e^{q\lag v\rag}\pa_t^{\al_0}\pa_\beta^\al({\bf I-P})[h_{n+1}](t)\right\|_{\bf D}
		\leq C(C_0)(1+t)^n.
	\end{align}
On the other hand, as discussed in Remark \ref{fn-hn}, by a similar argument as for deriving \eqref{em-cohn-s}, it can be demonstrated that $({\bf I-P})[f_{n}]$ is bounded in the way as $({\bf I-P})[h_{n}]$.

Thus \eqref{em-cohn-s} together with \eqref{em-ma-es} yields
	\begin{align}\label{em-f1-es}
		\sum\limits_{\al_0+|\al|+|\beta|\leq m+4k-2n-4}&\left\|\lag v\rag^{\ell-n-(2n-1)|\gamma+2s|/2}e^{q\lag v\rag}\pa_t^{\al_0}\pa_\beta^\al h_{n+1}(t)\right\|_{\bf D}\nonumber\\
		&+\sum\limits_{\al_0+|\al|\leq m+4k-2n-4}\left|\pa_t^{\al_0}\pa^\al [E_{n+1},B_{n+1}]\right|\\
		\leq& C(C_0)(1+t)^{n+1}.\nonumber
	\end{align}
	
	Finally, by setting $m=N+4$, the estimate \eqref{em-f1-es} for $0\leq n\leq 2k-2$ implies \eqref{em-w-fn-es}.

\end{proof}

\end{document}